\documentclass[12pt]{article}
\usepackage{amsmath,amssymb,amsfonts,amsthm}
\usepackage[all]{xy}

\usepackage[inner=2cm,outer=2cm]{geometry} 
\newcounter{item}[section]
\newcounter{kirshr}
\newcounter{kirsha}
\newcounter{kirshb}
\newenvironment{enumroman}{\setcounter{kirshr}{1}
\begin{list}{(\roman{kirshr})}{\usecounter{kirshr}} }{\end{list}}
\newenvironment{enumarab}{\setcounter{kirshb}{1}
\begin{list}{(\arabic{kirshb})}{\usecounter{kirshb}} }{\end{list}}
\newtheorem{theorem}{Theorem}[section]

\newtheorem{lemma}[theorem]{Lemma}
\newtheorem{corollary}[theorem]{Corollary}
\newenvironment{demo}[1]{\noindent{\bf #1.}\upshape\mdseries}
{\nopagebreak{\hfill\rule{2mm}{2mm}\nopagebreak}\par\normalfont}
\theoremstyle{definition}
\newtheorem{remark}[theorem]{Remark}

\newtheorem{example}[theorem]{Example}
\newtheorem{definition}[theorem]{Definition}

\def\C{{\mathfrak{C}}}
\def\Fm{{\mathfrak{Fm}}}

\def\Nr{{\mathfrak{Nr}}}
\def\Fr{{\mathfrak{Fr}}}
\def\Sg{{\mathfrak{Sg}}}
\def\Fm{{\mathfrak{Fm}}}
\def\Rd{{\mathfrak{Rd}}}
\def\Ig{{\mathfrak{Ig}}}
\def\CA{{\bf CA}}
\def\RCA{{\bf RCA}}
\def\K{{\bf K}}
\def\L{{\bf L}}

\def\QA{{\bf QA}}
\def\QEA{{\bf QEA}}

\def\(R)RA{{\bf (R)RA}}
\def\RA{{\bf RA}}

\def\Dc{{\bf Dc}}
\def\R{\mathbb{R}}
\def\Dc{{\bf Dc}}

\def\QA{{\bf QA}}
\def\QEA{{\bf QEA}}
\def\PA{{\bf PA}}

\def\R{\mathbb{R}}
\def\N{\mathbb{N}}

\def\R{\mathbb{R}}
\def\Dc{{\bf Dc}}

\def\Sc{{\bf Sc}}

\def\Cs{{\bf Cs}}

 \def\CA{{\sf CA}}

\def\M{{\mathfrak{M}}}

\def\RA{{\bf RA}}

\def\At{{\mathfrak{At}}}
\def\GG{\cal {G}}

\def\G{{\mathfrak{G}}}

\def\K{{\bf K}}
\def\QA{{\bf QA}}
\def\QEA{{\bf QEA}}

\def\tp{{\sf tp}}

\def\cyl#1{{\sf c}_{#1}}
\def\diag#1#2{{\sf d}_{#1#2}}

\def\s{{\sf s}}

\def\pa{$\forall$}
\def\pe{$\exists$}

\def\ef{Ehren\-feucht--Fra\"\i ss\'e}

\def\nodes{{\sf nodes}}

\def\restr #1{{\restriction_{#1}}}

\def\A{{\mathfrak{A}}}
\def\B{{\mathfrak{B}}}
\def\C{{\mathfrak{C}}}
\def\D{{\mathfrak{D}}}
\def\P{{\mathfrak{P}}}
\def\Fm{{\mathfrak{Fm}}}
\def\Ra{{\mathfrak{Ra}}}
\def\Nr{{\mathfrak{Nr}}}
\def\F{{\mathfrak{F}}}
\def\CA{{\bf CA}}
\def\RCA{{\bf RCA}}

\def\set#1{ \{#1\}}

\def\Ca{{\mathfrak Ca}}
\def\b#1{{\bar{ #1}}}
\def\pe{$\exists$}
\def\pa{$\forall$}
\def\Cm{{\mathfrak Cm}}
\def\Sg{{\mathfrak Sg}}

\def\At{{\sf At}}
\def\Id{{\sf Id}}

\def\rng{{\sf rng}}
\def\dom{{\sf dom}}

\def\w{{\sf w}}
\def\g{{\sf g}}
\def\y{{\sf y}}
\def\r{{\sf r}}
\def\cyl#1{{\sf c}_{#1}}

\def\diag#1#2{{\sf d}_{#1#2}}
\def\V{{\sf V}}
\def\de{Dedekind-MacNeille}

\def\ws{winning strategy}
\def\ef{Ehren\-feucht--Fra\"\i ss\'e}

\def\Rl{\mathfrak{Rl}}

\def\Mat{{\sf Mat}}

\def\y{{\sf y}}
\def\g{{\sf g}}
\def\r{{\sf r}}
\def\w{{\sf w}}

\def\Dc{{\sf Dc}}
\def\Sc{{\sf Sc}}

\def\CA{{\sf CA}}

\def\R{{\sf R}}

\def\L{{\mathfrak{L}}}

\def\Z{{\mathbb{Z}}}
\def\K{{\sf K}}
\def\RA{{\sf RA}}
\def\RCA{{\sf RCA}}
\def\PEA{{\sf PEA}}

\def\Df{{\sf Df}}
\def\Uf{{\sf Uf}}
\def\Rd{{\mathfrak{Rd}}}

\def\la#1{\langle#1\rangle}

\def\nodes{{\sf nodes}}

\def\QEA{{\sf QEA}}
\def\nodes{{\sf nodes}}

\def\de{Dedekind--MacNeille}

\def\s{{\sf s}}

\def\Tm{{\mathfrak{Tm}}}
\def\QA{{\sf QA}}
\def\M{\mathfrak{M}}

\def\E{\mathfrak{E}}
\def\PA{{\sf PA}}
\def\Bl{{\mathfrak{Bl}}}

\def\Lf{{\sf {Lf}}}
\def\Cs{{\sf Cs}}
\def\M{{\mathfrak M}}
\def\Mo{{\sf M}}
\title{Atom--canonicity, complete representations for cylindric--like algebras, and omitting types in clique guarded 
fragments of $L_{\omega,\omega}$ }

\author{Tarek Sayed Ahmed}

\begin{document}
\maketitle

\begin{abstract}
Fix a finite ordinal $n>2$. We show that there exists an atomic, simple and countable representable $\CA_n$, such that its \de\ completion is outside $\bold S\Nr_n\CA_{n+3}$.
Hence, for any finite $k\geq 3$, the variety
$\bold S\Nr_n\CA_{n+k}$ is not atom--canonical, so that the variety
of $\CA_n$s  having $n+k$--flat representations is not atom--canonical, too. We show, for finite $k\geq 3$,  that $\bold S_c\Nr_n\CA_{n+k}$ is not elementary, hence
the class of $\CA_n$s having complete $n+3$--smooth representations is not elementary.
We obtain analogous results by replacing flat and smooth, respectively, by (the weaker notion of) square; this give a
stronger result in both cases and here we can allow $k$ to be infinite.  Our results are proved using rainbow constructions for $\CA$s.
We lift the negative result on atom--canonicity to the transfinite.
We  also show that for any
ordinal $\alpha\geq \omega$, for any finite $k\geq 1,$
and for any $r\in \omega$, there exists an atomic algebra
$\A_r\in  \bold S\Nr_{\alpha}\CA_{\alpha+k}\sim \bold S\Nr_n\CA_{\alpha+k+1},$ such that
$\Pi_{r/U} \A_r\in \sf RCA_{\alpha},$ where $U$ is any non--principal ultrafilter  on $\omega$. 
Reaping the harvest of our algebraic results we investigate a plethora of omitting types theorems for variants of first logic 
including its finite variable fragments and its packed fragment.  
We show that there exists a countable, consistent, atomic 
and complete theory $T$ using $n$ variables, such that the 
non--principal type consisting of co--atoms, cannot be omitted in $n+3$--square, a {\it priori}, $n+3$--flat models. This implies that the omitting types theorem fails 
for the clique--guarded fragment of first order logic, an alternative formulation of the packed fragment.
In contrast, we show that if $T$ is any theory using $n$ variables such that $|T|\leq \lambda$, $\lambda$ a regular cardinal, admits 
elimination of quantifiers,  then $< 2^{\lambda}$ many {\it complete} non--principal types can be omitted in ordinary ($\omega$--square) 
models having cardinality $\leq |\lambda|$.  We provide an example for $|T|=2^{\aleph_0}$, showing that the condition of 
completeness cannot be dispensed with. 
We show that the omitting types theorem fails for finite first order expansions of finite variable fragments of first order logic, 
as long as the number of variables available  are $>2.$ 

\end{abstract}
\section{Introduction}

We follow the notation of \cite{1} which is in conformity with the notation adopted in the monograph \cite{HMT2}. 
In particular, $\CA_n$ denotes the class of cylindric algebras of dimension $n$, $\sf RCA_n$ denotes the class of
representable $\CA_n$s and for $m>n$, $\Nr_n\CA_m$ denotes the class of all $n$--neat reducts of $\CA_m$s.

Three cornerstones in the development of the theory of cylindric algebras due to Tarski, Henkin and Monk, respectively - the last
two involving the notion of neat reducts - are
the following. Tarski proved that every locally finite infinite dimensional cylindric algebra of infinite dimension is representable.
This is equivalent (in $\sf ZFC$) to G\"odels completeness theorem.

Henkin proved that for any ordinal $n$, $\bold S\Nr_n\CA_{n+\omega}= \sf RCA_{n}$, where $\bold S$ denotes the operation of forming subalgebras. 
When $n$ is infinite, this is a strong generalization of Godel's completeness theorem addressing 
extensions of first order logic allowing infinitary predicates.

Monk proved that for any ordinal $n>2$ and $k\in \omega$,  $\bold S\Nr_n\CA_{n+k}\neq \RCA_n$. In particular,
for each finite $n>2$, and $k\in \omega$, there is an algebra $\A_k\in \bold S\Nr_n\CA_{n+k}$ that is not representable.
Any non--trivial ultraproduct of such algebras will be in $\sf RCA_n$.
Hence, the variety $\sf RCA_n$ ($2<n<\omega$) is  not finitely axiomatizable. This implies that finite variable fragments 
of first order logic are severely incomplete, as long as the number of variables available are $\geq 3$. For any finite $m>2$, there is a valid $3$--variable 
first order formula that needs more than $m$ variables to be proved, in any fairly standard Hilbert--style axiomatization.  
Using an ingenious lifting argument \cite[Theorem 3.2.87]{HHbook2}, 
which we henceforth refer to as {\it  Monk's trick}  (to be used several times below),
Monk lifted this non--finite axiomatizability result to the transfinite.\\

{\bf Some queries on neat reducts:} 
Now let $\K$ be a class consisting of cylindric--like algebras (like for instance cylindric algebras or quasi--polyadic algebras), so that for every pair of ordinals
$n<m$,  $\K_{n}$ is a variety consisting of $n$--dimensional algebras and for $\B\in \K_m,$
$\Nr_{n}\B$ and $\Rd_{n}\B$ are defined, so that the former, the {\it $n$--neat reduct} of $\B$,
is a subalgebra of the latter, the {\it $n$--reduct} of $\B$. If $\A\subseteq \Nr_n\B$ and $\B\in \K_m$, then we say that $\A$ {\it neatly embeds} in $\B$,
and that $\B$ is an {\it $m$--dilation}, or simply a dilation
of $\A$, when $m$ is clear from context.

When reflecting about neat reducts, 
several possible questions cannot help but spring to mind, each with its own range of intuition.

\begin{enumerate}
\item Is it true that every algebra neatly embeds into another algebra having only one extra dimension?
having $k$ extra dimension $k>1$ ($k$ possibly infinite) ?
And could it, at all be possible, that an $n$--dimensional algebra
($n$ an ordinal) neatly embeds into another $n+k$--dimensional one, but does not neatly
embed into any $n+k+1$--dimensional algebra for some $k\in \omega$?

\item Assume that $\A\in \K_n$ neatly embeds into $\B$ having dimension $>n$, and assume further that $\A$ (as a set) generates $\B$, is then $\B$ unique up to isomorphisms
that fix $\A$ pointwise? does $\A$ exhaust the set of $n$--dimensional elements of $\B,$ namely, $\Nr_n\B$, the $n$--neat
reduct of $\B$? In this context $\B$ is called a {\it minimal} dilation of $\A$, because $\A$ generates $\B$.

\item  If $2<n<m\cap \omega$, and $\A\in \K_n$ neatly embeds into $\B\in \K_m$; does this induce some form of `concrete representation' of $\A$ using the spare dimensions of the
$m$--dilation
$\B$? We know it does when $\A\in \CA_n$ and $m=\omega$, that is when $\B\in \CA_{\omega}$. 
This gives a classical representation by Henkin's aforementioned {\it neat embedding theorem}. But when $n<m<\omega$,
does this perhaps give some kind of an {\it $m$ localized form of representation} that is an `$m$ approximation'
to a classical representation?  Could such relativized representations, if any,
diffuse negative undesirable properties of the classical ones, for example 
is it decidable to tell whether a finite algebra possesses such a relativized representation?

\item Now that we have, for every pair of ordinals $n<m$, the class
$\Nr_n\K_m$ in front of us, there is a pressing need  to classify
it.   For example, is it a variety, if not, is it
elementary, if not, is it perhaps pseudo--elementary and if so, is its elementary
theory recursively enumerable, finitely axiomatizable?

\item  Let $n<m$ be ordinals. Now we address the varieties $\bold S\Nr_n\K_m$ and the classes $\bold S_c\Nr_n\K_m$, where $\bold S_c$ denotes the operation 
of forming {\it complete} subalgebras:
\begin{itemize}
\item What do we know about the varieties $\bold S\Nr_n\K_m$?  Are they finitely axiomatizable? are they canonical, atom--canonical? are they closed
under \de\ completions, hence, being conjugated, they are Sahlqvist axiomatizable?
These are important questions addressing the modal formalism of the corresponding 
logic; that is, the multi $n$--dimensional modal logic, whose modal algebras are
in $\bold S\Nr_n\K_m$. What are the possible `natural' semantics for such a logic?

\item What do we know about the classes $\bold S_c\Nr_n\K_m$?
Are they elementary classes? A limiting case, when $n<\omega$, is the class of completely representable $\K_n$s, which is not elementary \cite{HH}.

Can we characterize (complete)
{\it relativized} representations, if any, for $\bold S\Nr_n\K_m  (\bold S_c\Nr_n\K_m$) when $n<m<\omega$, by
games similar to usual atomic games \cite[Definition 3.3.2]{HHbook2}, if we perhaps limit the number of pebbles and /or rounds used in the play, to finitely many?
\end{itemize}

\item Related to the two previous items: What are the impact of such results, if any,  on meta--logical properties of finite variable
fragments of first order logic and its modal formalism,
like, say, definability issues, completeness and omitting types?

\end{enumerate}

{\bf  Some answers:}  Henceforth in the introduction we focus on $\CA$s.  Once vexing outstanding problems in (algebraic) logic,
some of such questions for finite dimensions
were settled by the turn of the millennium, and others a few years later,
after thorough  dedicated trials, and dozens of publications providing
partial answers.\\

Now let us review the present status of such questions; which are settled, which are not, and also those questions 
whose naswers lend themselves to further natural generalizations that enriches the hitherto obtained answers 
providing new insights and opening new avenues.

Concerning the second question, it is known that minimal dilations may not be unique up to isomorphisms that fix
$\A$ pointwise. For an intense discussion of this rather anti--intuitive phenomena
the reader is referred to \cite{Sayed} where a somewhat convoluted proof is given depending on a strong result of Pigozzi's, namely that $\sf RCA_{\omega}$ 
does not have the amalgmation property. Below we give a more straightforward proof, exhibiting explicity such an $\A$. Infinitely many examples are 
given.

It is now known that for finite $n\geq 3$ and $k\geq 1$, the
variety $\bold S\Nr_n\CA_{n+ k+1}$ is not even finitely axiomatizable over $\bold S\Nr_n\CA_{n+k}$;
this was proved by Hirsch, Hodkinson and Maddux, answering problem 2.12 in \cite{HMT2}.

This refines considerably Monk's seminal result. It implies that for any finite $m>2$, there is a $3$--variable formula (using only one binary relation, a significant addition) 
that cannot be proved using $m$ variables, but it {\it can} be proved using $m+1$ variables. 
Monk's result does not tell us more than the obvious fact - baring in mind 
that  proofs are finite strings of formulas and formulas contain only finitely many variables - that this $3$ variable formula (that cannot be proved using $m$ variables), 
can be proved using $m+k$ variables, for some finite $k\geq 1$,  
that we do not have any control whatsoever on. In Monk's (algebraic) proof  this $k$  depends on a large uncontrollable Ramsey number.

This result was lifted to infinite dimensions, using Monk's trick,
by Robin Hirsch and the present author
\cite{t}, addressing other cylindric--like algebras, as well, like Pinter's substitution algebras and Halmos' quasi--polyadic algebras.
This also answers the first part of item (5). The rest of the questions in this item are answered negatively by Hodkinson for
finite $n>2$ and $m=\omega$ \cite{Hodkinson}. Below we will sharpen Hodkinson's result by showing that:\\

(*) The variety $\bold S\Nr_n\CA_{n+k}$ is not atom--canonical, for any ordinal $n>2$, and for any $k\geq 3$.\\

This will be done by a constructing a representable countable 
atomic $\A\in \CA_n$, such that $\Cm\At\A\notin \bold S\Nr_n\CA_{n+3}$, providing an answer to the first half of
item (6). For the second part, non--elementarity is known for $m=\omega$ \cite{strong}, from which we can easily deduce that there must be a finite $l<n$, such
that $\bold S\Nr_n\CA_m$ is not elementary for all $m\geq l$. To determine the least such $l$,  is a challenging open problem.
For this $l$ it must be the case that ${\sf Str}\bold S\Nr_n\CA_l=\{\F\in \At\CA_n: \Cm\At\F\in \bold S\Nr_n\CA_l\}$ is propery contained in $\At\bold S\Nr_n\CA_l$. 
We prove  this necessary (but not sufficient) 
condition for $l\geq n+3$. In this connection, we also re-prove a result of Bulian and Hirsch \cite{bh} extending 
the result in \cite{strong} to many cylindric--like algebras, including diagonal free $\CA$s, 
polyadic equality algebras, and, in fact, for any class in between these two classes, like for instance the class of Pinter's substitution algebras.

We will see that the semantics appropriate for the  `modal algebras' in $\bold S\Nr_n\CA_m$  for $2<n<m<\omega$, are 
what we call following Hirsch and Hodkinson for analogous results on relation algebras,  {\it $n$-clique--guarded semantics}, or {\it $n$--Gaifman hypergraph semantics}.
Such semantics are like classical semantics but only
locally on `$m$ squares'.  Roughly, if we zoom in adequately by a `movable window', there will come  a point, determined by $m$,  where we
mistake this severely relativized representation for a genuine one.

Concerning the fourth question, it is now also known that for $1<n<m$,  the class $\Nr_n\CA_m$ is not first order definable, least a variety, 
hence it is not pseudo--universal, but for $n<\omega$, it is pseudo--elementary, and its elementary theory is recursively enumerable. 
All these results are due to the present author,  the last two are proved below.  The first  refines N\'emeti's answer to
problem 2.11 in \cite{HMT2} and answers 4.4 in \cite{HMT2}.  Now fix $n$ with $2<n<\omega$. A natural variant of the class 
$\bold K=\{\A\in \CA_n: \Cm\At\A\in \RCA_n\}$, proved to be non--elementary in \cite{strong} (this is equivalent to non-elementarity of 
${\sf Str}\RCA_n$ is not elementary), 
is the class $\bold N=\{\A\in \CA_n: \text{ $\A$ is atomic and } \Cm\At\A\in \Nr_n\CA_{\omega}\}$. 

This class is properly contained in $\bold K$; 
it will be shown that there is a countable atomic algebra $\C\in \CA_n$ such that $\Cm\At\C\in \RCA_n\sim \Nr_n\CA_{\omega}$; 
and there is a countable atomic $\B\notin \Nr_n\CA_{\omega},$ such that $\Cm\At\B\in \Nr_n\CA_{\omega}$ 
so that  $\B$ in $\bold N\sim \Nr_n\CA_{\omega}$. 

We also show that the classes $\bold N$, $\bold S_c\Nr_n\CA_{\omega}$, and the class of completely representable $\CA_n$s coincide on atomic algebras having countably many atoms.
Finally, we show that, like both of $\bold K$ and $\Nr_n\CA_{\omega}$: \\

(**)   The class $\bold N$ is not closed under ultraroots, hence is not 
elementary.\\

Concerning the fifth,
we show that:\\

(***) For $2<n<\omega$ and finite $k\geq 3$, the class
$\bold S_c\Nr_n\CA_{n+k}$, and the class of completely representable $\CA_n$s are also not elementary,
substantially generalizing the result in \cite{HH}. \\

Games devised to characterize the {\it complete} 
$n$-clique--guarded semantics
for such classes will be devised below.

The remaining items, namely, (3) and (6) will be for their part addressed as the paper unfolds.
First, the answer is known for the limiting cases \cite{AU, HH}. In the former reference, it is proved that 
$2<n<\omega$, it is undecidable to tell whether a finite $\CA_n$ has a (classical) representation.

One of our results connects these two items.
We will show that, for finite $n>2$,  the Orey--Henkin omitting types theorem fails for first order logic restricted to the first $n$ variables $(L_n)$, even if allow severely relativized models.
There is a countable, atomic, $L_n$ theory, such that the single non--principal type consisting of co--atoms cannot be omitted in
$n+k$--smooth models; $n+k$--smoothness provides semantics for $\bold S\Nr_n\CA_{n+k}$ for any $k\geq 3$.
This result will follow from (*) and (**):  Both imply that there are countable and atomic $\sf RCA_n$s that have no complete $n+3$--flat representations.
In the case of (*) the constructed algebra is simple, giving that $T$ can further be chosen to be a complete theory. 
In fact, this non--principal type of co--atoms cannot be omitted, even if we allow $n+k$--square models $(k\geq 3)$. 
This readily implies that the omitting types theorem fails for 
the clique guarded fragment of first order logic, an alternative formulation of the packed fragment.\\

Both  (*), (**) (***) are proved using games  to be described shortly, 
played on (distinct) rainbow cylindric--like atom structures.
The class of $m$--square models for finite $m>n$ is  a strictly and substantially
larger class than the class of $m$--flat ones. In fact,  $m$--square representations provide semantics for an abstract variety $\sf  V=\bold S\Nr_n\sf \bold K_{m}$,
where commutativity of cylindrifiers may
fail in the $m$--dilations (in $\bold K_m$) of $\A\in \sf V$.  In case of polyadic equality algebras, $\bold K_m$ is
not so abstract, it is the variety $\sf G_m$
studied by Ferenczi \cite{Fer}. The variety of modal algebras $\sf G_m$ manifest nice modal behaviour, e.g. it is finitely axiomatizable, and it
has the finite algebra finite base property.

For finite dimensions, squareness behaves better as far as decidability issues are concerned. For example,
it is decidable to tell for $2<n<m<\omega$,  whether an $n$--dimensional finite algebra has an $m$--square representation in polynomial time, cf. \cite[Corollary, 12.32]{HHbook}, 
but this is not true for $n$--flatness when $n=3$ and $m\geq n+3$.
But in any case this discrepancy between squareness and flatness disappears at the limit.
Classical models are $\omega$--square. Algebraically, algebras having 
$\omega$--square representations are representable, though this might fail for {\it complete} $\omega$--square representations. 
For  finite $n>2$,  uncountable (representable) algebras of dimension $n$ having $\omega$--square representations, without 
having (classical) complete
representations, can (and will be) be constructed. This, however, cannot happen for atomic agebras 
having countably many atoms. As far as the omitting types theorem is concerned
for $L_n$, $2<n<\omega$, there is no discrepancy between $m$-- squareness and $m$--flatness, when $m\geq n+3$.
It fails dramatically for both. (In principal, it could hold for `squareness' but not 
for  `flatness').\\

{\bf Games:} Now fix the dimension $n$ to be finite $>2$.
Let $\A$ be an atomic $\CA_n$. An {\it atomic network} over $\A$ is a map $N:{}^n\Delta\to \At\A$, where $\Delta$ is a finite
non--empty set of nodes, denoted by $\nodes(N)$. Being basically a finite approximation to a representation,
$N$ has to satisfy certain local consistency properties:

If $x\in {}^n\nodes(N)$, and $i<j<n$, then $N(x)\leq {\sf d}_{ij}$ iff $x_i=x_j$.
If $x, y\in {}^n\nodes(N)$, $i<n$ and $x\equiv_i y$, then  $N(x)\leq {\sf c}_iN(y)$.
An atomic game can be played on atomic networks of an atomic algebra $\A$. It has $\omega$ rounds.
Suppose that we are at round $t>0$. Then \pa\ picks a previously played network $N_t$, $i<n,$ $a\in A$, $x\in {}^n\nodes(N_t)$, such that
$N_t(x)\leq {\sf c}_ia$. For her response, \pe\ has to deliver a network $M$
such that $M\equiv _i N$, and there is $y\in {}^n\nodes(M)$
that satisfies $y\equiv _i x$, and $M(y)=a$.
The $k$--rounded game is denoted by $G_k$ \cite[Definition 3.3.2]{HHbook2}.  A \ws\ for \pe\ in the $k$--rounded game can be coded in a first order sentence $\rho_k$ called $k$th Lyndon condition.
If \pe\ has a \ws\ in all $k$--rounded games, so that $\A\models \rho_k$, then $\A$ is representable, in fact it
is elementary equivalent to an algebra that is completely representable; on the atom structure of this
algebra \pe\ can win the $\omega$--rounded game.

In response to questions (6) and (7), we devise `truncated versions' $F^m$, $G^m_{\omega}$
of the above games. These games have $\omega$ rounds, but the number of nodes used during the play is limited to $m$ where $2<n<m$.
In $F^m$, \pa\ has the bonus to reuse the $m$ pebbles in play. When $m\geq \omega$, these games reduce to the usual
$\omega$--rounded atomic game described above. Such games test the semantical notion of the existence of {\it complete $m$--relativized} representations and 
the syntactical one of the existence of $m$--dilations for an algebra $\A\in \CA_n$, for which $\A$ 
{\it completely} embeds into (in the sense that suprema are preserved). 
A completeness theorem (to be proved below) tells us that these two notions are equivalent:  For $2<n<m<\omega$, 
$\A\in \CA_n$ has a complete $m$--relativized representation $\iff$ it has an $m$--dilation $\D$, so that 
$\A\subseteq \Nr_n\D$ and furthermore, for all $X\subseteq \A$, $\sum^{\D}X=1\implies \sum^{\B}X=1$.

We will show that if \pa\ has a \ws\ in $F^m$ for $m<\omega$,  played on atomic networks of an atomic algebra $\A\in \CA_n,$
then $\A$ does not have a complete $m$--flat representation, and if he wins
$G_{\omega}^n$, $\A$ does not have a complete $n$--square one. If $\A$ is finite, then we can delete complete.\\

We will use these games to prove three of our main results. The statements (*) and (***) are proved in 
the second and third items of theorem \ref{main}, and (**) is proved in theorem \ref{squareflat}.\\

For (*): A finite rainbow algebra $\D\in \CA_n$, on which \pe\ can win $G^{n+3}_{\omega}$ on its atom structure
(in finitely many rounds)  so that $\D\notin \bold S\Nr_n\CA_{n+3}$, least representable, will be embedded in the \de\ completion of a representable atomic (infinite) countable
$\A\in \CA_n$. From this, we conclude that $\Cm\At\A\notin \bold S\Nr_n\CA_{n+3},$ 
because $\D\subseteq \Cm\At\A$, and so $\bold S\Nr_n\CA_{n+k}$ 
is not atom--canonical for all $k\geq 3$.  Indeed, $\A\in \bigcap_{k\geq 3}\bold S\Nr_n\CA_{n+k}$, 
but $\Cm\At\A\notin \bold S\Nr_n\CA_{n+k}$ for any  $k\geq 3$. So although $\A$ has an ordinary $(\omega$--square) representation, 
its \de\ completion $\Cm\At\A$ does not have even an $n+3$--square one.\\

For (**) and (***):  We start by (***). We construct, for any finite $n>2$,  an atomic  rainbow algebra $\C\in \sf RCA_n$
with countably many atoms, on which \pa\ can win the $\omega$--rounded game $F^{n+3}$ (in the graph rainbow game \cite{HH} using and re-using $n+3$ nodes), 
and \pe\
can win $G_k$ for all finite $k$ of the same game. It will follow that $\C$ does not have a complete
$n+3$--flat  representation. Using ultrapowers followed by an elementary chain argument, we get that $\C$ is elementary
equivalent to a countable completely representable algebra $\B$, so that $\B$ has  a complete $k>n$
flat representation for all $k$, implying that the class of algebras having complete
$\geq n+k$--flat  representations, for any $k\geq 3$, is not elementary. 

Furthermore, we show that  $\Cm\At\B\in \Nr_n\CA_{\omega},$ so that $\B\in \bold N$.
Now we have $\C\equiv \B$ and $\Cm\At\C=\C$. Therefore $\Cm\At\C\notin \Nr_n\CA_{\omega}$, for else (having countably many atoms),
$\C$ will be completely representable, 
{\it a fortiori,}  it will have  a complete $n+3$--flat representation.  Wrapping up we get: 
$\C\equiv \B$, $\C\notin \bold N$ and $\B\in \bold N$, 
from which it readily follows that 
that $\bold N$ is not elementary, too.\\

{\bf Notions of representability via neat embeddings:}  We continue to fix finite $n>2$. 
The chapter \cite{HHbook2} is devoted to studying the following inclusions between various types of atom structures:
$${\sf CRCA}_n\subseteq {\sf LCA}_n\subseteq {\sf SRCA}_n\subseteq {\sf WRCA}_n.$$
The first is the class of completely representable atom structures, the second is the class of atom structures satisfying the Lyndon conditions, the 
third is the class of strongly representable atom structures, and the last is the class of weakly representable atom structures, 
all of dimension $n$. It  is shown in \cite{HHbook2} that  all inclusions are proper.

Now one can lift such notions from working on atom structures (the frame level) to working 
on the algebra level restricting his attension to atomic ones. 
The class of atomic $\CA_n$s
that satisfies the Lyndon conditions is an elementary class (by definition) 
that is not finitely axiomatizable. It is contained properly in the class of representable algebras
and it contains properly
the class of completely representable algebras;
in fact, it is the elementary closure of the last class.
If $\A\in \CA_n$ satisfies the Lyndon conditions, then so does its \de\ completion. In particular, $\Cm\At\A$  is representable, too.
So call an atomic algebra {\it strongly representable} if $\Cm\At\A$
is representable. It turns out that for finite dimension $>2$, not all 
atomic representable algebras are strongly representable; in fact this last class is not even elementary \cite{HHbook2}.

Call an atomic $\CA_n$ {\it weakly representable}, if it is just representable.
There are algebras that are weakly but not strongly representable, because the former class is obviously elementary.
The class of completely representable algebra coincide with $\bold S_c\Nr_n\CA_{\omega}$ on atomic algebra having countably
many atoms \cite{Sayedneat, Sayed}.  

By definition, the class of weakly representable algebras
coincides with the class $\bold S\Nr_n\CA_{\omega}$ on atomic algebras. However, below we will construct 
for each $2<n\leq \omega$, an atomic (representable) algebra in $\Nr_n\CA_{\omega}$ (necessarily having uncountably many atoms), that is 
not completely representable. But when $n$ is finite, any such algebra is {\it strongly} representable; in fact, it satisfies the Lyndon conditions. 
For finite dimensions $>2$, there are strongly representable 
algebras that do not satisfy the Lyndon conditions, because the latter class is elementary, 
and it is contained in the non--elementary former class 
\cite[Corollary 3.7.2]{HHbook2}.

Such notions,  originally formulated for atom structures of cylindric algebras, are studied extensively in \cite{HHbook2}. It is tempting therefore to investigate
such  various notions of  representability, lifted to the algebra level
in connection to neat embedding properties, baring in mind
that Henkin's neat embedding theorem characterizes the class of {\it all} representable algebras.
In this paper, we initiate this task, which is likely to be rewarding, 
but by no means do we end it. \\

{\bf To conclude:} We know that Hilbert style axiomatizations for $n$--variable fragments of first order logic, using only $m$ variables, are severely incomplete wih respect to
square Tarskian semantics, but  $m$--relativized models produce a perfect match between (this restricted) syntax and semantics.
This match, however, no longer holds, when we require
that such locally relativized representations preserve certain infinitary meets.

We learn from the results in this paper, that for $2<n<\omega$, 
negative properties that hold for $\sf RCA_n$, like non--finite axiomatizability, non-atom--canonicity, 
and non--elementarity of the class
of completely representable $\CA_n$s, persist to hold for $\CA_n$s 
having $m$--dilations, when $m$ is finite $\geq n+3$, 
equipped with $m$--relativized (complete) representations, and from our seemingly pure algebraic results, 
we infer that the omitting types theorem fails for (an alternative formulation of) 
the packed fragment of first order logic.

Very little has been said about the infinite dimensional case. 
Results involving notions like atom--canonicity, \de\ completions, complete representations 
for the infinite dimensional case, are extremely rare in algebraic 
\cite[Problem 3.8.3]{HHbook2}; in fact, almost non--existent.  
In this paper, we also pursue the adventurous endeavor 
of saying a little bit more, by stepping 
in the territory of the transfinite.  
As  sample, we show 
that $\RCA_{\omega}$ cannot be axiomatized by Sahqlvist equations.

Due to the length of the paper, it is divided into two parts. Roughly in the present first part we do the algebra and in the second part we do the logic 
using `bridge theorems' from results in the first part 
to results in the second part. In the first part all  most of our results address {\it both} finite and 
infinite dimensional algebras. 
In the second part we deal only with finite dimensional 
algebras.

\section{Preliminaries:}

{\bf Basic notation:} For a set $X$,  $\wp(X)$ denotes the set of all subsets of $X$, i.e. the powerset of $X$.
Ordinals will be identified with the set of smaller ordinals.
In particular,  for finite $n$, $n=\{0,\ldots, n-1\}$.
$\omega$ denotes the least infinite
ordinal which is the set of all finite ordinals.
For two given sets $A$ and $B$, ${}^AB$ denotes the set of functions from $A$ to $B$,
and $A\sim B=\{x\in A: x\notin B\}.$
If $f\in {}^AB$ and $X\subseteq A$ then $f\upharpoonright X$
denotes the restriction of $f$
to $X$. We denote by $\dom f$ and $\rng f$ the domain and range of a given function
$f$, respectively.  If $f, g$ are functions such that $\dom g=\dom f=n$, $n$ is an ordinal, 
and $i\in n$, then we write $f\equiv_i g$ if $f$ and $g$ agree off of  $i$, that is, 
$f(j)=g(j)$ for all $j\in n\sim \{i\}$.

We define composition so that the righthand function acts first, thus
for given functions $f,g$, $f\circ g(x)=f(g(x))$, whenever the left hand side is defined, i.e
when $g(x)\in \rng f$.
For a non-empty set $X$, $f(X)$ denotes the image of $X$ under $f$, i.e
$f(X)=\{f(x):x\in X\}.$
$|X|$ denotes the cardinality of $X$ and $Id_X$, or simply $Id$ when $X$ is
clear from context, denotes the identity function on $X$.
A set $X$ is {\it countable} if $|X|\leq \omega$; if $X$ and $Y$ are sets then $X\subseteq_{\omega}Y$ denotes that $X$
is a finite subset of $Y$.

Algebras will be denoted by
Gothic letters , and when we write $\A$ then we will be tacitly assuming
that $A$ will denote  the universe
of $\A$.
However, in some occasions we will identify (notationally)
an algebra and its universe.\\

{\bf Cylindric-like algebras:} Let $n$ be a finite ordinal. Then
$\PA_n$ $(\PEA_n)$ and $\Sc_n$ denote the classes of polyadic algebras
(with equality) and Pinter's algebras of dimension $n$. $\sf Df_n$ denotes the class of diagonal free cylindric algebras of
dimension $n$. Here the extra non--Boolean operations are just the $n$--cylindrifiers, so all algebras considered have a $\sf Df$ reduct.
The standard reference for all such classes of cylindric--like algebras is \cite{HMT2}.

If $n$ is an infinite ordinal, then $\QA_{n}$, ($\QEA_{n}$) denotes the class 
of quasi--polyadic (equality) algebras of dimension $n$.
These are term definitionally equivalent to $\PA_n (\PEA_n)$ for finite $n$, but for $n\geq \omega$, they 
are not quite like $\PA_{n}$ ($\PEA_{n})$, for their signature 
contains only substitutions  indexed by replacements and transpositions.  In fact, this is a huge difference; witness Figure 1.\\

For any ordinal $n$, $\Rd_{ca}\A$ denotes the cylindric reduct of $\A$ if it has one, $\Rd_{sc}\A$
denotes the $\Sc$ reduct of $\A$ if it has one, and 
$\Rd_{df}\A$ denotes the reduct of $\A$ obtained by discarding all the operations except for cylindrifiers.
It is always the case that $\Rd_{df}\A\in \sf Df_{n}$.

\begin{figure}
\[\begin{array}{l|l}
\mbox{class}&\mbox{extra operators}\\
\hline
\Df_n&\cyl i:i<n\\
\Sc_n& \cyl i, \s_i^j :i, j<n\\
\CA_n&\cyl i, \diag i j: i, j<n\\
\PA_n&\cyl i, \s_\tau: i<n,\; \tau\in\;^nn\\
\PEA_n&\cyl i, \diag i j,  \s_\tau: i, j<n,\;  \tau\in\;^nn\\
\QA_n, &  \cyl i, \s_{[i/j]}, \s_{[i, j]} :i, j<n  \\
\QEA_n &\cyl i, \diag i j, \s_i^j, \s_{[i, j]}: i, j<n
\end{array}\]
\caption{Non-Boolean operators for the classes\label{fig:classes}}
\end{figure}
For operators on classes of algebras, $\bold S$ stands for the operation of forming subalgebras, $\bold P$
stands for that of forming products, $\bf H$ for forming homomorphic images, ${\bf Up}$ for forming ultraproducts
and ${\bf Ur}$  for forming ultraroots. The smallest elementary class containing $\K$, namely, its elementary closure,
is denoted by ${\bf El}\K$. 
It is known that ${\bf UpUr}\K=\bf El\K.$ 
If $I$ is a non--empty set and $U$ is an ultrafilter over $\wp(I)$ and if $\A_i$ is some structure (for $i\in I$) we write either
$\Pi_{i\in I}\A_i/U$ or $\Pi_{i/U}\A_i$ for the ultraproduct of the $\A_i$s over $U$. 

We write $\subseteq$ for inclusion, and $\subset$ for
proper inclusion. 
$\prod$ denotes infimum, and $\sum$ denotes (its dual) supremum.  For algebras $\A$ and $\B$ having a Boolean reduct,
we write $\A\subseteq _c \B$, if for all $X\subseteq \A$, whenever $\sum^{\A} X=1$,
then $\sum^{\B} X=1$. 
Examples of such $\A$ and $\B$ are, when $\B$ is the \de\ completion of $\A$, and
when $\A\subseteq \B$
are finite algebras.  If $\A\subseteq \B$, we say that $\A$ is dense in $\B$,  if for all non-zero $b\in \B$, there exists a 
non-zero $a\in \A$ such that $a\leq b$. If $\A$ is dense in $\B$, then $\A\subseteq_c \B$, 
the converse is not true in general. A weaker form of density holds if $\A$ is atomic and 
$\A\subseteq_ c\B$, namely, for all non-zero $b\in \A$, there exists $a\in \At\A$ such that $a\cdot b\neq 0$.

For an algebra $\A$ having a Boolean reduct, and $X\subseteq \A$, we say that
$X$ {\it completely generates} $\A$, if whenever $X\subseteq \A'\subseteq_c \A$, $\A'$ a subalgebra of $\A$,
then $\A'=\A$. If $\A$ is a $\CA_n$, then it is always the case that $\Tm\At\A$ ($\Cm\At\A)$ is (completely) generated 
by $\At\A$.

For a class $\K$ having a Boolean reduct,
we write $\bold S_c\K$ for $\{\A: (\exists \B\in \K) \A\subseteq_c \B\}$, and
$\K\cap \At$ for the class of atomic algebras in $\K$.\\

Let $\K\in\set{\Df, \Sc, \CA, \PA, \PEA}$ and let $m<n$ be ordinals. For $\A\in \K_n$, the \emph{reduct to $m$ dimensions} $\Rd_m\A\in\K_m$ is 
obtained from $\A$ by discarding all operators with indices $m\leq i<n$.     
The \emph{neat reduct to $m$ dimensions}  is the algebra  $\Nr_m\A\in \K_m$ with universe $\set{a\in\A: m\leq i<n\rightarrow \cyl i a = a}$ 
where  all the operators are induced from $\A$ (see \cite[Definition~2.6.28]{HMT2} for the $\CA$ case). 
More generally, for $\Gamma\subseteq n$ we write $\Nr_\Gamma\A$ for the algebra whose universe 
is $\set{a\in\A: i \in n\setminus\Gamma\rightarrow \cyl i a=a}$ with all the operators $o$ of $\A$ where the indicies 
in $o$ is contained in $\Gamma$. \\
  
{\bf BAOs and games:} A variety $\V$ of $\sf BAO$s (Boolean algebras with operators)
is {\it atom--canonical} if whenever $\A\in \V$ and $\A$ is atomic, then the complex algebra of its atom structure, in symbols
$\Cm\At\A$ is also in $V$.
For $\A\in \V$, $\At\A$ will denote the set of atoms of $\A$ or the atom structure of $\A$.
Context will help to decide which one of them we mean. $\Uf\A$ denotes the ultrafilter atom structure (frame) of $\A$, based on its Stone space; we sometimes use $\Uf\A$ for the 
underlying set of the Stone space, that is, the set of all Boolean 
ultrafilters of $\A$.

If $\V$ is completely additive, then $\Cm\At\A$ is the \de\ completion of $\A$.
$\A$ is always dense in its \de\ completion
and so the \de\ completion of $\A$ is atomic $\iff$  $\A$ is atomic.
Not all varieties henceforth encountered are completely additive; the varieties $\Sc_n$ and
$\sf PA_n$ for $n>1$,
are not \cite{AGMNS}. There are atomic algebras in $\Sc_n$ and $\PA_n$, for any $n>1$,
that are not completely additive. Such algebras lack complete representations,
since a completely representable algebra is necessarily completely
additive, because sums  in such an algebra are just unions.

For a variety $\V$ of $\sf BAO$s, ${\sf Str}\V=\{\F: \Cm\F\in \V\}$.
If $\V$ is completely additive, then ${\sf Str}\V\subseteq \At\V$ because $\At\Cm\F=\F$.
It is always the case that $\At \V$ is elementary \cite[Theorem 2.84]{HHbook}, which might not be the case
with ${\sf Str} \V$, witness eg. \cite{HHbook2}.
${\sf StrV}$  is elementary if $\V$ is atom--canonical, because in this case we have
${\sf At} \V={\sf Str}\V$.  

We will use  rainbow constructions played on {\it atomic networks} of a rainbow cylindric--like atom structure consisting
of coloured graphs as defined in \cite{HH} (to be recalled next).
A network $N$ on an atomic $\PEA_n$ is cylindric network \cite[Definition 3.3.2]{HHbook2}
satisfying the following extra consistency condition (reflecting the polyadic substitutions corresponding to transpositions),
${\sf s}_{[i,j]}N(\bar{x})= N(\bar{x}\circ [i,j])$ for all $i\neq j<n$, and $\bar{x}\in {}^n\nodes(N)$.\\

The game $G^m_k(\At\A)$, or simply $G^m_k$ is the usual game played on atomic networks of an atomic algebra $\A\in \PEA_n$,
using $m$ nodes and having $k\leq \omega$ rounds. We write $G_k(\At\A)$, or simply $G_k$, if $m\geq \omega$.
$F^m(\At\A)$ is like $G^m_{\omega}(\At\A)$ except that
\pa\ has the option to re--use the available
$m$  nodes  (pebbles).
Both games are played on atomic networks using at most $m$ nodes
defined on an atomic algebra $\A\in \PEA_n$ say, where \pa\ is offered one cylindrifier move.
The usual atom game $G_k^{\omega}$ was described in the introduction.
In these two games, $G^m_k$ and $F^m$ $(m, k\leq\omega$), there are no moves other than the cylindrifier move.   
The additional polyadic `information' is coded in the networks (via ${\sf s}_{[i,j]}N(\bar{x})=N(\bar{x}\circ [i,j])$) not in the game. When $m=k=\omega$, we often
write  $G_{\omega}$ or simply $G$ for $G^m_k$.\\

{\bf Rainbows:} The most general exposition of $\CA$ rainbow constructions is given
in \cite[Section 6.2, Definition 3.6.9]{HHbook2} in the context of constructing atom structures from classes of models.
Our models are just coloured graphs \cite{HH}.

Let $A$, $B$ be two relational structures. Let $2<n<\omega$.
Then the colours used are:
\begin{itemize}

\item greens: $\g_i$ ($1\leq i\leq n-2)$, $\g_0^i$, $i\in A$,

\item whites : $\w_i: i\leq n-2,$

\item reds:  $\r_{ij}$ $(i,j\in B)$,

\item shades of yellow : $\y_S: S\text { a finite subset of } B$ or $S=B$.

\end{itemize}
A {\it coloured graph} is a graph such that each of its edges is labelled by the colours in the above first three items,
greens, whites or reds, and some $n-1$ hyperedges are also
labelled by the shades of yellow.
Certain coloured graphs will deserve special attention.
\begin{definition}
Let $i\in A$, and let $M$ be a coloured graph  consisting of $n$ nodes
$x_0,\ldots,  x_{n-2}, z$. We call $M$ {\it an $i$ - cone} if $M(x_0, z)=\g_0^i$
and for every $1\leq j\leq n-2$, $M(x_j, z)=\g_j$,
and no other edge of $M$
is coloured green.
$(x_0,\ldots, x_{n-2})$
is called  the {\it base of the cone}, $z$ the {\it apex of the cone}
and $i$ the {\it tint of the cone}.
\end{definition}
The rainbow algebra depending on $A$ and $B,$ from the class $\bold K$ consisting
of all coloured graphs $M$ such that:
\begin{enumerate}
\item $M$ is a complete graph
and  $M$ contains no triangles (called forbidden triples)
of the following types:
%\vspace{-.2in}
\begin{eqnarray}
&&\nonumber\\
%(1', x, y)&&\mbox{unless }x=y\label{forb:id}\\
(\g, \g^{'}, \g^{*}), (\g_i, \g_{i}, \w_i)
&&\mbox{any }1\leq i\leq  n-2,  \\
%(\g^j_0, \y, \w_f)&&\mbox{unless }f\in P, i\in dom(f)\\
(\g^j_0, \g^k_0, \w_0)&&\mbox{ any } j, k\in A,\\
%\label{forb:pim}(\g^i_0, \g^j_0, \r_{kl})&&\mbox{unless } \set{(i, k), (j, l)}\mbox{ is an order-}\\
%&&\mbox{ preserving partial function }\Z\to\N\nonumber\\
%\label{forb:pim2}(\g_i, \g_j, \r_{kl})&&\mbox{if } i=j\mbox{ but }k\neq l\\
%\label{forb: black}(\y,\y,\y), (\y,\y,\bb)\\
\label{forb:match}(\r_{ij}, \r_{j'k'}, \r_{i^*k^*})&&\mbox{unless }i=i^*,\; j=j'\mbox{ and }k'=k^*
\end{eqnarray}
and no other triple of atoms is forbidden.

\item If $a_0,\ldots,   a_{n-2}\in M$ are distinct, and no edge $(a_i, a_j)$ $i<j<n$
is coloured green, then the sequence $(a_0, \ldots, a_{n-2})$
is coloured a unique shade of yellow.
No other $(n-1)$ tuples are coloured shades of yellow. Finally, if $D=\set{d_0,\ldots,  d_{n-2}, \delta}\subseteq M$ and
$M\upharpoonright D$ is an $i$ cone with apex $\delta$, inducing the order
$d_0,\ldots,  d_{n-2}$ on its base, and the tuple
$(d_0,\ldots, d_{n-2})$ is coloured by a unique shade
$\y_S$ then $i\in S.$
\end{enumerate}

Let $A$ and $B$ be relational structures as above. Take the set $\sf J$ consisting of all surjective maps $a:n\to \Delta$, where $\Delta\in \bold K$
and define an equivalence relation on this set  relating two such maps iff they essentially define the same graph \cite{HH};
the nodes are possibly different but the graph structure is the same.
Let $\At$ be the set of equivalences classes. We denote the equivalence class of $a$ by $[a]$. Then define, for $i<j<n$, 
the accessibility 
relations corresponding to $ij$th--diagonal element, $i$th--cylindrifier, and substitution operator 
corresponding to the transposition $[i,j]$, as follows: 

(1) \ \  $[a]\in E_{ij} \text { iff } a(i)=a(j),$

(2) \ \ $[a]T_i[b] \text { iff }a\upharpoonright n\smallsetminus \{i\}=b\upharpoonright n\smallsetminus \{i\},$

 (3) \ \ $[a]S_{ij}[b] \text { iff } a\circ [i,j]=b.$

This, as easily checked, defines a $\PEA_n$
atom structure. The complex $\PEA_n$ over this atom structure will be denoted by
$\PEA_{A, B}$. The dimension of $\PEA_{A, B}$, always finite and
$>2$, will be clear from context. The game $F^{m}$ lifts to a game on coloured graphs, that  is like the graph games
$G^m_{\omega}$ \cite{HH}, where the number of nodes of graphs played during the $\omega$ rounded
game does not exceed $m$, but \pa\ has the option to re-use nodes. \\

Throughout the paper $L_n$ denotes first order logic restricted to the first $n$ variables,
and $L_{\infty, \omega}^n$ denotes the $n$ variable fragment of $L_{\infty, \omega}$ (allowing infinite conjunctions).

\subsection*{Layout}

In what follows our investigations address cylindric--like algebras. The paper is divided 
into two parts. The first is completely self contained. The second is self contained modulo cross reference to lemmata \ref{flat}, \ref{Thm:n} and 
the main theorem  \ref{main} in part one.

\subsection*{Part 1}

\begin{enumarab}

\item In section 3, using ideas of Hirsch and Hodkinson, we introduce the notion of $n$-clique--guarded representations for a given $\CA_n$,
and we study the notion of complete representability in connection to that of neat embeddings, for both finite and infinite dimensional algebras.
The main results in this section are lemmata \ref{flat}, \ref{Thm:n}, to be extensively used throughout the paper, and theorems  
\ref{complete}, \ref{complete2}, \ref{complete3} and \ref{SL} all 
characterizing existence of complete representations of algebras, 
possibly having uncountably many atoms, in terms of neat embedding properties.

\item Section 4 is the heart and soul of this part of the paper, and in fact it is a central theorem 
in the whole paper. 
We prove our main algebraic results, concerning non atom--canonicity for both finite and infinite dimensional algebras, 
and non--elementarity; in all cases for classes consisting of  algebras having a neat embedding property  (and clique--guarded representations).
The main theorems in this section are theorems \ref{main} and \ref{squareflat}. Theorem \ref{main} 
is the heart and soul of the first part of this paper.  A large portion of the second part part of the paper will be devoted to working out carefully some of 
its metalogical consequences. 

It has five items proving a diversity of results on atom--canonicity, finite axiomatizability, first order definability, 
of classes of algebras having a neat embedding property for both finite and infinite dimensions.
\end{enumarab}

\subsection*{Part 2}

\begin{enumarab}

\item In the fifth section, we continue our investigations on  connections between neat embeddings and various 
notions of representability. We reprove a result of Bulian and Hodkinson \cite{bh} showing that for any class 
$\K$ between $\Df$ nd $\PEA$, for any $2<n<\omega$,  the class ${\sf Str}\sf RK_n$
is not elementary.
The rest of the results (and proofs) in this section are new, 
theorem \ref{main2}. We formulate, prove and compare several theorems on completions, theorems \ref{conditional}, \ref{conditional2} and 
\ref{lyndon}.

\item In the last section, we apply our algebraic results obtained on non atom--canonicity and non--elementarity of various classes of 
algebras having a neat embedding property, to prove that various fragments of first order logic, including its 
finite variable fragments (together with finite first order expansions of such fragments) 
and its packed fragment, all fail an omitting types theorem in a strong sense. Such results are formulated and proved in theorems 
\ref{OTT}, \ref{con}, and \ref{firstorder}. A positive theorem on omitting uncountably many types in $L_n$ $(2<n<\omega)$ 
theories  is given in theorem \ref{Shelah}.

\end{enumarab}

\section*{Part one}

\section{Clique guarded semantics}

We define what we call  $n$-clique--guarded semantics only for $\CA_n$s. Such notions were introduced by Hirsch and Hodkinson for relation algebras. The transfer
to $\CA$s does not involve too much subtleties, but is also not entirely straightforward. 
The rest of the cases $(\Sc$s, $\PA$s, \PEA$s)$
can be dealt with the same way. It will always be the case, unless otherwise explicitly indicated,
that $1<n<m<\omega$;  $n$ denotes the dimension.

We identify notationally a 
set algebra with its universe.  Let $M$ be a {\it relativized representation} of $\A\in \CA_n$, that is, there exists an injective
homomorphism $f:\A\to \wp(V)$ where $V\subseteq {}^nM$ and $\bigcup_{s\in V} \rng(s)=M$. For $s\in V$ and $a\in \A$,
we may write $a(s)$ for $s\in f(a)$. This notation does not refer to $f$, but whenever used 
then  either $f$ will be clear from context, or immaterial in the context. We may also write $1^M$ for $V$.  

Let  $\L(A)^m$ be the first order signature using $m$ variables
and one $n$--ary relation symbol for each element of $\A$.  Allowing infinitary conjunctions, we denote the resulting signature taken in $L_{\infty, \omega}$
by $\L(A)_{\infty, \omega}^m$.

{\it An $n$--clique}, or simply a clique,  is a set $C\subseteq M$ such
$(a_0,\ldots, a_{n-1})\in V=1^M$
for all distinct $a_0, \ldots, a_{n-1}\in C.$
Let
$${\sf C}^m(M)=\{s\in {}^mM :\rng(s) \text { is an $n$ clique}\}.$$
Then ${\sf C}^m(M)$ is called the {\it $n$--Gaifman hypergraph}, or simply Gaifman hypergraph  of $M$, with the $n$--hyperedge relation $1^M$.
The {\it $n$-clique--guarded semantics}, or simply clique--guarded semantics,  $\models_c$, are defined inductively. 

Let $f$ be as above. For an atomic $n$--ary formula $a\in \A$, $i\in{}^nm$, 
and $s\in {}^mM$, $M, s\models_c a(x_{i_0},\ldots x_{i_{n-1}})\iff\ (s_{i_0}, \ldots s_{i_{n-1}})\in f(a).$ 
For equality, given $i<j<m$, $M, s\models_c x_i=x_j\iff s_i=s_j.$
Boolean connectives, and infinitary disjunctions,
are defined as expected.  Semantics for existential quantifiers
(cylindrifiers) are defined inductively for $\phi\in \L(A)^m_{\infty, \omega}$ as follows:
For $i<m$ and $s\in {}^mM$, $M, s\models_c \exists x_i\phi \iff$ there is a $t\in {\sf C}^m(M)$, $t\equiv_i s$ such that 
$M, t\models_c \phi$. 
\begin{definition}\label{cl}
Let $\A\in \CA_n$, $M$ a relativized representation of $\A$ and $\L(A)^m$  be as above. 
\begin{enumarab}
\item Then $M$ is said to be {\it $m$--square},
if witnesses for cylindrifiers can be found on $n$--cliques. More precisely,
for all  $\bar{s}\in {\sf C}^m(M), a\in \A$, $i<n$,
and for any injective map  $l:n\to m$, if $M\models {\sf c}_ia(s_{l(0)}\ldots, s_{l(n-1)})$,
then there exists $\bar{t}\in {\sf C}^m(M)$ with $\bar{t}\equiv _i \bar{s}$,
and $M\models a(t_{l(0)},\ldots, t_{l(n-1)})$.

\item $M$ is said to be {\it (infinitary) $m$--flat} if  it is $m$--square and
for all $\phi\in (\L(A)_{\infty, \omega}^m) \L(A)^m$, 
for all $\bar{s}\in {\sf C}^m(M)$, for all distinct $i,j<m$,
we have
$M\models_c [\exists x_i\exists x_j\phi\longleftrightarrow \exists x_j\exists x_i\phi](\bar{s}).$
\end{enumarab}
\end{definition}
The proof of the  following lemma can be distilled
from its $\sf RA$ analogue \cite[Theorem 13.20]{HHbook},  by reformulating deep concepts
originally introduced by Hirsch and Hodkinson for $\sf RA$s in the $\CA$ context, involving the notions of 
hypernetworks and hyperbasis. This can (and will) be done.

In the coming proof, we highlight
the main ideas needed to perform such a transfer from $\sf RA$s to $\CA$s
\cite[Definitions 12.1, 12.9, 12.10, 12.25, Propositions 12.25, 12.27]{HHbook}. 

In all cases, the $m$--dimensional dilation stipulated in the statement of the theorem, will have
top element ${\sf C}^m(M)$, where $M$ is the $m$--relativized representation of the given algebra, and the operations of the dilation
are induced by the $n$-clique--guarded semantics.

But first a technical definition.

\begin{definition}\label{sub} Let $m$ be a finite ordinal $>0$. An $\sf s$ word is a finite string of substitutions $({\sf s}_i^j)$ $(i, j<m)$,
a $\sf c$ word is a finite string of cylindrifications $({\sf c}_i), i<m$;
an $\sf sc$ word $w$, is a finite string of both, namely, of substitutions and cylindrifications.
An $\sf sc$ word
induces a partial map $\hat{w}:m\to m$:
\begin{itemize}

\item $\hat{\epsilon}=Id,$

\item $\widehat{w_j^i}=\hat{w}\circ [i|j],$

\item $\widehat{w{\sf c}_i}= \hat{w}\upharpoonright(m\smallsetminus \{i\}).$

\end{itemize}
If $\bar a\in {}^{<m-1}m$, we write ${\sf s}_{\bar a}$, or
${\sf s}_{a_0\ldots a_{k-1}}$, where $k=|\bar a|$,
for an  arbitrary chosen $\sf sc$ word $w$
such that $\hat{w}=\bar a.$
Such a $w$  exists by \cite[Definition~5.23 ~Lemma 13.29]{HHbook}.
\end{definition}

We also need the notion of $m$--dimensional hyperbasis. This hyperbasis is made up of $m$--dimensional hypernetworks.
An $m$--dimensional hypernetwork on the atomic algebra $\A$ is an $n$--dimensional  network $N$, with $\nodes(N)\subseteq m$, endowed with a set of labels $\Lambda$ for
hyperedges of length $\leq m$,
not equal to $n$ (the dimension), such that $\Lambda\cap \At\A=\emptyset$. We call a label in $\Lambda$ a non--atomic label.

Like in networks, $n$--hyperedges are labelled by atoms. In addition to the consistency properties for networks,
an $m$--dimensional hypernetwork should satisfy the following additional consistency rule involving non--atomic labels:
If $\bar{x}, \bar{y}\in {}^{\leq m}m$, $|\bar{x}|=|\bar{y}|\neq n$ and $\exists \bar{z}$, such that $\forall i<|\bar{x}|$,
$N(x_i,y_i,\bar{z})\leq {\sf d}_{01}$,
then $N(\bar{x})=N(\bar{y})\in \Lambda$. (We shall deal with hypernetworks in item (1) of theorem \ref{main}).

\begin{definition} Let $2<n<m<\omega$ and $\A\in \CA_n$ be atomic.
\begin{enumarab}

\item An $m$--dimensional basis $B$ for $\A$ consists of a set of $n$--dimensional networks whose nodes $\subseteq m$, satisfying 
the following properties: 

\begin{itemize}

\item For all  $a\in \At\A$, there is an $N\in B$ such that $N(0,1,\ldots, n-1)=a,$

\item The {\it cylindrifier property}: For all $N\in B$, all $i<n$,  all $\bar{x}\in {}^n\nodes(N)(\subseteq {}^nm)$, all $a\in\At\A$, such that
$N(\bar{x})\leq {\sf c}_ia$,  there exists $M\in B$, $M\equiv_i N$, $\bar{y}\in {}^n\nodes(M)$ such 
that $\bar{y}\equiv_i\bar{x}$ and $M(\bar{y})=a.$  We can always assume that $\bar{y}_i$ is a new node else one takes $M=N$. 

\end{itemize}

\item An $m$--dimensional  hyperbasis $H$ consists of $m$--dimensional hypernetworks, satisfying the above two conditions reformulated 
the obvious way for hypernetworks, in addition,  $H$ has an amalgamation property for overlapping hypernertworks off 
at most $2$ nodes; 
this property corresponds to commutativity of cylindrifiers: 

For all $M,N\in H$ and $x,y<m$, with $M\equiv_{xy}N$, there is $L\in H$ such that
$M\equiv_xL\equiv_yN$. Here $M\equiv_SN$, means 
that $M$ and $N$ agree off of $S$ \cite[Definition 12.11]{HHbook}.

\end{enumarab}
\end{definition}
The next example taken from \cite{AGMNS} alerts us to the fact that the diagonal free algebras we study, namely, $\Sc$s and $\PA$s 
may not  be necessarily 
completely additive. We cannot assume {\it a priori} that they are. 
So in the next lemma, and other lemmas whose proofs depend on complete addtivity of the ${\sf s}_i^j$s, 
we will have to assume explicity complete additivity in the hypothesis which is superfluous for any $\K$ between 
$\CA$ and $\PEA$.
\begin{example}\label{counter}
We give the outline. Let $|U|=\mu$ be an infinite set and $|I|=\kappa$ be a cardinal such 
that $Q_n$, $n\in \kappa$,  is a family of relations that form a partition of $U\times U$. 
Let $i\in I$, and let $J=I\sim \{i\}$. Then of course $|I|=|J|$. Assume that $Q_i=D_{01}=\{s\in V: s_0=s_1\},$
and that each $Q_n$ is symmetric; that is for any  $S_{[0,1]}Q_n=Q_n$ and furthermore, that $\dom Q_n=\rng Q_n=U$ for every $n\in \kappa$.
It is straightforward to show that such partitions exist.
Now fix $F$ a non--principal ultrafilter on $J$, that is $F\subseteq \mathcal{P}(J)$. 
For each $X\subseteq J$, define
\[
 R_X =
  \begin{cases}
   \bigcup \{Q_k: k\in X\} & \text { if }X\notin F, \\
   \bigcup \{Q_k: k\in X\cup \{i\}\}      &  \text { if } X\in F
  \end{cases}
\]
Let $$\A=\{R_X: X\subseteq J\}.$$
Notice that $|\A|\geq \kappa$. Also $\A$ is an atomic set algebra with unit $R_{J}$, and its atoms are $R_{\{k\}}=Q_k$ for $k\in J$.
(Since $F$ is non-principal, so $\{k\}\notin F$ for every $k$). This can be proved exactly like in \cite{AGMNS}. 
We check substitutions. Transpositions are clear (recall that $Q_k$'s are symmetric), so we check only replacements. It is easy to see that
\[
 {\sf s}_0^1(R_X)=
  \begin{cases}
   \emptyset & \text { if }X\notin F, \\
   R_{J}      &  \text { if } X\in F
  \end{cases}
\]

Now we show that ${\sf s}_0^1$ is not completely additive:
$$\sum \{{\sf s}_0^1(R_{k}): k\in J\}=\emptyset.$$
and
$${\sf s}_0^1(R_{J})=R_{J}$$
$$\sum \{R_{\{k\}}: k\in J\}=R_{J}=\bigcup \{Q_k:k\in J\}.$$
Thus $${\sf s}_0^1(\sum\{R_{\{k\}}: k\in J\})\neq \sum \{{\sf s}_0^1(R_{\{k\}}): k\in J\}.$$

\end{example}

Accordingly for a class $\K$ of $\sf BAO$s, 
we write $\K^{\sf ad}$ 
for the class of completely additive algebras in $\K$. We have just seen that for $\K\in \{\Sc, \PA\}$, 
$\K_2^{\sf ad}\neq \K_2$. In \cite{AGMNS} more examples can be found establishing that
$\K_n^{\sf ad}\neq \K_n$ for all $n>1$ replacing, in the infinite dimensional case, polyadic algebras 
by quasi-polyadic algebras.

One can define $m$--smooth representations as in \cite[Definition 13.12] {HHbook}.
But like infinitary $m$--flat representations,  they really do not add much to $m$--flatness as far as (ordinary as opposed to complete) 
representations are concerned.
But they do in the case of {\it complete} $m$--relativized representations, witness item (3) in the coming lemma.
An algebra having a complete $m$--flat representation, has an $m$--infinitary flat representation, 
but not necessarily a complete one. 

\begin{lemma}\label{flat}\cite[Theorems 13.45, 13.36]{HHbook}.
Assume that $2<n<m<\omega$ and that $\K$ is any class between $\Sc$ and $\PEA$. Let $\A\in \K_n$. Then the following hold:
\begin{enumarab}
\item $\A\in \bold S\Nr_n\K_m\iff \A$ has an $m$--smooth representation $\iff \A$ has has an  infinitary $m$--flat representation
$\iff \A$ has an $m$--flat representation $\iff \A^+$ has an $m$--dimensional hyperbasis.
\item If $\A$ is atomic, then $\A$ has a complete infinitary $m$--flat representation $\iff$ $\A$ has a complete $m$--smooth representation.
$\iff$ $\A\in \bold S_c\Nr_n(\K_m^{\sf ad}\cap \At)$ $\iff$ $\A$ has an $m$--dimensional hyperbasis. 
\item If $\A$ is atomic and has a complete $m$--flat representation, then $\A\in \bold S_c\Nr_n\K_m$, but the dilation in 
$\sf K_m$ may not be atomic.
\item If $\A$ has an $m$--square representation, 
then $\A^+$ has a complete $m$--square representation; a completely analogous result holds for $m$--flatness, 
infinitary  $m$--flatness, and $m$--smoothness.
\end{enumarab}
\end{lemma}

\begin{proof}

Fix $2<n<m<\omega$. For the first item:
As above, let $\L(A)^m$ denote the signature that contains
an $n$--ary predicate symbol for every $a\in A$.
We prove the easier direction which is a soundness theorem: The existence of $m$--flat representations, implies the existence of $m$--dilations.
Let $M$ be an $m$--flat representation of $\A$. We show that $\A\subseteq \Nr_n\D$, for some $\D\in \CA_m$,
and that $\A$ actually has an infinitary $m$--flat representation.

For $\phi\in \L(A)^m$,
let $\phi^{M}=\{\bar{a}\in {\sf C}^m(M):M\models_c \phi(\bar{a})\}$, where ${\sf C}^m(M)$ is the $n$--Gaifman hypergraph.
Let $\D$ be the algebra with universe $\{\phi^{M}: \phi\in \L(A)^m\}$ and with  cylindric
operations induced by the $n$-clique--guarded (flat) semantics. 
For $r\in \A$, and $\bar{x}\in {\sf C}^m(M)$, we identify $r$ with the formula it defines in $\L(A)^m$, and 
we write $r(\bar{x})^M\iff M, \bar{x}\models_c r$.

Then certainly $\D$ is a subalgebra of the $\sf Crs_m$ (the class
of algebras whose units are arbitrary sets of $m$--ary sequences)
with domain $\wp({\sf C}^m(M))$, so $\D\in {\sf Crs_m}$ with unit $1^{\D}={\sf C}^m(M)$.
Since $M$ is $m$--flat, then cylindrifiers in $\D$ commute, and so $\D\in \CA_m$.

Now define $\theta:\A\to \D$, via $r\mapsto r(\bar{x})^{M}$. Then exactly like in the proof of \cite[Theorem 13.20]{HHbook},
$\theta$ is a neat embedding, that is, $\theta(\A)\subseteq \Nr_n\D$.
It is straightforward to check that $\theta$ is a homomorphism.  We show that $\theta$ is injective.
Let $r\in A$ be non--zero. Then $M$ is a relativized representation, so there is $\bar{a}\in M$
with $r(\bar{a})$, hence $\bar{a}$ is a clique in $M$,
and so $M\models r(\bar{x})(\bar{a})$, and $\bar{a}\in \theta(r)$, proving the required.

$M$ itself might not be  infinitary $m$--flat, but one can build an infinitary $m$--flat representation of $\A$, whose base is an $\omega$--saturated model
of the consistent first order theory, stipulating the existence of an $m$--flat representation \cite[Proposition 13.17, Theorem 13.46 items (6) and (7)]{HHbook}.
This idea (of using saturation) will be given in more detail in the last item.\\

The converse, essentially a completeness theorem, that is a `truncated version' of Henkin's
neat embedding theorem: {\it Existence of $m$--dimensional dilations}
versus {\it existence of $m$--flat representations}.
One constructs an infinitary $m$--flat representation $M$ of $\A$
from an {\it $m$--dimensional hyperbasis} \cite[Definition 12.11]{HHbook}
a notion that can be adapted to $\CA$s, as illustrated next. But this $m$--dimensional hyperbasis 
is constructed  for the {\it canonical extension} $\A^+.$ 

The order of the proof is: Constructing an $m$--dilation, from which an $m$--dimensional 
hyperbasis is constructed, from which
the required $m$-- relativized representation is built.

We start by giving a fairly complete sketch of how an $m$--dimensional hyperbasis for the  canonical extension of $\A\in \CA_n$
is obtained from an $m$--dilation of $\A$ \cite[Definition 13.22, lemmata 13.33-34-35, Proposition 36]{HHbook}.
Suppose that $\A\subseteq \Nr_n\D$ for some $\D\in \CA_m$.
Then $\A^+\subseteq_c \Nr_m\D^+$, and $\D^+$ is atomic. We show that $\D^+$ has an $m$--dimensional hyperbasis.
First, it is not hard to see that for every $n\leq l\leq m$, $\Nr_l\D^+$ is atomic.
The set of non--atomic labels $\Lambda$ is the set $\bigcup_{k<m-1}\At\Nr_k\D^+$.

For each atom $a$ of $\D^+$, define a labelled  hypergraph $N_a$ as follows.
Let $\bar{b}\in {}^{\leq m}m$. Then if $|\bar{b}|=n$,  so that $\bar{b}$  has to get a label that is an atom of $\D^+$, one sets  $N_a(\bar{b})$ to be 
the unique $r\in \At\D^+$ such that $a\leq {\sf s}_{\bar{b}}r$; notation here
is given in definition \ref{sub}.
If $n\neq |\bar{b}| <m-1$, $N_a(\bar{b})$ is the unique atom $r\in \Nr_{|b|}\D^+$ such that $a\leq {\sf s}_{\bar{b}}r.$ Since
$\Nr_{|b|}\D^+$ is atomic, this is well defined. Note that this label may be a non--atomic one; it 
might not be an atom of $\D^+$. But by definition it is a permitted label.
Now fix $\lambda\in \Lambda$. The rest of the labelling is defined by $N_a(\bar{b})=\lambda$.
Then $N_a$ as an $m$--dimensional
hypernetwork, for each 
such chosen $a$,  and $\{N_a: a\in \At\D^+\}$ is the required $m$--dimensional hyperbasis.
The rest of the proof consists of a fairly straightforward adaptation of the proof \cite[Proposition 13.37]{HHbook},
replacing edges by $n$--hyperedges.\\

For the second item,
we need: If $\A$ is an atomic $\CA_n$ having an  $n$--relativized representation $M$, then 
$f:\A\to \wp(1^M)$ is a complete $n$--relativized representation of $\A$, carrying meets to set--theoretic intersections 
$\iff$ $f$ is atomic, in the sense that $\bigcup_{a\in \At\A}f(a)=1^M$. This is proved 
exactly like in the classical case \cite{HH}.

Now one works in $L_{\infty, \omega}^m$ instead of first order logic.
In this case, the dilation $\D$ having again top element the Gaifman hypergraph ${\sf C}^m(M)$, where $M$ the  complete $m$--flat representation,
will now have  (the larger) universe $\{\phi^M: \phi\in \L(A)^m_{\infty, \omega}\}$ with operations also induced by the $n$-clique--guarded semantics extended to
$L_{\infty, \omega}^m$.
$\D$ will be a $\CA_m$ as before, but this time, it will be {\it an atomic} one.  To prove atomicity, let $\phi^M$ be a non--zero element in $\D$.
Choose $\bar{a}\in \phi^M$, and consider the following infinitary conjunction (which we did not have before when working in $L_m$)
\footnote{There are set--theoretic subtleties involved here, that we prefer to ignore.}:
$\tau=\bigwedge \{\psi\in \L(A)_{\infty,\omega}^m: M\models_C \psi(\bar{a})\}.$
Then $\tau\in \L(A)_{\infty,\omega}^m$, and $\tau^{M}$ is an atom below $\phi^M$. 
The neat embedding will be an atomic one, hence
it will be a complete neat embedding \cite[p. 411]{HHbook}. This last statement is proved in the third item. 
(In case of $\Sc$s and $\PA$s the dilation will be  completely additive because sums in uch a dilation defined like above 
are just unions.)

Conversely, if $\A\subseteq _c \Nr_n\D$, $\D$ an atomic $m$--dilation, 
then the existence of an $m$--dimensional hyperbasis is  exactly as done 
above, when constructing an $m$--dimensional 
hyperbasis from the atoms of the canonical extension of an $m$--dilation,  
giving a complete $m$--infinitary flat representation.

For third item, assume that $\A$ has a complete $m$--flat representation $M$. Then one forms the dilation $\D\in \CA_m$ like before, 
but like in the first item, working in $\L(A)^m$, 
so that the dilation $\D$  might not be atomic. However, the neat embedding map $\theta$ (using the above notation) 
is atomic, hence  complete, because  $M$ is a complete $m$--flat representation. 
To see why, let $\phi\in \L(A)^m$ such that $\phi^M\neq 0$. Pick $\bar{a}\in \phi^M$. Since $M$ is a complete $m$--flat 
representation, and $\bar{a}\in {\sf C}^m(M)$,  so there exists an
$\alpha\in \At\A$, such that $M\models_c \alpha(\bar{a})$, so $\bar{a}\in \alpha(\bar{x})^M=\theta(\alpha)$. Hence,
$\theta(\alpha)\cdot \phi^M\neq 0$, and we are done.

Last item: This can be proved, using ideas of Hirsch and Hodkinson, by taking an $\omega$--saturated model of the consistent theory stipulating the 
existence of an $m$--square representation for $\A$, as the base of the complete $m$--
square representation for $\A^+$.
In more detail, let $M$ be an $\omega$--saturated model,
of this theory, then we show that it is a complete $n$--square representation.
One  defines an injective complete embedding $h: \A^+\to \wp(1^M).$
First note that the set $f_{\bar{x}}=\{a\in A: a(\bar{x})\}$
is an ultrafilter in $\A$, whenever $\bar{x}\in M$ and $M\models 1(\bar{x}).$
Now $\A^+=\Cm(\Uf\A)$. For $S\subseteq \Uf\A$, 
let $h(S)=\{\bar{x}\in 1^{M}: f_{\bar{x}}\in S\}.$

We check only injectivity using saturation. For the (ideas used in) rest of  the proof the reader is referred to \cite[Corollary 13.18]{HHbook}.
It suffices to show that for any ultrafilter $F$ of $\A$ 
which is an atom in $\A^+$, we have
$h(\{F\})\neq 0$.
Let $p(\bar{x})=\{a(\bar{x}): a\in F\}$. Then this type is finitely satisfiable.
Hence by $\omega$ saturation $p$ is realized in $M$ by $\bar{y}$, say.
Now $M\models 1(\bar{y})$ and $F\subseteq f_{\bar{x}}$,
since these are both ultrafilters, equality holds.
\end{proof}
Now we devise games testing neat embeddability.
The following lemma  will be used quite often. It is  a non--trivial adaptation of its relation algebra analogue,
witness \cite[Theorem 33]{r}, \cite[Proposition 12.25, Theorem 13.45-46]{HHbook},
to our present $\CA$ context.

\begin{lemma}\label{Thm:n}
Let $2<n<m$. If $\A\in \bold S_c\Nr_n\Sc_{m}^{\sf ad}$,
then \pe\ has a \ws\ in $F^m.$
In particular, if $\A$ is a finite $\Sc_n$, and \pa\ has a \ws\ in $F^{m}$ on $\At\A$,
then $\A\notin \bold S\Nr_n\Sc_{m}$; if in addition $\A\in \CA_n$, then $\A$ does not have an 
$m$--dimensional hyperbasis.
If \pa\ has a \ws\ in $G^m_{\omega},$ then $\A$  does not have a complete $m$--square representation. If $\A$ is finite
then it does not have an $m$--square representation and furthermore if $\A\in \CA_n$, then it does not have an $m$--dimensional basis.
\end{lemma}
\begin{proof}
We start by proving the first part ending up with a \ws\ for \pe\ in $F^m$. The proof  uses ideas
in \cite[Lemma 29, 26, 27]{r} formulated for relation algebras.\\

{\bf Proof of first part}:  Fix $2<n<m$. Assume that $\C\in\Sc_m$, $\A\subseteq_c\Nr_n\C$ is an
atomic $\Sc_n$ and $N$ is an $\A$--network with $\nodes(N)\subseteq m$. Define
$\widehat N\in\C$ by
\[\widehat N =
 \prod_{i_0,\ldots, i_{n-1}\in\nodes(N)}{\sf s}_{i_0, \ldots, i_{n-1}}{}N(i_0,\ldots, i_{n-1}).\]
Here the substitution operator is defined as in definition \ref{sub}.
For a network $N$ and  function $\theta$,  the network
$N\theta$ is the complete labelled graph with nodes
$\theta^{-1}(\nodes(N))=\set{x\in\dom(\theta):\theta(x)\in\nodes(N)}$,
and labelling defined by
$$(N\theta)(i_0,\ldots, i_{n-1}) = N(\theta(i_0), \theta(i_1), \ldots,  \theta(i_{n-1})),$$
for $i_0, \ldots, i_{n-1}\in\theta^{-1}(\nodes(N))$.
Now suppose that the ${\sf s}_i^j$s for  $i< j<m$ are completely additive in $\C$. This condition is superfluous for any $\K$ between $\CA$ and $\PEA$.
Then the following hold:
\begin{enumerate}
\item for all $x\in\C\setminus\set0$ and all $i_0, \ldots, i_{n-1} < m$, there is $a\in\At\A$, such that
${\sf s}_{i_0,\ldots, i_{n-1}}a\;.\; x\neq 0$,

\item for any $x\in\C\setminus\set0$ and any
finite set $I\subseteq m$, there is a network $N$ such that
$\nodes(N)=I$ and $x\;.\;\widehat N\neq 0$. Furthermore, for any networks $M, N$ if
$\widehat M\;.\;\widehat N\neq 0$, then
$M\restr {\nodes(M)\cap\nodes(N)}=N\restr {\nodes(M)\cap\nodes(N)},$

\item if $\theta$ is any partial, finite map $m\to m$
and if $\nodes(N)$ is a proper subset of $m$,
then $\widehat N\neq 0\rightarrow \widehat{N\theta}\neq 0$. If $i\not\in\nodes(N),$ then ${\sf c}_i\widehat N=\widehat N$.

\end{enumerate}

Since $\A\subseteq _c\Nr_n \C$, then $\sum^{\C}\At\A=1$. For (1), we have, by assumption,  ${\sf s}^i_j$ is a
completely additive operator (any $i, j<m$), hence ${\sf s}_{i_0,\ldots, i_{n-1}}$
is, too.
So $\sum^{\C}\set{{\sf s}_{i_0\ldots, i_{n-1}}a:a\in\At(\A)}={\sf s}_{i_0\ldots i_{n-1}}
\sum^{\C}\At\A={\sf s}_{i_0\ldots, i_{n-1}}1=1$,
for any $i_0,\ldots, i_{n-1}<m$.  Let $x\in\C\setminus\set0$.  Assume for contradiction
that  ${\sf s}_{i_0\ldots, i_{n-1}}a\cdot x=0$ for all $a\in\At\A$. Then  $1-x$ will be
an upper bound for $\set{{\sf s}_{i_0\ldots i_{n-1}}a: a\in\At\A}.$
But this is impossible
because $\sum^{\C}\set{{\sf s}_{i_0\ldots, i_{n-1}}a :a\in\At\A}=1.$\\

To prove the first part of (2), we repeatedly use (1).
We define the edge labelling of $N$ one edge
at a time. Initially, no hyperedges are labelled.  Suppose
$E\subseteq\nodes(N)\times\nodes(N)\ldots  \times\nodes(N)$ is the set of labelled hyperedges of $
N$ (initially $E=\emptyset$) and
$x\;.\;\prod_{\bar c \in E}{\sf s}_{\bar c}N(\bar c)\neq 0$.  Pick $\bar d$ such that $\bar d\not\in E$.
Then by (1) there is $a\in\At(\A)$ such that
$x\;.\;\prod_{\bar c\in E}{\sf s}_{\bar c}N(\bar c)\;.\;{\sf s}_{\bar d}a\neq 0$.
Include the hyperedge $\bar d$ in $E$.  We keep on doing this until eventually  all hyperedges will be
labelled, so we obtain a completely labelled graph $N$ with $\widehat
N\neq 0$.
it is easily checked that $N$ is a network.\\

For the second part of $(2)$, we proceed contrapositively. Assume that there is
$\bar c \in{}\nodes(M)\cap\nodes(N)$ such that $M(\bar c )\neq N(\bar c)$.
Since edges are labelled by atoms, we have $M(\bar c)\cdot N(\bar c)=0,$
so
$0={\sf s}_{\bar c}0={\sf s}_{\bar c}M(\bar c)\;.\; {\sf s}_{\bar c}N(\bar c)\geq \widehat M\;.\;\widehat N$.
A piece of notation. For $i<m$, let $Id_{-i}$ be the partial map $\{(k,k): k\in m\smallsetminus\{i\}\}.$
For the first part of (3)
(cf. \cite[lemma~13.29]{HHbook} using the notation in {\it op.cit}), since there is
$k\in m\setminus\nodes(N)$, \/ $\theta$ can be
expressed as a product $\sigma_0\sigma_1\ldots\sigma_t$ of maps such
that, for $s\leq t$, we have either $\sigma_s=Id_{-i}$ for some $i<m$
or $\sigma_s=[i/j]$ for some $i, j<m$ and where
$i\not\in\nodes(N\sigma_0\ldots\sigma_{s-1})$.
But clearly  $\widehat{N Id_{-j}}\geq \widehat N$ and if $i\not\in\nodes(N)$ and $j\in\nodes(N)$, then
$\widehat N\neq 0 \rightarrow \widehat{N[i/j]}\neq 0$.
The required now follows.  The last part is straightforward.\\

Using the above proven facts,  we are now ready to show that \pe\  has a \ws\ in $F^m$. She can always
play a network $N$ with $\nodes(N)\subseteq m,$ such that
$\widehat N\neq 0$.

In the initial round, let \pa\ play $a\in \At\A$.
\pe\ plays a network $N$ with $N(0, \ldots, n-1)=a$. Then $\widehat N=a\neq 0$.
Recall that here \pa\ is offered only one (cylindrifier) move.
At a later stage, suppose \pa\ plays the cylindrifier move, which we denote by
$(N, \langle f_0, \ldots, f_{n-2}\rangle, k, b, l).$
He picks a previously played network $N$,  $f_i\in \nodes(N), \;l<n,  k\notin \{f_i: i<n-2\}$,
such that $b\leq {\sf c}_l N(f_0,\ldots,  f_{i-1}, x, f_{i+1}, \ldots, f_{n-2})$ and $\widehat N\neq 0$.
Let $\bar a=\langle f_0\ldots f_{i-1}, k, f_{i+1}, \ldots f_{n-2}\rangle.$
Then by  second part of  (3)  we have that ${\sf c}_l\widehat N\cdot {\sf s}_{\bar a}b\neq 0$
and so  by first part of (2), there is a network  $M$ such that
$\widehat{M}\cdot\widehat{{\sf c}_lN}\cdot {\sf s}_{\bar a}b\neq 0$.
Hence $M(f_0,\dots, f_{i-1}, k, f_{i-2}, \ldots$ $, f_{n-2})=b$,
$\nodes(M)=\nodes(N)\cup\set k$, and $\widehat M\neq 0$, so this property is maintained.
The second part, follows immediately by noting 
that $\A=\A^+$, if $\A$ is finite,  and that if $\A\in \bold S\Nr_n\Sc_m\implies \A^+\in \bold S_c\Nr_n\Sc_m$.\\

{\bf Proof of the rest}: For the last required, we proceed contrapositively, and we deal only with $\CA$s. 
Like the order of the difficult implication 
in the proof of the first item of lemma \ref{Thm:n}, dealing with completeness, the order of the proof is as follows:
We first construct an $m$--dilation that might not be a $\CA_n$; cylindrifiers may not commute. From the dilation 
we construct an $m$--dimensional basis. Then (instead of proving relativized representability) 
we prove the harder implication in the following equivalence (which is from left to right): 
An atomic $\A\in \CA_n$ has  an $m$--dimensional basis $\iff$  \pe\ has  a \ws\ in $G^m(\At\A)$. 

Here we assume that $m$ is finite (we will use the result in this restricted form only).
So fix $2<n<m< \omega$.  Assume that $\A$ is an atomic $\CA_n$ having a complete $m$--square representation.
We will show that \pe\ has a \ws\  in $G_{\omega}^m$.

As before, let $M$ be a complete $m$--square representation of $\A$.
One constructs the $m$--dimensional dilation $\D$ using $L_{\infty, \omega}^n$ formulas 
from a complete $m$--square representation exactly like in the proof of lemma \ref{flat}.
The neat embedding map  $\theta:\A\to \D$ is the same, defined  via $r\mapsto r(\bar{x})^{M}$.
Here the $m$--neat reduct of $\D$ is defined like the $\CA$ case,  even though the dilation $\D$ {\it may not be a $\CA_m$} for we
do not necessarily have commutativity of cylindrifiers, because there is no guarantee that $M$ is
$m$--flat. As before $\theta$ is an injective homomorphism into $\Nr_n\D$, and $\D$ is atomic.

For each $\bar{a}\in 1^{\D},$ define \cite[Definition 13.22] {HHbook} a labelled
hypergraph $N_{\bar{a}}$ with nodes $m$, and
$N_{\bar{a}}(\bar{x})$ when $|\bar{x}|=n$, is the unique atom of $\A$
containing the tuple of length $m>n$,\\
$(a_{x_0},\ldots, a_{x_{1}},\ldots, a_{x_{n-1}}, a_{x_0}\ldots,\ldots a_{x_0}).$
It is clear that if $s\in 1^{\D}$ and $i, j<m$,
then $s\circ [i|j]\in 1^{\D}$.

Hence the above definition is sound. Indeed, if $\Psi: m\to m$ is defined by 
$\Psi(1)=x_1,\dots$,  $\Psi(n-2)=x_{n-2},$
and $\Psi(i)=x_0$ for $i\in m\sim \{1,\ldots, n-2\}$, then $\Psi$ is not injective, hence it is a composition of
replacements, so $\bar{a}\circ \Psi=(a_{x_0},\ldots, a_{x_{1}},\ldots, a_{x_{n-1}}, a_{x_0}\ldots\ldots a_{x_0})\in 1^{\D}.$
It is also easy to see, since $\A\subseteq \Nr_n\D$, that if $\bar{a}=(a_0, \ldots, a_{m-1})\in 1^{\D}$,
$i_0, \ldots, i_{n-1}<m$
and  $\bar{b}\in 1^{\D}$ is such that $\bar{b}\upharpoonright n\subseteq \bar{a}$, then for
all atoms $r\in \A$,  $\bar{b}\in r\iff  r=N_{\bar{a}}(i_0,\ldots, i_{n-1}).$
Furthermore, \cite[Lemma 13.24]{HHbook}  $N_{\bar{a}}$ is a network.

Let $H$ be the symmetric closure
of $\{N_a: \bar{a}\in 1^M\}$, that is $\{N\theta: \theta:m\to m, N\in H\}$.
Then $H$ is an $m$--dimensional basis \cite[Lemma 13.26]{HHbook} as defined in the proof of lemma \ref{Thm:n}.
Recall that $H$ `eliminates cylindrifiers' in the following sense: For all $a\in \At\A$, $i<n$ and $N\in H$, for all $\bar{x}\in {}^n\nodes(N)$, whenever 
$N(\bar{x})\leq {\sf c}_ia$, then there is an $M\in H$,  with $M\equiv _i N$, and 
$\bar{y}\in {}^n\nodes(M)$, such that  $\bar{y}\equiv_i \bar{x}$ 
and $M(\bar{y})=a.$

Now \pe\ can win $G_{\omega}^m$ by always
playing a subnetwork of a network in the constructed $H$.
In round $0$, when \pa\ plays
the atom $a\in \A$, \pe\ chooses $N\in H$ with $N(0,1,\ldots, n-1)=a$ and plays $N\upharpoonright n$.

In round $t>0$, inductively if the current network is $N_{t-1}\subseteq M\in H$, then no matter how \pa\ defines $N$, we have
$N\subseteq M$ and $|N|<m$, so there is $z<m$, with $z\notin \nodes(N)$.
Assume that  \pa\ picks $x_0,\ldots, x_{n-1}\in \nodes(N)$, $a\in \At\A$ and $i<n$ such that
$N(x_0,\ldots, x_{n-1})\leq {\sf c}_ia$, so $M(x_0, \ldots  x_{n-1})\leq {\sf c}_ia$,
and hence (by the properties of $H$), there is $M'\in H$ with
$M'\equiv _i M$ and $M'(x_0, \ldots, z, \ldots,  x_{n-1})=a$, with $z$ in the $i$th place.
Now \pe\ responds with the restriction of $M'$
to $\nodes(N)\cup \{z\}$.\\ 

The last part dealing with finiteness 
is obvious. Non--existence of an $m$--dimensional hyperbasis in the first case, and an $m$--dimensional basis in the second 
follows from the proof of lemma \ref{flat}, by noting that if  $\A$ is finite, then $\A^+=\A$.
\end{proof}

The following corollary is the $\CA$ analogue of \cite[Theorem 13.45, item (1)--(4)]{HHbook}. Its proof can be easily discerned below the surface 
of the proofs of lemma \ref{flat}, and second part
of lemma \ref{Thm:n}.
\begin{corollary}Let $2<n<m<\omega$. Let $\A\in \CA_n$. Then the following are equivalent:
\begin{enumarab}
\item $\A^+$ has an $m$--dimensional basis.
\item $\A^+$ has a complete $m$--square representation.
\item $\A^+\subseteq_c \Nr_n\D$, where $\D$ is a set algebra having top element $V\subseteq {}^mU$, for some non-empty set $U$, 
such that if $s\in V$,  and $i<j<m$, then $s\circ [i|j]\in V$. The operations in the $m$--dilation 
$\D$ are like the operations in cylindric set algebras of dimension $m$, 
but relativized to the top element $V.$ (For example: for $X\subseteq V$ and $i<m$, ${\sf c}_iX=\{t\in V: \exists s\in V: s\equiv_i t\}$).
\item  \pe\ has a \ws\ in $G^m_{\omega}(\Uf\A)$.
\end{enumarab}
\end{corollary}
If we consider $\PEA_m$, we get that the dilation $\D$ is in $\sf G_m$, and we require that the  basis $B$ is symmetric $(N\in B\implies N\theta\in B$).
The analogous results for infinitary $m$--flatness, is replacing (symmetric) basis by (symmetric) 
hyperbasis, $\D$ is a $\CA_m(\sf PEA_m)$,  and \pe\ has a \ws\ in a `hyperbasis game' 
which has an additional  amalgamation 
move \cite[Theorem 13.5, items (5)--(10)]{HHbook} corresponding to commutativity of 
cylindrifiers in the dilation.

Let $2<n<m<\omega$. Let $\A\in \bold S\Nr_n\CA_m$, then $\A=\Fm_T$ 
where $T$ is a first order theory in the signature containing one $n$--ary {\it restricted} atomic formula for each
$a\in A$ \cite[Theorem 4.3.28 (1)]{HMT2} and equality. An $n$-ary restricted atomic formula is one of the form $R(x_0,\ldots x_{n-1})$ where $R$
is a relation symbol of arity $n$ and the variables occur in their natural order, so we might as well write 
$R$ for $R(x_0,\ldots, x_{n-1})$.  Here the quotient is defined via provability  with respect to the 
proof system in  \cite[pp. 157]{HMT2} formulated for 
restricted formulas, {\it using $m$ variables in the derivations}, 
and the rules of modus ponens and generalization. We say that $T$ is $m$--consistent if 
$\Fm_T\neq 0$.

\begin{definition}\label{validity}
Let $2<n<m<\omega$. 
Let $\A\in \CA_n$ be countable.  Then the formula $\phi\in \L(A)^m$, is {\it $\A$-$m$--flat valid}, 
if for all $m$--flat representations 
$M$ of $\A$, for all $s\in {\sf C}^{m}(M)$,  $M, s\models_c   \phi$. 
 \end{definition}

Now we can easily obtain the following completeness theorems, one for each $m\geq n$, using the 
simple proof system in \cite{HMT2}. We write $\vdash_k$ for provability using 
$k$ variables.
\begin{theorem}\label{validity2} Let everything be as above. Then for $\A\in \bold S\Nr_n\CA_m$, $\A=\Fm_T$ say, the formula   
$\phi$ in $\L(A)^n$ is $\A$ $m$--flat valid $\iff$  $T\vdash_m \phi$.
\end{theorem}
\begin{proof} Soundness is by induction on the length of proof by noting that axioms are valid in $n$--clique guarded semantics and that both 
modus ponus and generalization (the rules of inference adopted in the aforemention proof system) 
`preserves' $n$--clique guarded semantics; for example if $\phi$ and $\phi\to \psi$ are $\A$ $m$--flat valid, 
then $\psi$ is also $\A$ $m$--flat  valid. 

Now we prove  (completeness) $\implies$. The argument is the standard textbook argument. Let $\A$ and $T$ be such that $\A=\Fm_T$, and assume that
$\phi\in \L(A)^m$ is $\A$ $m$---valid.  Assume for contradiction that 
it is not the case that $T\vdash _m \phi$. Then $T\cup \{\neg \phi\}$ is $m$--consistent, 
so it has an $m$--flat model $M$, hence $M$ is an $m$--flat representation of $\A$ for which there exists $s\in {}C^m(M)$ 
such that $M, s\models \neg \phi$.  
This contradicts that $\phi$ is $\A$ $m$--flat valid.     

In more detail, we have 
$\A=\Fm_T\in  \bold S\Nr_n\CA_m$ and $a=(\neg \psi_T)\neq 0$ in $\A$, 
hence using an $m$--dilation of $\A$, like in the proof of lemma \ref{flat}, one constructs 
an $m$--dimensional hyperbasis of $\A^+$, then one  constructs the required $m$--flat representation $M$ of $\A$ from this $m$--dimensional hyperbasis, via $\theta$ say,
that is $\theta:\A\to \wp(1^M)$ is a homomorphism, 
and $\theta(a)\neq 0$. 
\end{proof}

\section{Complete representations}

We generalize the characterization in \cite[Theorem 5.3.6]{Sayedneat}
to the infinite dimensional case addressing complete representations using 
weak set algebras, via complete neat embeddings dealing with algebras that can possibly have 
uncountably many atoms, witness theorem \ref{complete3}. 
For finite dimensions, we characterize complete representability of algebras having countably many atoms, 
via the games $G_{\omega}$ and $F^{\omega}$.

Then we recall a result from \cite{SL} 
that marks the boundaries of our result obtained in the first item
of our main theorem \ref{main}.

Let $\alpha$ be an arbitary ordinal. A {\it weak space of dimension $\alpha$} is a set $V$ of the form $\{s\in {}^{\alpha}U: |\{i\in \alpha: s_i\neq p_i\}|<\omega\}$
where $U$ is a non--empty set and $p\in {}^{\alpha}U$.
We denote $V$ by $^{\alpha}U^{(p)}$. 

$\sf GwKs_{\alpha}$ is the class of {\it generalized weak set algebras}
as defined in \cite[Definition 3.1.2, item (iv)]{HMT2} for $\CA$s; adapted the obvious way to other $\K$s.  We have $\sf GwKs_{\alpha}={\bf SP}{\sf WKs}_{\alpha}$,
where $\sf WKs_{\alpha}$ denotes the class of weak set algebra of dimension $\alpha$. The top elements of $\sf GwKs_{\alpha}$s are
{\it generalized weak spaces} of dimension $\alpha$; these are disjoint unions of weak spaces of the same dimension.

For $\alpha\geq \omega$, we let $\sf DKc_{\alpha}$ denote the class of {\it dimension complemented $\K_{\alpha}$s}, 
so that $\A\in {\sf DKc}_{\alpha}\iff\ \alpha\sim \Delta x$ is infinite for every $x\in \A$. 

\begin{definition}\label{omit}
Let $\alpha$ be any ordinal.
Then $\A\in \K_{\alpha}$ is {\it completely representable} if there exists $\B\in \sf GwKs_{\alpha}$ and 
an  isomorphism $f:\A\to \B$ (respecting the $\K$ concrete operations on $\wp(V)$)
such for all $X\subseteq \A$, $f(\prod X)=\bigcap f(X)$ whenever $\prod X$ exists.
\end{definition}
For finite $\alpha$ the above definition is the usual one \cite{HH}. 
We use the fact that, like in relativized representations,  complete representations are atomic ones \cite{HH}; the characterization in {\it opcit} 
lifts easily to the infinite dimensional case. In this context, 
as expected, if $\A$ is an atomic $\K_{\alpha}$, then an isomorphism $f:\A\to \wp(V)$, 
$V$ a generalized weak space, is {\it atomic}, if $\bigcup_{a\in \At\A}f(a)=V$. For a class $\bold K$ of $\sf BAO$s, 
we write 
$\bold K^{\sf atc}$ for $(\bold K\cap \At)^{\sf ad},$ that is, $\bold K^{\sf atc}$ is 
the class of atomic completely additive algebras in $\bold K$.

\begin{theorem}\label{complete} Let $\alpha$ be any countable ordinal (possibly infinite) and $\A\in \K_{\alpha}$.
If $\A$ is  atomic and has countably many atoms,
then $\A$ is completely representable $\iff \A\in \bold S_c\Nr_{\alpha}\K_{\alpha+\omega}^{\sf atc}$. The implication $\implies$ holds without any restriction on 
the cardinality of the atoms. 
\end{theorem}
\begin{proof} Assume that $\A\in \bold S_c\Nr_{\alpha}\K_{\alpha+\omega}^{\sf ad}$ has countably many atoms. Then $\A\subseteq_c \Nr_{\alpha} \D$
for some $\D\in \K_{\alpha+\omega}$, where $\D$ is completely additive. We want to show that $\A$ is completely representable.
We can assume that $\A=\Tm\At\A$, because $\Tm\At\A$ is dense in $\A$, hence $\Tm\At\A\subseteq_ c\A\subseteq_c \Nr_{\alpha}\D$,
so that $\Tm\At\A\subseteq_c \Nr_{\alpha}\D$.
Furthermore, a complete representation  of  $\Tm\At\A$ induces a
complete representation of $\A$. In particular, we can assume that $\A$ is countable.
Let $\B=\Sg^{\D}\A$.  Then  $\B$ is countable, too,  
$\A\subseteq \Nr_n\B$ and $\B\in \sf DKc_{\alpha+\omega}$. This can be proved inductively, by noting that  $\Sg^{\B}\A$ 
is generated by elements whose dimensions sets are included in $\alpha$, 
and for $i,j<\alpha+\omega$ and $x\in \B$, we have $\Delta {\sf c}_ix\subseteq \Delta x$, and if present in the signature: 
$\Delta {\sf d}_{ij}=\{i,j\}$, $\Delta {\sf s}_i^jx\subseteq \Delta x\cup \{i,j\}$ and 
$\Delta {\sf s}_{[i,j]}x\subseteq \Delta\cup \{i,j\}$, so that $|\alpha+\omega\sim \Delta x|\geq \omega$ for all 
$x\in \Sg^{\B}\A$.

We use \cite[Theorem 3.2.4]{Sayed}, replacing $\Fm_T$ in {\it op.cit} by $\B$, so we need to make sure 
that certain suprema exist when we do not have diagonal elements. We cannot assume {\it a priori} that they do, witness example \ref{counter}. In fact, 
the available 
proofs we have seen depend on the existence of diagonal elements. 
We assume without loss that all substitution operations corresponding to replacements are in the signature. 

Let $X=\At\A$. We claim that for any $i< j\in \alpha+\omega$, $\sum {\sf s}_i^j{}^{\B}X=1$. If not, then there exists $b\in \B$, $b<1$, 
such that ${\sf s}_i^j{}^{\B}x\leq b$ 
for all $x\in X$. But then  $b\in \D$ is an upper bound of $\{{\sf s}_i^j{}^{\D}x: x\in X\}$ in $\D$. But, by assumption, 
$\D$ is completely additive, hence $1={\sf s}_i^j{}^{\D}(\sum X)=\sum_{x\in X}{\sf s}_i^j{}^{\D}x\leq b<1,$ 
which is impossible.
Let $Y$ be the set of co--atoms of $\A$, then by the above we readily obtain 
$\prod {\sf s}_i^j(Y)=0$ in $\B$ for all $i< j\in \alpha+\omega$.

We also need to make sure that for all $k\in \beta=\alpha+\omega$,
${\sf c}_kx=\sum_{l\in \beta}{\sf s}_l^kx$ in $\B$, a sum that is valid when we have  diagonal elements \cite[Theorem 1.11.6]{HMT2}. 

As it happens, the proof for $\CA$s does not use diagonal elements; it lifts as is to Pinter's algebras. Fix $k<\beta$. Then for all $l\in \beta$, we have ${\sf s}_l^kx\leq {\sf c}_kx$.
Conversely, assume that $y\in \B$ is an upper bound for $\{{\sf s}_l^kx: l\in \beta\}$. 
Let $l\in \beta\sim (\Delta x\cup \Delta y)$; such an $l$ exists, because $\B$ is dimension 
complemented and $x, y\in \B$.
Hence, we get that ${\sf c}_lx=x$ and ${\sf c}_ly=y$.  But then ${\sf c}_l{\sf s}_l^kx\leq y$, and so ${\sf c}_kx\leq y$.
We have proved that the above suprema holds in all cases. 
 
Now by complete additivity of 
the substitution operations ${\sf s}_{\tau}$s, $\tau$ a finite transformation on $\beta$ 
(${\sf s}_{\tau}$ is defined as a composition of finitely many completely additive 
substitution operations  corresponding  to replacements),
one has all the machinery of the the proof of theorem \cite[Theorem 3.2.4]{Sayed}, so one can omit the one non--principal type of co--atoms,
getting the required complete representation.
In more detail, let $S$ be the Stone space of $\B$ and let $V={}^{\beta}\beta^{(Id)}$. 
For $b\in \B$, let $N_b$ denote the basic clopen set $\{F\in \Uf\B: b\in F\}$. Let $\bold G=\bigcup_{i\in \beta}\bigcup_{x\in \B}\bold G_{j,x}$ and 
$\bold H=\bigcup_{\tau\in V}\bold H_{\tau}$, where
$\bold G_{j, x}=N_{{\sf c}_jx}\sim \bigcup _{i\notin \Delta x}N_{{\sf s}_i^jx}$ and
$\bold H_{\tau}=S\sim \bigcup_{x\in \At\B} N_{{\sf s}_{\tau}x}$ are nowhere dense sets in the Stone topology. 
Then given $0\neq a\in  \B$, one can find using 
an ultrafilter  $F\in (S\sim \bold H\cup \bold G)\cap N_a$, since $\bold H\cup \bold G$ is meager (a countable union of nowhere dense sets), 
so by the Baire Category theorem for compact Hausdorff spaces, $S\sim (\bold H\cup \bold G)$ is dense in the Stone topology. 
This ultrafilter can be used to build a  complete representation  $f:\B\to \wp(W)$ 
such that $f(a)\neq 0$; $f$ is defined, using the notation of \cite[p.216]{Sayed}, via  $f(x)=\{\bar{\tau}\in W: {\sf s}_{\tau}^{\B}x\in F\}$ $(x\in \B)$.
Here $W=V/\sim$, where $\sim$ is the equivalence relation defined on $V$ by 
$s\sim t\iff {\sf d}_{s(i) t(i)}\in F$ for all $i<\beta$.

Here the countability condition is essential, 
witness the last item of 
theorem \ref{main}. \\

Now we prove $\implies$. Here we do not require that $\A$ has countably many atoms. 
The proof is very similar to the proof of \cite[Theorem 29]{r}, except that we deal with weak set algebras when dealing with infinite dimensions. 
Furthermore, we have to check that the dilations defined during the proof are not merely atomic, but also 
{\it completely additive}. 
Assume that $M$ is the base of a complete representation of $\A$, whose
unit is a weak generalized space,
that is, $1^M=\bigcup {}^nU_i^{(p_i)}$ $p_i\in {}^{\alpha}U_i$, where $^{\alpha}U_i^{(p_i)}\cap {}^{\alpha}U_j^{(p_j)}=\emptyset$ for distinct $i$ and $j$, in some
index set $I$, that is, we have an isomorphism $t:\B\to \C$, where $\C\in \sf GKs_{\alpha}$ 
has unit $1^M$, and $t$ preserves arbitrary meets carrying
them to set--theoretic intersections.

For $i\in I$, let $E_i={}^{\alpha}U_i^{(p_i)}$. Take  $f_i\in {}^{\alpha+\omega}U_i^{q_i}$ where $q_i\upharpoonright \alpha=p_i$
and let $W_i=\{f\in  {}^{\alpha+\omega}U_i^{(q_i)}: |\{k\in \alpha+\omega: f(k)\neq f_i(k)\}|<\omega\}$.
Let ${\C}_i=\wp(W_i)$. Then $\C_i$ is atomic; indeed the atoms are the singletons. 
Since the sums in $\C_i$ are just unions, then $\C_i$ 
is completely additive. 

Let $x\in \Nr_{\alpha}\C_i$, that is ${\sf c}_ix=x$ for all $\alpha\leq i<\alpha+\omega$.
Now if  $f\in x$ and $g\in W_i$ satisfy $g(k)=f(k) $ for all $k<\alpha$, then $g\in x$.
Hence $\Nr_{\alpha}\C_i$
is atomic;  its atoms are $\{g\in W_i:  \{g(i):i<\alpha\}\subseteq U_i\}.$

Define $h_i: \A\to \Nr_{\alpha}\C_i$ by
$h_i(a)=\{f\in W_i: \exists a'\in \At\A, a'\leq a;  (f(i): i<\alpha)\in t(a')\}.$
Let $\D=\bold P _i \C_i$. Let $\pi_i:\D\to \C_i$ be the $i$th projection map.
Now clearly  $\D$ is atomic, because it is a product of atomic algebras,
and its atoms are $(\pi_i(\beta): \beta\in \At(\C_i))$. Furthermore, $\D$ is completely additive, since its components are. 
Now  $\A$ embeds into $\Nr_{\alpha}\D$ via $J:a\mapsto (\pi_i(a) :i\in I)$. If $x\in \Nr_{\alpha}\D$,
then for each $i$, we have $\pi_i(x)\in \Nr_{\alpha}\C_i$, and if $x$
is non--zero, then $\pi_i(x)\neq 0$. By atomicity of $\C_i$, there is an $\alpha$--ary tuple $y$, such that
$\{g\in W_i: g(k)=y_k\}\subseteq \pi_i(x)$. It follows that there is an atom
of $b\in \A$, such that  $x\cdot  J(b)\neq 0$, and so the embedding is atomic, hence complete.
We have shown that $\A\in \bold S_c\Nr_{\alpha}\K_{\alpha+ \omega}^{\sf atc}$
and we are done.  It readily follows that $\A$ is atomic, 
because it is dense in the atomic dilation.
\end{proof}

Given any cardinal $\kappa$, possibly infinite, $\A\in \K_m$, with $m<\omega$, 
then (complete) $\kappa$--square representations of $\A$ can be easily 
defined \cite[Definition 17.22]{HHbook}.

If $\omega\leq \kappa<\lambda$, an algebra having a complete $\lambda$--square
representation, may not have a complete $\kappa$--square one. 
The rainbow algebra of dimension $n$, for any $2<n<\omega$, $\A=\PEA_{\lambda, \kappa}$ witnesses this. 
Any complete $\kappa$--square representation of $\A$ will force a `$\kappa$ red clique' 
indexed by the $\lambda$ greens which is impossible because the  indices of reds must match within the red clique.

\begin{theorem}\label{complete2} 
If $m<\omega$, and  $\A\in \bold S_c\Nr_{m}\K_{\omega}^{\sf ad}$ is atomic, then \pe\ has a \ws\ in both $F^{\omega}$ and $G_{\omega}$.
Furthermore, $\A$ has a complete $\omega$--square representation, an ordinary representation, and a complete 
$\omega$--square representation. If $\A$ has countably many atoms, then it has a complete (ordinary) representation.
\end{theorem}
\begin{proof} Let $m<\omega$ and $\A\in \bold S_c\Nr_{m}\K_{\omega}^{\sf ad}$, so that $\A\subseteq_c\Nr_m\D$, $\D\in \K_{\omega}$
and $\D$ is completely additive. The
\ws' s for $F^{\omega}$ and $G_{\omega}$
are like in lemma \ref{Thm:n}; \pe\ plays networks $N$, such $\nodes(N)\subseteq \omega$, maintaining
the property that $\widehat{N}\neq 0$, using the $\omega$--dilation $\D$ which is completely additive,
so lemma \ref{Thm:n} applies.

Now we focus on $\CA$s.
The other cases are completely analogous undergoing the 
obvious modifications.  Let $\A\in \CA_m$ and assume that \pe\ has a \ws\ in $G_{\omega}(\At\A)$. Extending the notion of finite dimensional basis 
in the obvious way (witness too the proof of the first part of the item (4) of theorem \ref{main}), we will build an 
{\it $\omega$--dimensional basis for $\A$}, from which the $\omega$--
square complete representation of $\A$  
will be built.  An $\omega$--dimensional basis $B$ for $\A$, consists of a set of $n$--dimensional networks on $\A$ 
satsifying the closure properties for a finite-dimensional basis, but here for
$N\in B$, we have $\nodes(N)\subseteq_{\omega} \omega$.
 
For a network $N$ and a map $\theta:\omega\to N$, 
let $H$ be the set of all such $N\theta$ (as defined in the first item of lemma \ref{Thm:n}), where 
$N$ occurs in some play of $G_{\omega}(\At\A)$ 
which \pe\ uses her \ws. 

Then, it can be checked without too much difficulty that $H$ is an $\omega$--dimensional basis for $\A$. 
Given a hypergrah $M$  with $m$--hyperedge relation $E$, then
$C\subseteq M$ is a clique if for all injective map $s: m\to C$, we have $s\in  E$. 
Then a complete $\omega$--square representation  can be obtained in a step--by--step way, requiring inductively
in step $t$, that for any finite clique $C$ of $M_t$, where the hyperedge relation is 
$1^{M_t}$, $|C|<\omega$, there is
a network $N$ in the base, and an embedding $\theta :N\to M_t$ (of hypergraphs) 
such that $\rng \theta\supseteq C$.
Then $M_t$  satisfies the first three items (by replacing $M$ by $M_t)$ and the limit hypergraph, 
which is the base required of the representation
will be fixed (along the way) to satisfy \cite[Proposition 13.37, Lemma 17.24]{HHbook}:
\begin{enumerate}

\item each $m$--hyperedge of of $M$ is labelled by an atom of $\A,$ 

\item  $M(\bar{x})\leq {\sf d}_{ij} \iff x_i=x_j$, 

\item for any clique $\bar{x}\in M$ of arbitrary finite length $>m,$
there is a unique $N\in H$, such that $\bar{x}$ is labelled by $N$, and we write this as
$M(\bar{x})=N,$\\

\item if $l\geq m$ is finite, $x_0,\ldots, x_{l-1}\in M$ and $M(\bar{x})=N\in H$,
then for all $i_0,\ldots, i_{m-1}<l$, we have $(x_{i_0}\ldots x_{i_{m-1}})$
is a hyperedge, and
$M(x_{i_0}, x_{i_1}, \ldots,  x_{i_{m-1}})=N(i_0,\ldots, i_{m-1})\in \At\A,$ 

\item  $M$ is {\it symmetric} (closed under substitutions; $N\in M\implies N\theta\in M$, any $\theta$),

\item  if $\bar{x}$ is a clique of arbitrary length,  $k<|x|$ and $N\in H$, then $M(\bar{x})\equiv_k N$ $\iff $ 
there is a $y\in M$ such that
$M(x_0,\ldots, x_{k-1}, y, x_{k+1}, \ldots, x_l, \ldots)=N$ (with $y$ in the $k$th place),

\item for every $N\in H$, for every finite $l\geq m$, there are $x_0,\ldots,  x_{l-1}\in M$, $M(\bar{x})=N.$
\end{enumerate}
Then $M$ will be an $\omega$--square complete representation of $\A$,
defined for $r\in \A$ via 
$$M\models r(\bar{x})\iff \bar{x}\text{ is an $m$--hyperedge of $M$ and } M(\bar{x})\leq r.$$
$M$ is a complete representation, because every $m$--hyperedge is labelled by an atom, so that $M$ is an atomic, 
hence complete representation. 

Now fix $2<m<\omega$. Then it is not hard to show that countable algebras having $\omega$--square 
representations are representable. Also, it is not too hard to show that the class of $\K_m$s 
having $\omega$--square representations is a discriminator variety, and 
that it coincides
with $\sf RK_m$ on countable simple algebras, hence the two varieties are equal.  
So if  $\A\in \K_m$ (regardless of its cardinality)  has an $\omega$--square representation, then
it is representable, witness too, item (2) of theorem \ref{squareflat}. This, as mentioned above, is not true for complete $\omega$--square representations. 

The last part follows immediately from \cite[Theorem 3.3.3]{HHbook}, by noting that  the $\kappa$ in {\it opcit} is, in our present situation,  
the least infinite cardinal  $\omega$; then one uses the 
fact that \pe\ has a \ws\ in $G_{\omega}$ established 
above.

\end{proof}

$F^{\omega}$ and $G_{\omega}$ are the same game, 
they test complete  $\omega$--square  representations.
When truncated to $m<\omega$ nodes, $F^m$ tests complete $m$--flatness, while $G^m_{\omega}$
tests  complete $m$--squareness,  witness lemma \ref{Thm:n}. In this sense, 
$G_{\omega}$ is a limiting cases,  when we deal with countable algebras.

We have proved the following corollary which is the $\K$ analogue of \cite[Theorem 29]{r} counting in infinite dimensions in the first part.
\begin{corollary}
Let $\alpha$ be an arbitary countable ordinal. Then $\A\in \K_{\alpha}$ is completely representable 
$\implies \A\in \bold S_c\Nr_{\alpha}\K_{\alpha+\omega}^{\sf atc}\implies  \A$ is atomic and $\A\in \bold S_c\Nr_{\alpha}\K_{\alpha+\omega}^{\sf ad}$.
Furthermore, if  $\alpha<\omega$, and $\A$ is completely representable, then \pe\ has  a \ws\ in $G_{\omega}$ and $F^{\omega}.$ 
All  reverse implications hold, if $\A$ has countably many atoms.
\end{corollary}
We will show, in item (4) of theorem \ref{main},  that the second implication in the second line cannot be reversed if $\A$ has uncountably many atoms.
But the first implication in the same line can be reversed, modulo an extra condition on the dilation, as shown in the first item of the next theorem \ref{complete3}, addressing 
only $\QEA$s that can possibly 
have uncountably many atoms. 
We do not know whether this extra 
condition can be omitted. In the second item of the same theorem other conditions are imposed on dilations which are strict expansions of 
quasi polyadic algebras (with and without equality), 
enforcing complete representability for finite dimensional polyadic and polyadic equality algebras, that are complete 
subneat reducts of such dilations.
But first a lemma:

\begin{lemma}\label{join} 
\begin{enumarab}
\item Let $\alpha$ be an ordinal and $\A\in \K_{\alpha}$. 
If $\A=\Nr_{\alpha}\D$, where $\D\in {\sf DKc}_{\alpha+\omega}$, then $\A\subseteq_c \D.$ 
\item \cite[Lemma 2.16]{HHbook} Let $\alpha$ be any ordinal. If $\A, \D\in \K_{\alpha}$ and 
$\A\subseteq_c \D$ and $\D$ is atomic, then $\A$ is atomic. 
\item If $\A, \D\in \K_{\alpha}$, $\A$ is atomic, then $\A\subseteq_c \D\iff$ for all non--zero $d\in \D$, 
there exists $a\in \At\A$ such that $a\cdot d\neq 0$ \cite[Lemma 2.17]{HHbook}.
If $\A$ is dense in $\B$ and $\B$ is atomic then $|\At\A|\geq |\At\B|$. 
\item \cite[Lemma 17]{r} If $\D$ is completely representable, then so is $\A$.
\end{enumarab}
\end{lemma}
\begin{proof}

(1) Assume that $S\subseteq \A$ and $\sum ^{\A}S=1$, and for contradiction, that there exists $d\in \D$ such that
$s\leq d< 1$ for all $s\in S$. Since $\D$ is dimension complemented, then 
 we can conclude that  $J=\Delta d\sim \alpha$ is finite; we assume that it is not empty.
So let  $J=\{m_1, \ldots, m_n\}$ $(m_i\in \alpha+\omega\sim \alpha)$.
Put  $t=-{\sf c}_{m_1}\ldots {\sf c}_{m_n}(-d)$.
We claim that $t\in \Nr_{\alpha}\D$ and $s\leq t<1$ for all $s\in S$. This contradicts
$1=\sum^{\A}S$.
The first required follows from the fact that all indices in $\alpha+\omega\sim \alpha$ 
that occur in $d$ are cylindrified using easy manipulations of cylindric axioms.
In more detail, let $\beta=\alpha+\omega$. 
Let $i\in \beta \sim \alpha$. If $i\notin \{m_1, \ldots m_n\}$ then using that ${\sf c}_i(x\cdot {\sf c}_iz)={\sf c}_ix\cdot {\sf c}_iz$,
we have ${\sf c}_it=t$.
Indeed, $y\in \A=\Nr_{\alpha}\B$, so $\Delta y\subseteq \alpha$, hence ${\sf c}_iy=y$ and also ${\sf c}_i-{\sf c}_{m_1}=-{\sf c}_{m_1}$.
Now assume that $i\in \{m_1, \ldots, m_n\}$. Then
\begin{align*}
{\sf c}_it&={\sf c}_i(-{\sf c}_{m_1}\ldots {\sf c}_{m_n} (-d))\\
 &=  {\sf c}_i-{\sf c}_{m_1}\ldots {\sf c}_{m_n} (-d)\\
&=  {\sf c}_i\cdot -{\sf c}_i{\sf c}_{m_1}\ldots {\sf c}_{m_n}( -d)\\
&= -{\sf c}_i{\sf c}_{m_1}\ldots {\sf c}_{m_n}( -d)\\
&=-{\sf c}_{m_1}\ldots {\sf c}_{m_n}( -d)\\
&=t
\end{align*}
In the first line we use
${\sf c}_iy=y$. From first to second line we use ${\sf c}_i(x\cdot z)={\sf c}_ix\cdot z$. From second to third we use
idempotency, namely, ${\sf c}_i{\sf c}_ix={\sf c}_ix$ and finally, we use
that cylindrifiers are complemented operators, namely, ${\sf c}_i(-{\sf c}_ix)=-{\sf c}_ix.$
We have proved that $t\in \A=\Nr_{\alpha}\B$ and we are 
done wth the first part of the required.

If $s\in S$, we show that $s\leq t$.
By $s\leq d$, we have  $s\cdot -d=0$.
Hence $0={\sf c}_{m_1}\ldots {\sf c}_{m_n}(s\cdot -d)=s\cdot {\sf c}_{m_1}\ldots {\sf c}_{m_n}(-d)$, so
$s\leq -{\sf c}_{m_1}\ldots {\sf c}_{m_n}( -d)$, hence  $s\leq t$ as required. We finally check that $t<1$. If not, then
$1=-{\sf c}_{m_1}\ldots {\sf c}_{m_n}(-d)$ and so $ {\sf c}_m\ldots {\sf c}_{m_n}(-d)=0$.
But $-d\leq {\sf c}_m\ldots {\sf c}_{m_n}(-d)$,  so $1\cdot -d\leq  {\sf c}_m\ldots {\sf c}_{m_n}(-d)=0.$
Hence $1\cdot -d =0$ and this contradicts that $d<1$.\\

(2) Let $a\in A$ be non-zero. Then since $\D$ is atomic, 
there exists an atom $d\in D$, such that $d\leq a$. Let $F=\{x\in A: x\geq d\}$. Then $F$ is an ultrafilter of $\A$. 
It is clear that $F$ is a filter. Now assume that $c\in A$ and $c\notin F$, then $-c\cdot d\neq 0$, so $0\neq -c\cdot d\leq d$, hence $-c\cdot d=d$, 
because $d$ is an atom in $\B$, thus $d\leq -c$, and we get by definition that $-c\in F$.  We have shown that $F$ is an ultrafilter. 
We now show  that $F$ is a principal ultrafilter in $\A$, 
that is, it is generated by an atom. Assume for contradiction that it is not, so that $\prod^{\A} F$ exists because $F$ is an ultrafilter, 
and $\prod ^{\A}F=0$ because it is non--principal. 
But $\A\subseteq_c \D$, so we obtain $\prod^{\A}F=\prod^{\D}F=0$. This contradicts that $0<d\leq x$ for all $x\in F$.
Thus $\prod^{\A}F=a'$, $a'$ is an atom in $\A$, $a'\in F$ 
and $a'\leq a$, because $a\in F$. \\

(3) Assume that $\A\subseteq_c \D$ and $\A$ is atomic. Let $X=\At\A$.
Then $\sum^{\A}X=\sum^{\D}X=1$. If $b\neq 0$ in $\D$, then $b\cdot \sum^{\D}X=b$, hence there exists $a\in X$ such that $a\cdot b\neq 0$, for else $-b\neq 1$ 
will be an upper bound of $X$ in $\D$ which is imposible.
Conversely, assume that $\sum^{\A}S=1$ and for contradiction that there exists $b'\in \D$, $b'<1$ such that 
$s\leq b'$ for all $s\in S$. Let $b=1-b'$ then $b\neq 0$, hence by assumption there exists an atom  
$a\in \A$ such that $a\cdot b\neq 0$, i.e $a\cdot (1- b')\neq 0$. If $a\leq s$ for some $s\in S$, then $a\leq b'$ which is impossible.
So if  $s\in S$, then $a\cdot s\neq a$, so $a\cdot s<a$ and because $a$ is an atom, we get that $a\cdot s=0$. This happens 
for every $s\in S$, implying that $a=0$. But this is  impossible, because $a$ is an atom.

Now assume further that $\A$ is dense in $\B$ and $\B$ is atomic. Then by density 
for every atom $b\in \B$, there is a 
non--zero $a'\in \A$, such that $a\leq b$ in $\B$. Since $\A$ is atomic, there is an atom $a\in \A$ such that $a\leq a'\leq b$. 
But $b$ is an atom of  $\B$,  and $a$ is non--zero in $\B$, too, so it must be the case that $a=b\in \At\A$.
Thus $\At\B\subseteq \At\A$ and we are done.
\\  

(4) Let $\A\subseteq_c \D$ and 
assume that $\D$ is completely representable.  We want to show that $\A$ is completely representable. 
Let $f:\D\to \wp(V)$ be a complete representation of
$\D$, that is, $f$ is an embedding into the generalized weak set algebra with universe $\wp(V)$ preserving arbitrary suprema carrying them 
to set--theoretic unions. 
We claim that $g=f\upharpoonright \A$ is a complete representation of $\A$. Let $X\subseteq \A$ be such that $\sum^{\A}X=1$. 
Then by $\A\subseteq_c \D$, we have  $\sum ^{\D}X=1$. Furthermore, for all $x\in X(\subseteq \A)$ we have $f(x)=g(x)$, so that 
$\bigcup_{x\in X}g(x)=\bigcup_{x\in X} f(x)=V$, since $f$ is a complete representation, 
and we are done. 
\end{proof}

In the next proof, we use the two following elementary facts about Boolean algebras. We already used the one implication of the first in the proof of the first item 
of theorem \ref{complete}. 
\begin{itemize}
\item If $\B$ is a Boolean algebra, with Stone space $S$,  $X\subseteq \B$, then $\sum^{\B}X=1\iff$ 
$S\sim \bigcup_{x\in X}N_x$ is nowhere dense in the Stone topology; here $N_x=\{F\in \Uf\A: x\in F\}$ is a basic clopen set in $S$. We will use $\implies$.  
\item If $\B$ is atomic, then the set of principal ultrafilters in  $\Uf\B$ 
is dense in the Stone topology. 
\end{itemize}
If $\alpha\geq \omega$ and $\A\in \K_{\alpha}$ then it is always the case that 
${\sf s}_l^kx\leq {\sf c}_kx$ for all $x\in A$ and $k, l<\alpha$. Sometimes ${\sf c}_k$ is the {\it least} upper bound of the set $\{{\sf s}_l^kx: l, k\in \alpha, l\neq k\}$. 
This happens in some significant cases like for example locally finite and dimension complemented algebras of dimension $\alpha$. 
Furthermore, any algebra $\A\in \K_{\alpha}$ (not necessarily dimension complemented) for
which ${\sf c}_kx=\sum {\sf s}_l^kx$ for all distinct $k, l\in \alpha$ and $x\in A$, is representable 
\cite[Last item of Theorem 3.2.11]{HMT2}; this will be also proved in our next theorem under the additional assumption of atomicity. 

If $\A\in \K_n$ has a (finite dimensional) dilation $\D$ of dimension $m<\omega$, $(n<m)$, then
using the argument in the first part of lemma \ref{join}, we get that $\A\subseteq_c \Nr_n\D\iff \A\subseteq_c \D$, 
so, in particular, if $\A$ is atomic 
with set of atoms $X$, and $\A\subseteq_c\Nr_n\D$, then $\sum ^{\D}X=1$. Using the argument
in the second part of lemma \ref{join}, we get that if $\D$ is atomic, then so is $\A$.

But if $\D$ is {\it not dimension complemented} and $\A\subseteq _c \Nr_{\alpha}\D$, we do not know whether  
it follows that $\A\subseteq _c \D$. Lemma \ref{join} does not help here. 
We cannot see why it should be the case that $\Nr_n\D\subseteq _c \D$, because if $d\in D$ then 
$\Delta d\sim n$ could be infinite and so the technique used in {\it op.cit} does not work here. (In principal, another argument might work giving the same result). 
The converse though can be easily shown to be true, that is if 
$\A\subseteq_c \D$, then $\A\subseteq_c \Nr_n\D$,
because if $X\subseteq \A$, is such that $\sum ^{\A}X=1$, then $1=\sum^{\D}X\leq \sum ^{\Nr_n\D}X\leq \sum ^{\A}X=1.$

Let $\alpha$ be an infinite ordinal. 
Let ${\sf CPES}_{\alpha}$ be the class of {\it strong polyadic equality algebras of dimension $\alpha$} as introduced by 
Ferenczi in \cite[Definition 6.3.8]{Fer}. These are like $\sf PEA_{\alpha}$s except that the signature and axiomatization are 
restricted only to finite cylindrifiers. Cylindrfication on infinite subsets of $\alpha$ is not allowed. But all substitution operations
${\sf s}_{\tau}$, $\tau\in {}^{\alpha}\alpha$ are present in the new signature. 
Let ${\sf CPS}_{\alpha}$ be their diagonal free reducts defined by restricting the signature and axiomatization of ${\sf PA}_{\alpha}$
to  finite cylindrifiers. Again here all substitution operations are present.

\begin{theorem} \label{complete3}
Let $\alpha$ be any ordinal.  
\begin{enumarab}
\item Let $\K\in \{\PA, \PEA\}$. If $\A\subseteq \Nr_{\alpha}\D$, $\D\in \K_{\alpha+\omega}\cap \At,$  
$\A\subseteq_ c\D$ and for all $x\in \D$ for all $k<\alpha+\omega$, 
${\sf c}_kx=\sum_{l\in \alpha+\omega} {\sf s}_l^kx$, then $\A$ is completely representable.
\item If $\alpha<\omega$ and $\A\subseteq \Nr_{\alpha}\D$, $\D\in \PA_{\omega}^{\sf atc}$, $\A\subseteq_c \D$, 
then $\A$ is a completely representable $\PA_{\alpha}$. The same result holds for 
$\PEA$s under the weaker assumption that the dilation $\D$ is only atomic. We can replace $\PA_{\omega}$ and $\PEA_{\omega}$ by ${\sf CPS}_{\omega}$ and 
${\sf CPES}_{\omega}$, respectively.   
\end{enumarab}
\end{theorem}

\begin{proof} 

We prove the first item: By $\A\subseteq_c \D$, we get by lemma \ref{join} that $\A$ is atomic because $\D$ is. 
For $\alpha\leq 1$, any atomic algebra in $\bold S_c\Nr_\alpha\K_{\omega}=\sf RK_{\alpha}$ is completely representable; 
in particular, any algebra in $\bold S_c\Nr_\alpha(\K_{\omega}\cap \At)$ 
is completely representable, so we assume that $\alpha>1$. Let $\A\in \K_\alpha$ and $\D\in \K_{\alpha+\omega}\cap \At$ 
be as in the hypothesis.

Before embarking on the technical details, 
we first give the general idea inspired by the idea used in \cite{au}. Here we cannot appeal to the Baire category theorem as we did in the first part 
of the proof of theorem \ref{complete} using \cite[Theorem 3.2.4]{Sayed},   because
$\A$ may have uncountably many atoms. Nevertheless, our proof is still topological using the elementary fact that in atomic Boolean algebras
principal ultrafilters lie outside nowhere dense subsets of the Stone space. The topological approach is not necessary, but it makes the proof shorter.

We represent $\D$ using a principal ultrafilter in $\D$. Because $\D$ is atomic, such principal ultrafilters are dense in 
the Stone space $S$ of $\D$; we have `many' of them.
So we can pick one that preserves a given set of infinite joins, namely those  given in the hypothesis 
(these have to do with eliminating cylindrifiers) and other joins that reflect omitting substituted versions
of the co-atoms carrying them to empty set-theoretic intersections in the concrete representation of $\D$.  
In $\D$ these joins, give rise to nowhere dense sets
in $S$, so that any principal ultrafilter of $\D$ containing the non-zero element 
$c$ will be outside these sets, and will give the required representation of $\D$.  
At least one such ultrafilter exists because there is an atom below $c$.

Now we implement the details of the above outline. Let $c\in \A$ be non--zero, and for brevity we denote $\At\A$ by $X$. We will show that there exists a $\C\in \sf WKs_{\alpha}$ and  
a homomorphism $f:\A\to \C$, such that $f(c)\neq 0$ and $\bigcup_{x\in X}f(x)= 1^{\C}$, by which we will be done.

For brevity, let $\beta=\alpha+\omega$.
Let $V$ be the weak space $^{\beta}\beta^{(Id)}=\{\tau\in {}^\beta\beta: |\{i\in \beta: \tau(i)\neq i\}|<\omega\}$.  
Then every $\tau\in V$ defines a unary substitution operation ${\sf s}_{\tau}^{\D}$, that is a Boolean endomorphism of 
the Boolean reduct of $\D$.   Every $\sf s_{\tau}$ is a finite 
composition of substitution operations corresponding to replacements and transpositions; so they are completely additive, since a 
composition of completely additive operations is completely additive, hence every such ${\sf s}_{\tau}$ is a complete 
Boolean endomorphism of the Boolean reduct of $\D$.

We have $\sum^{\D}X=1$, because $\A$ is atomic, so that $\sum ^{\A}X=1$ and $\A\subseteq_c\D$.
Furthermore, in $\D$ the following joins exist for all $i<\beta$, $x\in D$ and $\tau\in V$ the first by hypothesis and the second by complete additivity of $\tau$: 
\begin{itemize}
\item ${\sf c}_{i}x=\sum^{\D}_{j\in \beta} {\sf s}_j^i x$, 
\item $\sum {\sf s}^{\D}_{\tau}X=1$.
\end{itemize}
Now let $S$ be the Stone space of $\D$, 
whose underlying set is $\Uf \D$. The set of principal Boolean ultrafilters of $\D$ (those generated by the atoms)  
are isolated points in the Stone topology, and they form a dense set 
since $\D$ is atomic.

For $a\in \D$, let $N_a$ denote the clopen set (in $S$) of all Boolean ultrafilters containing $a$.
From the suprema existing in above two items,  we get that 
for all $i\in \beta$, $x\in D$ and $\tau\in V$, the sets
\begin{itemize}
\item $G_{i,x}=N_{{\sf c}_{i}x}\sim \bigcup_{j\in \omega} N_{{\sf s}_j^ix},$
\item $G_{X, \tau}=S\sim \bigcup_{x\in X}N_{{\sf s}_{\tau}x}.$
\end{itemize}
are nowhere dense in $S$. 
Let $F$ be a principal ultrafilter in $S$ containing $c$. 
This ultrafilter exists, since $\D$ is atomic, so there is an atom $a$ below $c$; just take the 
ultrafilter generated by $a$.
Now $N_a=\{F\}$,  the latter is a basic open set, so $F$ lies outside nowhere dense sets in $S$, that is for any such nowhere dense set $T\subseteq S$, say, 
we have $N_a\cap T=\emptyset$.  In particular, $F\notin G_{i, x}$ and 
$F\notin G_{X,\tau}$ for all $i\in \beta, x\in \B$ and $\tau\in V$. 

Assume that $\K=\QEA$. Then we factor out $\beta$ by $E\subseteq \beta\times \beta$ defined as follows: 
$(i, j)\in E\iff {\sf d}_{ij}^{\D}\in F.$
Then it is easy to check, using elementary properties of diagonal elements, and filters that 
$E$ is an equivalence relation on $\beta$.
For $t:m\to \beta/\sim$, $m\leq \beta$, write $t=\bar{s}$, if $s:m\to \beta$ is such that 
$t(i)=s(i)/E$  for all $i<m$. 

Now  define for $a\in A$,
$f(a)=\{\bar{\tau}\in {}^{n}(\beta/E)^{(\bar{Id})}: {\sf s}_{\tau\cup Id}^{\B}d\in F\}.$
We claim that the map $f:\A\to \wp(^{n}(\beta/E)^{\bar{Id}})$ is well defined.
It suffices to show  that for all $\sigma, \tau\in V$, 
if $(\tau(i), \sigma(i))\in E$ for all $i\in \beta$ and $a\in A$, then 
${\sf s}_{\tau}a\in F\iff {\sf s}_{\sigma}a\in F.$  This can be proved by induction on
$|\{i\in \beta: \tau(i)\neq \sigma(i)\}| (<\omega)$.  
If $J=\{i\in \beta: \tau(i)\neq \sigma(i)\}$ is empty, the result is obvious. 
Otherwise assume that $k\in J$. 
We introduce a helpful piece of notation.
For $\eta\in V (={}^{\beta}\beta^{(Id)})$, let  
$\eta(k\mapsto l)$ stand for the $\eta'$ that is the same as $\eta$ except
that $\eta'(k)=l.$ 
Now take any 
$\lambda\in  \{\eta\in \beta: (\sigma){^{-1}}\{\eta\}= (\tau){^{-1}}\{\eta\}=
\{\eta\}\}\smallsetminus \Delta a.$
Recall that $\Delta a=\{i\in \beta: {\sf c}_ix\neq x\}$ and that $\beta\setminus \Delta a$
is infinite because $\Delta a\subseteq n$, so such a $\lambda$ exists. Now 
we freely use properties of substitutions for cylindric algebras.
We have by \cite[1.11.11(i)(iv)]{HMT2}
(a) ${\sf s}_{\sigma}x={\sf s}_{\sigma k}^{\lambda}{\sf s}_{\sigma(k\mapsto \lambda)}x,$
and (b)
${\sf s}_{\tau k}^{\lambda}({\sf d}_{\lambda, \sigma k}\cdot  {\sf s}_{\sigma} x)
={\sf d}_{\tau k, \sigma k} {\sf s}_{\sigma} x,$
and (c)
${\sf s}_{\tau k}^{\lambda}({\sf d}_{\lambda, \sigma k}\cdot {\sf s}_{\sigma(k\mapsto \lambda)}x)
= {\sf d}_{\tau k,  \sigma k}\cdot {\sf s}_{\sigma(k\mapsto \tau k)}x,$
and finally (d)
${\sf d}_{\lambda, \sigma k}\cdot {\sf s}_{\sigma k}^{\lambda}{\sf s}_{{\sigma}(k\mapsto \lambda)}x=
{\sf d}_{\lambda, \sigma k}\cdot {\sf s}_{{\sigma}(k\mapsto \lambda)}x.$
Then by (b), (a), (d) and (c), we get,
\begin{align*}
{\sf d}_{\tau k, \sigma k}\cdot {\sf s}_{\sigma} x
&=  {\sf s}_{\tau k}^{\lambda}({\sf d}_{\lambda,\sigma k}\cdot {\sf s}_{\sigma}x)\\
&={\sf s}_{\tau k}^{\lambda}({\sf d}_{\lambda, \sigma k}\cdot {\sf s}_{\sigma k}^{\lambda}
{\sf s}_{{\sigma}(k\mapsto \lambda)}x)\\
&={\sf s}_{\tau k}^{\lambda}({\sf d}_{\lambda, \sigma k}\cdot {\sf s}_{{\sigma}(k\mapsto \lambda)}x)\\
&= {\sf d}_{\tau k,  \sigma k}\cdot {\sf s}_{\sigma(k\mapsto \tau k)}x.
\end{align*}
By $F$ is a filter and $\tau k E\sigma k,$ we conclude that
$${\sf s}_{\sigma}x\in F \iff{\sf s}_{\sigma(k\mapsto \tau k)}x\in F.$$
The conclusion follows from the induction hypothesis.

We have proved that $f$ is well defined.
We now check that $f$ is a homomorphism, i.e. it 
preserves the operations.
It is straightforward  to check that $f$ preserves diagonal elements.  The relation $E$ is defined exactly for that.
Preservation of the Boolean operations is equally straightforward using the properties of $F$.
Preservation of cylindrifiers follows from how $F$  was picked. The ultrafilter $F$ `eliminates cylindrifiers' 
in the sense that:  for all $i<\beta$ (in particular, for $i<n$), for all $a\in \D$ (in particular, for $a\in \A$), 
if ${\sf c}_ia\in F$, then there exists $j\notin \Delta a$, 
such that ${\sf s}_i^j{}^{\D}a\in F$. We provide more details.
For $\sigma\in {}^{n}\beta^{(Id)}$, write $\sigma^+$ for $\sigma \cup Id_{\beta\sim n}\in {}^{\beta}\beta^{(Id)}.$ 
\begin{itemize}
\item Boolean operations:  Since $F$ is maximal we have  
$\bar{\sigma}\in f(x+y)\iff {\sf s}_{\sigma^+}(x+y)\in F\iff {\sf s}_{\sigma^+}x+{\sf s}_{\sigma^+}y\in F\iff 
{\sf s}_{\sigma^+} x \text { or } {\sf s}_{\sigma^+} y\in F\iff  
\bar{\sigma}\in f(x)\cup f(y).$

We now check complementation.
$$\bar{\sigma} \in f(-x) \iff {\sf s}_{\sigma^+}(-x)\in F \iff-{\sf s}_{\sigma^+x}\in F 
\iff {\sf s}_{\sigma^+} x\notin F \iff \bar{\sigma}\notin f(x).$$

\item Diagonal elements: Let $k,l<n$. Then we have:   
$\sigma\in f{\sf d}_{kl}\iff {\sf s}_{\sigma^{+}}{\sf d}_{kl}\in F \iff
{\sf d}_{\sigma k, \sigma l}\in F
\iff \sigma k E \sigma l \iff
\sigma k/E=\sigma l/E \iff 
\bar\sigma(k)=\bar\sigma(l)
\iff\bar{\sigma} \in {\sf d}_{kl}.$

\item Cylindrifications: Let $k<n$ and $a\in A$.  Let $\bar{\sigma}\in {\sf c}_kf(a)$. 
Then for some $\lambda\in \beta$, we have
$\bar{\sigma}(k \mapsto  \lambda/E)\in f(a)$ 
hence 
${\sf s}_{\sigma^+(k\mapsto \lambda)}a\in F$ 
It follows from  the inclusion $a\leq {\sf c}_ka$ that , 
${\sf s}_{\sigma^+(k\mapsto \lambda)}{\sf c}_ka \in F$
so ${\sf s}_{\sigma^+}{\sf c}_ka\in F.$
Thus ${\sf c}_kf(a)\subseteq f({\sf c}_ka.)$

We prove the other more difficult inclusion that uses the condition of eliminating cylindrifiers.
Let $a\in A$ and $k<n$. Let $\bar\sigma'\in f{\sf c}_ka$ and let $\sigma=\sigma'\cup Id_{\beta\sim n}$. Then 
${\sf s}_{\sigma}^{\D}{\sf c}_ka={\sf s}_{\sigma'}^{\D}{\sf c}_ka\in F.$
Let $\lambda\in \{\eta \in \beta:  \sigma^{-1}\{\eta\}=\{\eta\}\}\smallsetminus \Delta a$, such a $\lambda$ exists because $\Delta a$ is finite, and 
$|\{i\in \beta:\sigma(i)\neq i\}|<\omega.$
Let $\tau=\sigma\upharpoonright \beta\smallsetminus \{k,\lambda\}\cup \{
(k,\lambda), (\lambda,k)\}.$
Then in $\D$, we have
${\sf c}_{\lambda}{\sf s}_{\tau}a={\sf s}_{\tau}{\sf c}_ka={\sf s}_{\sigma}{\sf c}_ka\in F.$
By the construction of $F$, there is some $u(\notin \Delta({\sf s}_{\tau}^{\D}a))$  
such that ${\sf s}_{u}^{\lambda}{\sf s}_{\tau}a\in F,$ so ${\sf s}_{\sigma(k\mapsto u)}a\in F.$
Hence $\sigma(k\mapsto u)\in f(a),$ from which we 
get that  $\bar{\sigma}'\in {\sf c}_k f(a)$.

\item Substitutions: Direct since substitution operations are Boolean endomorphisms

\end{itemize}
We also have $f(c)\neq 0$, because clearly $\bar{Id}\in f(c)$.  
Now it remains to show that the hitherto defined non--zero homomorphism 
$f$ is an atomic, hence, a complete representation.

We have $F\notin G_{X,\tau}$ for every $\tau\in V,$ 
hence for every $\bar{s}\in {}^{\alpha}(\beta/E)^{(\bar{Id})}$, 
there exists $x\in X(=\At\A)$, such that ${\sf s}_{s\cup Id_{\beta\sim \alpha}}^{\B}x\in F$, 
from which we get the required, namely, that  
$\bigcup_{x\in X}f(x)={}^\alpha(\beta/E)^{(\bar{Id})}.$

If $\K=\QA$ then we let $E$ be the identity relation and proceed as above.

Now for the second item:  We denote the finite ordinal $\alpha$ by $n$ $(<\omega)$.  
Assume that $\A\subseteq_c \A'$, $\A\subseteq \Nr_n\A'$,  $\A'\in \PA_{\omega}$ is atomic and completely additive.
Then by \cite{au} $\A'$ is completely representable. Hence there exists an isomorphism $f:\A'\to \wp(V)$, where $V$ is a generalized space, $f$ 
preserves the polyadic operations, such
that  $\bigcup_{x\in \At\A'}f(x)=V$. Let $V=\bigcup_{i\in I}{}^{\omega}U_i$ where $I$ is a non-empty indexing set, and for $i\neq j\in I$, 
$U_i\cap U_j=\emptyset$.   Let $V'=\bigcup_{i\in I}{}^{n}U_i$. For $a\in \A$, define $g(a)=\{s\in V': (\exists s'\in f(a))(s'\upharpoonright n=s)\}$. 
It is straightforward to check that $g$ is  a non--zero 
homomorphism. Preservation of the polyadic operations is straightforward.
To show that $g$ is non-zero, let $a\in \A$ be non-zero. Let $s\in f(a)$. 
Then $s\upharpoonright n\in g(a)$; we have shown that $g$ is actually an embedding. 

Now $f$ is a complete representation, hence it is an atomic one, so if $s\in {}^nU$, then there exists $\alpha\in \At\A'$ such that
$s\cup Id\in f(\alpha)$.
We will show that $g$ is also an atomic, hence a complete representation of $\A$.
Assume, for contradiction, that there exists $s\in {}^nU$, such that $s\notin g(\alpha)$ for each $\alpha\in \At\A$.  But $f$ is an atomic representation of $\A'$, hence
$s\cup Id\in f(\beta)$ for some atom $\beta$ of $\A'$. 
Let $G=\{b\in \A: b\geq \beta\}$. Then clearly $\prod^{\A}G=0$ and $\prod^{\A'}G\geq \beta$, which contradicts
that $\A\subseteq_c \A'$.  So $\bigcup_{a\in \At\A}g(a)={}^nU$ and 
we have proved that $g$ is an atomic, hence a complete representation of $\A$. 

Now we deal with $\PEA$s. The ordinal $n$ remains to be finite. Assume $\A\in \PEA_{n}$, $\A\subseteq \Nr_n\A'$, $\A\subseteq_c \A'$ 
and $\A'\in \PEA_{\omega}\cap \At$. Since $\A'\in \PEA_{\omega}$, then it is completely additive.
We know from the above argument that $\Rd_{pa}\A'$ is completely representable, and that the restriction of this complete representation, as defined above, to $\Rd_{pa}\A$ 
gives a complete representation of $\Rd_{pa}\A$. 

To represent the diagonal elements in $\A$, we need to know for a start {\it how the diagonal free reduct of the completely additive atomic dilation 
$\A'$ is completely represented.} 
Let $\C$ be this reduct, so that in particular $\C\in \PA_{\omega}$, and let $c\in \C$ be non-zero.
We will find (identifying set algebras with their domains), a complete $\PA_{\omega}$ representation $f:\C\to \wp({}^{\omega}U)$, 
for some non-empty set $U$, such that $f(c)\neq 0$. 

We use the argument in \cite{au} which freely uses the terminology in \cite{DM}. 
Let $\mathfrak{m}$ be the local degree of $\C$, $\mathfrak{c}$ its effective cardinality and 
$\mathfrak{n}$ be any cardinal such that $\mathfrak{n}\geq \mathfrak{c}$ 
and $\sum_{s<\mathfrak{m}}\mathfrak{n}^s=\mathfrak{n}$. 
Then there exists an atomic  $\B\in \PA_{\mathfrak{n}}$ such that 
that $\C=\Nr_{\omega}\B$, cf. \cite[theorem 3.10]{DM}, 
and the local degree of $\B$ is the same as that of $\C$.
Furthermore, because $\A'\in \PEA_{\omega}$, we have $\B\in \PEA_{\mathfrak{n}}$. One can define for all $i<j<\mathfrak{n}$, the diagonal element ${\sf d}_{ij}^{\B}$ as in 
\cite[Theorem 5.4.17]{HMT2}; but for a while we concentrate only on 
the polyadic reduct of $\B$,  which we denote also by $\B$; that is we forget about diagonal elements that will be used at the last part of the proof.

Let $\Gamma\subseteq \alpha$ and $p\in \C$. Then in $\B$ we have, see \cite[proof of theorem 6.1]{DM}: 
${\sf c}_{(\Gamma)}p=\sum\{{\sf s}_{\bar{\tau}}p: \tau\in {}^{\alpha}\mathfrak{n},\ \  \tau\upharpoonright \alpha\sim\Gamma=Id\}.$
Here  for a transformation $\tau$ with domain $\omega$ and range included in $\mathfrak{n}$,
$\bar{\tau}=\tau\cup Id_{\mathfrak{n}\sim \omega}$. 
Let $X$ be the set of atoms of $\C$. Since $\C$ is atomic, then  $\sum^{\C} X=1$. By 
$\C=\Nr_{\alpha}\B$, we also have $\sum^{\B}X=1$.
By complete additivity we have 
for all $\tau\in {}^{\omega}\mathfrak{n},$ ${\sf s}_{\bar{\tau}}^{\B}X=1.$

Now working in the Stone space $S$ of $\B$, using the notation as in the first item of the proof,
 we have  by the suprema evaluated above for all $\Gamma\subseteq \alpha$, $p\in A$ and $\tau\in {}^{\alpha}\mathfrak{n}:$
$G_{\Gamma,p}=N_{{\sf c}_{(\Gamma)}p}\sim \bigcup_{{\tau}\in {}^{\alpha}\mathfrak{n}} N_{s_{\bar{\tau}}p},$
and
$G_{X, \tau}=S\sim \bigcup_{x\in X}N_{s_{\bar{\tau}}x}$
are nowhere dense. 
Let $F$ be a principal ultrafilter of $S$ containing $c$; exists because $\C$ is atomic. 
Then as before, $F\notin G_{\Gamma, p}$, $F\notin G_{X,\tau},$ 
for every $\Gamma\subseteq \alpha$, $p\in $
and $\tau\in {}^{\alpha}\mathfrak{n}$.
Now define for $a\in \C$
$f(a)=\{\tau\in {}^{\omega}\mathfrak{n}: {\sf s}_{\bar{\tau}}^{\B}a\in F\}.$
Then $f$ is the required complete representation of $\C$.
 
Now we want to completely represent $\A\subseteq_c \Nr_n\A'=\Nr_n\B$. 
Recall that $\A\in \PEA_n$, 
so we need to preserve diagonal elements as well. 
As in the first item of the proof (putting the diagonal elements of $\B$ to use),  
for $i, j\in \mathfrak{n}$, set $iEj\iff {\sf d}_{ij}^{\B}\in F$.

Then, exactly as before, $E$ is an equivalence relation on $\mathfrak{n}$ and for $s, t\in {}^n\mathfrak{n}$, 
if $(s(i), t(i))\in E$ for each $i<n$, and $a\in \A$, 
then ${\sf s}_{s\cup Id}^{\A'}a\in F\iff {\sf s}_{t\cup Id}^{\A'}a\in F$.  

For $a\in \A$, define  $g(a)=\{\bar{t}\in {}^n(\mathfrak{n}/E): {\sf s}_{t\cup Id}^{\B}a\in F\},$ 
where $\bar{t}(i/E)=t(i)$,  then again it can be proved exactly like in the first item of 
the proof, that  $g$ is well--defined, and in fact $g$ is 
the required complete representation of $\A$ as a $\PEA_n$. 

Now we assume that $\A\subseteq \Nr_n\D$, $\A\subseteq_c \D$, 
where $\D\in {\sf CPS}_{\omega}$ is atomic and completely additive. The idea in the proof for $\PA$s, was dilating
$\D$ to $\B$ possessing `enough spare dimensions' such that $\A=\Nr_n\D=\Nr_n\B'$ and then completely representing $\D$ by picking a principal ultrafilter
in the Stone space of $\B$  outside a nowhere dense set determined by two sets of joins; the first 
involving eliminating cylindrifiers via substitutions, 
and the second consists of permuted version of the atoms. The complete representation of $\A$ was 
the restriction of the thereby obtained complete reprentation of $\D$ to 
$\A$ using the second item of lemma \ref{join} since $\A\subseteq_c \D$.

We show that this can be done with  ${\sf CPS}$s, too. By a {\it transformation system} we mean a quadruple of the form $(\A, I, G, {\sf S})$ where $\A$ is an 
algebra of any signature, 
$I$ is a non-empty set,
$G$ is a subsemigroup of $(^II,\circ)$ (the operation $\circ$ denotes composition of maps) 
and ${\sf S}$ is a homomorphism from $G$ to the semigroup of endomorphisms of $\A$. 
Elements of $G$ are called transformations. 

Substitutions in $\D$ induce a homomorphism
of semigroups $S:{}^\alpha\alpha\to {\sf End}(\D)$ via $\tau\mapsto {\sf s}_{\tau},$ where ${\sf End}(\D)$ is the semigroup of
Boolean endomorphisms on
$\D$.
The operation on both semigroups is composition of maps; the latter is the semigroup of endomorphisms on
$\D$. For any set $X$, let $F(^{\omega}X,\D)$
be the set of all maps from $^{\omega}X$ to $\D$ endowed with Boolean  operations defined pointwise and for
$\tau\in {}^\omega\omega$ and $f\in F(^{\omega}X, \D)$, put ${\sf s}_{\tau}f(x)=f(x\circ \tau)$.
This turns $F(^{\omega}X,\D)$ to a transformation system.

The map $H:\D\to F(^{\omega}\omega, \D)$ defined by $H(p)(x)={\sf s}_xp$ is
easily checked to be an embedding. Assume that $\beta\supseteq \omega$. Then 
$K:F(^{\omega}\omega, \D)\to F(^{\beta}\omega, \D)$
defined by $K(f)x=f(x\upharpoonright \omega)$ is an embedding, too.
These facts are not too hard to establish
\cite[Theorems 3.1, 3.2]{DM}.
Call $F(^{\beta}\omega, \D)$ a minimal functional dilation of $F(^{\omega}\omega, \D)$.
Elements of the big algebra, or the (cylindrifier free)
functional dilation, are of form ${\sf s}_{\sigma}p$,
$p\in F(^{\beta}\omega, \A)$ where $\sigma$ restricted 
to $\omega$ is an injection \cite[Theorem 4.3-4.4]{DM}.

Let $\B$ be the algebra obtained from $\D$, by discarding its cylindrifiers, then taking  a minimal functional dilation,
dilating $\A$ to a regular cardinal $\mathfrak{n}$\footnote{If $\kappa$ is a cardinal, then {\it the cofinality of $\kappa$},
in symbols ${\sf cf}\kappa$, is  the least cardinal $\lambda$ such that $\kappa$ is the union of $\lambda$ sets each having cardinality $<\kappa$. The cardinal
$\kappa$ is {\it regular} if ${\sf cf}\kappa=\kappa$.}, such that $|\mathfrak{n}|>\omega$
and $|\mathfrak{n}\sim \omega|=\mathfrak{n}$.
Then  one re-defines cylindrifiers in the dilation $\B$ by setting for each  $i\in  \mathfrak{n}:$
$${\sf c}_{i}{\sf s}_{\sigma}^{\B}p={\sf s}_{\rho^{-1}}^{\B} {\sf c}_{(\rho(i)\cap \sigma(\omega))}^{\D}
{\sf s}_{(\rho\sigma\upharpoonright \omega)}^{\D}p.$$
Here $\rho$ is a any permutation such that $\rho\circ \sigma(\omega)\subseteq \sigma(\omega).$
The definition is sound, that is, it is independent of $\rho, \sigma, p$; furthermore, it agrees with the old cylindrifiers in $\D$.
Identifying algebras with their transformation systems
we have $\D\cong \Nr_{\omega}\B$, via $H$ defined
for $f\in \D$ and $x\in {}^{\mathfrak{n}}\omega$ by,
$H(f)x=f(y)$ where $y\in {}^{\omega}\omega$ and $x\upharpoonright \omega=y$,
\cite[Theorem 3.10]{DM}.
This dilation also has Boolean reduct isomorphic to $F({}^\mathfrak{n}\omega, \D)$, in particular, it is atomic because $\A$ is atomic
(a product of atomic Boolean algebras is atomic). 
Then we have
for all $j< \mathfrak{n}$, $p\in \B$: 
${\sf c}_{j}p=\sum_{i<\mathfrak{n}} {\sf s}_i^jp$
and, by complete additivity, for all $\tau\in {}^{\mathfrak{n}}\mathfrak{n}$:
${\sf s}_{\tau}^{\B}\At\A={\sf s}_{\tau}^{\B}\At\D=1.$
Now assume that $d\in \A$, is such that $d\neq 0$. Form the nowhere dense sets corresponding to the above joins and 
then pick a principal ultrafilter containing $d$; one that is generated by an atom below $d$.  The complete represenation of $\A$ is now defined like before.
In case of $\A\subseteq \Nr_n\D$, $\A\subseteq_c \D$ and $\D\in {\sf CPES}_{\omega}$, one dilates the $\sf CPS$ reduct of $\D$, call it $\Rd\D$, to $\mathfrak{n}$
dimensions as above, obtaining 
$\B\in \sf CPS_{\mathfrak{n}}$, such that $\Rd\D=\Nr_{\omega}\B$, then $\B$ admits an expansion $\B^+$ by diagonal elements, such that  
$\B\in {\sf CPES}_{\mathfrak{n}}$, $\D=\Nr_{\omega}\B^+$,
and $\A\subseteq \Nr_n\D=\Nr_n\B^+$. The rest is exactly like the $\PEA$ case. One picks 
a principal ultrafilter $F$ in $\B^+$ containing the given non-zero element of $\A$, then 
one defines, for $i,j\in \mathfrak{n}$, $(i, j)\in E\iff {\sf d}_{ij}\in F$, and finally  one sets (with notation $\bar{\tau}$ as above) 
$f(x)=\{\bar{\tau}\in {}^n({\mathfrak{n}}/E): {\sf s}_{\tau\cup Id}^{\B}\in F\}$ $(x\in A)$.   
\end{proof}

Now the following easy example shows us that the condition imposed in the hypothesis of the first item of the previous theorem \ref{complete3} is 
not necessary.

\begin{example}\label{cr}
Let $\bold 0$ the constant zero sequence with domain $\omega$.
Let $\D\in \QEA_{\omega}$ be the full weak set algebra with top element $^{\omega}\omega^{(\bold 0)}$. 
Then, it is easy to show that for any $n<\omega$, $\Nr_n\A$
is completely representable.  
Let $X=\{\bold 0\}\in \D$. Then for all $i\in \omega$, we have ${\sf s}_i^0X=X$. 
But $(1, 0,\ldots )\in {\sf c}_0X$, so that $\sum_{i\in \omega} {\sf s}_i^0X=X\neq {\sf c}_0X$.  
Hence the condition stipulated in the first item  of theorem \ref{complete3} does not hold.

Now fix $1<n<\omega$.  If we take $\D'=\Sg^{\D}\Nr_n\D$, then $\D'$ of course will still be a weak set algebra, and it will be localy finite, so that 
${\sf c}_ix=\sum^{\D}{\sf s}_j^ix$ for all $i<j<\omega$ by the reasoning used in the first item of theorem \ref{complete}.
However, $\D'$ will be atomless as we proceed to show. Assume for contradiction  that it is not, and let $x\in \D'$ be an atom. 
Choose $k, l\in \omega$ with $k\neq l$ and ${\sf c}_kx=x$, this is possible since $\omega\sim \Delta x$ is infinite. 
Then ${\sf c}_k(x\cdot  {\sf d}_{kl})=x$, so $x\cdot {\sf d}_{kl}\neq 0$. 
But $x$ is an atom, so $x\leq {\sf d}_{kl}$. This gives that $\Delta x=0$, 
and by \cite[Theorem 1.3.19]{HMT2} $x\leq -{\sf c}_k-{\sf d}_{kl}$. 
It is also easy to see that $({\sf c}_k -{\sf d}_{kl})^{\D'}={}^{\omega}\omega^{(\bold 0)}$, 
from which we conclude that $x=0$, 
which is a contradiction. 
\end{example}

We know from theorem \ref{complete} that if $n<\omega$ and $\A\in \bold S_c\Nr_n\CA_{\omega}$ is atomic with countably many atoms, then 
$\A$ is completely representable. This may not hold if $\A$ has uncountably many atoms as shown in the last item of theorem \ref{main}, though this might be true if the 
$\omega$--dilation is atomic.
And indeed for $\PA$s and $\PEA$s the countability condition can be omitted when the dilations have all substitution operations and are atomic
and completely additive (a condition that is superfluous for $\PEA$s). The next theorem gives a huge class of completely representable 
$\PA_n$ and $\PEA_n$s $(n<\omega)$  possibly having uncountably many atoms.

\begin{corollary} \label{complete4}
\begin{enumarab}
\item If $\D\in \PA_{\omega}$ is atomic and completely additive and $n<\omega$, then any $\B\in \PA_n$, such that 
$\B\subseteq_c \Nr_n\D$ is a completely representable $\PA_n$. Hence, for any $n<\omega$, we have (using the notation introduced before theorem \ref{complete}), 
$\bold S_c\Nr_n\PA_{\omega}^{\sf atc}\subseteq \sf CRPA_n$. Furthermore, If $\C=\Sg^{\D}\B$, 
then $\C$ is locally finite, $\B\subseteq_c \Nr_n\C\subseteq_c\C$. If $-{\sf c}_0-{\sf d}_{01}\neq 0$ in $\B$,  
then $\C\nsubseteq_c \D$. 
\item The same result holds for $\PEA$s without the condition of complete additivity, so that 
for any $n<\omega$,  $\bold S_c\Nr_n(\PEA_{\omega}\cap \At)\subseteq \sf CRPEA_n$.
\end{enumarab}
\end{corollary}
\begin{proof} It suffices to 
show, by theorem \ref{complete3},  that $\Nr_n\D\subseteq_c \D$ because then it will follow that 
$\B\subseteq_c \D$.  Let $\A=\Nr_n\D$. Assume that $S\subseteq \A$ and $\sum ^{\A}S=1$, and for contradiction, that there exists $d\in \D$ such that
$s\leq d< 1$ for all $s\in S$. Let  $J=\Delta d\sim n$. 
Put  $t=-{\sf c}_{(J)}(-d)$.
Then exactly as in the proof of the first item of lemma \ref{join}, we have  $t\in \Nr_{n}\D$.
We now show that $s\leq t<1$ for all $s\in S$, which contradicts $1=\sum^{\A}S$.
The proof is also completely analogous to the proof of the corresponding part in the first item of lemma \ref{join}, but using 
the possibly infinite cylindrifier ${\sf c}_{(J)}$; that is $J=\Delta d\sim n$ 
can be infinite.  
But all the same we give the details. 
If $s\in S$, we show that $s\leq t$.
By $s\leq d$, we have  $s\cdot -d=0$.
Hence $0={\sf c}_{(J)}(s\cdot -d)=s\cdot {\sf c}_{(J)}(-d)$, so
$s\leq -{\sf c}_{(J)}(-d)$, hence  $s\leq t$ as required. Assume for contradiction that $t=1$, then 
$1=-{\sf c}_{(J)}(-d)$ and so $ {\sf c}_{(J)}(-d)=0$.
But $-d\leq {\sf c}_{(J)}(-d)$,  so $1\cdot -d\leq  {\sf c}_{(J)}(-d)=0.$
Hence $1\cdot -d =0$ and this contradicts that $d<1$.

If $\C$ is as in the last part, then $\C$ is locally finite, because it is generated by $n$--dimensional elements, namely, 
the elements of $\B$  (witness the first part of theorem \ref{complete}). If $\sum ^{\B}S=1$, then since $\B\subseteq \C\subseteq \D$, we have
$1=\sum^{\D}S\leq \sum^{\C}S\leq \sum^{\B}S=1$, and by the first item of lemma \ref{j}, $\Nr_n\C\subseteq_c \C$.

Using exactly the argument in example \ref{cr}, we get that $\C$, being locally finite,  is atomless, 
so obviously  it is not atomic.  But $\D$ is atomic, so by the second item of lemma \ref{j}, it cannot be the case
that $\C\subseteq_c \D$ for this will force $\C$ to be atomic.

For the second item, we also have $\Nr_n\D\subseteq_c \D$ hence applying theorem \ref{complete3}, 
we get the required.

\end{proof}

\begin{remark} We do not know whether the converse inclusion ${\sf CRA}_n\subseteq \bold S_c\Nr_n\PA_{\omega}^{\sf atc}$ holds, but it seems to be highly unlikely.
The following observation might be useful here.
If $\D\in \sf PA_{\omega}^{atc}$ is infinite, 
then $\D$ would have to be uncountable.  To see why, first the set of atoms is of course infinite, too. 
Now for distinct  permulations $\tau, \sigma$ on $\omega$, and $x$ an atom,  
${\sf s}_{\sigma}x\neq {\sf s}_{\tau}x$, for else $x={\sf s}_{\sigma^{-1}\circ \tau}x$; the right hand side is an atom, 
so it has to be $x$ which means that $\sigma^{-1}\circ \tau=Id$, hence 
$\sigma=\tau$. Thus for each $x\in \At\D$, $J=\{{\sf s}_{\tau}x: \tau\in S_{\infty}\}\subseteq \At\D,$ 
and so $|\At\D|\geq |J|=|S_{\infty}|=\omega!$ (Here $S_{\infty}$ is the symmetric group on $\omega$).
However,  from the proof of theorem \ref{complete}, we get that  if $\B$ is completely representable, then it has an uncountable $\omega$--dilation
$\C$ that is a generalized weak set algebra such that $\B\subseteq_c \Nr_n\C$. Now $\C$, sure enough,  atomic and completely additive, 
is {\it only closed under  substitutions corresponding to finite transformations}.  It might not allow a polyadic structure (closure under infinitary substitutions). 
One could perhaps replace this $\omega$--dilation by a {\it generalized set algebra $\C'$} which of course allows a polyadic structure, 
but in this case one could well lose the condition  $\A\subseteq_c \Nr_n\C'$.\\
Let us summarize some of what have said in the following  theorem which actually  gives what 
we want if we can replace `$\A$ dense in $\B$' by 
the weaker $\A\subseteq_c \B$.
\end{remark}
\begin{theorem} If $\A\in \PA_n$ is atomic, countable
and completely representable, then there can be no $\D\in \PA_{\omega}^{\sf atc}$, such 
that $\A$ is dense in $\D$. The same holds for 
$\PEA_n$.  (But it might be the case that $\A\subseteq_c \D$.) 
\end{theorem}
\begin{proof} Assume for contradiction that there is such a $\D$. 
From the brief discussion in the above remark, we have $|\At\D|>\omega$. Using third item in lemma \ref{j} we get  
$|\At\A|\geq |\At\D|$, which is a contradiction because $\A$ is countable, 
and  and we are done.
\end{proof}

{\bf Question:} Do we have an analogous theorem for ${\sf CPS}_{\omega}$ and ${\sf CPES}_{\omega}$?  The argument used above does not work to show that 
if $\A\in {\sf CPS}_{\omega}$, then  $\Nr_n\D\subseteq_c \D$, because there may exist an upper bound $d<1$ 
in $\D$ of $S\subseteq \Nr_n\D$ such that $\sum^{\Nr_n\D}S=1$ and $J=\Delta d\sim n$ is infinite. In the present context of ${\sf CPS}_{\omega}$s, 
the infinite cylindrifier ${\sf c}_{(J)}$ is not defined to 
pull the element $d$ back to the neat reduct $\Nr_n\D$.\\

We identify set algebras with their domain. 
Let $\alpha$ be an ordinal.  Assume that $\A\in \K_{\alpha}$ is completely representable. Then, by definition, $\A\subseteq_c\wp(V)$
where $V$ is a generalized weak space of dimension $\alpha$,
$V=\bigcup_{i\in I}{}^{\alpha}U_i^{(p_i)}$ for some non--empty set $I$ and $p_i\in {}^{\alpha}U_i$, 
and the $^{\alpha}U_i^{(p_i)}s$'s are pairwise disjoint,  
hence $\A\subseteq_c \bold P_{i\in I}\wp(^{\alpha}U_i^{(p_i)}).$ 
Conversely, assume that $\A\subseteq_c \bold P_{i\in I}\wp(^{\alpha}U_i^{(p_i)}).$
Then  $\B=\bold P_{i\in I}\wp(^{\alpha}U_i^{(p_i)})\cong \wp(V)$, where $V$ is the disjoint
union of the $^{\alpha}U_i^{(p_i)}s$, is clearly completely representable. By the second item of lemma \ref{join},
we get that $\A$ is completely representable, too.

We conclude that $\sf CRK_{\alpha}=\bold S_c \bold P {\sf WFs_{\alpha}}$, 
where 
$${\sf WFs}_{\alpha}=\{\A\in {\sf Ws}_{\alpha}: A=\wp(V) \text { for some weak space $V$ of dimension $\alpha$}\}.$$ 
Note that both of $\bold P$ and $\bold S_c$ preserve atomicity.  
An analogous characterization is obtained in \cite[Theorem 18]{r} for relation algebras. 
The main novelty here is that the characterization lifts to the transfinite replacing set algebras by {\it weak} set algebras which is a common 
practice in algebraic logic.

In the last item of theorem \ref{main}, we will construct an atomless $\C\in \CA_{\omega}$ such that for each $2<n<\omega$, 
$\B=\Nr_n\C\in \Nr_n\CA_{\omega}\subseteq \RCA_n,$ 
is atomic but not completely representable (so that by theorem \ref{complete} $\B$ has uncountably many atoms). 
All of the above prompt:\\

{\bf Question:} Let $\alpha$ be an ordinal $>2$ and let $\K$ be any class between $\Sc$ and $\QEA$. 
Is $\bold S_c\Nr_{\alpha}\K_{\alpha+\omega}^{\sf atc}\subseteq \sf CRK_{\alpha}$?\\
 
We know that the reverse inclusion holds by theorem \ref{complete}. Also in the present context in the polyadic case, the dilations are 
{\it quasi-polyadic algebras}, we do not have all substitution operations, 
the case we dealt with in the previous theorem. The relation algebra analogue of this question is to the best of our knowlege also unsettled. It was posed 
by Robin Hirsch as \cite[Problem 32]{r}.

Our next theorem \ref{SL} collects, refines some known results on neat embeddings and prove slightly new ones particularly 
in the last item. 
But first a definition.

\begin{definition}\cite[Definition 5.2.1]{Sayedneat} Let $\A\in \RCA_{\alpha}$. Then $\A$ has the $UNEP$ (short for unique neat embedding property)
if for all $\A'\in \CA_{\alpha}$, $\B$, $\B'\in \CA_{\alpha+\omega},$
isomorphism $i:\A\to \A'$, embeddings  $e_A:\A\to \Nr_{\alpha}\B$
and $e_{A'}:\A'\to \Nr_{\alpha}\B'$ such that $\Sg^{\B}e_A(A)=\B$ and $\Sg^{\B'}e_{A'}(A)'=\B'$, there exists
an isomorphism $\bar{i}:\B\to \B'$ such that $\bar{i}\circ e_A=e_{A'}\circ i$.
\end{definition}

Recall that if $\alpha<\beta$, $\A\in \K_{\alpha}$ and $\A\subseteq \Nr_{\alpha}\B$, $\B\in \K_{\beta}$, 
then $\B$ is called a $\beta$--dilation of $\A$. If $\A$ generates $\B$, then $\B$ is called {\it minimal $\beta$--dilation}, or simply a minimal dilation of 
$\A$.\\

In the second item in the next theorem \ref{SL}, we show that minimal dilations may not be unique up to isomorphisms that fixes the dilated algebra 
pointwise a result that is proved via a complicated argument using a chain of implications \cite[Theorem 5.2.4]{Sayedneat} 
aplied to a deep result of Pigozzi's on the failure of the amalgamation property for 
subvarieties of $\CA_{\alpha}$s. The proof in {\it op.cit} does not exhibit explicity the dilated algebra; here we
present such an algebra taken from \cite{Conference}.\\ 

The first item is proved in \cite{SL}. The third item is proved in several publications of the author; 
a model--theoretic proof for $\CA$s is given in \cite{Sayedneat} for finite $\alpha>1$, witness the relevant 
references in {\it op.cit} for the rest of the cases. Here we give a uniform proof for $\Nr_{\alpha}\K_{\beta}$ for any pair of ordinals 
$1<\alpha<\beta\cap \omega$. The parts on pseudo--elementarity of $\Nr_{n}\K_{\omega}$, 
recursvie axiomatizabilty the elementary theory of $\Nr_n\K_{\omega}$ $(1<n<\omega)$, 
and non--finite axiomatizability of ${\bf El}\Nr_n\CA_{\omega}$ for $2<n<\omega$, formulated (and proved) 
in the last item 
are,  to the best of our knowledge,  new. \\ 

From the first item it easily follows that for any $1<n<m\leq \omega$, 
$\Nr_n\K_{m}\neq \bold S_c\Nr_n\K_m$. The $\RA$ analogue of this result is proved in \cite{r} 
only for $m=\omega$. It is (to the best of our knowledge) 
known that for $m>2$, $\Ra\CA_{m}\neq \bold S_c\Ra\CA_{m}\iff m\geq \omega$ \cite[Problem 37]{r}.

\begin{theorem}\label{SL}
\begin{enumarab} 
\item Let $\alpha$ be any ordinal $>1$.
Then for every infinite cardinal $\kappa\geq |\alpha|$, there exist completely representable algebras 
$\B, \A\in \QEA_{\alpha}$, that are weak set algebras, such that $\At\A=\At\B$, $|\At\B|=|\B|=\kappa$, 
$\Rd_{sc}\B\notin {\bf El}\Nr_{\alpha}\Sc_{\alpha+1}$,
$\A\in \Nr_n\QEA_{\omega}$,  and $\Cm\At\B=\A$, 
so that $|\A|=2^{\kappa}$.

\item For $\alpha\geq \omega$, 
there exists an algebra $\A\in \sf RCA_{\alpha}$ that does not have the $UNEP$. In particular, there exists $\A\in \RCA_{\alpha}$
that has two minimal $\alpha+\omega$--dilations that are not isomorphic via an isomorphism that
fixes $\A$ pointwise, and dually, there are  two non-isomorphic representable algebras of dimension $\alpha$ 
that generate the same $\alpha+\omega$--dilation.

\item For any ordinal $\alpha>1$, and any unountable cardinal $\kappa\geq |\alpha|$, there exist completely representable algebras $\A, \B\in \QEA_{\alpha}$, 
that are weak set algebras, such that $|\A|=|\B|=\kappa$, $\A\in \Nr_{\alpha}\QEA_{\alpha+\omega}$,  $\Rd_{sc}\B\notin \Nr_{\alpha}\Sc_{\alpha+1}$, 
$\A\equiv_{\infty, \omega}\B$ and $\At\A\equiv_{\omega, \infty}\At\B$.  For any algebra $\C$, such that $\At\B=\At\C$, $\C\equiv_{\infty, \omega}\A$.

\item Let everything be as in hypothesis of the previous item. Then there is  an $x\in \B$ such that if $\B\subseteq \C(\in \K_{\alpha})$, $\At\B=\At\C$ and 
$|\Rl_x\B|=|\Rl_x\C|$ (relativization here is with respect to the Boolean order), then $\Rd_{sc}\C\notin \Nr_n\Sc_{n+1}$. So any algebra $\C\in \QEA_{\alpha}$ 
such that $\B\subseteq \C$, $|\Rl_x\B|=|\Rl_x\C|$
and $\At\B=\At\C$, will be in $\sf RQEA_{\alpha}$ and it will witnesses that $\Nr_{\alpha}\K_{\beta}$ is not closed under $\equiv_{\infty, \omega}$, {\it a fortiori}, under $\equiv$

\item Let $0<n<m\cap \omega$, where $m$ is an ordinal ($>n$). The  class $\Nr_n\K_m$ is elementary and  pseudo--universal $\iff$ $n\leq 1$.
It is closed under ${\bf HPUp}$, but  for $n>1$ it is not closed $\bold S_c$ (hence not closed under $\bold S$), 
nor under $\equiv_{\infty, \omega}$ and $\bf Ur$. 
$\Nr_n\K_m$ is pseudo--elementary and its elementary theory is recursively enumerable.  
For any $2<n<\omega$, for any class $\bold K$, such that $\Nr_n\K_{\omega}\subseteq \bold K\subseteq \sf RK_n$, $\bf El K$
is not  finitely axiomatizable.

\end{enumarab}
\end{theorem}

\begin{proof}  
(1)   Fix an infinite cardinal $\kappa\geq |\alpha|$.  Let $FT_{\alpha}$ denote the set of all finite transformations on an ordinal $\alpha$.
Assume that $\alpha>1$.  Let $\F$ be field of characteristic $0$ such that $|\F|=\kappa$.
$V=\{s\in {}^{\alpha}\F: |\{i\in \alpha: s_i\neq 0\}|<\omega\}$ and let 
${\A}$ have universe $\wp(V)$ with the usual concrete operations,
Then clearly $\wp(V)\in \Nr_{\alpha}\sf QEA_{\alpha+\omega}$.
Let $y$ denote the following $\alpha$-ary relation:
$y=\{s\in V: s_0+1=\sum_{i>0} s_i\}.$
Let $y_s$ be the singleton containing $s$, i.e. $y_s=\{s\}.$
Define
Let ${\B}=\Sg^{\A}\{y,y_s:s\in y\}.$ Clearly $|\B|=\kappa$.
Now clearly $\B$ and $\A$ having same top element $V$ share the same atom structure, namely, the singletons.
Thus $\Cm\At\B=\A$. Finally, as proved in \cite{SL}, we have
$\Rd_{sc}\B\notin {\bold  El}\Nr_{\alpha}{\sf Sc}_{\alpha+1}$. \\

(2) The second part of this item readily follows from the previous one. We have $\B\in \RCA_{\alpha}$ so there exists $\D\in \CA_{\alpha+\omega}$,
such that $\B\subseteq \Nr_{\alpha}\D$, and we can assume
that $\B$ generates
$\D$, so that in fact $\D\in {\sf DcCA}_{\alpha+\omega}\subseteq \sf RCA_{\alpha+\omega}$. 
Then $\B$ and $\Nr_{\alpha}\D$ are non-isomorphic algebras, 
since $\B$ is not in $\Nr_{\alpha}\CA_{\alpha+1}\supseteq \Nr_{\alpha}\CA_{\alpha+\omega}$, but they 
generate the same dilation $\D$.

Now for the first part of this item, we will be sketchy. Details can be found in \cite{Conference}.
Let $\A=\Fr_4\CA_{\alpha}$ with $\{x,y,z,w\}$ its free generators. Let $X_1=\{x,y\}$ and $X_2=\{x,z,w\}$.
Let $r, s$ and $t$ be defined as follows:
$$ r = {\sf c}_0(x\cdot {\sf c}_1y)\cdot {\sf c}_0(x\cdot -{\sf c}_1y),$$
$$ s = {\sf c}_0{\sf c}_1({\sf c}_1z\cdot {\sf s}^0_1{\sf c}_1z\cdot -{\sf d}_{01}) + {\sf c}_0(x\cdot -{\sf c}_1z),$$
$$ t = {\sf c}_0{\sf c}_1({\sf c}_1w\cdot {\sf s}^0_1{\sf c}_1w\cdot -{\sf d}_{01}) + {\sf c}_0(x\cdot -{\sf c}_1w),$$
where $ x, y, z, \text { and } w$ are the first four free generators
of $\A$.
Then $r\leq s\cdot t$.
Let $\D=\Fr_4\RCA_{\alpha}$ with free generators $\{x', y', z', w'\}$.
Let  $\psi: \A\to \D$ be defined by the extension of the map $t\mapsto t'$, for $t\in \{x,y,x,w\}$.
For $i\in \A$, we denote $\psi(i)\in \D$ by $i'$.
Let $I=\Ig^{\D^{(X_1)}}\{r'\}$ and $J=\Ig^{\D^{(X_2)}}\{s'.t'\}$, and let
$$L=I\cap \D^{(X_1\cap X_2)}\text { and }K =J\cap \D^{(X_1\cap X_2)}.$$
Then $L=K$, and $\A_0=\D^{(X_1\cap X_2)}/L$  can be embedded into
$\A_1=\D^{(X_1)}/I$ and $\A_2=\D^{(X_2)}/J$,
but there is no amalgam even in $\CA_{\omega}.$ 
If $\A_0$ has the unique neat embedding property that it lies in the amalgamation
base of $\RCA_{\alpha}$ \cite{Sayedneat}.
Hence the algebra $\A_0$ does not have the unique neat embedding property.
More generally, if $\D_{\beta}$ is taken as the free
$\RCA_{\alpha}$ on $\beta$ generators, so that our algebra in the previous theorem is just $\D_4$,
where $\beta\geq 4$, then the algebra constructed from $\D_{\beta}$ as above,
will not have the unique neat embedding property. Now using the chain of implications in \cite[Theorem 5.2.4]{Sayedneat}, 
we get the result in the first part differently, since the negation of item(iv) in the cited theorem taking $L=\RCA_{\alpha}$ 
is true,  and this implies the negation of the first, which, in turn, easily implies
the required. \\

(3) We start by giving the general idea for the $\CA$s case when $1<\alpha<\omega$ paving the way for a general proof that addresses 
simultaneously any $\K$ between $\Sc$ and $\PEA$.
The  $\alpha$--dimensional set algebras $\A, \B$ constructed in \cite[Theorem 5.1.4]{Sayedneat} can be modified to be atomic by requiring that the 
interpretation of the uncountably --many tenary relations taken to have cardinality $\kappa$ in the model $\sf M$ constructed  in \cite[Lemma 5.3.1]{Sayedneat} 
are disjoint (not only distinct).   Then $\A$ and $\B$ will be also completely representable, $\A\in \Nr_{\alpha}\CA_{\omega}$, 
$\B\notin \Nr_{\alpha}\CA_{\alpha+1}$, and $\A\equiv_{\omega, \infty}\B$. 
We now give a unified proof that for any $\K$ such that
$\Sc\subseteq \K\subseteq \QEA$, for any $1<n<\beta\cap\omega$, the class $\Nr_n\K_{\beta}$ 
is not elementary.
Fix $1<n<\omega$. 
Let $L$ be a signature consisting of the unary relation
symbols $P_0,P_1,\ldots, P_{n-1}$ and
uncountably many $n$--ary predicate symbols. $\Mo$ is as in \cite[Lemma 5.1.3]{Sayedneat}, but the tenary relations are replaced by 
$n$--ary ones, and we require that the interpretations of the $n$--ary relations in $\Mo$ 
are {\it pairwise disjoint} not only distinct. This can be fixed. 
In addition to pairwise
disjointness of $n$--ary relations, we require their symmetry, 
that is, permuting the variables does not change
their semantics.

For $u\in {}^n n$, let $\chi_u$
be the formula $\bigwedge_{u\in {}^n n}  P_{u_i}(x_i)$. We assume that the $n$--ary relation symbols are indexed by (an uncountable set) $I$
and that  there is a binary operation $+$ on $I$, such that $(I, +)$ is an abelian group, and for distinct $i\neq j\in I$,
we have $R_i\circ R_j=R_{i+j}$. 
For $n\leq k\leq \omega$, let $\A_k=\{\phi^{\Mo}: \phi\in L_k\}(\subseteq \wp(^k\Mo))$,
where $\phi$ is taken in the signature $L$, and $\phi^{\Mo}=\{s\in {}^k\Mo: \Mo\models \phi[s]\}$. 

Let $\A=\A_n$, then $\A\in \sf Pes_n$ by the added symmetry condition.
Also $\A\cong \Nr_n\A_{\omega}$; the isomorphism is given by
$\phi^{\Mo}\mapsto \phi^{\Mo}$. The map is obviously an injective homomorphism; it is surjective, because $\Mo$ 
(as stipulated in \cite[ item (1) of lemma 5.1.3]{Sayedneat}), 
has quantifier elimination.

For $u\in {}^nn$, let $\A_u=\{x\in \A: x\leq \chi_u^{\Mo}\}.$ Then
$\A_u$ is an uncountable and atomic Boolean algebra (atomicity follows from the new disjointness condition)
and $\A_u\cong {\sf Cof}(|I|)$, the finite--cofinite Boolean algebra on $|I|$.
Define a map $f: \Bl\A\to \bold P_{u\in {}^nn}\A_u$, by
$f(a)=\langle a\cdot \chi_u\rangle_{u\in{}^nn+1}.$

Let $\P$ denote the
structure for the signature of Boolean algebras expanded
by constant symbols $1_u$, $u\in {}^nn$, ${\sf d}_{ij}$, and unary relation symbols
${\sf s}_{[i,j]}$ for each $i,j\in n$.
Then for each $i<j<n$, there are quantifier free formulas
$\eta_i(x,y)$ and $\eta_{ij}(x,y)$ such that
$\P\models \eta_i(f(a), b)\iff b=f({\sf c}_i^{\A}a),$
and $\P\models \eta_{ij}(f(a), b)\iff b=f({\sf s}_{[i,j]}a).$

The one corresponding to cylindrifiers is exactly like the $\CA$ case \cite[pp.113-114]{Sayedneat}.
For substitutions corresponding to
transpositions, it is simply $y={\sf s}_{[i,j]}x.$  The diagonal elements and the Boolean operations are easy to interpret. 
Hence, $\P$ is interpretable in $\A$, and the interpretation is one dimensional and
quantifier free. For $v\in {}^nn$, by the Tarski--Sk\"olem downward theorem, 
let  $\B_v$ be a countable elementary subalgebra of $\A_v$. (Here we are using the countable signature of $\PEA_n$).
Let $S_n (\subseteq {}^nn)$ be the set of permuations in $^nn$.

Take $u_1=(0, 1, 0, \ldots, 0)$ and $u_2=(1, 0, 0, \ldots, 0)\in {}^nn$. 
Let  $v=\tau(u_1,u_2)$  where $\tau(x,y)={\sf c}_1({\sf c}_0x\cdot {\sf s}_1^0{\sf c}_1y)\cdot {\sf c}_1x\cdot {\sf c}_0y$. We call $\tau$ an approximate
witness. It is not hard to show that $\tau(u_1, u_2)$ is actually the composition of $u_1$ and $u_2$, 
so that $\tau(u_1, u_2)$ is the constant zero map; which we denote by $\bold 0$; it is also in $^nn$. 
Clearly for every $i<j<n$,  ${\sf s}_{[i,j]} {}^{^nn}\{\bold 0\}=\bold 0\notin \{u_1, u_2\}$.

We can assume without loss that 
the Boolean reduct of $\A$ is the following product:
$$\A_{u_1}\times \A_{u_2}\times  \A_{\bold 0}\times\bold P_{u\in V\sim J} \A_u,$$
where $J=\{u_1, u_2, \bold 0\}$.
Let $$\B=((\A_{u_1}\times \A_{u_2}\times \B_{\bold 0}\times\bold P_{u\in V\sim J} \A_u), 1_u, {\sf d}_{ij}, {\sf s}_{[i,j]}x)_{i,j<n},$$
inheriting the same interpretation.  Then by the Feferman--Vaught theorem,
we get that
$\B\equiv \A$.
Now assume for contradiction, that $\Rd_{sc}\B=\Nr_n\D,$ with $\D\in \Sc_{n+1}$.
Let $\tau_n(x,y)$, which we call an {\it $n$--witness}, be defined by ${\sf c}_n({\sf s}_n^1{\sf c}_nx\cdot {\sf s}_n^0{\sf c}_ny).$
By a straightforward, but possibly tedious computation, one can obtain  $\Sc_{n+1}\models \tau_n(x,y)\leq \tau(x,y)$ 
so that the approximate witness {\it dominates} the $n$--witness.

The term $\tau(x,y)$ does not use any spare dimensions, and it `approximates' the term $\tau_n(x,y)$ that
uses the spare dimension $n$. The algebra $\A$ can be viewed as splitting the atoms of the atom structure $(^nn, \equiv_, \equiv_{ij}, D_{ij})_{i,j<n}$ each to uncountably many atoms.
On the other hand, $\B$ can be viewed as splitting the same atom structure, each  atom -- except for one atom that is split into countably many atoms -- 
is also split into uncountably many atoms (the same as in $\A)$. 

On the `global' level, namely, in the complex algebra of the finite
(splitted) atom structure $^nn$, these two terms are equal, the approximate witness
is the $n$--witness. The complex algebra $\Cm{}(^{n}n)$ does not `see' the $n$th dimension.
But in the algebras $\A$ and $\B$, obtained after splitting,  the $n$--witness becomes then a {\it genuine witness},  not an approximate one.  The approximate witness 
{\it strictly dominates} the $n$--witness.  The $n$--witness using the spare dimension $n$, detects the cardinality twist that $L_{\infty, \omega}$, {\it a priori}, 
first order logic misses out on.
If the $n$--witness were term definable (in the case we have a full neat reduct of an algebra in only one extra dimension), then
it takes two uncountable component to an uncountable one,
and this is not possible for $\B$, because in $\B$, the target 
component  is forced to be 
countable. 

Now for $x\in \B_{u_1}$ and  $y\in \B_{u_2}$, we have
$$\tau_n^{\D}(x, y)\leq \tau_n^{\D}(\chi_{u_1}, \chi_{u_2})\leq \tau^{\D}(\chi_{u_1}, \chi_{u_2})=\chi_{\tau^{\wp(^nn)}}(u_1,u_2)=\chi_{\tau(u_1, u_2)}=\chi_{\bold 0}.$$
But for $i\neq j\in I$,
$\tau_n^{\D}(R_i^{\sf M}\cdot \chi_{u_1}, R_j^{\sf M}\cdot \chi_{u_2})=R_{i+j}^{\sf M}\cdot \chi_v$, and so    $\B_{\bold 0}$ will be uncountable,
which is impossible.

We now show that \pe\ has a \ws\ in an \ef\ back--and--forth game over the now atomic
$(\A, \B).$  At any stage of the game, if \pa\ places a pebble on one of
$\A$ or $\B$, \pe\ must place a matching pebble on the other
algebra.  Let $\b a = \la{a_0, a_1, \ldots, a_{m-1}}$ be the position
of the pebbles played so far (by either player) on $\A$ and let $\b b = \la{b_0, \ldots, b_{m-1}}$ be the the position of the pebbles played
on $\B$.

Denote $\chi_u^{\Mo}$, by $\bold 1_u$. Then \pe\ has to
maintain the following properties throughout the
game:
\begin{itemize}

\item for any atom $x$ (of either algebra) with $x\cdot \bold 1_{\bold 0}=0$, , then $x \in a_i$ iff $x\in b_i$,

\item $\b a$ induces a finite partition of $\bold 1_{\bold 0}$ in $\A$ of $2^m$
 (possibly empty) parts $p_i:i<2^m$ and the $\b b$
induces a partition of  $1_{u}$ in $\B$ of parts $q_i:i<2^m$ such that 
$p_i$ is finite iff $q_i$ is
finite and, in this case, $|p_i|=|q_i|$.
\end{itemize}

It is easy to see that \pe\ can maintain these two properties in every round. In this back--and--forth game, \pe\ will always find a matching pebble, 
because the pebbles in play are finite.  For each $w\in {}^nn$ 
the component $\B_w=\{x\in \B: x\leq \bold 1_v\}(\subseteq \A_w=\{x\in \A: x\leq \bold 1_v\}$) 
contains infinitely many atoms. 
For any $w\in V$, $|\At\A_w|=|I|$, while 
for  $u\in V\sim \{\bold 0\}$, $\At\A_u=\At\B_u$. For 
$|\At\B_{\bold 0}|=\omega$, but it is still an infinite set.
Therefore $\A\equiv_{\infty}\B$.

It is clear that the above argument
works for any $\C$ such that $\At\C=\At\B$, hence $\B\equiv_{\infty, \omega}\C.$\\

(4) Take $x=\bold 1_{\bold 0}$. 
If $\B\subseteq \C$, $\At\B=\At\C$ and $|\Rl_x\C|=|\Rl_x\B|=\omega$, then the same `cardinality twist' used above works for $\C$, because 
$u, u' \in {}^nn$, $\bold 1_u\cap \bold 1_{u'}=\emptyset$ and for $u\in {}^nn\sim \{\bold 0\}$,  $\Rl_{\bold 1_u}\B\subseteq \Rl_{\bold 1_u}\C$, so $|\Rl_{\bold 1_u}\C|>\omega.$
More explicitly,  assume for contadiction that $\Rd_{sc}\C\in \Nr_n\Sc_{n+3}$. The genuine witness then witnesses that   $|\Rl_{\bold 1_{\bold 0}}\C|>\omega$, 
which is a contradiction.\\

(5)  For $1<n<m\cap \omega$,
the class $\Nr_n\K_{m}$ is not elementary, hence not pseudo--universal. 
For $n\leq 1$, and $m>n$, $\Nr_n\K_m=\K_n$ is a (finitely axiomatizable) variety. 
Closure under $\bf P$ and $\bf Up$
is easy, by noting that if $(\A_i: i\in I)$ is a family of $\K_m$s, $I$ a non-empty indexing set, $U$ an ultrafilter on $I$, 
then $\bold P_{i\in I}\Nr_n\A_i=\Nr_n\bold P_{i\in I}\A_i$ and 
$\Nr_n\Pi_{i/U}\A_i\cong  \Pi_{i/U}\Nr_n\A_i$. 
(All results in this paragraph formulate for $1<n<m\cap \omega$, works for any pair of ordinals $1<\alpha<\beta$ with same proofs).

Since for $1<n<m\cap \omega$, the class 
$\Nr_n\K_m$ is not elementary, and it is closed under $\bf Up$, then by the Keisler--Shelah ultrapower theorem, it is not
closed under $\bf Ur$.

Closure under ${\bf H}$ is not as easy as $\bold P$ and $\bf Up$. 
We  might as well prove it for any pair of ordinals $\alpha<\beta$, because again 
the proof is the same for any such pair. 
Let $\A=\Nr_{\alpha}\C$ with $\C\in \K_{\beta}$, and let $h:\A\to \B$ be a surjective homomorphism. 
We can assume that $\C$ is generated by $\A$. Let $I=ker f$, it suffices to to show that $\A/I\in \Nr_{\alpha}\K_{\beta}$. 
Let $J$ be the ideal of $\C$ generated by $I$.
Then  $I=J\cap \A$  \cite[Theorem 2.6.71, and Remark 2.6.73]{HMT2}. 
Obviously, $\B\cong \A/I\subseteq \Nr_{\alpha}(\C/J)$.  Now we will show that this last inclusion is not a proper one, by which we will be done. 

Let $x\in \Nr_{\alpha}(\C/J)$,   $x=c/J$ say, with $c\in \C$. Then $\Delta c$  uses only finitly many dimensions in
$\beta\sim \alpha$, because $\A$ generates $\C$. Let $\Gamma$ be this finite set, and 
put $c' ={\sf c}_{(\Gamma)}c$.  Then clearly $c'\in A$ and $x=c/I=c'/J$. 
Hence, $\B=\Nr_{\alpha}(\C/J)$, 
so that, as required, the inclusion is not a proper one.

For $n>1$, $\Nr_n\K_m$ is not closed under $\bold S_c$, hence under $\bold S$ since $\Rd_{\sf K}\B$  
completely representable, hence by theorem \ref{complete} it is
in  $\bold S_c\Nr_n\K_{\omega}\subseteq \bold S_c\Nr_n\K_{n+k}$ for any $k\leq \omega$. 
On the other hand, $\Rd_\K\B$  is not in ${\bf El}\Nr_n\K_{n+1}\supseteq \Nr_n\K_{\omega}$. 

For pseudo--elementarity  (for any finite $n$),  one easily adapts \cite[Theorem 21]{r} by defining  $\Nr_n\K_m$ in a two--sorted theory, when $1<n<m<\omega$, and a three-sorted one, when
$1<n<m$ where $m$ is infinite. The first part is easy.  For the second part; one uses a sort for a $\K_n$
$(c)$, the second sort is for the Boolean reduct of a $\K_n$ $(b)$
and the third sort for a set of dimensions $(\delta).$

For any infinite ordinal $\beta$, the defining theory for $\Nr_n\K_{\beta}=\Nr_n{\sf K}_{\omega}$,
includes sentences requiring that the constants $i^{\delta}$ for $i<\omega$
are distinct and that the last two sorts define
a $\K_\omega$. There is a function $I^b$ from sort $c$ to sort $b$ and sentences forcing  that $I^b$ is injective and
respects the $\K_n$ operations. For example, for all $x^c$
$I^b({\sf c}_i x^c)= {\sf c}_i^b(I^b(x^c)).$ The last requirement is that $I^b$ maps {\it onto} the set of $n$ dimensional elements. This can  be easily expressed
via
$$\forall y^b((\forall z^{\delta}(z^{\delta}\neq 0^{\delta},\ldots (n-1)^{\delta}\implies  c^b(z^{\delta}, y^b)=y^b))\iff \exists x^c(y^b=I^b(x^c))).$$
In all cases, it is clear that any algebra of the right type is the first sort of a model of this theory.
Conversely, a model for this theory will consist of an $\K_n$  (sort $c$),
and a $\K$ algebra whose dimension is the cardinality of
the $\delta$-sorted elements which is infinite.
Thus this three--sorted theory defines the class of neat reducts;
furthermore, it is clearly recursive.
The rest follows from \cite[Theorem 9.37]{HHbook}.

Now we prove non--finite axiomatizability of $\bf El K$ for any class $\bold K$ 
between $\Nr_n\K_{\omega}$ and $\sf RK_n$ for $2<n<\omega$. 
Note that the inclusion ${\bf El}\Nr_n\K_{\omega}\subseteq \sf RK_n$ is proper by first item of 
theorem \ref{SL} for any $n>1$.  Fix $2<n<\omega$ and $\bold K$ as specified above. For $3\leq n,i<\omega$, with $n-1\leq i, {\C}_{n,i}$ denotes
the finite ${\sf PEA}_n$ associated with the cylindric atom structure as defined on \cite[pp. 95]{HMT2}.
Then by \cite[Theorem 3.2.79]{HMT2}
for $3\leq n$, and $j<\omega$,
$\Rd_3{\C}_{n,n+j}$ can be neatly embedded in a
${\sf PEA}_{3+j+1}$ \ \  (1).

By \cite[Theorem 3.2.84]{HMT2}), we have for every $j\in \omega$,
there is an $3\leq n$ such that $\Rd_{df}\Rd_{3}{\C}_{n,n+j}$
is a non--representable $\sf Df_3$ \ \  (2).

Now suppose that $m\in \omega$. By (2),
there is a $j\in \omega\sim 3$ so that $\Rd_{df}\Rd_3{\C}_{j,j+m+n-4}$
is  not a representable $\sf Df_3$. 

By (1) we have
$\Rd_3{\C}_{j,j+m+n-4}\subseteq \Nr_3{\B_m}$, for some
${\B_m}\in {\sf PEA}_{n+m}$. We can assume that $\Rd_3{\C}_{j,j+m+n-4}$ generates $\B_m$, so that $\B_m$ is finite.
Put ${\A}_m=\Nr_n\B_m$,
then $\A_m$ is finite, too, and $\Rd_{df}{\A}_m$ is not representable,
{\it a fortiori} $\Rd_{\sf K}{\A}_m\notin \sf RK_{n}$. Therefore $\Rd_{\sf K}\A_m\notin {\bf El}\Nr_n{\sf K}_{\omega}$.

Let $\C_m$ be an algebra similar to ${\sf QEA}_{\omega}$'s such that $\B_m=\Rd_{n+m}\C_m$.
Then $\A_m=\Nr_n\C_m$.  (Note that $\C_m$ cannot belong to $\QEA_{\omega}$ for else $\A_m$ will be representable).
If $F$ is a non--trivial ultrafilter on $\omega,$ we have 
$$\Pi_{m\in \omega}\A_m/F=\Pi_{m\in \omega}(\Nr_n\C_m)/F=\Nr_n(\Pi_{m\in \omega}\C_m/F).$$ 
But  $\Pi_{m\in \omega}\C_m/F\in {\sf QEA}_{\omega},$
we conclude that  ${\sf K}_n\sim {\bf El}\bold K$ is not closed under ultraproducts, 
because $\Rd_{sc}\A_m\notin {\sf RSc}_n\supseteq {\bf El}\Nr_n\Sc_{\omega}$ and 
$\Pi_{m\in \omega}\A_m/F\in \Nr_n\QEA_{\omega}\subseteq \bf El K$.

\end{proof}

\section{Main Theorem}

{\it In what follows, unless otherwise explicitly indicated,  $\K$ denotes any class between $\Sc$ and $\PEA$.}\\

Let $m$ be finite $>2$. Recall that the finitely axiomatizable variety
$\sf G_m$ consists of set algebras whose top elements are of the form
$V\subseteq {}^mU$, such that if $s\in V$ and $\tau: m\to m$, then $s\circ \tau\in V$ \cite{Fer}.

For $2<n<m\leq \omega$, let $\bold L_{n,m}=\PEA_n\cap \bold S\Nr_n{\sf G}_m$. Notice that 
$\sf RPEA_n\subseteq \bold S\Nr_n\PEA_m\subseteq \bold L_{n, m},$ 
because $\PEA_m\subseteq {\sf G}_m$ \cite{Fer}, 
and we will  prove next that the inclusions are all  proper. 
In fact, it can also be proved using sophisticated rainbow constructions 
that $\bold L_{n, m}$ is {\it not finitely axiomatizable} over $\bold L_{n,m+1}$, cf. \cite[Section 17.4]{HHbook} for the $\sf RA$ 
analogue.  It will be proved below that $\bold L_{n,m}$ is a canonical variety,  and that 
$\bold L_{n,m}=\{\A\in \PEA_n: \A \text { has an $m$--square representation}\}.$\\

In the first item of the next theorem, we prove the $\CA$ analogue of the main result in formulated in \cite[Theorem 45]{r}  corrected in \cite{r2} in the way to be recalled next.
Addressing relation algebras, the corrected result states that 
any class $\bold K$ such that $\bold S_c\Ra\CA_{\omega}\subseteq \bold K\subseteq \bold S_c\Ra\CA_5$ is not elementary. The proof presented in \cite{r} 
does not allow us to remove $\bold S_c$
from the left hand side, like was  mistakenly stated in \cite{r}. In particular, it is {\it still not known} whether $\Ra\CA_{\omega}$, or for that matter $\Ra\CA_n$ for finite $n>2$, 
is elementary or not. However, as shown in theorem \ref{SL}, the analogous result (and much more) for many cylindric-like algebras is known. 
Indeed, for any pair of ordinals $1<\alpha<\beta$ (infinite included), for any $\K$ between $\Sc$ and $\PEA$, 
the class $\Nr_{\alpha}\K_{\beta}$ is not elementary.

If $\A\in \sf CRK_n$, then there exists by theorem \ref{complete} a completely additive 
atomic (dilation) $\B\in \K_{\omega}$, such that $\A\subseteq _c\Nr_n\B$, so the that inclusion 
$\sf CRK_n\subseteq \bold S_c\Nr_n\K_{n+3}^{\sf ad}$ stipulated in the first item 
is valid.   Furthermore, it  is proper by \cite{t}, where in the proof of \cite[Theorem 1.1]{t}, for each $2<n<m<\omega$,
a finite algebra $\A(n,m)$ is constructed such that $\A(n,m)\in \Nr_n\PEA_m$, and $\Rd_{sc}\A(n,m)\notin \bold S\Nr_n\Sc_{m+1}$. 
Suppose that $\A(n,m)=\Nr_n\D$, $\D\in \PEA_m$, and we can assume that $\A(n,m)$ generates $\D$, 
so that $\D$ is finite, hence $\Rd_{\K}\D$ is completely additive. Thus $\Rd_{\K}\A(n,m)\in \Nr_n\K_m^{\sf ad}\subseteq \bold S_c\Nr_n\K_m^{\sf ad}$ but 
$\Rd_{\K}\A(n,m)\notin \bold S\Nr_n\K_m\supseteq  \bold S_c\Nr_n\K_m\supseteq {\sf CRK}_n.$ 
The last $\supseteq$ follows from theorem \ref{complete}.\\

In item (2),  we show that is $\bold L_{n,m}$ for $2<n<m<\omega$ and $m\geq n+3$, is 
{\it not} atom--canonical, and in item (3) we show that $\bold S\Nr_m\PEA_n$ is not finitely axiomatizable over $\bold L_{m,n}$. 
Other  (strongly related) results proved in item (2)  
solve an open problem originating  with the present author, first announced by Hirsch and Hodkinson in \cite[Problem 12, pp.637]{HHbook}
and re--appearing  in  \cite{1}, cf. \cite[pp.131]{Sayed}.

The results proved in the first and second items strengthen substantially the classical results 
in \cite{Hodkinson, HH} on the notions of atom--canonicity and complete representations for the class $\sf RCA_n$ ($2<n<\omega$), respectively.
Our results address the (strict) approximations $\bold S\Nr_n\CA_{n+k}$ and $\bold S_c\Nr_n\CA_{n+k}$ for $k\geq 3$, 
together with their $\K$ analogues, for any $\K$ between $\Sc$ and $\PEA$.

Ourr construction in the second item, showing that, for $2<n<\omega$ and $k\geq 3$, $\bold S\Nr_n\CA_{n+k}$ is not atom--canonical,  has affinity
to the proofs of lemmata \cite[17.32, 17.34, 17.36]{HHbook},
exhibiting  a typical {\it blow up and blur construction} in the sense of \cite{ANT} 
for rainbow relation algebras 
proving that $\bold S\sf Ra\CA_k$,
$k\geq 6$ is not atom--canonical, but the term blow up and blur is not used there.

One might be tempted to obtain the result in this item by lifting the construction
of relation algebra analogue, using the construction of Hodkinson in \cite{AU}, which constructs $\CA$s (of every finite dimension $>2$)
from atom structures of $\sf RA$s.

But we hasten to add that this {\it cannot be done} with the construction in \cite{AU} as it stands, because the atomic $\sf RA$ does not embed in
$\sf Ra$ reduct of the atomic $\CA_n$ constructed from it, if $n\geq 6$.  Furthermore, there cannot be a lifting argument that associates to {\it any} 
given relation algebra $\R$ 
an $n$--dimensional $\B_n\in \CA_n$ 
for every $n\geq 3$, such that
$\R\subseteq \Ra\B_n$, for then if one starts with a non-representable $\R\in \sf RA$, then $\R\in \bigcap_{n\geq 3}\bold S\Ra\CA_n$ which
means that $\R\in \sf RRA$.
Since the relation algebra on which the construction in  lemmata \cite[17.32, 17.34, 17.36]{HHbook} hinges 
is not representable, 
we have decided to start from scratch blowing up and bluring 
a finite rainbow $\CA_n$. 

Nevertheless, our construction
relies on the model--theoretic construction in \cite{Hodkinson},
except that the latter is {\it not} a blow up and blur construction.\\

Now we are ready for:
\begin{theorem}\label{main}

Unless otherwise indicated, $n$ will be a finite ordinal $>2$. 

\begin{enumarab}

\item  (a)  If $m\geq n+3$,
then $\bold S_c\Nr_n\K_m^{\sf ad}$ is not elementary. In fact, any class $\bold K$,
such that $\sf CRK_n\subseteq \bold K\subseteq \bold S_c\Nr_n\K_m^{\sf ad}$,
is not elementary.  In particular, if $\K$ has a $\CA$ reduct, then  any $\bold K$ such that 
$\sf CRK_n\subseteq \bold K\subseteq \bold S_c\Nr_n\K_m$ is not elementary. For $\QA$s, we can replace $\QA_m^{\sf ad}$ 
by $\{\A\in \QA_m: {\sf s}_0^1 \text { is completely additive in }\A\}.$

\item Let  $\bold N=\{\A\in \K_n^{\sf ad}\cap \At: \Cm\At\in \Nr_n\K_{\omega}\}$.  Then $\bold N\neq \Nr_n\K_{\omega}\cap \At$, 
$\bold N\subseteq \bold S_c\Nr_n\K_{\omega}^{\sf ad}\cap \At$, $\bold N\cap {\sf Count}\subseteq {\sf CRK_n}$; the last two inclusions are proper, 
and $\bold N$ is 
not elementary.  If $\bold K_1={\sf CRK}_n\cap {\bf El}\Nr_n\K_{\omega}$ 
and $\bold K_2=\bold S_c\Nr_n\K_{\omega}\cap {\bf El}\Nr_n\K_{\omega}\cap \At$, 
then these are distinct classes  that lie 
{\it strictly} between $\Nr_n\K_{\omega}$ and $\bold S_c\Nr_n\K_{\omega}$,  but $\bold K_1$ and $\bold K_2$ 
coincide on algebras having countably many atoms. 

\item  The varieties $\bold S\Nr_n\K_m$ and $\bold L_{n,m}$ are not atom--canonical, when $m\geq n+3$.
In fact, there exists a countable, simple, atomic $\A\in \sf RPEA_n$ (so that $\A\in \bold I\sf Pes_n$), such that $\Cm\At\Rd_{sc}\A\notin \bold S\Nr_n\Sc_{n+3}$
and $\Cm\At\A\notin \bold S\Nr_n{\sf G}_{n+3}$.

\item  Let $\alpha$ be any ordinal $>2$ possibly infinite.  Then for any $r\in \omega$, and $k\geq 1$, there exists $\A_r\in \bold S\Nr_{\alpha}\QEA_{\alpha+k}$
such that $\Rd_{ca}\A_r\notin \bold S\Nr_{\alpha}\CA_{\alpha+k+1}$
and $\Pi_{r/U}\A_r\in \sf RQEA_{\alpha}$ for any non--principal ultrafilter
$U$ on $\omega$. For any $k\geq 3$, 
the variety $\bold S\Nr_{\alpha}\K_{\alpha+k}$ is not atom--canonical and for any $2<n<m<\omega$, 
the variety $\bold S\Nr_n\PEA_m$ is not finitely axiomatizable over $\bold L_{n, m}.$ 

\item For any ordinal $\alpha>2$, there exists an atomic algebra $\A\in\Nr_{\alpha}\K_{\alpha+\omega}$ that is 
not completely representable. For finite dimensions, such algebras are constructed uniformly from one relation algebra 
$\R$ possessing an $\omega$--dimensional basis. More precisely, for any cardinal $\kappa$ there exists an atomless $\C\in \QEA_{\omega}$, and an atomic relation algebra $\R$ having 
$2^{\kappa}$  many atoms, with no complete representation, such that $\R=\Ra\C$, and for all $n<\omega$, $\Nr_n\C\in {\bf El}{\sf CRK_n}$ 
(hence it is atomic),  but the  $\Df$ reduct of $\Nr_n\C$ does not have a 
complete representation $\iff$ $n>2$; for $n\leq 2$, $\Nr_n\C\in \sf CRQEA_n$. 
Hence for $2<n<\omega$, $\Nr_n\K_{\omega}^{\sf ad}\cap \At\nsubseteq {\sf CRK}_{n}.$ 

\end{enumarab}
\end{theorem}

\begin{proof}

{\bf Proof of (1)}: We use a rainbow construction inspired by the rainbow construction in \cite{r} for relation algebras.
We will show that any class $\sf L$ between ${\sf CRK}_n$
and $\bold S_c\Nr_n\K_{n+3}^{\sf ad}$ is not elementary.
From our construction, we will  get that $\sf CRDf_n$ 
is not elementary, too.

The proof consists of three parts:
\begin{itemize}
\item In the first part, we construct  a representable (atomic) rainbow--like algebra based on $\Z$ and $\N$ viewed as relational structures with usual order, 
denoted below by $\PEA_{\Z, \N}$ 
and show that \pe\ has a \ws\ in $G_k(\At\PEA_{\Z, \N})$ for each finite $k$.
\item In the second part, we show that \pa\ has a \ws\ in the $\omega$--rounded game $F^{n+3}$ played on atomic networks (coloured graphs) of  $\PEA_{\Z, \N}.$
\item In the third part, we use the \ws\ of \pe\ in a game stronger than $G_k$ denoted below by $H_k$ (for each $k\in \omega$) to obtain finer 
results.
\end{itemize}
Let $\A$ be the $n$--dimensional
rainbow cylindric algebra $R(\Gamma)$  \cite[Definition 3.6.9]{HHbook2}
where $\Gamma=\omega$, the reds  are the set ${\sf R}=\{\r_{ij}: i<j<\omega\},$ and the greens
${\sf G}=\{\g_i:1\leq i <n-1\}\cup \{\g_0^i: i\in \Z\}$, but we have a new
forbidden triple connecting two greens and one red, namely, the triple  $(\g^i_0, \g^j_0, \r_{kl})$ is forbidden, unless
${(i, k), (j, l)}$ is an order preserving partial function from
$\Z\to\N$ with the usual order $<$. Here we identify $\omega$ with $\N$.

We will refer to this new rainbow--like algebra of dimension $n$ by $\PEA_{\Z, \N}$ (not only in this item).
Adapting the technique in \cite{r},  we start by showing that \pe\ has a \ws\ in $G_k$ for all finite $k,$
hence, using ultrapowers followed by an elementary chain argument,
we obtain $\B$ such  $\PEA_{\Z,\N}\equiv \B$
and \pe\ has a \ws\ in $G$ on $\At\B$, so by \cite[Theorem 3.3.3]{HHbook2}, $\B$ is completely representable, thus by theorem \ref{complete},
we get that $\B\in \bold S_c\Nr_n\PEA_{n+3}.$

This known argument of forming ultrapowers followed by an elementary chain argument in rainbow constructions 
will be used several times below, so we briefly give the idea. For brevity let $\C=\PEA_{\Z, \N}$. Now for $k<\omega$, \pe\ has a \ws\ $\sigma_k$ in $G_k(\At\C)$.
We can assume that $\sigma_k$ is deterministic.
Let $\D$ be a non--principal ultrapower of $\C$.  Then \pe\ has a \ws\ $\sigma$ in $G(\At\D)$ --- essentially she uses
$\sigma_k$ in the $k$'th component of the ultraproduct so that at each
round of $G(\At\D)$,  \pe\ is still winning in co-finitely many
components, this suffices to show she has still not lost.

We can also assume that $\C$ is countable. If not then replace it by its subalgebra generated by the countably many atoms
(the term algebra); \ws\ s here will depend only on the atom structure, so they persist
for both players. Now one can use an
elementary chain argument to construct countable elementary
subalgebras $\C=\A_0\preceq\A_1\preceq\ldots\preceq\ldots \D$ in this manner.
One defines  $\A_{i+1}$ be a countable elementary subalgebra of $\D$
containing $\A_i$ and all elements of $\D$ that $\sigma$ selects
in a play of $G(\At\D)$ in which \pa\ only chooses elements from
$\A_i$. Now let $\B=\bigcup_{i<\omega}\A_i$.  This is a
countable elementary subalgebra of $\D$, hence necessarily atomic,  and \pe\ has a \ws\ in
$G(\At\B)$, so by \cite[Theorem 3.3.3]{HHbook2}, $\B$ is completely representable; furthermore $\B\equiv \C$.\\

{\bf (a) \pe\ s \ws\ in $G_k(\At\PEA_{\Z, \N})$ for each finite $k$:} 
Let $M_0, M_1,\ldots, M_r$, $r<k$ be the coloured graphs at the start of a play of $G_k$ just before round $r+1$.
Assume inductively, that \pe\ computes a partial function $\rho_s:\Z\to \N$, for $s\leq r$.
Here her `red moves' are further restricted by the newly added forbidden triple, so her strategy is not quite like
usual rainbow \ws's.
Inductively
for $s\leq r$:

\begin{enumerate}
\item    If $M_s(x,y)$ is green then \pa\  coloured it,

\item   $\rho_0\subseteq \ldots \rho_r\subseteq\ldots,$

\item   $\dom(\rho_s)=\{i\in \Z: \exists t\leq s, x, x_0, x_1,\ldots, x_{n-2}
\in \nodes(M_s)\\
\text { where the $x_i$'s form the base of a cone, $x$ is its appex and $i$ its tint}\}.$
The domain consists of the tints of cones created at an earlier stage,

\item  $\rho_s$ is order preserving: if $i<j$ then $\rho_s(i)<\rho_s(j)$. The range
of $\rho_s$ is widely spaced: if $i<j\in \dom\rho_s$ then $\rho_s(i)-\rho_s(j)\geq  3^{k-r}$, where $k-r$
is the number of rounds remaining in the game,

\item   For $u,v,x_0\in \nodes(M_s)$, if $M_s(u,v)=\r_{\mu,\delta}$, $M_s(x_0,u)=\g_0^i$, $M_s(x_0,v)=\g_0^j$,
where $i,j$ are tints of two cones, with base $F$ such that $x_0$ is the first element in $F$ under the induced linear order,
then $\rho_s(i)=\mu$ and $\rho_s(j)=\delta,$

\item $M_s$ is a a  coloured graph,

\item   If the base of a cone $\Delta\subseteq M_s$  with tint $i$ is coloured $y_S$, then $i\in S$.

\end{enumerate}
To start with, if \pa\ plays $a$ in the initial round then $\nodes(M_0)=n$, the
hyperedge labelling is defined by $M_0(0,1,\ldots, n-1)=a$.
In response to a cylindrifier move for some $s\leq r$, involving a $p$ cone, $p\in \Z$,
\pe\ must extend $\rho_r$ to $\rho_{r+1}$ so that $p\in \dom(\rho_{r+1})$
and the gap between elements of its range is at least $3^{k-r-1}$. Properties (3) and (4) are easily
maintained in round $r+1$. Inductively, $\rho_r$ is order preserving and the gap between its elements is
at least $3^{k-r}$, so this can be maintained in a further round.
If \pa\ chooses a green colour, or green colour whose suffix
that already belong to $\rho_r$, there would be fewer
elements to add to the domain of $\rho_{r+1}$, 
which makes it easy for \pe\ to define $\rho_{r+1}$.

Now assume that at round $r+1$, the current coloured graph is $M_r$ and that   \pa\ chose the graph $\Phi$, $|\Phi|=n$
with distinct nodes $F\cup \{\delta\}$, $\delta\notin M_r$, and  $F\subseteq M_r$ has size
$n-1$.  We can  view \pe\ s move as building a coloured graph $M^*$ extending $M_r$
whose nodes are those of $M_r$, together with the new node $\delta,$ and whose edges are edges of $M_r$ together with edges
from $\delta$ to every node of $F$.
Now \pe\ must extend $M^*$ to a complete graph $M^+$ on the same nodes and
complete the colouring giving a coloured graph $M_{r+1}=M^+$.
In particular, she has to define $M^+(\beta, \delta)$ for all nodes
$\beta\in M_r\sim F$, such that all of the above properties are maintained.
She does this as follows:

(1) If $\beta$ and $\delta$ are both apexes of two cones on $F$.
The two cones
induce the same linear ordering on $F$. We have $\beta\notin F$, but it is in $M_r$, while $\delta$ is not in $M_r$,
and  $|F|=n-1$.
By the rules of the game  \pe\ has no choice but to pick a red colour. \pe\ uses her auxiliary
function $\rho_{r+1}$ to determine the suffices, she lets $\mu=\rho_{r+1}(p)$, $b=\rho_{r+1}(q)$
where $p$ and $q$ are the tints of the two cones based on $F$,
whose apexes are $\beta$ and $\delta$. Then she sets $M^+(\beta, \delta)=\r_{\mu,b}$
maintaining property (5), and so $\delta\in \dom(\rho_{r+1})$
maintaining property (4). We check consistency to maintain property (6).

Consider a triangle of nodes $(\beta, y, \delta)$ in the graph $M_{r+1}=M^+$.
For $y\in \nodes(M_r)$, assume that the edges $M^+(y,\beta)$ and $M^+(y,\delta)$ are coloured green with
distinct superscripts $p, q$. This does not contradict
forbidden triangles of the form involving $(\g_0^p, \g_0^q, \r_{kl})$, because $\rho_{r+1}$ is constructed to be
order preserving.  Now assume that
$M_{r+1}(\beta, y)$ and $M_{r+1}(y, \delta)$ are both red (some $y\in \nodes(M_r)$).
Then \pe\ chose the red label $M_{r+1}(y,\delta)$, for $\delta$ is a new node.
We can assume that  $y$ is the apex of a cone with base $F$ in $M_r$. If not then $M_{r+1}(y, \delta)$ would be coloured
$\w$ by \pe\   and there will be no problem as shown next. All properties will be maintained.
Now $y, \beta\in M_r$, so by by property (5) we have $M_{r+1}(\beta,y)=\r_{\rho+1(p), \rho+1(t)}.$
But $\delta\notin M_r$, so by her strategy,
we have  $M_{r+1}(y,\delta)=\r_{\rho+1(t), \rho+1(q)}.$ But $M_{r+1}(\beta, \delta)=\r_{\rho+1(p), \rho+1(q)}$,
and we are done.  This is consistent triple, and so have shown that
forbidden triples of reds are avoided.\\

(2) If there is no $f\in F$ such that $M^*(\beta, f), M^*(\delta,f)$ are coloured $g_0^t$, $g_0^u$ for some $t, u$ respectively,
then \pe\ defines $M^+(\beta, \delta)$ to be $\w_0$.\\

(3)  If this is not the case, and  for some $0<i<n-1$ there is no $f\in F$ such
that $M^*(\beta, f), M^* (\delta, f)$ are both coloured $\g_i$,
she chooses $\w_i$ for  $M^+{(\beta,\delta)}$.\\

It is clear that the choices in the last two items  avoid all forbidden triangles (involving greens and whites).
She has not chosen green maintaining property (1).   For colourings $n-1$ tuples by shades of yellow, this
can be done like any (all) of \cite[p.16]{Hodkinson} or \cite[p.844]{HH}
or \cite{HHbook2}, maintaining property (7).\\

{\bf (b) \pa\ s \ws\ in $F^{n+3}$:} We show that \pa\ has a \ws\ in $F^{n+3}$, the argument used is the $\CA$ analogue of \cite[Theorem 33, Lemma 41]{r} except that now $F^{n+3}$
is played on coloured graphs.
The rough idea is that the added forbidden triple will restrict \pe\ s choices of reds and it
will allow  \pa\ to win the $\omega$ rounded game $F^{n+3}$ by pebbling
$\N$ and its mirror image $-\N=\{-a: a\in \Z\}$ using and
re-using the $n+3$ pebbles on the board, in this way, forcing \pe\
a decreasing sequence in $\N$,
so she loses in $\omega$ rounds.

In more detail, in the initial round \pa\ plays a graph $M$ with nodes $0,1,\ldots, n-1$ such that $M(i,j)=\w_0$
for $i<j<n-1$
and $M(i, n-1)=\g_i$
$(i=1, \ldots, n-2)$, $M(0, n-1)=\g_0^0$ and $M(0,1,\ldots, n-2)=\y_{\Z}$. This is a $0$ cone.
In the following move \pa\ chooses the base  of the cone $(0,\ldots, n-2)$ and demands a node $n$
with $M_2(i,n)=\g_i$ $(i=1,\ldots, n-2)$, and $M_2(0,n)=\g_0^{-1}.$
\pe\ must choose a label for the edge $(n+1,n)$ of $M_2$. It must be a red atom $r_{mk}$, $m, k\in \N$. Since $-1<0$, then by the `order preserving' condition
we have $m<k$.
In the next move \pa\ plays the face $(0, \ldots, n-2)$ and demands a node $n+1$, with $M_3(i,n)=\g_i$ $(i=1,\ldots, n-2)$,
such that  $M_3(0, n+2)=\g_0^{-2}$.
Then $M_3(n+1,n)$ and $M_3(n+1, n-1)$ both being red, the indices must match.
$M_3(n+1,n)=r_{lk}$ and $M_3(n+1, r-1)=r_{km}$ with $l<m\in \N$.

In the next round \pa\ plays $(0,1,\ldots n-2)$ and re-uses the node $2$ such that $M_4(0,2)=\g_0^{-3}$.
This time we have $M_4(n,n-1)=\r_{jl}$ for some $j<l<m\in \N$.
Continuing in this manner leads to a decreasing
sequence in $\N$.

We have proved that  any $\sf L$ between $\bold S_c\Nr_n\K_{\omega}$ and $\bold S_c\Nr_n\K_3^{\sf ad}$
is not elementary, and  that the same holds, if we replace the former class by $\sf CRK_n$, since $\B$ is in
$\bold S_c\Nr_n\K_{\omega}\cap \sf CRK_n$ by theorem \ref{complete}. 
We also have the $\Sc$ reduct of $\PEA_{\Z, \N}$ is not
in $\bold S_c\Nr_n\Sc_{n+3}^{\sf ad}$ 
by lemma \ref{Thm:n}.
By complete additivity of $\PEA$s and $\CA$s, 
we get that  for $\K$ having a $\CA$ reduct, any class $\sf L$, such that
$\sf CRK_n\subseteq \sf L\subseteq \bold S_c\Nr_n\K_{n+3}$, is not elementary. ${\sf CRDf}_n$ is not elementary because
$\C$ is completely generated by the set $\{x\in C: \Delta x\neq n\}$ hence its $\sf Df$ reduct is not completely representable 
and obviously $\Rd_{df}\C\equiv \Rd_{df}\B$ and $\Rd_{df}\B$ is completely representable.

Assume that $\Rd_{qa}\C\subseteq_c\Nr_n\D$ for some $\D\in \QA_{n+3}$ where only 
${\sf s}_0^1{}^{D}$ is completely additive.
Then every substitution operation corresponding to a replacement in $\D$, can be obtained from
a composition of finitely many substitution operations involving only one replacement 
${\sf s}_0^1$  and all the rest are substitution operations that correspond to transpositions.
To prove this, we can assume without loss that $i\neq 0, 1$. 
Then computing we get: 
$${\sf s}_{[1,i]}{\sf s}_0^1x
={\sf s}_0^i{\sf s}_{[1, i]}x\text { so }{\sf s}_{[1,i]}{\sf s}_0^1{\sf s}_{[1, i]}x={\sf s}_0^i{\sf s}_{[1, i]}{\sf s}_{[1,i]}x={\sf s}_0^ix,$$
$${\sf s}_i^0x={\sf s}_{[0,i]}{\sf s}_0^1x
={\sf s}_{[i,0]}{\sf s}_{[1, i]}{\sf s}_0^1{\sf s}_{[i,1]}x  
\text { and }{\sf s}_{[0, j]}{\sf s}_i^0x 
={\sf s}_i^j {\sf s}_{[0,j]}x.$$
Continuing the computation:
\begin{align*}
{\sf s}_i^jx&={\sf s}_{[0,j]}{\sf s}_{[0,j]}{\sf s}_i^jx\\
&={\sf s}_{[0,j]}{\sf s}_i^0{\sf s}_{[0,j]}x\\
&={\sf s}_{[i,0]} {\sf s}_{[1,i]}  {\sf s}_0^1 {\sf s}_{[1,i]} {\sf s}_{[0,j]}x.
\end{align*}
We have shown that:
$${\sf s}_i^j={\sf s}_{[0, j]}\circ {\sf s}_{[i,0]}\circ {\sf s}_{[1,i]}\circ  {\sf s}_0^1\circ {\sf s}_{[1,i]}\circ {\sf s}_{[0,j]}.$$

All such substitution operations are completely additive,
the ones involving transpositions are in fact self--conjugate,  
hence we get that $\D$ is completely additive, 
which is impossible.
\\

{\bf (c) More difficult  games for \pe\ and stronger results:} We introduce a stronger game played, denoted below by $H$, played on {\it hypernetworks} (to be defined next) 
aspiring a stronger result. This game
is the $\CA$ analogue of a game devised by Robin Hirsch for relation algebras in \cite{r}.  A few non--trivial modifications are needed. Fix $2<n<\omega$.
Let $\C=\PEA_{\Z, \N}$ and $\B$ be as above.  Using a \ws\ for \pe\ 
in $H_m$  ($H$ restricted to $m$ rounds $m<\omega$), we  further show that $\B$ can be chosen to satisfy that
(+) $\At\B\in \At\Nr_n\QEA_{\omega}$ and $\Cm\At\B\in \Nr_n\QEA_{\omega}$, and we discuss the possibility
of removing $\At$ from (+) getting the (stronger) result that $\B\in \Nr_n\QEA_{\omega}$,
witness theorem \ref{SL}, and remark \ref{remark}. In \cite{r}, $\At$ was removed from an analogous membership relation, namely, for a certain relation algebra $\R$,
from $\At\R\in \At\Ra\CA_{\omega}$, it was
inferred that  $\R\in \Ra\CA_{\omega}$ \cite[Theorem 39, 45]{r}.   
As mentioned before the proof of this theorem, this is a mistake that was corrected in \cite{r2},  by weakening the result erraneously 
proved in \cite{r}. Here we prove a result stronger than the result anounced in \cite{r2}, but still weaker than the alledged result in \cite{r}, which together with
its $\CA$ analogue, remains open for both $\RA$s and $\CA$s, though for 
$\CA$s many special cases, like that proved in the second item pf theorem \ref{SL}, are already known.
But the main result in \cite{r} in its full generality remains unproved.

For an atomic network and for  $x,y\in \nodes(N)$, we set  $x\sim y$ if
there exists $\bar{z}$ such that $N(x,y,\bar{z})\leq {\sf d}_{01}$.
Define the  equivalence relation $\sim$ over the set of all finite sequences over $\nodes(N)$ by
$\bar x\sim\bar y$ iff $|\bar x|=|\bar y|$ and $x_i\sim y_i$ for all
$i<|\bar x|$. (It can be easily checked that this indeed an equivalence relation).

A \emph{ hypernetwork} $N=(N^a, N^h)$ over an atomic polyadic equality algebra
consists of a network $N^a$
together with a labelling function for hyperlabels $N^h:\;\;^{<
\omega}\!\nodes(N)\to\Lambda$ (some arbitrary set of hyperlabels $\Lambda$)
such that for $\bar x, \bar y\in\; ^{< \omega}\!\nodes(N)$
if $\bar x\sim\bar y \Rightarrow N^h(\bar x)=N^h(\bar y).$
If $|\bar x|=k\in \N$ and $N^h(\bar x)=\lambda$, then we say that $\lambda$ is
a $k$-ary hyperlabel. $\bar x$ is referred to as a $k$--ary hyperedge, or simply a hyperedge.
We may remove the superscripts $a$ and $h$ if no confusion is likely to ensue.

The is a one--to--one correspondence between networks and coloured graphs \cite[Second half of pp 76]{HHbook2}.
If $\Gamma$ is a coloured graph, then by $N_{\Gamma}$
we mean the corresponding network.
A hyperedge $\bar{x}\in {}^{<\omega}\nodes (\Gamma)$ of length $m$ is {\it short}, if there are $y_0,\ldots, y_{n-1}$
that are nodes in $N_{\Gamma}$, such that
$N_{\Gamma}(x_i, y_0, \bar{z})\leq {\sf d}_{01}$
or $\ldots N_{\Gamma}(x_i, y_{n-1},\bar{z})\leq {\sf d}_{01}$
for all $i<|x|$, for some (equivalently for all)
$\bar{z}.$
Otherwise, it is called {\it long.}
A hypergraph $(\Gamma, l)$
is called {\it $\lambda$--neat} if $N_{\Gamma}(\bar{x})=\lambda$ for all short hyperedges.

Concerning his moves, this game with $m$ rounds, denoted by $H_m$, \pa\ can play a cylindrifier move, like before but now played on $\lambda_0$---neat hypernetworks.
Also \pa\ can play a \emph{transformation move} by picking a
previously played hypernetwork $N$ and a partial, finite surjection
$\theta:\omega\to\nodes(N)$, this move is denoted $(N, \theta)$.  \pe's
response is mandatory. She must respond with $N\theta$.
Finally, \pa\ can play an
\emph{amalgamation move} by picking previously played hypernetworks
$M, N$ such that
$M\restr {\nodes(M)\cap\nodes(N)}=N\restr {\nodes(M)\cap\nodes(N)},$
and $\nodes(M)\cap\nodes(N)\neq \emptyset$.
This move is denoted $(M,
N).$
To make a legal response, \pe\ must play a $\lambda_0$--neat
hypernetwork $L$ extending $M$ and $N$, where
$\nodes(L)=\nodes(M)\cup\nodes(N)$.

It can be shown that  \pe\ has a \ws\ in $H_{m}(\At\PEA_{\Z, \N})$ for each finite $m$.
The analogous proof for relation algebras is rather long
\cite[p.25-31]{r} so we will be sketchy. We have already dealt with the graph part.

We now have to extend his strategy dealing with $\lambda$-- neat hypernetworks, where $\lambda$ is constant label.
Labelling hyperedges is exactly like in \cite{r}.
In a play, \pe\ is required to play $\lambda$ neat hypernetworks, so she has no choice about the
hyperedges for short edges, these are labelled by $\lambda$. In response to a cylindrifier move by \pa\
extending the current hypergraph providing a new node $k$,
and a previously played coloured graph $M$
all long hyperedges not incident with $k$ necessarily keep the hyperlabel they had in $M$.
All long hyperedges incident with $k$ in $M$
are given unique hyperlabels not occurring as the hyperlabel of any other hyperedge in $M$.
In response to an amalgamation move, which involves two hypergraphs required to be amalgamated, say $(M,N)$
all long hyperedges whose range is contained in $\nodes(M)$
have hyperlabel determined by $M$, and those whose range is contained in $\nodes(N)$ have hyperlabels determined
by $N$. If $\bar{x}$ is a long hyperedge of \pe\ s response $L$ where
$\rng(\bar{x})\nsubseteq \nodes(M)$, $\nodes(N)$ then $\bar{x}$
is given
a new hyperlabel, not used in any previously played hypernetwork and not used within $L$ as the label of any hyperedge other than $\bar{x}$.
This completes her strategy for labelling hyperedges.

We turn to the remaining amalgamation moves. We need some notation and terminology taken from \cite[pp.25]{r}.
Every edge of any hypernetwork has an {\it owner \pa\ or \pe}, namely, the one who coloured this edge.
We call such edges \pa\ edges or \pe\ edges. Each long hyperedge $\bar{x}$ in a hypernetwork $N$
occurring in the play has {\it an envelope} $v_N(\bar{x})$ to be defined shortly.
In the initial round,  \pa\ plays $a\in \alpha$ and \pe\ plays $N_0$
then all irreflexive edges of $N_0$ belongs to \pa.
There are no long hyperedges in $N_0$. If in a later move,
\pa\ plays the transformation move $(N,\theta)$
and \pe\ responds with $N\theta$, then owners and envelopes are inherited in the obvious way.
If \pa\ plays a cylindrifier move requiring a new node $k$ and \pe\ responds with $M$ then the owner
in $M$ of an edge not incident with $k$ is the same as it was in $N$
and the envelope in $M$ of a long hyperedge not incident with $k$ is the same as that it was in $N$.

All  edges $(l,k)$
for $l\in \nodes(N)\sim \{k\}$ belong to \pe\ in $M$.
if $\bar{x}$ is any long hyperedge of $M$ with $k\in \rng(\bar{x})$, then $v_M(\bar{x})=\nodes(M)$.
If \pa\ plays the amalgamation move $(M,N)$ and \pe\ responds with $L$
then for $m\neq n\in \nodes(L)$ the owner in $L$ of a edge $(m,n)$ is \pa\ if it belongs to
\pa\ in either $M$ or $N$, in all other cases it belongs to \pe\ in $L$.
If $\bar{x}$ is a long hyperedge of $L$
then $v_L(\bar{x})=v_M(x)$ if $\rng(x)\subseteq \nodes(M)$, $v_L(x)=v_N(x)$ and  $v_L(x)=\nodes(M)$ otherwise.
This completes the definition of owners and envelopes.
The next claim, basically, reduces amalgamation moves to cylindrifier moves.
By induction on the number of rounds one can show:\\

{\bf Claim}:\label{r} 
\begin{itemize}

\item Let $M, N$ occur in a play of $H_m$, $0<m\in \omega.$ in which \pe\ uses the above labelling
for hyperedges. Let $\bar{x}$ be a long hyperedge of $M$ and let $\bar{y}$ be a long hyperedge of $N$.
Then for any hyperedge $\bar{x}'$ with $\rng(\bar{x}')\subseteq v_M(\bar{x})$, if $M(\bar{x}')=M(\bar{x})$
then $\bar{x}'=\bar{x}$, 

\item If $\bar{x}$ is a long hyperedge of $M$ and $\bar{y}$ is a long hyperedge of $N$, and $M(\bar{x})=N(\bar{y}),$
then there is a local isomorphism $\theta: v_M(\bar{x})\to v_N(\bar{y})$ such that
$\theta(x_i)=y_i$ for all $i<|x|$, 

\item For any $x\in \nodes(M)\sim v_M(\bar{x})$ and $S\subseteq v_M(\bar{x})$, if $(x,s)$ belong to \pa\ in $M$
for all $s\in S$, then $|S|\leq 2$.
\end{itemize}

Next,  we proceed inductively with the inductive hypothesis exactly as before, except that now each $N_r$ is a
$\lambda$ neat hypergraph. All other inductive conditions are the same (modulo this replacement). Now,
we have already dealt with hyperlabels for long and short
hyperedges, we dealt with the graph part of the first hypergraph move.
All what remains is the amalgamation move. With the above claim at hand,
this turns out an easy task to implement guided by \pe\ s
\ws\ in the graph part.

We consider an amalgamation move $(N_s,N_t)$ chosen by \pa\ in round $r+1$.
We finish off with edge labelling first.   \pe\ has to choose a colour for each edge $(i,j)$
where $i\in \nodes(N_s)\sim \nodes(N_t)$ and $j\in \nodes(N_t)\sim \nodes(N_s)$.

Let $\bar{x}$ enumerate $\nodes(N_s)\cap \nodes(N_t).$
If $\bar{x}$ is short, then there are at most two nodes in the intersection
and this case is similar to the cylindrifier move, she uses $\rho_s$ for the suffixes of the red.
If not, that is if $\bar{x}$ is long in $N_s$, then by the claim
there is a partial isomorphism $\theta: v_{N_s}(\bar{x})\to v_{N_t}(\bar{x})$ fixing
$\bar{x}$. We can assume that
$v_{N_s}(\bar{x})=\nodes(N_s)\cap \nodes (N_t)=\rng(\bar{x})=v_{N_t}(\bar{x}).$
It remains to label the edges $(i,j)\in N_{r+1}$ where $i\in \nodes(N_s)\sim \nodes (N_t)$ and $j\in \nodes(N_t)\sim \nodes(N_s)$.
Her strategy is similar to the cylindrifier move. If $i$ and $j$ are tints of the same cone she choose a red using $\rho_s$,
If not she  chooses   a white.
She never chooses a green.
Then she lets $\rho_{r+1}=\rho_r$  maintaining the inductive hypothesis.

Concerning the last property to be maintained, and that is
colouring $n-1$ types property (7). Let $M^+=N_s\cup M_s$, which is the graph whose edges are labelled according to the rules of the game,
we need to label $n-1$ hyperedges by shades of yellow.
For each tuple $\bar{a}=a_0,\ldots a_{n-2}\in {M^+}^{n-1}$, $\bar{a}\notin N_s^{n-1}\cup M_s^{n-1}$,  with no edge
$(a_i, a_j)$ coloured green (we have already labelled edges), then  \pe\ colours $\bar{a}$ by $\y_S$, where
$$S=\{i\in \Z: \text { there is an $i$ cone in $M^*$ with base $\bar{a}$}\}.$$
We have shown that  \pe\ has a \ws\ in $H_{m}$ on $\At\PEA_{\Z, \N}$ for each finite $m$.
By taking an ultrapower followed by an elementary chain argument
(like before), we get that \pe\ has a \ws\ on $\alpha=\At\B$ in the $\omega$ rounded game $H$.

By taking an ultrapower followed by an elementary chain argument
(like before), we get that \pe\ has a \ws\ on $\alpha=\At\B$ in the $\omega$ rounded game $H$.

Fix some $a\in\alpha$. Using \pe\ s \ws\ in the game of neat hypernetworks, one defines a
nested sequence $M_0\subseteq M_1,\ldots$ of neat hypernetworks
where $M_0$ is \pe's response to the initial \pa-move $a$, such that:
If $M_r$ is in the sequence and $M_r(\bar{x})\leq {\sf c}_ib$ for an atom $b$ and some $i<n$,
then there is $s\geq r$ and $d\in\nodes(N_s)$
such that  $M_s(\bar{y})=b$ such that $\bar{y}_i=d$ and $\bar{y}\equiv_i \bar{x}$.
In addition, if $M_r$ is in the sequence and $\theta$ is any partial
isomorphism of $M_r$, then there is $s\geq r$ and a
partial isomorphism $\theta^+$ of $N_s$ extending $\theta$ such that
$\rng(\theta^+)\supseteq\nodes(M_r)$.

Now let $M_a$ be the limit of this sequence, that is $M_a=\bigcup M_i$, the labelling of $n-1$ tuples of nodes
by atoms, and the hyperedges by hyperlabels done in the obvious way.
Let $L$ be the signature with one $n$-ary relation for
each $b\in\alpha=\At\B$, and one $k$--ary predicate symbol for
each $k$--ary hyperlabel $\lambda$.

We work in $L_{\infty, \omega}^n.$
For fixed $f_a\in\;^\omega\!\nodes(M_a)$, let
$U_a=\set{f\in\;^\omega\!\nodes(M_a):\set{i<\omega:g(i)\neq
f_a(i)}\mbox{ is finite}}$.
Now we  make $U_a$ into the universe an $L$ relativized structure ${\cal M}_a$ like in \cite[Theorem 29]{r} except that we allow a clause for infinitary disjunctions.
In more detail, we have 
For $b\in\alpha,\; l_0, \ldots, l_{n-1}, i_0 \ldots, i_{k-1}<\omega$, \/ $k$--ary hyperlabels $\lambda$,
and all $L$-formulas $\phi, \phi_i, \psi$,
\begin{eqnarray*}
{\cal M}_a, f\models b(x_{l_0}\ldots,  x_{n-1})&\iff&{\cal M}_a(f(l_0),\ldots,  f(l_{n-1}))=b,\\
{\cal M}_a, f\models\lambda(x_{i_0}, \ldots,x_{i_{k-1}})&\iff&  {\cal M}_a(f(i_0), \ldots,f(i_{k-1}))=\lambda,\\
{\cal M}_a, f\models\neg\phi&\iff&{\cal M}_a, f\not\models\phi,\\
{\cal M}_a, f\models (\bigvee_{i\in I} \phi_i)&\iff&(\exists i\in I)({\cal M}_a,  f\models\phi_i),\\
{\cal M}_a, f\models\exists x_i\phi&\iff& {\cal M}_a, f[i/m]\models\phi, \mbox{ some }m\in\nodes({\cal M}_a).
\end{eqnarray*}
We check that the rest of the proof in \cite{r} survives this non--trivial  change. We are now 
working with (weak) set algebras  whose semantics is induced by $L_{\infty, \omega}$ formulas in the signature $L$, 
instead of first order ones. 
For any such $L$-formula $\phi$, write $\phi^{{\cal M}_a}$ for
$\set{f\in\;^\omega\!\nodes(M_a): {\cal M}_a, f\models\phi}.$  
Let
$D_a= \set{\phi^{{\cal M}_a}:\phi\mbox{ is an $L$-formula}}$ and
$\D_a$ be the set algebra with universe $D_a$. Then $\D_a$ is locally finite, that is, the dimension set of any element in $\D_a$ is finite,
because (the two `sorts' of) formulas use only finitely many variable.
Let $\D=\bold P_{a\in \alpha} \D_a$. However, $\D$ is a generalized weak set algebra, that  might not be locally finite.
Exactly as in \cite{r}, it can be proved that $\Nr_{n}\D$ is atomic and $\alpha\cong\At\Nr_{n}\D$ --- the isomorphism
is $b \mapsto (b(x_0, x_1,\dots, x_{n-1})^{\D_a}:a\in \alpha).$
We have shown that $\At\B\in \At\Nr_n\QEA_{\omega}$.
The last part is a fairly straightforward adaptation of the technique used in \cite[Theorem 29]{r}, building a representable
$\omega$ dilation
$\D$ from   a `saturated set' of $\lambda$--neat hypernetworks, using games, replacing binary relation symbols with $n$--ary ones
and working in $L_{\infty, \omega}$ instead of $L_{\omega, \omega}$.
Now we reap the harvest of our `infinitary addition'.  Because we are working in $L_{\infty, \omega},$ infinite disjuncts exist in $\D_a$ $(a\in \alpha)$,
hence, they exist too in the dilation $\D=\bold P_{a\in\alpha}\D_a$. 
Therefore $\D$ is complete.
Now by lemma \ref{join}, we have 
$\Nr_n\D_a\subseteq _c \D_a$ from which  we get:
$$\Nr_n\D=\Nr_n(\bold  P_{a\in \alpha}\D_a)=\bold P_{a\in \alpha}\Nr_n\D_a\subseteq_c \bold P_{a\in \alpha}\D_a=\D.$$
Because $\D$ is complete, and $\Nr_n\D\subseteq_c  \D$, then $\Nr_n\D$ is complete.
But $\Nr_n\D\subseteq \Cm\At\B$ is dense  in $\Cm\At\B$,  because they share the same atom structure.
It readily follows, from the completeness of $\Nr_n\D$, that $\Nr_n\D=\Cm\At\B$, 
so that $\Cm\At\B\in \Nr_n\QEA_{\omega}$ as required.

Now the  $\sf Df$ reduct of $\C=\PEA_{\Z, \N}$ is also not completely representable, for
the set $\{x\in  \C: \Delta c \neq n\}$ completely generates $\C$,
it follows that $\Rd_{df}\C$ is also not completely
representable \cite[Proposition 4.10]{Hodkinson}. Hence, we readily infer that the class of completely representable
${\sf Df}_n$s is not elementary,  since obviously $\Rd_{df}\C\equiv \Rd_{df}\B$ and the latter is ${\sf CRDf}_n$.\\

{\bf Proof of (2):} The atomic algebra $\Rd_{\K}\B$ given in the first item of theorem \ref{SL} is outside 
$\Nr_n\K_{n+1}\supseteq \Nr_n\K_{\omega}$, 
but $\Cm\At\B\in \Nr_n\K_{\omega}$, hence $\Rd_{\K}\B\in \bold N$ and is outside 
$\Nr_n\K_{\omega}$. 

We have $\Rd_{\K}\PEA_{\Z, \N}=\Cm\At\Rd_{\K}\PEA_{\Z,\N}\notin \bold S_c\Nr_n \K_{\omega}^{\sf ad}=\bold S_c\Nr_n\K_{\omega}$, 
{\it a fortiori} it is not in $\Nr_n\K_{\omega}.$
But in  the proof of the first item of theorem \ref{main}, we have shown that 
there is a countable completely representable $\B\in \PEA_n$ such that $\B\equiv \PEA_{\Z, \N}$ and 
$\Cm\At\B\in \Nr_n\QEA_{\omega}\cap {\sf CRQEA}_n$, 
so $\bold N$ is not elementary. 

If $\A\in \bold N$, then $\Cm\At\A\in \Nr_n\K_{\omega}\subseteq \bold S_c\Nr_n\K_{\omega}^{\sf ad}$, the last is gripped, hence 
$\A\in \bold S_c\Nr_n\K_{\omega}$. If $\A\in \N\cap {\sf Count}$, then $\Cm\At\A\in \Nr_n\CA_{\omega}$ has countably
many atoms,  hence by the first item of theorem \ref{complete},  $\Cm\At\in {\sf  CRK}_n$; 
the latter is gripped, hence  $\A\in {\sf CRK}_n$.
The strictness of the two inclusions as in the statement of this item is witnessed by $\PEA_{\Z, \N}$.\\

{\bf Proof of (2).  Idea of Proof:}  
We deal only with $\CA$s in this sketch of the idea of proof. The rainbow algebra based on $A$ (the greens)  and $B$ (the reds) is denoted by $\CA_{A, B}$.
Fix $2<n<\omega$. The gist of the idea makes use of the model--theoretic techniques of Hodkinson's used in \cite{Hodkinson} conjuncted
with a blow up and blur construction in
the sense of \cite{ANT}.

In \cite{Hodkinson}, Hodkinson proves that $\sf RCA_n$ ($2<n<\omega)$
is not atom--canonical. Hodkinson's proof is semantical; ours is syntactical implemented by blowing up and blurring
a finite rainbow polyadic-equality algebra, in which `the number' of greens is $n+1$ and
the reds $n$.

The blow up and blur addition,  will allow us to refine and indeed strengthen
Hodkinson's result,  showing that for any class $\K$, such that $\bold S\Nr_n\CA_{n+3}\subseteq \K\subseteq \RCA_n$, $\K$
is  not atom--canonical.
This applies to the infinitely many varieties $\bold S\Nr_n\CA_{n+k}$, $k\geq 3$ \cite{t}.

Here the dimension $d=n+3$, which is the least $d$ for which we could prove that $\bold S\Nr_n\CA_{n+d}$ is not atom--canonical,
is determined by the number of greens, which we denote by  $\sf num(g)$ used in our rainbow construction;  ${\sf num({\sf g})}=n+1$. We have
$n+3={\sf num(\sf g)}+2$, where $2$  is the increase occurring in the number of nodes (pebbles)
resulting from  lifting  the \ws\ of \pa\ in the \ef\ forth private game between \pa\ and \pe\
on the complete irreflexive graphs
$n+1$ and $n$, to the number of nodes used by \pa\ in the graph game on
the rainbow algebra $\CA_{n+1, n}$.  

Blowing up and blurring $\CA_{n+1, n}$, by splitting the `red atoms' each into $\omega$ many, will give a representable
algebra $\A$ similar to the term algebra used by
Hodkinson; in fact the only difference is  that we use only $n+1$ greens not infinitely many.
As long as their number outfits the reds, \pa\ can win in a finite rounded game. The \de\ completion of $\At$, call it $\C$, will be
outside $\bold S\Nr_n\CA_{n+3}$, because $\CA_{n+1, n}$ is outside $\bold S\Nr_n\CA_{n+3}$
by the fact that  \pa\  has a \ws\  in $G^{n+3}$ using only $n+3$ nodes {\it without} the need to reuse them, witness lemma \ref{Thm:n}  
and $\CA_{n+1, n}$ embeds into $\C$.  So although $\A$ is ($\omega$--square) representable,  
by the first item of lemma \ref{Thm:n},  its \de\ completion $\C$ does not even have an $n+3$--flat representation.

The proof is divided into two parts:

\begin{itemize}
 \item In the first part we `blow up and blur' a finite rainbow algebra
$\D$ (denoted in the proof below by $\PEA_{n+1, n}$)
by splitting some of the atoms (the red ones),  each into infinitely many, getting a weakly representable atom structure $\At$,
and we embed $\D$ into  the complex algebra $\Cm\At$, the \de\ completion of $\Tm\At$.

\item In the second part, we show that \pa\ has a \ws\ in  $G^{n+3}(\At\D)$, hence
using lemma \ref{Thm:n}, we get that $\Rd_{sc}\D\notin
\bold S\Nr_n\Sc_{n+3}$. Here, unlike the first item,  additivity does not interfere 
because $\D$ is finite.
\end{itemize}

{\bf (a) Blowing up and blurring a finite rainbow algebra:} Take the finite rainbow algebra where the reds $\sf R$ is the complete irreflexive graph $n$, and the greens
are  ${\sf G}=\{\g_i:1\leq i<n-1\}
\cup \{\g_0^{i}: 1\leq i\leq n+1\}$, but endowed with the polyadic operations.
Call this finite rainbow $\PEA_n$, $\PEA_{n+1, n}$ and denote its atom structure by $\At_f$.
One  then splits the red atoms
of the finite rainbow algebra of $\PEA_{n+1, n}$ each into infinitely many, getting a weakly representable atom structure $\At$, that is
the term algebra $\Tm\At$ is representable. 

The resulting atom structure (after splitting),  call it $\At$, is the rainbow atom structure
that is like the atom structure of the weakly representable algebra $\A$ constructed in \cite{Hodkinson}, the only difference is that we have $n+1$ greens
and not infinitely many as is the case in \cite{Hodkinson}.

The rainbow signature \cite[definition 3.6.9]{HHbook2} $L$ now consists of $\g_i: 1\leq i<n-1$, $\g_0^i: 1\leq i\leq n+1$,
$\w_i: i<n-1$,  $\r_{kl}^t: k<l< n$, $t\in \omega$,
binary relations, and $n-1$ ary relations $\y_S$, $S\subseteq n+1$.

There is a shade of red $\rho$; the latter is a binary relation that is outside the rainbow signature,
but it labels coloured graphs during a `rainbow game'.
 \pe\ can win the rainbow $\omega$ rounded game
and build an $n$--homogeneous model $M$ by using $\rho$ when
she is forced a red;  \cite[Proposition 2.6, Lemma 2.7]{Hodkinson}. In essence $\rho$ 
can be identified with the non--principal ultrafilter consisting of all cofinite sets of red graphs, that is graphs that has at least one red rainbow label,
of $\Tm\At\A$ that is used as a colour to completely represent the canonical extension of 
$\Tm\At\A$. $\Tm\At\A$ itself is not, and cannot be, 
completely representable.

To build $M$, the class of coloured graphs is considered in
the signature $L\cup \{\rho\}$ with the two additional forbidden triples
$(\r, \rho, \rho)$ and $(\r, \r^*, \rho)$, where $\r, \r^*$ are any reds, but in forming $\At\A$ later on, $\rho$ is `deleted'. 
Let $\GG$ be the class of all models of this {\it extended rainbow first order theory}.
The extra shade of red $\rho$  will be used as a label.

The construction of $M$ is done in a step--by--step way, which can be implemented  via an 
$\omega$-- rounded game between \pa\ and \pe. The required $M$, in the expanded signature, 
will be the countable limit of the play. In this game \pe\ uses a rainbow strategy
in her choice of labels, but playing $\rho$ whenever forced a red label.
The required $M\in \GG$ is formed exactly as in \cite{Hodkinson}.

Now deleting the one available red shade, set
$W = \{ \bar{a} \in {}^n M : M \models ( \bigwedge_{i < j <n} \neg \rho(x_i, x_j))(\bar{a}) \},$
and for $\phi\in L_{\infty, \omega}^n$, let
$\phi^W=\{s\in W: M\models \phi[s]\}.$
Here $W$ is the set of all $n$--ary assignments in
$^nM$, that have no edge labelled by $\rho$.
Let $\A$  be the relativized set algebra with domain
$\{\varphi^W : \varphi \,\ \textrm {a first-order} \;\ L_n-
\textrm{formula} \}$  and unit $W$, endowed with the 
usual concrete operations. 

Classical semantics for $L_n$ rainbow formulas and their
semantics by relativizing to $W$ coincide \cite[Proposition 3.13]{Hodkinson}.
This depends essentially on \cite[Lemma 3.10]{Hodkinson} which says that any permutation $\chi$ of $\omega\cup \{\rho\}$,
$\Theta^{\chi}$ as defined in  \cite[Definitions 3.9, 3.10]{Hodkinson} is an $n$ back--and--forth system.

Hence $\A\in {\bold I}{\sf Pes}_n$, that is, $\A$ is isomorphic to a set algebra of dimension $n$, so $\A$ 
is simple, in fact its $\Df$ reduct is simple.
From now we forget about $\rho$; it will play no further role. 
We have $\At\A=\At\Tm$, and $\Tm\At\subseteq \A$, hence $\Tm\At$ is representable.
The atoms of both are the coloured graphs whose edges are not labelled by $\rho$.
These atoms are uniquely determined by $\sf MCA$ formulas in the rainbow signature of $\At$  as in
\cite[Definition 4.3]{Hodkinson}.

Let $\D=\{\phi^W: \phi\in L_{\infty, \omega}^n\}$
\cite[Definition 4.1]{Hodkinson}
with the operations defined like on $\A$ the usual way.
$\Cm\At$ is complete and, so like in \cite[Lemma 5.3]{Hodkinson}
we have an isomorphism from $\Cm\At$  to $\D$ defined
via $\Psi: X\mapsto \bigcup X$.

Now we embed $\PEA_{n+1,n}$ into $\Cm\At$.
Roughly one takes every coloured graph $\Gamma$ which is an atom in $\PEA_{n+1, n}$
to $\bigcup \phi'^M\in \D\cong \Cm\At$
where $\phi'$ is a {\it copy} of $\Gamma$. We explain what we mean by a copy.
For brevity, we write $\r$ for both $\r_{jk}$ and $\r_{jk}^l$ ($j<k<n, l\in \omega$); 
the superscripts and double indices will be clear from context.

We regard the rainbow signature of $\PEA_{n+1, n}$ as a sub-signature of the rainbow signature of $\A$, by identifying any 
red binary relation $\r$ in the former signature with $\r^0$ in the latter. (Here $\rho$ does not exist). 
The non--red rainbow symbols are the same in both signatures.
We write $M_a$ for the element of $\At_f$ or $\At$ for which
$a:n\to M$ is a surjection; $M$ a coloured graph in the signature of $\A$. 
Then $M_b\in \At$ is a copy of $M_a\in \At_f$ 
iff
\begin{itemize}
\item $a(i)=a(j)\Longleftrightarrow b(i)=b(j),$

\item $M_a(a(i), a(j))=\r\iff M_b(b(i), b(j))=\r^l, l\in \omega,$

\item $M_a(a(i), a(j))=M_b(b(i), b(j))$, if they are not red,

\item $M_a(a(k_0),\dots, a(k_{n-2}))=M_b(b(k_0),\ldots, b(k_{n-2}))$, whenever
defined.
\end{itemize}
Now the map $\Theta: \At_f\to \Cm\At$ defined by:
$$\Theta(\{M_a\})= \{ M_{b}: \text { $M_{b}$  is a copy of $M_a$}\}$$
induces an embedding from ${\sf PEA}_{n+1, n}$ into $\Cm\At$.

Consider now the map $\eta$ defined from $\At_f\to \D$ 
via $\{M_a\}\mapsto \bigcup \phi'^M,$ 
$a:n\to M$ where $M\in {\sf CRG}_f(\subseteq {\sf CRG})$ and
$\phi'$ the $\sf MCA$ formula
obtained  from $M\in {\sf CGR}$, by replacing a red edge $\r_{jk}$, if any,
by any red relation symbol 
$r_{jk}^l$, $j< k<n$, and $l\in \omega$, 
respecting consistent  red triangles. Then $\eta$ is the same embedding just defined from $\At_f\to \Cm\At$,
modulo the isomorphism $\Psi$  defined above between $\D$ and $\Cm\At$, 
in the sense that $\eta=\Theta\circ \Psi^{-1}.$\\

{\bf  (b) \pa\ s \ws\ in $F^{n+3}$ on $\At\Rd_{sc}\PEA_{n+1, n}$}: We start by  showing that \pa\ has \ws\ first in an  \ef\ forth  private game played between \pe\ and \pa\ on the complete
irreflexive graphs $n+1$ and $n$.
In each round $0,1\ldots n+1$, \pe\ places a  new pebble  on  element of $n+1$.
The edges relation in $n+1$ is irreflexive so to avoid losing
\pe\ must respond by placing the other  pebble of the pair on an unused element of $n+1$.
After $n+1$ rounds there will be no such element,
and she loses in the next round.

This game lifts to a graph game  \cite[pp.841]{HH} on $\At_f$ which  in this
case is equivalent to the graph version of $F^{n+3}$ but unlike the situation in item (2), here
\pa\ does not need to re-use pebbles, so that the game is actually $G^{n+3}$.
We show that  \pa\ can win the graph game on $\At_f$
using the standard rainbow strategy \cite{HH}. He bombards \pe\ with cones have the same base and green tints,
forcing \pa\ to play an inconsistent triple of reds whose indices do not match.

In his zeroth move, \pa\ plays a graph $\Gamma$ with
nodes $0, 1,\ldots, n-1$ and such that $\Gamma(i, j) = \w_0 (i < j <
n-1), \Gamma(i, n-1) = \g_i ( i = 1,\ldots, n-2), \Gamma(0, n-1) =
\g^0_0$, and $ \Gamma(0, 1,\ldots, n-2) = \y_{n+1}$. This is a $0$-cone
with base $\{0,\ldots, n-2\}$. In the following moves, \pa\
repeatedly chooses the face $(0, 1,\ldots, n-2)$ and demands a node
$\alpha$ with $\Phi(i,\alpha) = \g_i$, $(i=1,\ldots, n-2)$ and $\Phi(0, \alpha) = \g^\alpha_0$,
in the graph notation -- i.e., an $\alpha$-cone, $\alpha\leq n+2$,  on the same base.
\pe\ among other things, has to colour all the edges
connecting new nodes created by \pa\ as apexes of cones based on the face $(0,1,\ldots, n-2)$. By the rules of the game
the only permissible colours would be red. Using this, \pa\ can force a
win, using $n+3$ nodes.

Thus by lemma \ref{Thm:n}, $\Rd_{sc}\PEA_{n+1, n}\notin
\bold S\Nr_n\Sc_{n+3}$. Since $\PEA_{n+1,n}$ embeds into $\Cm\At\A$,
hence $\Rd_{sc}\Cm\At=\Cm\Rd_{sc}\At$
is outside $\bold S\Nr_n\Sc_{n+3}$, too.
By lemma \ref{Thm:n}, $\A$ hence $\C$, are outside $\bold S\Nr_n\sf G_{n+3}$.
Note that the diagonal free reduct of $\A$ has no complete representation, too \cite[Proposition 4.10]{AU}.

This idea of embedding a rainbow algebra in the \de\ completion of a representable term algebra obtained by blowing up and blurring the rainbow algebra, 
{\it also works} in the case of Hodkinson's construction by {\it blowing up and blurring the non--representable 
{\it infinite} rainbow  cylindric algebra  $\CA_{\omega,n}$} which has greens $\g_0^i: i<\omega$ and reds $\r_{kl}$, $k<l<n$ by splitting the reds like we did, thereby obtaining the term 
algebra $\Tm\At\A\subseteq \A$ used in \cite{Hodkinson}.
The algebra $\CA_{\omega, n}$ is embeddable in $\Cm\At\A$, with $\A$ as
defined in \cite{Hodkinson} (with infinitely many greens), 
so we readily conclude that $\Cm\At\A$ is not representable.

In fact, for any $2<n<k$, such that  $n+1\leq k\leq \omega$, blowing up and bluring the rainbow algebra $\CA_{k, n}$, 
by splitting the red atoms with double distinct indices coming from $n$, as we did, 
gives, a representable term algebra $\A$, call it ${\sf Split}(\CA_{k, n}, \r, \omega)$, short for splitting each red $\r=\r_{ij}$ $(i<j<n)$ in $\CA_{k, n}$ 
into $\omega$ many copies.  
The \de\ completion $\C_k=\Cm\At[{\sf Split}(\CA_{k, n}, \r, \omega)]$
is outside  $\bold S\Nr_n\CA_{2+k}$, 
hence by theorem \ref{flat}, $\C_k$ does not have an $2+k$--flat representation.  
When $k$ is infinite, then by $2+k$ we mean ordinal addition, 
so that $2+k=k$.

Hodkinson's proof in {\it op.cit} is different. He proves that
a certain tuple in the representation, assuming a representation exists, of $\Cm\At[{\sf Split}(\CA_{\omega,n}, \r, \omega)]$ cf., \cite[Lemma 5.7]{Hodkinson}
will be the base of infinitely many cones, that will be used to force an inconsistent triple of reds, concluding that it is  not representable.

Conversely, Hodkinson's semantical argument works here, too. Assuming the existence of only {\it an $2+k$--flat representation of 
$\Cm\At[{\sf Split}(\CA_{k, n}, \r, \omega))$} 
it can be proved that a certain tuple in this relativized representation will be the base of 
{$2+k$ many cones}, forcing an inconsistent triple of reds.  
(In such an argument  a simple combinatorial application of the Pigeon hole principle is used.)
Using the completeness theorem ` $\A$ has an  $m$--dilations $\iff$ $\A$ has an $m$--flat repesentation'  
the two approaches (proofs) 
are two sides of the same coin.\\

{\bf Proof of (3):} The proof is divided into two parts. In the first we deal with (non--finite) axiomatizability. 
And in the second
we deal with atom--canonicity for infinite dimensional algebras.\\

{\bf (a) Non--finite axiomatizability:}
At the start we need to recall a piece of notation from \cite{HMT2}. Let $m\leq n$ be ordinals and let $\rho:m\rightarrow n$ be an injection.
For any $n$--dimensional algebra $\B$ (substitution, cylindric or quasi--polyadic algebra with or without equality)
we define an $m$-dimensional algebra $\Rd^\rho\B$, with the same universe and Boolean structure as
$\B$, where the $(ij)$th diagonal of $\Rd^\rho\B$ is $\diag {\rho(i)}{\rho(j)}\in\B$
(if diagonals are included in the signature of the algebra), the $i$th cylindrifier is $\cyl{\rho(i)}$, the $i$ for $j$
replacement operator is the operator $\s^{\rho(i)}_{\rho(j)}$ of $\A$, the $ij$ transposition operator is $\s_{\rho(i)\rho(j)}$
(if included in the signature), for $i, j<m$.  It is easy to check, for $\K\in\set{\Df,\Sc, \CA, \sf QA, \sf QEA}$,
that if $\B\in\K_n$ then $\Rd^\rho\B\in\K_m$.    Also, for $\B\in\K_n$ and $x\in \B$,
we define $\Rl_x\B$ by `restriction to $x$', so the universe is the set of elements of $\B$ below $x$, where the Boolean unit is $x$,
Boolean zero and sum are not changed, Boolean complementation is relative to $x$,
and the result of applying any non-Boolean operator is obtained by using the operator for $\B$
and intersecting with $x$. It is not always the case
that $\Rl_x\B$ is a $\K_{n}$ (we can lose commutativity of cylindrifiers). However, if $x$ is {\it  rectangular}, in the sense that all $i<j<n$, 
${\sf c}_ix\cdot {\sf c}_jx=x$, then $\Rl_x\B\in \K_n$. This is used below.\\

The idea used here is the same idea used in \cite{t}; both are instances of Monk's trick. But here the result that we lift for finite dimensions is stronger than 
that obtained for finite dimensions in \cite{t}, hence
our result obtained for infinite dimensions is stronger than that obtained in \cite{t} when restricted to any
$\K$ between $\CA$ and $\QEA$. 

Let $\mathfrak{C}(m,n,r)=\Ca(H_m^{n+1}(\A(n,r), \omega)),$
consisting of all $n+1$--wide $m$--dimensional
wide $\omega$ hypernetworks \cite[Definition 12.21]{HHbook}
on $\A(n,r)$  as defined in \cite[Definition 15.2]{HHbook},  is a $\CA_m$, and it can be easily expanded
to a $\PEA_m$, since $\C(m, n, r)$ is symmetric; it allows natural polyadic operations
corresponding to transpositions.

Furthermore, for any $r\in \omega$ and $3\leq m\leq n<\omega$, $\C(m,n,r)\in \Nr_m{\sf PEA}_n$, $\Rd_{ca}\C(m,n,r)\notin {\bold  S}\Nr_m{\sf CA_{n+1}}$
and $\Pi_{r/U}\C(m,n,r)\in {\sf RPEA}_m$ by easily
adapting \cite[Corollaries 15.7, 5.10, Exercise 2, pp. 484, Remark 15.13]{HHbook}
to the $\PEA$ context.

Let  $3\leq m<n$.
Take
$$x_n=\{f\in H_n^{n+k+1}(\A(n,r), \omega); m\leq j<n\to \exists i<m, f(i,j)=\Id\}.$$
Then $x_n\in \C(n,n+k,r)$ and ${\sf c}_ix_n\cdot {\sf c}_jx_n=x_n$ for distinct $i, j<m$.
Furthermore (*),
$I_n:\C(m,m+k,r)\cong \Rl_{x_n}\Rd_m {\C}(n,n+k, r)$
via the map, defined for $S\subseteq H_m^{m+k+1}(\A(m+k,r), \omega)),$ by
$$I_n(S)=\{f\in H_n^{n+k+1}(\A(n,r), \omega):  f\upharpoonright {}^{\leq m+k+1}m\in S,$$
$$\forall j(m\leq j<n\to  \exists i<m,  f(i,j)=\Id)\}.$$
We have proved the (known) result for finite ordinals $>2$.
To lift the result to the transfinite,
we proceed like in \cite{t}, using Monk's trick.

Let $I=\{\Gamma: \Gamma\subseteq \alpha,  |\Gamma|<\omega\}$.
For each $\Gamma\in I$, let $M_{\Gamma}=\{\Delta\in I: \Gamma\subseteq \Delta\}$,
and let $F$ be an ultrafilter on $I$ such that $\forall\Gamma\in I,\; M_{\Gamma}\in F$.
For each $\Gamma\in I$, let $\rho_{\Gamma}$
be an injective function from $|\Gamma|$ onto $\Gamma.$

Let ${\C}_{\Gamma}^r$ be an algebra similar to $\QEA_{\alpha}$ such that
$\Rd^{\rho_\Gamma}{\C}_{\Gamma}^r={\C}(|\Gamma|, |\Gamma|+k,r)$
and let
$\B^r=\Pi_{\Gamma/F\in I}\C_{\Gamma}^r.$
Then we have $\B^r\in \bold \Nr_\alpha\QEA_{\alpha+k}$ and
$\Rd_{ca}\B^r\not\in \bold S\Nr_\alpha\CA_{\alpha+k+1}$.
These can be proved exactly like the proof of the first two items in \cite[Theorem 3.1]{t}. The second part uses (*).

We know 
from the finite dimensional case that $\Pi_{r/U}\Rd^{\rho_\Gamma}\C^r_\Gamma=\Pi_{r/U}\C(|\Gamma|, |\Gamma|+k, r) \subseteq \Nr_{|\Gamma|}\A_\Gamma$,
for some $\A_\Gamma\in\QEA_{|\Gamma|+\omega}=\QEA_{\omega}$.
Let $\lambda_\Gamma:\omega\rightarrow\alpha+\omega$
extend $\rho_\Gamma:|\Gamma|\rightarrow \Gamma \; (\subseteq\alpha)$ and satisfy
$\lambda_\Gamma(|\Gamma|+i)=\alpha+i$
for $i<\omega$.  Let $\F_\Gamma$ be a $\QEA_{\alpha+\omega}$ type algebra such that $\Rd^{\lambda_\Gamma}\F_\Gamma=\A_\Gamma$.
Then $\Pi_{\Gamma/F}\F_\Gamma\in\QEA_{\alpha+\omega}$, and we have proceeding like in the proof of item 3 in \cite[Theorem 3.1]{t}:
\begin{align*}
\Pi_{r/U}\B^r&=\Pi_{r/U}\Pi_{\Gamma/F}\C^r_\Gamma\\
&\cong \Pi_{\Gamma/F}\Pi_{r/U}\C^r_\Gamma\\
&\subseteq \Pi_{\Gamma/F}\Nr_{|\Gamma|}\A_\Gamma\\
&=\Pi_{\Gamma/F}\Nr_{|\Gamma|}\Rd^{\lambda_\Gamma}\F_\Gamma\\
&=\Nr_\alpha\Pi_{\Gamma/F}\F_\Gamma.
\end{align*}
But $\B=\Pi_{r/U}\B^r\in \bold S\Nr_{\alpha}\QEA_{\alpha+\omega}$
because $\F=\Pi_{\Gamma/F}\F_{\Gamma}\in \QEA_{\alpha+\omega}$
and $\B\subseteq \Nr_{\alpha}\F$, hence it is representable (here we use the neat embeding theorem).
Now it can be easily shown that 
that $\bold S\Nr_{\alpha}\CA_{\alpha+k+ l}$ is not axiomatizable
by a finite schema over  $\bold S\Nr_{\alpha}\CA_{\alpha+k}$ in the sense of \cite[Definition 5.4.12]{HMT2} for any $l\geq 1$.

In \cite[Theorem 3.1]{t}, the ultraproduct was proved to be in
$\bold S\Nr_{\alpha}\K_{\alpha+k+1}$ for $\K$ between $\Sc$ and $\QEA$,  
a strict superset of $\sf RK_{\alpha}$. In fact, the result here is `infinitely stronger'. Using a L\'os argument, we have $\sf RK_{\alpha}$
cannot be axiomatized by a finite schema over
$\bold S\Nr_{\alpha}\K_{\alpha+m}$ 
for any finite $m\geq 0$, and for any 
$\K$ between $\CA$ and $\QEA$.

Before atom--canonicity we prove the last part, since we make use of the above algebras. Fix $2<n<m<\omega$. 
Note that from our above arguments, it follows directly  that $\sf RPEA_n$ is not finitely axiomatizable over $\bold S\Nr_n\PEA_m$.

Now write $\mathfrak{C}_r$ for the $m$--dimensional polyadic equality 
$\mathfrak{C}(m,n,r)$ not to clutter notation. The parameters $m$ and $n$ will be clear from context.
Given $k$, then for any $r\geq  k^2$, we have \pe\ has a \ws\ in $G^k$
on $\A(n,r)$ \cite[Remark 15.13]{HHbook}. This implies using ultraproducts and an elementary chain argument that \pe\
has a \ws\ in the $\omega$--rounded game,
in an elementary substructure of $\Pi_{r/U}\A(n,r)/F$,
hence the former is representable, and then so is the latter because
${\sf RRA}$ is a variety. To show that $\Pi_{r/U} \mathfrak{C}^r/F$ is also representable, it suffices to note
that there is  a representation of
an algebra $\A\prec \Pi_{r/U} \A(n,r)/F$  that embeds {\it all}
$m$ dimensional hypernetworks, respecting $\equiv_i$ for all $i<m$ \cite[Exercise 2, p.484]{HHbook}.

Now \pe\ has a \ws\ in $G^k_{\omega}(\A(n,r))$ when $r\geq k^2$,
hence, $\A(n,r)$ embeds into a complete atomic relation algebra having
a $k$--dimensional relational basis by \cite[Theorem 12.25]{HHbook}. But this induces a \ws\ for \pe\ on $\At\mathfrak{C}_r$ in $k'\geq k$, $k'\in \omega$, rounds
in the basis cylindric
atomic game $G^{k'}_{\omega}(\mathfrak{C}_r)$ (with $k'$ nodes and $\omega$ rounds, defined in the second part of the proof of lemma \ref{Thm:n}) 
so that $\mathfrak{C}_r\in \sf \bold L_{m, k'}$ when $r\geq k'{^2}$.
So if $n\geq m+2$, $k\geq 3$, and $r\geq k'^2$, then we have
$\mathfrak{C}_r\in  ({\bold L}_{m,n+1}\cap \bold S\Nr_m\PEA_n)\sim \bold S\Nr_m\PEA_{n+1}$,
and $\Pi_{r/U} \mathfrak{C}_r/F\in \sf RPEA_m\supseteq \bold S\Nr_m\PEA_{n+1}$, 
and we are done.\\

{\bf (b)  Atom canonicity:} For non atom--canonicity for infinite dimensions, we use 
Monk's trick again,
lifting the construction in item (2) to the transfinite.
To simplify matters, without losing the gist of the idea, we take $\alpha$ to be the least infinite ordinal $\omega$, and we restrict our attention to cylindric algebras.
The idea for both uncountable ordinals and other algebras is essentially the same (using ultraproducts as in item (4) above).
For each finite $n\geq 3$, let $\A_n$
be an atomic countable simple representable
algebra such that  $\B_n=\Cm\At\A_n\notin \bold S\Nr_n\CA_{n+3}.$ We know by item (3) that such algebras exist.

Let $\A_n^+$ be an algebra having the signature of $\CA_{\omega}$
such that $\Rd_n\A_n^+=\A_n$.
Analogously, let $\B_n^+$ be an algebra having the signature
of $\CA_{\omega}$ such that $\Rd_n\B_n^+=\B_n$, and we require in addition that $\B_n^+=\Cm(\At\A_n^+)$.
Let $\B=\Pi_{i\in \omega}\B_i^+/F$.
As before we have
$\A=\Pi_{i\in \omega}\A_i^+/F\in \RCA_{\omega}$.
Furthermore,
\begin{align*}
\Cm\At\A&=\Cm(\At[\Pi_{i\in \omega}\A_n^+/F])\\
&=\Cm[\Pi_{i\in \omega}(\At\A_n^+)/F)]\\
&=\Pi_{i\in \omega}(\Cm(\At\A_n^+)/F)\\
&=\Pi_{i\in \omega}\B_n^+/F\\
&=\B.
\end{align*}
We also have $\B\in \CA_{\omega}$. We now show that $\B$ is outside $\bold S\Nr_{\omega}\CA_{\omega+3}$.
Assume for contradiction that $\B\in \bold S\Nr_{\omega}\CA_{\omega+3}$.
Then $\B\subseteq \Nr_{\omega}\C$ for some $\C\in \CA_{\omega+3}$.
Let $3\leq m<\omega$ and  $\lambda:m+3\rightarrow \omega+3$ be the function defined by $\lambda(i)=i$ for $i<m$
and $\lambda(m+i)=\omega+i$ for $i<3$.

Then $\Rd^\lambda\C\in \CA_{m+3}$ and $\Rd_m\B\subseteq \Rd_m\Rd^\lambda\C$.
Suppose for the time being that $\B_m$ embeds into $\Rd_m\B_{t}$, whenever $3\leq m<t<\omega,$ via
$I_t: \B_m\to \Rd_m\B_t$.
Let $\iota( b)=(I_{t}b: t\geq m )/F$ for  $b\in \B_m$.
Then $\iota$ is an injective homomorphism that embeds $\B_m$ into
$\Rd_m\B$.  By the above we have $\Rd_{m}\B\in {\bf S}\Nr_m\CA_{m+3}$, hence  $\B_m
\in \bold S\Nr_{m}\CA_{m+3}$, too
which is a contradiction, and we will be done.

Now we prove our assumption finishing the proof.
Let $t_1$ be the rainbow signature of $\B_m$ and $t_2$ be the rainbow signature of $\B_t$; then
identifying algebras with their universes,
we have, modulo isomorphisms (witness item (2)),
$\B_m=\{\phi^{W_1}: \phi\in L_{\infty, \omega}^m\}$ and $\B_t=\{\phi^{W_2}: \phi\in L_{\infty, \omega}^t\}$ where
$W_1 = \{ \bar{a} \in {}^m M : M \models ( \bigwedge_{i < j < n,
l < n} \neg \rho(x_i, x_j))(\bar{a}) \},$
and $W_2$ is defined analogously by replacing $m$ by $t$.
We can assume without loss that $t=m+1$, then we proceed inductively.

Fix $m$ as above. Let $\eta:m-1\to m$ be an injection. Define $\eta^+:m\to m+1$ by
$\eta^+\upharpoonright m-1=\eta$ and $\eta^+(m-1)$ is the unique
$a\in m\sim \rng(\eta)$. Then $\eta^+: m\to m+1$ is an injection such that $m\notin \rng(\eta^+)$.
Conversely, let $\eta: m\to m+1$ be such that $m\notin \rng(\eta)$. Let
$\eta_+=\eta\upharpoonright m-1$, Then $\eta_+:m-1\to m$
is an injection. Also $(\eta_+)^+=\eta$.

Now let $In(A,B)$ denote the set of all injective functions from $A$ to $B$.
Using the correspondence established in the last paragraph, one maps coloured graphs corresponding to 
$\sf MCA$ formulas in the signature $t_1$, that is, atoms in $\B_m$
to coloured graphs corresponding to $\sf MCA$ formulas in $\B_t$, that is atoms in $\B_t$ as follows.

Let
$\alpha=\bigwedge_{i \neq j < m} \alpha_{ij}(x_i, x_j) \land \bigwedge_{\mu\in In(m-1, m)}\eta_{\mu}(x_{\mu(0)},\ldots, x_{\mu(m-2)}),$
be an $\sf MCA$ formula in the signature $t_1$, so that for each $i,j<m$, $\alpha_{ij}$ is either $x_i=x_i$ or $R(x_i,x_j)$ a binary relation symbol in the rainbow signature,
and for each $\mu:(m-1)\to m$, $\eta_{\mu}$ is either $y_S(x_{\mu(0)},\ldots x_{\mu(m-2)})$ for some $y_S$ in the signature,
if for all distinct $i,j<m$, $\alpha_{\mu(i), \mu(j)}$ is not equality nor green, otherwise it is
$x_0=x_0$ \cite{Hodkinson}.
Then $\alpha^{W_1}$ is mapped
to $\alpha'^{W_2}$
where $\alpha'$ is the following $\sf MCA$ formula in the signature $t_2$:
$\bigwedge_{i \neq j < m} \alpha_{ij}(x_i, x_j)\land \bigwedge_{i<m} x_m=x_i \land \psi_1\land \psi_2$
where $$\psi_1= \bigwedge_{\mu^+\in In(m, m+1), m\notin \rng(\mu^+)}
\eta_{\mu}(x_{\mu^+(0)},\ldots, x_{\mu^+(m-1)})$$
and $\psi_2$ is the same conjunction except that we take those $\mu^+$ for which $m\in \rng(\mu^+)$.
The map is extended (to non--atoms) the obvious way.\\

{\bf Proof of (4):} We first prove the required for any finite $n>2$. Then using Monk's trick we lift it to the transfinite. We use 
the construction in \cite[Remark 31]{r}.\\ 

{\bf (a) Finite dimensional algebras:}
The proof in \cite{r} is only sketched. Here we give
the details since it will be used in the second part of the paper violating an omitting types theorem. The example shows that the maximality 
condition in \cite[Theorem 3.2.9]{Sayed}, reformulated  and proved in theorem \ref{Shelah} below, 
cannot be omitted.  Witness too, the last paragraph in \cite{Sayed}.

We specify the atoms and forbidden triples. The atoms are $\Id, \; \g_0^i:i<2^{\kappa}$ and $\r_j:1\leq j<
\kappa$, all symmetric.  The forbidden triples of atoms are all
permutations of $({\sf Id}, x, y)$ for $x \neq y$, \/$(\r_j, \r_j, \r_j)$ for
$1\leq j<\kappa$ and $(\g_0^i, \g_0^{i'}, \g_0^{i^*})$ for $i, i',
i^*<2^{\kappa}.$  In other words, we forbid all the monochromatic triangles.

Write $\g_0$ for $\set{\g_0^i:i<2^{\kappa}}$ and $\r_+$ for
$\set{\r_j:1\leq j<\kappa}$. Call this atom
structure $\alpha$. Let $\A$ be the term algebra on this atom
structure. We claim that $\A$, as a relation algebra,  has no complete representation.
Assume for contradiction that $\A$ has a complete representation $M$.  Let $x, y$ be points in the
representation with $M \models \r_1(x, y)$.  For each $i< 2^{\kappa}$, there is a
point $z_i \in M$ such that $M \models \g_0^i(x, z_i) \wedge \r_1(z_i, y)$.

Let $Z = \set{z_i:i<2^{\kappa}}$.  Within $Z$ there can be no edges labelled by
$\r_0$ so each edge is labelled by one of the $\kappa$ atoms in
$\r_+$.  The Erdos-Rado theorem forces the existence of three points
$z^1, z^2, z^3 \in Z$ such that $M \models \r_j(z^1, z^2) \wedge \r_j(z^2, z^3)
\wedge \r_j(z^3, z_1)$, for some single $j<\kappa$.  This contradicts the
definition of composition in $\A$ (since we avoided monochromatic triangles).
Let $S$ be the set of all atomic $\A$-networks $N$ with nodes
$\omega$ such that $\{\r_i: 1\leq i<\kappa: \r_i \text{ is the label
of an edge in $N$}\}$ is finite.
Then it is straightforward to show $S$ is an amalgamation class, that is for all $M, N
\in S$ if $M \equiv_{ij} N$ then there is $L \in S$ with
$M \equiv_i L \equiv_j N$, witness \cite[Definition 12.8]{HHbook} for notation.
We have $S$ is symmetric, that is, if $N\in S$ and $\theta:\omega\to \omega$ is a finitary function, in the sense
that $\{i\in \omega: \theta(i)\neq i\}$ is finite, then $N\theta$ is in $S$. It follows that the complex
algebra $\Ca(S)\in \QEA_\omega$.
Now let $X$ be the set of finite $\A$-networks $N$ with nodes
$\subseteq\kappa$ such that:

\begin{enumerate}
\item each edge of $N$ is either (a) an atom of
$\A$ or (b) a cofinite subset of $\r_+=\set{\r_j:1\leq j<\kappa}$ or (c)
a cofinite subset of $\g_0=\set{\g_0^i:i<2^{\kappa}}$ and

\item  $N$ is `triangle-closed', i.e. for all $l, m, n \in \nodes(N)$ we
have $N(l, n) \leq N(l,m);N(m,n)$.  That means if an edge $(l,m)$ is
labelled by $\sf Id$ then $N(l,n)= N(m,n)$ and if $N(l,m), N(m,n) \leq
\g_0$ then $N(l,n)\cdot \g_0 = 0$ and if $N(l,m)=N(m,n) =
\r_j$ (some $1\leq j<\omega$) then $N(l,n)\cdot \r_j = 0$.
\end{enumerate}

For $N\in X$ let $N'\in\Ca(S)$ be defined by
$$\set{L\in S: L(m,n)\leq
N(m,n) \mbox{ for } m,n\in \nodes(N)}.$$
For $i\in \omega$, let $N\restr{-i}$ be the subgraph of $N$ obtained by deleting the node $i$.
Then if $N\in X, \; i<\omega$ then $\cyl i N' =
(N\restr{-i})'$.
The inclusion $\cyl i N' \subseteq (N\restr{-i})'$ is clear.
Conversely, let $L \in (N\restr{-i})'$.  We seek $M \equiv_i L$ with
$M\in N'$.  This will prove that $L \in \cyl i N'$, as required.

Since $L\in S$ the set $T = \set{\r_i \notin L}$ is infinite.  Let $T$
be the disjoint union of two infinite sets $Y \cup Y'$, say.  To
define the $\omega$-network $M$ we must define the labels of all edges
involving the node $i$ (other labels are given by $M\equiv_i L$).  We
define these labels by enumerating the edges and labeling them one at
a time.  So let $j \neq i < \kappa$.  Suppose $j\in \nodes(N)$.  We
must choose $M(i,j) \leq N(i,j)$.  If $N(i,j)$ is an atom then of
course $M(i,j)=N(i,j)$.  Since $N$ is finite, this defines only
finitely many labels of $M$.  If $N(i,j)$ is a cofinite subset of
$\g_0$ then we let $M(i,j)$ be an arbitrary atom in $N(i,j)$.  And if
$N(i,j)$ is a cofinite subset of $\r_+$ then let $M(i,j)$ be an element
of $N(i,j)\cap Y$ which has not been used as the label of any edge of
$M$ which has already been chosen (possible, since at each stage only
finitely many have been chosen so far).  If $j\notin \nodes(N)$ then we
can let $M(i,j)= \r_k \in Y$ some $1\leq k < \kappa$ such that no edge of $M$
has already been labelled by $\r_k$.  It is not hard to check that each
triangle of $M$ is consistent (we have avoided all monochromatic
triangles) and clearly $M\in N'$ and $M\equiv_i L$.  The labeling avoided all
but finitely many elements of $Y'$, so $M\in S$. So
$(N\restr{-i})' \subseteq \cyl i N'$.

Now let $X' = \set{N':N\in X} \subseteq \Ca(S)$.
Then the subalgebra of $\Ca(S)$ generated by $X'$ is obtained from
$X'$ by closing under finite unions.
Clearly all these finite unions are generated by $X'$.  We must show
that the set of finite unions of $X'$ is closed under all cylindric
operations.  Closure under unions is given.  For $N'\in X$ we have
$-N' = \bigcup_{m,n\in \nodes(N)}N_{mn}'$ where $N_{mn}$ is a network
with nodes $\set{m,n}$ and labeling $N_{mn}(m,n) = -N(m,n)$. $N_{mn}$
may not belong to $X$ but it is equivalent to a union of at most finitely many
members of $X$.  The diagonal $\diag ij \in\Ca(S)$ is equal to $N'$
where $N$ is a network with nodes $\set{i,j}$ and labeling
$N(i,j)=\sf Id$.  Closure under cylindrification is given.
Let $\C$ be the subalgebra of $\Ca(S)$ generated by $X'$.
Then $\A = \Ra(\C)$.
To see why, each element of $\A$ is a union of a finite number of atoms,
possibly a co-finite subset of $\g_0$ and possibly a co-finite subset
of $\r_+$.  Clearly $\A\subseteq\Ra(\C)$.  Conversely, each element
$z \in \Ra(\C)$ is a finite union $\bigcup_{N\in F}N'$, for some
finite subset $F$ of $X$, satisfying $\cyl i z = z$, for $i > 1$. Let $i_0,
\ldots, i_k$ be an enumeration of all the nodes, other than $0$ and
$1$, that occur as nodes of networks in $F$.  Then, $\cyl
{i_0} \ldots
\cyl {i_k}z = \bigcup_{N\in F} \cyl {i_0} \ldots
\cyl {i_k}N' = \bigcup_{N\in F} (N\restr{\set{0,1}})' \in \A$.  So $\Ra(\C)
\subseteq \A$.
$\A$ is relation algebra reduct of $\C\in\QEA_\omega$ but has no
Let $n>2$. Let $\B=\Nr_n \C$. Then
$\B\in \Nr_n\QEA_{\omega}$, is atomic, but has no complete representation; in fact, because it is generated by its two dimensional elements, 
and its dimension is at least three, its
$\Df$ reduct is not completely representable \cite[Proposition 4.10]{AU}.

Now by theorem \ref{Thm:n} \pe\ has a \ws\ in $G_{\omega}(\At\B)$, hence  \pe\ has a \ws\ in $G_k(\At\B)$ for all $k<\omega$. 
Using ultrapowers and an elementary chain argument, we get
that $\B\equiv \C$, so that $\C$ is atomic and \pe\ has a \ws\ in $G_{\omega}(\At\C)$. Since $\C$ is countable
then by \cite[Theorem 3.3.3]{HHbook2} it is completely representable. We have 
proved that $\B\in {\bf El}{\sf CRK}_n$. We note that $\C$ is atomless because, as stated in \cite[Remark 31]{r}, for any $N\in X$, we can add an extra node 
extending
$N$ to $M$ such that $\emptyset\subset M'\subset N'$, so that $N'$ cannot be an atom.

If $n\leq 2$, then $\sf CRK_n$ is elementary; it coincide with the class of atomic completely additve algebras in $\sf RK_n$, 
hence in this case $\Nr_n\C$ is completely representable.
Indeed, complete additivity of an operator on an atomic $\sf BAO$ is a first order property.
If $\A$ is an atomic $\sf BAO$ and $f$ is a unary modality on $\A$,
then the formula
$y\neq 0\to \exists x(\At(x)\land f(x)\cdot  y\neq 0),$
where $\At(x)$ is the first order
formula asserting that $x$ is an atom, forces $f$ to be completely additive.
To see why, let $\A$ be an atomic algebra
with set of atoms $X$, and with a unary modality $f$ in its signature that validates the above stipulated formula.
Then it suffices to show that
$\sum_{x\in X}f(x)=1 (=f(\sum X)$).  If not, let $a=1-\sum_{x\in X} f(x)$.
Then $a\neq 0$. But then  (using this formula), there exists $x'\in X$, such that
$0\neq f(x')\cdot a
=f(x')\cdot (1 - \sum_{x\in X}f(x))
=f(x')-\sum_{x\in X}f(x),$
which is impossible.\\

{\bf (b) Infinite dimensions:} Now, using Monk's trick one more time,  we lift the previous result to the transfinite.
We use that for any finite ordinal $k$ and any infinite ordinal
$\beta$, we have (*) $\Nr_n\QEA_{\omega}=\Nr_n\QEA_{\beta}$ \cite[Theorem 2.6.35]{HMT2}.

For each finite $k\geq 3$, let $\C(k)$ be an uncountable algebra in $\Nr_k\QEA_{\omega}$ 
that is not completely representable; such algebras 
were constructed above. 
Let $I=\{\Gamma: m\subseteq \Gamma\subseteq \alpha,  |\Gamma|<\omega\}$.
Define $F$ an ultrafilter on $\wp(I)$, $\rho_{\Gamma}$ $(\Gamma\subseteq \alpha)$
as in the previous item.
 Let ${\C}_{\Gamma}$ be an algebra similar to $\QEA_{\alpha}$ such that
$\Rd^{\rho_\Gamma}{\C}_{\Gamma}={\C}(|\Gamma|)=\Nr_{|\Gamma|}\D_{\Gamma}\in \QEA_{|\Gamma|},$
where $\D_{\Gamma}\in \QEA_{\alpha+\omega}$. Such a $\D_{\Gamma}$ exists by (*). 
Let  
$\B=\Pi_{\Gamma/F}\C_{\Gamma}.$
We will prove that $\B\in \bold \Nr_\alpha\QEA_{\alpha+\omega},$ 
and that $\Rd_{sc}\B$ is not completely representable.

Let $\sigma_{\Gamma}$ be an injective function
$\omega\rightarrow(\alpha+\omega)$ such that $\rho_{\Gamma}\subseteq \sigma_{\Gamma}$
and $\sigma_{\Gamma}(|\Gamma|+i)=\alpha+i$ for every $i<\omega$. Let $\A_{\Gamma}$ be an algebra similar to a
$\QEA_{\alpha+\omega}$ such that
$\Rd^{\sigma_\Gamma}\A_{\Gamma}=\D_{\Gamma}$.
Then from \cite[Lemma 3.2]{t}, we have
$\Pi_{\Gamma/F}\A_{\Gamma}\in \QEA_{\alpha+\omega}$.
We prove that $\B=\Nr_\alpha\Pi_{\Gamma/F}\A_\Gamma$.  
For each $\Gamma\in I$ we have: 
\begin{align*}
\Rd^{\rho_{\Gamma}}\C_{\Gamma}&=\C(|\Gamma|)\\
&\cong\Nr_{|\Gamma|}\D_{\Gamma}\\
&=\Nr_{|\Gamma|}\Rd^{\sigma_{\Gamma}}\A_{\Gamma}\\
&=\Rd^{\sigma_\Gamma}\Nr_\Gamma\A_\Gamma\\
&=\Rd^{\rho_\Gamma}\Nr_\Gamma\A_\Gamma.
\end{align*}
Then from \cite[Lemma 3.3]{t}, using a standard Lo\'s argument we have:
$$\Pi_{\Gamma/F}\C_\Gamma\cong\Pi_{\Gamma/F}\Nr_\Gamma\A_\Gamma=\Nr_\alpha\Pi_{\Gamma/F}\A_\Gamma.$$
We are through with the first required.

For the second part, we proceed as follows.  Assume for contradiction
that $\Rd_{sc}\B$ (which is atomic) is completely representable, with isomorphism $f$ establishing the complete representation.
Identifying set algebras with their domain, we have $f: \Rd_{sc}\B\to \wp(V)$, where $V$ is a generalized weak space.
Let $3\leq m<\omega$. Then  $\C=\Rd_m\Rd_{sc}\B$ is completely representable, via $f\upharpoonright \C$, by noting that 
$\Rd_m\wp(V)\cong \wp(W)$ for some $W$; a generalized space of dimension $m$, and that this isomorphism preserves infinite
intersections. 

In more detail, let $U$ be the base of $\D=\wp(V)$, that is $U=\bigcup_{s\in V}\rng(s)$ and let $d\in \D$, $d\neq 0$. 
Fix $x\in d$. For each $s\in {}^nU$, let $s^+=x\upharpoonright (\alpha\sim n)\cup s$, 
then $s^+\in {}^{\alpha}U$. For $Y\in \D$, let $g_d(Y)=\{y\in {}^nU: y^+\in Y\}$. 
Then $g_d:\Rd_m\wp(V)\to \wp(^nU)$ is a homomorphism that preserves infinite intersections, and 
$g_d(d)\neq 0$. By taking the direct product of images (varying non-zero $d\in \wp(V)$), we get an embedding  
$g:\Rd_m\wp(V)\to \C$, where $\C\in \sf Gs_m$ and $g$ 
preserves infinite intersections.  Then $g\circ (f\upharpoonright \C)$ 
is the desired complete representation of $\C=\Rd_m\Rd_{sc}\B$.

We have from the finite dimensional case that $\sf R=\Ra\C$, $\C\in \QEA_{\omega}$. Then, for any $2<m<n<\omega$, the identity map 
is  a complete embedding from $\C(m)=\Nr_m\C$ to $\Rd_{m}\Nr_n\C=\Rd_m\C(n)$.
Fix $2<m<\omega$. For each $\Gamma\in I,   m\subseteq \Gamma$, let $I_{|\Gamma|}$ be a complete embedding 
from ${\C}(m)$ into $\Rd_m {\C}(|\Gamma|),$ we know that such an embedding exists. Let $\iota( b)=(I_{|\Gamma|}b: \Gamma\in I)/F$ for  $b\in \C(m)$.
Then $\iota$ is a complete embedding 
from $\C(m)$ into $\Rd_m\B=\Rd_m(\Pi_{\Gamma/F}\C_{\Gamma}$).
But $\Rd_{sc}\C(m)$ is a complete subalgebra of $\Rd_m\Rd_{sc}\B$, 
hence $\Rd_{sc}\C(m)$  
is completely representable, too which is a contradiction.

\end{proof}

\section*{Part two}

\section{ Notions of representability and completions}

Before formulating the next theorem we need some preparation.

Let $2<n<\omega$. Let $\K$ be any class between $\Sc$ and $\PEA$. 
Then $\sf SRK_n$ denotes the class of {\it strongly representable} $\sf K_n$s.  An algebra
$\A\in \sf SRK_n$ if  {\it $\A$ is atomic, completely additive and $\Cm\At\A\in \sf RK_n$.} 
Complete additivity is superfluous for $\CA$s and $\PEA$s. $\sf CRK_n$ denotes the class of 
completely representable $\K_n$s.
$\sf LK_n$ denotes the class of algebras satisfying the Lyndon condition
adapted to the given signature.

\begin{definition}
Let $\Gamma=(G, E)$ be a graph and $C$ be a non-empty set of `colours'.
\begin{enumerate}

%\item A set $X\subseteq G$ is said to be \textit{independent} if $E\cap(X\times X)=\phi$.

\item A function $f:G\to C$ is called a {\it $C$--colouring} of $\Gamma$ if $(v,w)\in E\implies f(v)\neq f(w).$

\item The \textit{chromatic number} $\chi(\Gamma)$ of $\Gamma$ is the smallest $\kappa<\omega$
such that $G$ can be partitioned into $\kappa$ independent sets,
and $\infty$ if there is no such $\kappa$. 
Equivalently, $\chi(\Gamma)$ is the size of the smallest finite set $C$ such that there exists a $C$--colouring of $\Gamma$, if such a $C$ exists 
and infinity if there is no such $C$.
\end{enumerate}
\end{definition}
We will also need the following strengthening of \cite[Lemma 5.1.50]{HMT2}. In the statement of the lemma $\alpha$ is a finite ordinal. The proof follows easily from the proofs of 
\cite[Theorem 5.1.39, Lemmas 5.1.48-49, Theorems 5.1.50-51]{HMT2}.  The only change is that we need  the $E$ defined next, to be 
closed under infinite intersections.  Luckily this readily follows from its definition. 

\begin{lemma}\label{j} Assume the hypothesis of \cite[lemma 5.1.49]{HMT2}. In particular, $\A\in \CA_{\alpha}$ and  $\Rd_{df}\A$ is representable. Let 
$E=\{X\in A:  \forall x, y\in {}^{\alpha}U\& \forall i<\alpha, x_iRy_i \implies (x\in X\iff y\in X)\}.$  ($R$ is defined in the lemma cited above). 
Then $\{x\in \A: \Delta x\neq \alpha\}\subseteq E$,  $E$  closed under the operations of $\A$, $E$ is further
closed under infinite intersections, and if $\E$ is the subalgebra of $\A$ with domain $E$, then $\E\in \sf RCA_{\alpha}$. 
In particular, if $\A\in \CA_{\alpha}$, $\Rd_{df}\A$ is representable and $\A$ is generated from $\{x\in X: \Delta x\neq n\}$ allowing infinite unions, 
then $\A$ is representable. 
\end{lemma} 
\begin{proof} Let $X_j: j\in J$ be in $E$. We will show that $\bigcap_{j\in J}X_j\in E$. 
Let $x, y\in {}^{\alpha}U$ such that $x_iRy_i$ for all $i<\alpha$, and assume that $x\in \bigcap_{j\in J} X_j$. Then $x\in X_j$ for every $j\in J$. 
Now fix  $i\in J$. Then $x\in X_i$, and $X_i\in E$, so by definition of $E$ we get that  $y\in X_i$. Since $i$ was arbitrary, 
we get that $y\in \bigcap_{j\in J}X_j.$ By symmetry we are done.
\end{proof}
For a class $\bold K$ of algebras,  we let $\bold K\cap {\sf Fin}=\{\A\in \bold K: |\A|<\omega\}$
and $\bold K\cap {\sf Count}=\{\A\in \bold K: |\A|\leq \omega\}$.

\begin{definition}\label{grip} Let $2<n<\omega$.
\begin{enumarab}

\item A class $\bold K\subseteq \K_n$ is {\it gripped by its atom structures} or briefly gripped, 
if whenever $\At\A\in \At\bold K\implies \A\in \bold K$. $\bold K$ is {\it strongly gripped} if $\At\A\in {\bf El}\At\bold K\implies \A\in \bold K$.
\item An $\omega$--rounded game $\bold G$ is {\it gripping} for a class $\bold K\subseteq \K_n,$
if whenever $\alpha\in \At\bold K$ is countable, and \pe\ has a \ws\ in $\bold G(\alpha),$ then
for any $\A\in \K_n$, such that
$\At\A=\alpha\implies \A\in \bold K$. We say that $\bold G$ {\it grips} $\bold K$.
\item A class $\bold K\subseteq \sf RK_n$ is {\it finitely dense via atom structures}, or simply finitely dense in $\sf RK_n$, 
if $\At(\bold K\cap \sf Fin)=\At({\sf RK}_n\cap \sf Fin)$.
\end{enumarab}
\end{definition}

\begin{example}
\begin{enumarab}

\item By the first item of theorem \ref{SL} for $1<n<m$, the class $\Nr_n\K_m$ is not gripped, while ${\sf CRK_n}$ for $n<\omega$ is gripped, 
but is not strongly gripped by the first item of theorem \ref{main}.
\item For $2<n<\omega$, $G_{\omega}$ grips
$\sf CRA_n$, $\sf LCA_n$, $\sf SRCA_n$ and $\sf RCA_n$ but it does not grip $\Nr_n\CA_{\omega}$ by first item of theorem \ref{SL}.
To see why, using the cylindric reduct of the algebras $\A$ and $\B$ in this item, continuing to denote them by $\A$ and $\B$,
we have $\B$ is completely representable, and has countably many atoms, therefore by \cite[Theorem 3.3.3]{HHbook2}, 
\pe\ has a \ws\ in $G_{\omega}$, furthermore  $\At\B=\At\A\in \Nr_n\CA_{\omega}$, 
but $\B\notin {\bf El}\Nr_n\CA_{\omega}\supseteq \Nr_n\CA_{\omega}$. 
Note that by the second item of theorem \ref{SL} the last inclusion is proper.

\item All of the above four classes are finitely dense in $\sf RCA_n$ because for a finite algebra $\A$, 
$\A\in {\sf CRCA}_n\iff \A\in \sf RCA_n$. We will show below that $\Nr_3\K_{\omega}$ is not finitely dense in 
$\sf RK_3$.\\
\end{enumarab}
\end{example}
Concerning the result proved in item (ii), it is known that for $2<n<\omega$, 
$\bold S\Nr_n\K_{n+1}$ is atom--canonical,  because it has a finite Sahlqvist axiomatization. 
The case $\bold S\Nr_n\K_{n+2}$ remains unknown (at least to us).
Now assume that $k$ is finite $\geq 1$. 
Let $${\sf V}_k=\bold S\Nr_n\K_{n+k}.$$
Then $\sf Str V_k=\At \sf V_k\iff \sf V_k$ is atom--canonical $\iff
\sf V_k$ is closed under \de\ completions $\iff \bold S\Cm\At\sf V_k=V_k\iff {\bf HSP}\At\sf V_k=\sf V_k$ \cite[Theorem 2.88]{HHbook}.

\begin{theorem} \label{squareflat}
Let $2<n<\omega$ and $\K$ be any class between $\Sc$ and $\PEA$. Then the following hold:
\begin{enumarab}

\item  For $k\geq 3$, and $\K\in \{\CA, \PEA\}$,
$\sf V_k$ is not any of the above. 

\item  For $3\leq k<\omega$, the classes of $\K_n$s, having complete: $n+k$--flat, infinitary--$n+k$ flat, $n+k$--smooth, representations
are not elementary, for any $k\geq 3$. If $\K$ has a $\CA$ reduct, then the class of $\K_n$s 
having complete $n+k$--flat representations is not elementary.

\item For $\K$ between $\Sc$ and $\PEA$, the classes of $\K_n$s having $n+k$--relativized representations as in the previous item,  are varieties, 
that are not finitely axiomatizable for $k\geq 2$ and, together with the $\K_n$s having $n+k$--square representations
are all canonical varieties, but they are not atom--canonical for $k\geq 3$. In particular, 
$\bold L_{n,n+k}=\PEA_n\cap \bold S\Nr_n{\sf G}_{n+k}=\{\A\in \PEA_{n+k}: \A \text{ has an $n+k$--square representation}\}$ 
is a canonical variety, that is not atom--canonical for $k\geq 3$. Furthermore, $\sf RPEA_n=\bigcap_{k\in \omega}\bold L_{n, n+k}.$ 
\item For any infinite cardinal $\kappa$, for any $\K$ between $\Sc$ and $\PEA$, the class of algebras having complete 
$\kappa$--square representations is not elementary.

\item There is a finite $k>n$ such that ${\sf Str}({\sf V}_l)$ is not elementary for all $l\geq k$.

\item The classes $\sf CRK_n$, $\sf LK_n$ and ${\sf SRK}_n$ are gripped  
and their elementary closure is finitely axiomatizable $\iff n\leq 2$, and with the exception of 
$\sf LK_n$, are  not strongly gripped for $n>2$. 
The class $[\sf RK_{n}\cap \At]^{\sf ad}$ is  gripped $\iff n\leq 1$ or $n=2$ 
and $\K\in \{\CA, \PEA\}.$   If $2<n<m<\omega$, then any class $\K$ between $\CA$ and $\sf PEA$, the class 
$\bold S_c\Nr_n[\K_m\cap \At]$ is gripped,  but not strongly gripped if $m\geq n+3$.

\item Let  $\bold N=\{\A\in \K_n^{\sf ad}\cap \At: \Cm\At\in \Nr_n\K_{\omega}\}$.  Then $\bold N\neq \Nr_n\K_{\omega}\cap \At$, 
$\bold N\subseteq \bold S_c\Nr_n\K_{\omega}^{\sf ad}\cap \At$, $\bold N\cap {\sf Count}\subseteq {\sf CRK_n}$; the last two inclusions are proper, 
and $\bold N$ is 
not elementary.  Let $\bold K_1={\sf CRK}_n\cap {\bf El}\Nr_n\K_{\omega}$ 
and $\bold K_2=\bold S_c\Nr_n\K_{\omega}\cap {\bf El}\Nr_n\K_{\omega}\cap \At$. Then these are distinct classes  that lie 
{\it strictly} between $\Nr_n\K_{\omega}$ and $\bold S_c\Nr_n\K_{\omega}$,  but $\bold K_1$ and $\bold K_2$ 
coincide on algebras having countably many atoms.

\item For $1<n<m$, the class $\Nr_n\K_m$ is not gripped. There is a non--atomic game $\bold G$ that grips $\Nr_n\CA_{\omega}$

\item The class $\Nr_3\K_{\omega}$ is not finitely dense via atom structures in $\sf RK_3$.
\end{enumarab}
\end{theorem}
\begin{proof} 
Fix $2<n<\omega$. \\

(1) $\sf V_k$ is canonical because if $\A\subseteq \Nr_n\D$, with $\D\in \K_m$, 
then $\A^+\subseteq_c \Nr_n\D^+$. It is not atom--canonical by item (2) of theorem \ref{main}. The first required then follows. \\

(2) We first show that an existence of a complete infinitary $m$--flat representation of $\A\in \Sc_n$ 
implies the existence of an atomic  {\it completely additive} $m$--dilation. We have already dealt with atomicity in lemma \ref{Thm:n}.
Let $\A\in \Sc_n$ and let $M$ be a complete infinitary $m$--flat representation of $\A$. Then, as in the proof of the second item of lemma \ref{Thm:n}, 
one forms the dilation $\D$ with universe $\{\phi^{\M}: \phi\in \L(A)_{\infty, \omega}^m\}$ but counting in {\it equality}.

This algebra is an atomic $\CA_m$, hence it is completely additive. Now $\A$ completely embeds 
into $\Rd_{sc}\Nr_n\D=\Nr_n\Rd_{sc}\D$, via the same map denoted by $\theta$ in the proof of the first 
item of lemma \ref{Thm:n}, defined via $r\mapsto r(\bar{x})^M$, since this latter map is still atomic in this new context (using the same reasoning in the earlier  context). 
Then $\Rd_{sc}\D$ is of course still atomic and completely additive. 

Let  $\C=\PEA_{\Z, \N}\in \PEA_n$ be as in item (1) of theorem \ref{main}. 
Then $\Rd_{sc}\C\notin \bold S_c\Nr_n\Sc_{n+3}^{\sf ad}$, 
hence by item (2) in lemma \ref{flat} and the above, with the fact that, for all finite $k\geq 3$, $\bold S_c\Nr_n\K_{n+k}^{\sf ad}\subseteq \bold S_c\Nr_n\K_{n+3}^{\sf ad},$ 
$\Rd_{sc}\C$, $\Rd_{\sf K}\C$, does not have a complete $n+k$--flat representation, {\it a priori}, a complete  infinitary $n+k$--flat 
(nor an $n+k$--smooth) representation.  But $\C$ is elementary equivalent to a countable $\B$, which is completely 
representable, hence $\Rd_{\K}\B$, has all types 
of $m>n$ relativized representations. 

If $\K$ has a $\CA$ reduct then $\K_m^{\sf ad}=\K_m$ for all $m$, and so $\Rd_{\sf K}\C\notin \bold S_c\Nr_n\K_{n+3}$, hence by 
the above argument, and now using the third item of 
lemma \ref{flat}, we get that $\Rd_{\K}\A$ does not have a complete $m$--flat representation, but $\Rd_{\K}\A\equiv \Rd_{\K}\B$, 
and $\Rd_{\K}\B$ has a complete  $m$--flat representation, because it is completely representable.\\

(3) For the third required, all such classes, except for the one involving squareness, are equal by first item of lemma \ref{flat} to $\bold S\Nr_n\K_{n+k}$. 
For finite  $k\geq 1$,  the variety $\bold S\Nr_n\K_{n+k+1}$ is not finitely axiomatizable
over $\bold S\Nr_n\K_{n+k}$ \cite[Corollary 3.3]{t}. 

We now show that the class $\sf V$ of $\sf K_n$ having $m$--square representations is a canonical variety.
When $\K=\PEA$, then $\V=\bold L_{n, m}$.
We start by showing that $\sf V$ is a variety. 

Closure under subalgebras is obvious. 
Products is also easy;  given $\A_i$s, $i\in I$  ($I$ a non--empty indexing set),  
with $m$--square representations, we have 
$\A_i^+$ has a complete $m$--square representation, 
by last item in lemma \ref{Thm:n}, 
so $\A_i^+$ has an $m$--dimensional basis $\M_i$, by the proof of item (ii) in lemma \ref{Thm:n}.
We can assume that the $\A_i^{+}$s are  pairwise disjoint; accordingly
let $\C$ be the complex algebra over the atom structure of disjoint unions of those of the $\A_i^{+}$.
Then $\C$ is complete and atomic; $\bold P_{i\in I}\A_i$ embeds in $\C$, and
$\bigcup \M_i$ is the desired basis. 

Now we check closure under homomorphic images. The proof is the $\CA$ analogue of \cite[Proposition 12.31]{HHbook} 
formulated for relation algebras and attributed to Maddux.
Assume that $\A$ has an $m$--square representation 
and let $h:\A\to \B$  be surjective. We want to show that $\B$ has an $m$--square representation, too.
Again,  by last item of lemma \ref{Thm:n}, we have $\A^+$ has an $m$--dimensional basis $\M$.   
Let $K= ker(h)$  and let $z=-\sum^{\A^+}K$. (The last sum exists, because $\A^+$ is complete).
Let $\D$ be the relativization of $\A^+$ to $z.$
Then $\D$ is complete and atomic.

Define  $g :\A^+\to \D$ by $a\mapsto a\cdot z.$ Then $\B$ embeds in $\D$ via $b\mapsto g(a)$, 
for any $a\in h^{-1}[b]$.
Hence  $\{N\in \M: N(\bar{x})\in \D\}$ 
is a basis for $\D$, and since $\B$ (up to isomorphism) is a subalgebra of $\D$, 
we get that $\B$ has an $m$--square representation, because $\D$ does.

Canonicity is straightforward by the proof of lemma \ref{Thm:n}.    
If $\A$ has an $m$--square representation, then $\A^+$ has an $m$--dimensional basis, 
hence it has an $m$--square representation, too;  
and in fact a complete one.  Finally, non atom--canonicity follows directly from item (iii) of theorem \ref{main}.

If $\A\in \sf RPEA_n$, then $\A\in \bold S\Nr_n\sf PEA_{n+k}\subseteq \bold L_{n,n+k}$ for any $k\in \omega$, 
hence $\sf RPEA_n\subseteq \bigcap_{k\in \omega}\bold L_{n,n+k}$. 
Conversely,  assume that  $\A\in \bigcap_{k\in \omega}\bold L_{n,n+k}$ is countable. Then $\A$ has an $m$--square representation for every finite
$m\geq n$. Using a compactness argument, one shows that $\A$ has an $\omega$--square representation. 
One can build, in a step--by--step manner, a chain of larger and larger finite cliques (as defined in the proof of theorem \ref{complete2}) resolving more and more
`cylindrification'  defects, so that their union will have no defects. Being a clique, the union will
be a classical representation (here countability of $\A$ is essential).  Since a classical  representation is obviously $\omega$--square,
we get that $\bigcap_{k\in \omega}\bold L_{n,n+k}$ and $\sf RPEA_m$ coincide on simple countable algebras. 
But each is a variety (since $\bold L_{n,n+k}$ is a variety for each $k\in \omega$), with discriminator term ${\sf c}_{(n)}x$, 
so that they are both discriminator varieties,  agreeing on countable (simple) algebras, 
so they are equal.\\

(4)  Now let $\K$ be any class between $\Sc$ and $\PEA$, and let $\kappa$ be any infinite cardinal. 
Let $\C=\PEA_{\Z, \N}$. 
Then $\Rd_{\K}\C$ does not have a complete $\kappa$--square representation, 
for such a representation necessarily induces an (ordinary) complete representation of $\Rd_{\K}\C$, 
because $\C$ has countably many atoms. But
$\Rd_{\K}\B$ has a complete $\kappa$--square representation, 
because it is completely representable. \\

(5) We start by reproving a result of Bulian and Hodkinson \cite{bh}. We show  
that the class $\sf SRK_n$ $(2<n<\omega)$ is not elementary.
Define the polyadic operations corresponding to transpositions
using the notation in \cite[Definition 3.6.6]{HHbook2},
in the context of defining atom structures from
classes of models by
$R_{{\sf s}_{[ij]}}=\{([f], [g]): f\circ [i,j]=g\}$. This is well defined.
In particular, we can (and will) assume that $\M(\Gamma)$ defined in
\cite[Top of p.78]{HHbook2} 
has the signature of polyadic equality algebra. 

For a graph $G$, recall that $\chi(G)$ denotes its chromatic number. 
We claim that if $\chi(\Gamma)=\infty$, then $\M(\Gamma)$ is representable as a  polyadic
equality algebra and if $\chi(\Gamma)<\infty$,
then $\Rd_{df}\M(\Gamma)$ is not representable by lemma \ref{j}.

Here  networks are changed by adding a consistency condition for the substitution operators corresponding to transpositions.
But the atomic (usual) game is the same. 
In more detail, in the proof of lemma \cite[Lemma 3.64]{HHbook2}, replacing $S$ by $\At\M(\Gamma)$,
then in the possibly transfinite game  $G^{|\At\M(\Gamma)|+\omega}\M(\Gamma)$, we do not need to worry about the starting point. 
But the reponse of 
\pe\ to the current  $N\downarrow$ (using the notion in the proof of the above cited lemma) 
played by \pa\ during the play, is easily seen to be symmetric, in 
the sense that  for all $i<j<n$, it satisfies
${\sf s}_{[i,j]}M(\bar{x})=M(\bar{x}\circ [i,j]),$ 
given (inductively)  that $N\downarrow$ is symmetric.  So symmetry is preserved at successor ordinals.
Limits obviously preserves symmetry, 
too, so $\M(\Gamma)\in \sf RPEA_n$, and in fact, 
the representation defined, induces a complete 
representation of $\M(\Gamma)^+$.

The $\sf Df$ reduct of $\M(\Gamma)$ is not representable, because $\M(\Gamma)$ 
is generated   completely by $\{x\in \M(\Gamma): \Delta x\neq n\}$.This prompts (borrowing the terminology from Hirsch and Hodkinson):
\begin{itemize}
\item A graph $\Gamma$ is {\it good} if $\chi(\Gamma)=\infty$.
A Monk's algebra $\M(\Gamma)$ is {\it  good }if $\Gamma$ is good, so that it is representable as a $\sf PEA_n$.
\item A graph $\Gamma$ is {\it bad} if $\chi(\Gamma)<\infty$.
A Monk's algebra $\M(\Gamma)$ is {\it  bad} if
$\Gamma$ is bad,  so that its $\sf Df$
reduct is not representable.
\end{itemize}
Then Using Erdos' probabalistic graphs,  
one constructs a sequence of good Monk algebras converging to a bad one.
As before, for each finite
$\kappa$ there is a finite graph $G_{\kappa}$ with
$\chi(G_{\kappa})>\kappa$ and with no cycles of length $<\kappa$.
Let $\Gamma_{\kappa}$ be the disjoint union of the $G_{l}$ for
$l>\kappa$. Clearly, $\chi(\Gamma_{\kappa})=\infty$, and so
the {\it polyadic equality algebra} $\mathfrak{M}(\Gamma_{\kappa})$ is representable.
Now let $\Gamma$ be a non-principal ultraproduct
$\Pi_{D}\Gamma_{\kappa}$ for the $\Gamma_{\kappa}$s. For $\kappa<\omega$, let $\sigma_{\kappa}$ be a
first-order sentence of the signature of the graphs stating that
there are no cycles of length less than $\kappa$. Then
$\Gamma_{l}\models\sigma_{\kappa}$ for all $l\geq\kappa$. By
Lo\'{s}'s Theorem, $\Gamma\models\sigma_{\kappa}$ for all
$\kappa$. So $\Gamma$ has no cycles, and hence by $\chi(\Gamma)\leq 2$.
Thus $\mathfrak{Rd}_{df}\mathfrak{M}(\Gamma)$
is not representable.  
(The 
algebra $\Tm(\rho(\Gamma))$ 
is representable, because the class of weakly representable atom structures is elementary \cite[Theorem 2.84]{HHbook}.)\\

(6) Let $\A_i$ be the sequence of (strongly) representable $\QEA_n$ with $\Cm\At\A_i=\A_i$
such that the diagonal free reduct of $\A=\Pi_{i/U}\A_i$ is not strongly representable as constructed above.
Hence $\Rd_{\K}\Cm\At\A=\Cm\At\Rd_{\K}\A\notin \bold S\Nr_n\K_{\omega}=\bigcap_{i\in \omega}\bold S\Nr_n\K_{n+i}$, 
so $\Rd_{\K}\Cm\At\A\notin \bold S\Nr_n\K_{l}$ for all $l>k$, for some $k\in \omega$, $k>n$. 
But for each such $l$, $\Rd_{\K}\A_i\in \bold S\Nr_n\K_l\subseteq \sf RK_n$, 
so $\Rd_\K\A_i$ is a sequence of algebras such that $\Cm\At(\Rd_\K\A_i)=\Rd_{\K}\Cm\At\A_i=\Rd_\K\A_i\in \bold S\Nr_n\K_{l}$, but 
$\Cm(\At(\Rd_\K(\Pi_{i/U}\A_i))=\Cm\At(\Rd_\K\A)\notin \bold S\Nr_n\K_l$, for all $l\geq k$.\\

(7) It is easy to see that all of ${\sf CRK}_n$, $\sf LK_n$ and $\sf SRK_n$ $(n\in \omega)$ are  gripped. 
If $n\leq 2$, all classes coincide with atomic completely additive algebras in $\sf RK_n$, hence are all elementary
and finitely axiomatizable. For $n>2$, and any $\bold K$ such that $\Nr_n\K_{\omega}\subseteq \bold K\subseteq \sf RK_n$, 
${\bf El K}$ 
is not finitely axiomatizable as proved in the last item of theorem \ref{SL}.

Let $2<n<\omega$. Then $\sf LK_n$ is strongly gripped, because it is elementary and gripped, $\sf CRK_n$ is not strongly gripped, 
for using the algebras in item (1) of theorem \ref{main},
$\At\PEA_{\Z,\N}\equiv \At\B$, where $\B$ is completely representable but $\PEA_{\Z, \N}$ is not. 

${\sf SRK}_n$ is not strongly gripped, because there is a strongly representable atom structure $\beta$,
that is elementary equivalent to an atom structure $\alpha$ 
that is not strongly representable.

By the second item  of theorem \ref{main}, $\sf RK_{n}$ for any $n>2$ 
is not gripped, for there is a weakly representable $\K_{n}$ atom structure,  that is not strongly representable.
Furthermore, there  are algebras based on this atom structure that  are completely additive, so we obtain 
the analogous result for 
$[\sf RK_{\alpha}\cap \sf At]^{\sf ad}.$ In particular, both classes are not strongly gripped.
For $n=2$, and $\K\in \{\PA, \Sc\}$, the class $[\sf RK_2\cap \At]^{\sf ad}$ is not gripped, because there exists an atomic algebra $\A\in \sf RK_2$ that is not completely additive, witness
example \ref{counter}.
Hence $\A$ and $\Cm\At\A$ share the same atom structure, $\A$ is not completely additive, while $\Cm\At\A$ is completely additive. 
It is  easy to see, that for $n\leq 1$, $\K_{\alpha}\cap \At=\sf RK_{n}\cap \At=[\sf RK_{n}\cap \At]^{\sf ad}$ is gripped.
On the other hand, for $n=2$ and $\K\in \{\CA, \PEA\}$,  $[\sf RK_{2}\cap \At]^{\sf ad}=
\sf RK_2\cap \At$ is gripped,   because any  weakly representable atom structure of dimension $2$ is strongly representable
since $\sf RK_2$ is a Sahlqvist conjugated variety; in particular, 
it is completely additive.

Now fix $2<n<m<\omega$ and $\K$ between $\CA$ and $\PEA$. 
Assume that $\At\A\in \At\bold S_c\Nr_n[\K_m\cap \At]$. We want to show that $\A\in \bold S_c\Nr_n[\K_m\cap \At]$. 
We can assume without loss that $\A$ is simple. 
If not, then one works with the simple components, like in the proof of  theorem \ref{complete}.
We have 
$\Tm\At\A\in \bold S_c\Nr_n[\K_m\cap \At]$ and is also simple; furthermore $\Tm\At\A\subseteq_c \A\subseteq_c \Cm\At\A$ and $\At\A=\At\Tm\At\A=\At\Cm\At\A$. 
By lemma \ref{flat}, $\Tm\At\A$ has a complete infinitary $m$--flat representation $M$. So let $f:\Tm\At\A\to \wp(1^M)$ be an 
injective homomorphism,   where $M$ is a complete infinitary $m$--flat representation of $\Tm\At\A$.
For  $a\in \Cm\At\A$, $a=\sum\{x\in \At\A: x\leq a\}$, say,  
define $\bar{f}(a)=\bigcup \{f(x): x\in \At\A\}$. Then  $\bar{f}\upharpoonright \A\to \wp(1^M)$
gives a complete infinitary $m$--flat 
representation of $\A$, so  again by lemma \ref{flat}, 
we get that  $\A\in \bold S_c\Nr_n[\K_m\cap \At]$ and we are done. 

Assume further that 
$m\geq n+3$. Then $\At{\sf PEA}_{\Z,\N}\in {\bf El}\At\bold S_c\Nr_n[\PEA_m\cap \At]$,  
because $\At\PEA_{\Z, \N}\equiv \At\B\in \At\bold S_c\Nr_n[\QEA_{\omega}\cap \At]$,
by theorem \ref{complete}, 
but $\Rd_{ca}\PEA_{\Z, \N}\notin \bold S_c\Nr_n\CA_m$, {\it a fortiori}, it is not in $\bold S_c\Nr_n[\CA_m\cap \At]$.
We have shown that $\bold S_c\Nr_n[\K_m\cap \At]$ is not strongly gripped. \\

For the second part. Let $\B$ be the algebra in the second item of theorem \ref{SL}. Then $\B\in  \bold K_1\cap \bold K_2\sim \Nr_n\K_{\omega}$.
On the other hand, the $\B$ constructed in the first item of the same theorem, is outside ${\bf El}\Nr_n\K_{\omega}$, but is completely
representable, hence $\B\in \bold S_c\Nr_n\K_{\omega}\sim \bold K_1\cup \bold K_2$.
Furthermore, $\bold K_1\subseteq \bold K_2$ by theorem \ref{complete}, but 
$\bold K_1\neq \bold K_2$ by the $\K_n$ constructed  in the first part of the last item of theorem \ref{main} 
which is in $\bold K_1\sim \bold K_2$, because 
it is in $\Nr_n\K_{\omega}$, but lacks 
a complete  representation. The last part follows from theorem \ref{complete}.\\

(8)  The class $\Nr_n\K_m$ ($1<n<m)$ is not gripped by the first item of theorem \ref{SL}. Now we work only with $\CA$s.
We  define a non--atomic $\omega$ rounded game $\bold G$ stronger than $H$ used in item (1) of theorem \ref{main} (permitting non-atomic labels)
that grips $\Nr_n\CA_{\omega}$, in the sense of the first item of definition \ref{grip}, that is, 
if  $\B\in \CA_n$, $|\At\B|\leq \omega$, $\At\B\in \At\Nr_n\CA_{\omega}$
and \pe\ has a \ws\ in $\bold G(\B)$, then $\B\in \Nr_{n}\CA_{\omega}$.

The game is played on both $\lambda$--neat hypernetworks defined in the proof of item (1) of theorem \ref{main} and complete labelled graphs (possibly by non--atoms) 
with no consistency conditions.
The play at a certain point, like in $H$ as in item (1) of theorem \ref{main},  will be a  
$\lambda$--neat hypernetwork, call its
network part $X$, and we write
$X(\bar{x})$ for the atom that labels the edge $\bar{x}$.

{\it An $n$-- matrix} is a finite complete graph with nodes including $0, \ldots, n-1$
with all edges labelled by {\it arbitrary elements} of $\B$. No consistency properties are assumed.
\pa\ can play an arbitrary $n$--matrix  $N$, \pe\ must replace $N(0, \ldots, n-1),$  by
some element $a\in \B$. 

The final move is that \pa\ can pick a previously played $n$--matrix $N$, and pick any  tuple $\bar{x}=(x_0,\ldots, x_{n-1})$
whose atomic label is below $N(0, \ldots, n-1)$.
\pe\ must respond by extending  $X$ to $X'$ such that there is an embedding $\theta$ of $N$ into $X'$
 such that $\theta(0)=x_0\ldots , \theta(n-1)=x_{n-1}$ and for all $i_0, \ldots i_{n-1} \in N,$ we have
$$X(\theta(i_0)\ldots, \theta(i_{n-1}))\leq N(i_0,\ldots, i_{n-1}).$$
This ensures that in the limit, the constraints in
$N$ really define the element $a$.

If \pe\ has a \ws\ in $\bold G(\At\B)$, for $\B\in \CA_n$,
then the extra move on matrices, ensures that  that every $n$--dimensional element generated by
$\B$ in the dilation $\D\in \RCA_\omega$
that is constructed in the play, is an element of $\B$ so that $\B$ 
{\it exhausts all $n$--dimensional elements}
of $\D$, thus $\B\cong \Nr_n\D$ $(\D$ as in the first item of theorem \ref{main})
and so $\B\in \Nr_n\CA_{\omega}$. \\

(9) Now we  construct a finite $\A\in \sf RK_3\sim \Nr_3\K_{\omega}$. Take the finite polyadic equality
algebra $\A$ consisting of $3$ by $3$ matrices (as defined by Maddux, see e.g.  \cite[pp.221]{HMT2})
over any integral (hence simple)  non--permutational relation algebra $\R$.
Such relation algebras exist  \cite[Theorem 36]{r}. The algebra $\A$  is of course completely representable.
Let $\alpha=\At\Rd_{sc}\A$. Assume for contradiction that $\B\in \Nr_3\K_{\omega}$ and $\At\B=\alpha$.

Then we have $(*) \ \  \At\Rd_{sc}\A\in \At\Nr_3\Sc_{\omega}.$
We claim that $\Rd_{sc}\A$ has a 
$3$--homogeneous complete representation, which is impossible, because $\sf R$
does not have a homogeneous representation.

It can be shown, using (*) together with arguments similar to \cite[Theorems 33, 34, 35]{r},
that \pe\ has a \ws\ in an $\omega$--rounded game  played on $\At\A$
where  \pa\ is offered three moves, as in the proof of item (1) in theorem \ref{main}, but played on networks rather than 
$\lambda$--neat hypernetworks. 
We do not have hyperedges labelled by non--atoms. This game also has three moves, namely,
a cylindrifier move, a transformation move,  and an amalgamation move
but now in amalgamation moves, the networks \pa\  chooses can overlap only on
at most $3$ nodes.

We use the notation in the proof of lemma \ref{Thm:n} given in the first part of the paper. 
Like in lemma \ref{Thm:n} and last item in theorem \ref{complete}, the \ws\ for  \pe\ 
is to play networks $N$ with $\nodes(N)\subseteq \omega$ such that
$\widehat N\neq 0$ (here $\widehat N$ is as defined in the proof of lemma \ref{Thm:n}).

In the initial round, let \pa\ play $a\in \At$.
\pe\ plays a network $N$ with $N(0, 1, 2)=a$. Then $\widehat N=a\neq 0$.
The response to the cylindrifier move is like in the first part of lemma \ref{Thm:n}.

For transformation moves: if \pa\ plays
$(M, \theta),$ then it is easy to see that we have $\widehat{M\theta}\neq 0$,
so this response is maintained in the next round.

We write $\bar{i}$ for $\{i_0, i_1, i_2\}$, if it occurs as a set, and we write ${\sf s}_{\bar{i}}$ short for
${\sf s}_{i_0}{\sf s}_{i_1}{\sf s}_{i_{2}}$.
Finally, if  \pa\ plays the amalgamation move $(M,N)$ where $\nodes(M)\cap \nodes(N)=\{\bar{i}\}$,
then $M(\bar{i})=N(\bar{i})$.  Let $\mu=\nodes(M)\sim \bar{i}$ and $v=\nodes(N)\sim \bar{i}$
Then ${\sf c}_{(v)}\widehat{M}=\widehat{M}$ and
${\sf c}_{(u)}\widehat{N}=\widehat{M}$.
Hence
${\sf c}_{(u)}\widehat{M}={\sf s}_{\bar{i}}M(\bar{i})={\sf s}_{\bar{i}}N(\bar{i})={\sf c}_{(v)}\widehat{N}$
so ${\sf c}_{(v)}\widehat{M}=\widehat{M}\leq {\sf c}_{(u)}\widehat{M}={\sf c}_{(v)}\widehat N$
and $\widehat{M}\cdot \widehat{N}\neq 0.$  So there is $L$ with $\nodes(L)=\nodes(M)\cup \nodes(N)\neq 0$,
and $\widehat{L}\cdot x\neq 0$, where $\widehat{M}\cdot \widehat{N}=x$,
thus $\widehat{L}\cdot \widehat{M}\neq 0$ 
and consequently $\widehat{L}\restr {\nodes(M)}=\widehat{M}\restr {\nodes(M)}$,
hence $M\subseteq L$ and similarly $N\subseteq L$, so that $L$ is the required amalgam.

We can assume that $\A$ is simple, because $\sf R$ is.
Let $a\in \A$ be non-zero.
\pe\ uses her \ws\  to define a sequence
of networks $N_0\subseteq \ldots N_r\subseteq \omega$, such  that $N_0$ is \pe's response  to \pa's move choosing $a$ in the initial round.

This sequence respects the cylindrifier
move, in the sense that if $N_r(\bar{x})\leq {\sf c}_ib$ for $\bar{x}\in \nodes(N_r)$, then there
exists $N_s\supseteq N_r$ and a node $k\in \omega\sim N_r$ such that $N_s(\bar{y})=b$; where
$\bar{y}\equiv_i \bar{x}$ and $\bar{y}_i=k,$
and it also respects partial isomorphisms, in the sense that if
if $\bar{x}, \bar{y}\in \nodes(N_r)$ such that $N_r(\bar{x})=N_r(\bar{y})$,
then there is a finite surjective map $\theta$ extending $\{(x_i, y_i): i<n\}$ mapping onto $\nodes(N_r)$
such that $\dom(\theta)\cap \nodes(N_r)=\bar{y}$,
and finally there is an extension $N_s\supseteq N_r$, $N_r\theta$ (some
$s\geq r$).

Let $N_a$ be the limit of such networks.
Define the required representation $\cal N$ of $\A$ having
domain  $\bigcup_{a\in A}\nodes(N_a)$, by
$$S^{\cal N}=\{\bar{x}: \exists a\in A, \exists s\in S, N_a(\bar{x})=s\},$$
for any subset $S$ of $\At\A$.
Then this is a  complete $3$--homogeneous  representation of $\A$. We have shown that 
$\alpha\notin \At(\Nr_3\K_{\omega}\cap \sf Fin)$, and we are done.
\footnote{It seems likely that the result lifts to higher dimension by using
the lifting argument  in \cite{AU}.}
\end{proof}

\begin{theorem} \label{fail} Assume that $2<n<\omega$. Then \pa\ can win $\bold G_k(\At\Rd_{ca}\PEA_{\Z, \N})$ for $k\geq 3$.
If ${\bf El}\Nr_n\CA_{\omega}\subseteq \bold S_c\Nr_n\CA_{n+k}$ for some $k\geq 3$, then the following $\CA$ analogue of \cite[Theorem 45]{r} is 
false. Any class $\bold K$ such that $\Nr_n\CA_{\omega}\subset \bold K\subset \bold S_c\Nr_n\CA_{n+3}$ is not elementary. This  also implies that 
if  $\At$,  is a countable atom structure such that \pe\ has a \ws\ in $\bold G_k(\At)$ 
for all $k\in \omega$, then \pe\ has a \ws\ in $F^{\omega}(\At)=G_{\omega}(\At)$.\\
\end{theorem}
\begin{proof} We omit the proof of the first part. For the second part $\bold K={\bf El}\Nr_n\CA_{\omega}$ is a counterexample. The third part 
follows from the second, since such an atom structure $\At$ will witness
 that any class $\bold K$ such that $\Nr_n\CA_{\omega}\subseteq \bold K\subseteq \bold S_c\Nr_n\CA_{\omega}$ is not elementary.
To see why, let $\C=\Cm\At\in \CA_n$. A \ws\ for \pe\ in $\bold G_k$ for all $k\in \omega$ implies that \pe\ has 
a \ws\ in $\bold G(\At\B)$ for some countable atomic $\B$, such that $\B\equiv \C$. But since $\bold G$ grips $\Nr_n\CA_{\omega}$, we 
get that $\B\in \Nr_n\CA_{\omega}$. Now assume for contradiction hat  \pa\ has a \ws\ in $G_{\omega}(\At\C)$, then by \cite[Theorem 3.3.3]{HHbook2}, upon noting that $\At$ is countable,
$\C$ has no complete representation, hence by the first item of theorem \ref{complete}, $\C\notin \bold S_c\Nr_n\CA_{\omega}$. 
So if $\bold K$ is between $\Nr_n\CA_{\omega}$ and $\bold S_c\Nr_n\CA_{\omega}$, then $\B\in \bold K$, 
$\C\notin \bold K$ and $\C\equiv \B$, thus $\bold K$ is not elementary, which is a contradiction since ${\bf El}\Nr_n\CA_{\omega}(\subseteq \bold S_c\Nr_n\CA_{\omega}$ by assumption)
is  elementary.
\end{proof}

The range of $\bold K$ strictly between $\Nr_n\CA_{\omega}$ and $\bold S_c\Nr_n\CA_{\omega}$ is the `critical range'.  
We know, by the second item of theorem \ref{SL} and first item of theorem \ref{main},
that the endpoints are not elementary, and we also 
know by the first item of theorem \ref{main}, that any class $\bold L$ 
between $\bold S_c\Nr_n\CA_{\omega}$ and $\bold S_c\Nr_n\CA_{n+3}$ is not elementary. However, we do not know whether $\bold K_1$ and $\bold K_2$ 
defined in the item (7) of the previous  theorem \ref{squareflat}, proved in this last theorem to lie
in this critical range, are elementary or not. If our assumption in the previous 
theorem is correct, then $\bold K_2={\bf El}\Nr_n\CA_{\omega}$ would be elementary. This assumption does not tell anything about $\bold K_1$ as far as 
elementarity is concerned.
But, on the other hand, if there exists a countable atom structure $\At$, such that \pe\ has a \ws\ in $\bold G_k(\At)$ for all $k\in \omega$,
and \pa\ has a \ws\ in $G_{\omega}(\At)$, then $\Cm\At\notin \bold S_c\Nr_n\CA_{\omega}\supseteq \bold K_1\cup \bold K_2$, but
$\Cm\At\in {\bf El}(\Nr_n\CA_{\omega}\cap {\sf CRCA}_n)={\bf El K}_1\subseteq {\bf El K}_2$, implying 
that both $\bold K_1$ and  $\bold K_2$  are not elmentary.\\

\begin{theorem}\label{main2} Let $2<n<\omega$. Then the following inclusions hold:
$$\Nr_n\K_{\omega}^{\sf ad}\cap \At\subset {\bf El}\Nr_n\K_{\omega}^{\sf ad}\cap \At
\subset {\bf El}\bold S_c\Nr_n\K_{\omega}^{\sf ad}\cap \At={\sf LK}_n={\bf El}{\sf CRK}_n\subset {\sf SRK_n}$$
$$\subset
{\bf Up}{\sf SRK_n}={\bf Ur}{\sf  SRK_n}={\bf El}{\sf SRK}_n
\subset [{\bf S}\Nr_n\K_{\omega}\cap \At]^{\sf ad}.$$
\end{theorem}
\begin{proof}
It is known \cite[Proposition 2.90]{HHbook} that
${\bf Up}{\sf SRK_n}={\bf Ur}{\sf SRK_n}={\bf El}{\sf SRK}_n.$
The algebras in ${\sf SRK}_n$ are atomic and completely additive by definition.
Complete additivity of an operator on an atomic $\sf BAO$ is a first order property.
Thus ${\bf El}{\sf SRK_n}$ is contained in $[\sf RK_n\cap \At]^{\sf ad}.$

The first and second strict inclusions follow from items (2) and (1), respectively, of theorem \ref{SL}.  
The atomic $\B\in {\bf El}\Nr_n\K_{\omega}$ given in the second item of theorem \ref{SL} described fully in last item of theorem \ref{main}, 
is outside $\Nr_n\K_{\omega}$, and 
the atomic $\B$ in the first item of the same theorem, 
is in $\bold S_c\Nr_n\K_{\omega}^{\sf ad} \subseteq {\bf El}\bold S_c\Nr_n\K_{\omega}^{\sf ad}$,  
but is  not in ${\bf El}\Nr_n\K_{\omega}\supseteq {\bf El}\Nr_n\K_{\omega}^{\sf ad}$.
Now we turn to strictness of the remaining two inclusions. 
Since $\sf LK_n\subseteq \sf SRK_n$, the first is elementary by definition, the second is not,  hence the strictness
of the last inclusion in the first line, and the first one in the second line.

To prove the strictness of the last inclusion, namely, ${\bf El}{\sf SRK}_n\subset [\sf RK_n\cap \At]^{\sf ad}$, take  $\omega$--many disjoint copies of the 
$n$ element graph with nodes $\{1,2,\ldots, n\}$ and edges
$0\to 1$, $1\to 2$, and $\dots n-1\to n$. Then of course $\chi(\G)<\infty$.  Now  $\G$  has an $n-1$ first order definable colouring.
Since $\M(\G)$ with atom structure $\rho(\G)$, as defined in \cite[Definition 3.6.3, pp. 77-78]{HHbook2} and modified above to the other cylindric--like algebras 
approached here, 
is not representable, then the first
order subalgebra $\F(\G)$ in the sense of \cite[pp.456 item (3)]{HHbook} is also not representable, because $\G$ is 
first order interpretable in $\rho(\G)$. Here $\F(\G)$ is the subalgebra of $\M(\G)$ 
consisting of all sets of atoms in $\rho(\G)$ of the form $\{a\in \rho(\G)\}: \rho(\G)\models \phi(a, \bar{b})\}(\in \M(\G))$ 
for some first order formula $\phi(x, \bar{y})$ of the signature
of $\rho(\G)$ and some tuple $\bar{b}$ of atoms.  It is easy to check that $\F(\G)$ is indeed a subalgebra of $\M(\G)$, and that
$\Tm(\rho(\G))\subseteq \F(\G)\subseteq \M(\G)$. 

But $\F(\G)$ is strictly larger than the  term algebra.
Indeed, the term algebra as a $\sf PEA_n$ can be shown to be representable, 
by an argument similar to that used in \cite{weak}, witness too \cite[pp.485, Exercise 14.2(8)]{HHbook}.
We readily conclude that $\rho(\G)\notin \At{\bf El}{\sf SRK}_n$ but $\rho(\G)\in \At(\sf RK_n^{\sf ad})$,  
so ${\bf El}\sf SRK_n\neq [\sf RK_n\cap \At]^{\sf ad}$.

Now we consider the equalities in the first line. The second  is straightforward.
The first equality is also fairly straightforward, but we give a proof. If $\A$ is in $\bold L=\bold S_c\Nr_n\K_{\omega}^{\sf ad}\cap \At$,
then by lemma \ref{Thm:n}, \pe\ has a \ws\ in $F^{\omega}(\At\A)$, hence in $G_{\omega}(\At\A),$
and so obviously in $G_k(\At\A)$ for all $k\in \omega$, so $\A$ satisfies the Lyndon conditions. 
We have shown that $\bold L\subseteq \sf LK_n$. 
The latter is elementary, so ${\bf El}\bold L\subseteq \sf LK_n$.

Conversely, if $\A\in \sf LK_n$, then \pe\ has a \ws\ in $G_k(\At\A)$ for all $k\in \omega$, 
so using ultrapowers followed by an elementary chain argument, we get a countable $\B$, 
such that \pe\ has a \ws\ in $G_{\omega}(\At\B)$, hence  by \cite[Theorem 3.3.3]{HHbook2}, we get that $\B\in \sf CRK_n,$
and  $\B\equiv \A$.  By theorem \ref{complete}, there is an atomic (dilation) $\D\in \K_{\omega}^{\sf ad}$, such that  $\B\subseteq_c \Nr_n\D$, 
hence, we get using the notation in theorem \ref{complete}, that $\B\in \bold S_c\Nr_n\K_{\omega}^{\sf atc}$, so
$\A\in {\bf El S_c}\Nr_n\K_{\omega}^{\sf atc}\subseteq {\bf ElS_c}\Nr_n\K_{\omega}^{\sf ad}\cap \At$, 
and we are done.

\end{proof}

Our next theorem is conditional. We need to fix and recall some notation. 
Let $2<n<\omega$. For a relation algebra atom structure $\alpha$, $\beta=\sf Mat_n(\alpha)$ denotes the set of all basic $n$--dimensional
matrices on $\alpha.$  Sometimes $\beta$ is a $n$--dimensional cylindric basis \cite[Definition 12.35]{HHbook}.
This means that $\beta$ is  the atom--structure of a $\CA_n$ with accessibility relations corresponding
to the cylindric operations
defined as in \cite[Definition 12.17]{HHbook}, so that $\Cm(\beta)\in \CA_n$.
 
We refer the reader to  \cite[Definition 3.1]{ANT}
for the exact definition of an {\it $n$--blur} $(J, E)$ for a relation algebra $\R$. In particular, 
$J\subseteq \wp(\At\R)$ and $E\subseteq {}^3\omega$.
$(J5)_n$ is the stronger condition obtained from 
$(J4)_n$ in \cite[Definition 3.1]{ANT} by replacing $\exists T$ by
$\forall T$.  We say that write $(J, E)$  is a {\it strong $n$--blur} for $\R$
If the condition $(J4)_n$ in the definition of an $n$--blur is replaced by $(J5)_n$.  

\begin{theorem}\label{conditional} 
Let $2<m<n<\omega$.  Assume that  there exists a finite integral relation algebra $\R\notin \bold S\Ra\sf CA_{n+1}$
that has a strong $n$--blur  $(J, E)$. 
Let $\At$ be the infinite atom structure obtained by blowing up and
blurring $\R$, in the sense of \cite[p.72]{ANT}. Then $\sf Mat_{n}(\At)$ is an $n$--dimensional  cylindric  basis, and
there exist  cylindric algebras $\C_m$ and
$\C_{n}$ in $\sf RCA_m$ and $\sf RCA_{n}$, respectively,
such that  $\Tm\sf Mat_m(\At)\subseteq \C_m$ and $\Tm\sf Mat_{n}(\At)\subseteq \C_{n}$, $\C_m=\Nr_m\C_{n}$
$\Cm\C_m\notin \bold S\Nr_m\CA_{n+1}$.\footnote{In fact, we have $\Rd_{sc}\mathfrak{Cm}\C_m\notin \bold S\Nr_m\Sc_{n+1}.$} 
Furthermore, if $\R$ is generated by a single element,  then it can be fixed that
both $\C_m$ and $\C_n$ are also generated by a single $2$ dimensional
element and that  $\Tm\At=\Ra\C_m=\Ra\C_n$.
\end{theorem}
\begin{demo}{Sketch} We give a sketch of proof. We use the flexible construction in \cite{ANT}. The details skipped can be easily recovered from \cite{ANT}. 
The proof is like that in \cite{ANT} by blowing up and blurring $\R$ instead of blowing up and blurring
the `Maddux algebra $\M$' defined on p.84 \cite[Theorems 1.2, 3.2, lemmas 4.3, 5.1]{ANT}. The algebra $\M$ 
will be used below in the proof of theorem \ref{firstorder}.

One starts  with the finite relation algebra $\R$.
Then this algebra is blown up and blurred. It is blown up by splitting the atoms each to infinitely many.
It is blurred by using a finite set of blurs or colours, call this set of blurs  $J$. This is expressed consicely by the product $\At=\omega\times (\At \R\sim \{\Id\})\times J$,
which will define an infinite atom structure of a new
relation algebra $\cal R$; the term algebra on $\At$. One can view such a product as a ternary matrix with $\omega$ rows, and fixing $n\in \omega$,  one gets the rectangle
$\At \R\times J$.
Then two partitions are defined on $\At$, call them $P_1$ and $P_2$.
Composition is defined on this new infinite atom structure; it is induced by the composition in $\R$, and a tenary relation $E$
on $\omega$, that synchronizes which three rectangles sitting on the $i,j,k$ $E$--related rows compose like the original algebra $\R$.
This relation is definable in the first order structure $(\omega, <)$.
The first partition $P_1$ is used to show that $\R$ embeds in the complex algebra of this new atom structure, so
the complex algebra cannot be in $S\Ra\CA_{n+1}.$

The second partition $P_2$ divides $\At$ into $\omega$ sided finitely many rectangles, each with base $W\in J$,
and the universe of the  algebra $\cal R$ having atom structure  $\At$, are the sets that intersect co-finitely with every member of this partition.
On the level of the term algebra $\R$ is blurred, so that the embedding of the small algebra into
the complex algebra via taking infinite joins, do not exist in $\cal R$,  for only finite and co-finite joins exist
in this algebra.

$\cal R$ can be represented using the finite number of blurs. These correspond to non-principal ultrafilters
in the Boolean reduct, which are necessary to
represent this term algebra, for the principal ultrafilter alone would give a complete representation,
hence a representation of the complex algebra, and this is impossible.
Thereby, in particular,  an atom structure that
is weakly representable but not strongly representable is obtained.
In fact this representation gives a complete representation of the canonical extension of $\cal R$.

Because $(J, E)$ is an $n$--blur, the set of $n$ basic matrices on $\At$ is an  $n$-dimensional cylindric basis 
\cite[Item(iii), Theorem 3.2]{ANT} and so the $n$--dimensional
basic matrices form an atom structure of a $\D\in \CA_n$ that is also only weakly representable (not strongly representable).
Because $(J, E)$ is  strong $n$--blur,  the required algebra $\C_n$, as defined in \cite[Lemma 4.3]{ANT}, denoted by ${\sf Bb}_n(\R, J, E)$ 
is an algebra that lies between $\D$ and $\Cm\At\D$, $\C_n\neq \Cm\At\D$ and $\C_n$ has atom structure ${\sf Mat}_n(\cal R)$.

Now $(J, E)$ is a strong $m$--blur, too (since $m<n$)
it can be shown, as in \cite[Theorem 4.4]{ANT}, 
that the algebra $\C_m={\sf Bl}_m(\R, J,E)$, having atom structure $\sf Mat_m(\At\cal R)$ (same $\cal R$) 
is isomorphic to $\Nr_m\C_n$. First an embedding $h:\Rd_m\Cm{\sf Mat}_n(\At\R)\to \Cm{\sf Mat}_m(\At\R)$ via 
$x\mapsto \{M\upharpoonright m: M\in x\}$ is defined and it is shown that 
$h\upharpoonright  \Nr_m\C_n $ is an isomorphism onto $\C_m$ \cite[p.80]{ANT}. 

Though $\C_m\in \RCA_m\cap \Nr_m\CA_n$, 
the complex algebra $\Cm\At\C_m$ is not in $\bold S\Nr_m\CA_{n+1}$, 
because $\Cm\At({\cal R})\notin \bold S\Ra\CA_{n+1}$, since $\R$ embeds into $\Cm(\At{\cal R})$, $\R\notin \bold S\Ra\CA_{n+1}$ and  
$\Cm\At\cal R$ embeds into $\Ra\Cm\At\C_m$. 
To make the algebra one generated this
entails using infinitely many tenary relations \cite[pp. 84-86]{ANT}. 
\end{demo}

\begin{remark}
If there exists a finite $\R$ with $m$-blur and $\R\notin \bold S\Ra\CA_{m+2}$ , then $\bold S\Nr_m\CA_{m+2}$ will not be atom--canonical. 
For in this case, we will have $\Cm\At{\cal R}\notin \bold S\Ra\CA_{m+2}$, 
$\Tm(\Mat_m(\At{\cal R}))\in \RCA_m$, but its \de\ completion, namely,
$\C=\Cm(\Mat_m(\At{\cal R}))$ is outside $\bold S\Nr_m\CA_{m+2}$ because $\R$ embeds into $\Ra\C$.
\end{remark}

Our next theorem provides, like the previous theorem, but in a totally different way, 
weakly representable atom structures that are not strongly representable 
for both $\CA_n$s $(2<n<\omega)$ and $\sf RA$s in one go.

The result was proved in \cite{weak}, where only the second graph in the following proof  was used ($\omega$ disjoint union of $N$ cliques, $N\geq n(n-1)/n$, 
where $n$ is finite dimension $>2$). 
Here we give another new graph that gives the same result.
The reason we give a sketchy proof, besides introducing the new graph, is its model--theoretic affinity to the proof of item (2) of main theorem, 
and it will give the us opportunity to compare Monk-like constructions to rainbow 
ones. Though the model theory involved in both proofs are almost identical, the algebraic differences to be stressed upon in the next subsection are subtle,
real and highly significant.  Also the construction used by enlarging $N$ getting infinitely many graphs with increasing chromatic number, 
gives another flexible construction (like that used in theorem \ref{conditional}), based on bad graphs having finite chromatic number, 
converging to a good one having infinite chromatic number.

Using ultraproducts of weakly representable algebras based on such bad graphs, and the construction in theorem \ref{conditional} applied to the algebra $\M$
(which is the way it appeared in \cite{ANT}), by varying 
the parameter of non-identity atoms in $\M$ (like done with $N$), letting it grow without bound,
we will prove 
the upcoming corollary \ref{lyndon} on non-finite axiomatizability of various classes of representable algebras 
of finite dimension $>2$. 

\begin{theorem}\label{conditional2} Let $2<n<\omega$. Then there exists an atomic  countable
relation algebra $\R$, such that $\Mat_n(\At\R)$ forms an $n$--dimensional cylindric basis, $\A=\Tm{\sf Mat}_n(\At\R)\in \sf RCA_n$, while
even the diagonal free reduct of the \de\ completion of $\A$, namely, $\Cm{\sf Mat_n}(\At\R$) is not representable. In particular, 
${\sf Mat}_n(\At\R)$ is a weakly, but not strongly representable 
$\CA_n$ atom structure.
\end{theorem}
\begin{proof} 
Fix  a finite $n>2$ (the dimension). Let $\G$ be a graph.
Let $\rho$ be a `shade of red'; we assume that $\rho\notin \G$. Let $L^+$ be the signature consisting of the binary
relation symbols $(a, i)$, for each $a \in \G \cup \{ \rho\}$ and
$i < n$.  Let $T$ denote the following (Monk) theory in this signature:
$M\models T$ iff
for all $a,b\in M$, there is a unique $p\in (\G\cup \{\rho\})\times n$, such that
$(a,b)\in p$ and if  $M\models (a,i)(x,y)\land (b,j)(y,z)\land (c,k)(x,z)$, $x,y, z\in M$, then $| \{ i, j, k \}> 1 $, or
$ a, b, c \in \G$ and $\{ a, b, c\} $ has at least one edge
of $\G_l$, or exactly one of $a, b, c$ -- say, $a$ -- is $\rho$, and $bc$ is
an edge of $\G$, or two or more of $a, b, c$ are $\rho$.
We denote the class of models of $T$ which can be seen as coloured undirected
graphs (not necessarily complete) with labels coming from
$(\G\cup \{\rho\})\times n$  by $\GG$.
Now specify $\G$ to be either:
\begin{itemize}
\item the graph with nodes $\N$ and edge relation $E$
defined by $(i,j)\in E$ if $0<|i-j|<N$, where $N\geq n(n-1)/2$ is a positive number.

\item or the $\omega$ disjoint union of $N$ cliques, same $N$.
\end{itemize}
Then there is a countable ($n$--homogeneous) coloured graph  (model) $M\in \GG$
with the following property \cite[Proposition 2.6]{Hodkinson}:

If $\triangle \subseteq \triangle' \in \GG$, $|\triangle'|
\leq n$, and $\theta : \triangle \rightarrow M$ is an embedding,
then $\theta$ extends to an embedding $\theta' : \triangle'
\rightarrow M$.

To construct such an $M$, we use a simple game \cite{Hodkinson, weak}.
There are $\omega$ rounds. Let $ \Gamma \in \GG$ be the graph constructed up to
a point in the game after finitely many rounds. The starting point can be taken to be the empty graph.  $\forall$ chooses $\triangle \in \GG$ of
size $<n$, and an embedding $\theta : \triangle \rightarrow
\Gamma$. He then chooses an extension $ \triangle \subseteq
\triangle^+ \in \GG$, where $| \triangle^+ \backslash \triangle |
\leq 1$. $\exists$ must respond with an extension $
\Gamma \subseteq \Gamma^+ \in \GG$ such that $\theta$ extends to an
embedding $\theta^+ : \triangle^+ \rightarrow \Gamma^+$.

We may assume  that $\Delta$ is not empty, $\Delta\neq\Delta^+$, and that
 $\forall$ played $( \Gamma \upharpoonright F, Id_F, \triangle^+)$; $F=\rng(\theta)$
where $\Gamma \upharpoonright F \subseteq \triangle^+ \in \GG$,
$\triangle^+ \backslash F = \{\delta\}$, and $\delta \notin \Gamma$,
$\forall$'s move builds a labelled graph $ \Gamma^*
\supseteq \Gamma$, whose nodes are those of $\Gamma$ together with
$\delta$, and whose edges are the edges of $\Gamma$ together with
edges from $\delta$ to every node of $F$.

Now $ \exists$ must extend $ \Gamma^*$ to a complete
graph on the same nodes and complete
the colouring yielding a graph
$ \Gamma^+ \in \GG$.

She plays as follows:
The set of colours of the labels
in $\{\triangle^+(\delta, \phi) : \phi \in F \} $ has cardinality at most
$n-1$. Let  $i < n$ be a ``colour" not in this set. \pe\ labels $(\delta, \beta) $ by $(\rho, i)$ for every $ \beta \in
\Gamma \backslash F$. This completes the definition of $ \Gamma^+$.
It remains to check that this strategy works--that the conditions
from the definition of $\GG$ are met.

For this,  it is
sufficient to note that, if $\phi \in F$ and $ \beta \in \Gamma \backslash F$, then
the labels in $ \Gamma^+$ on the edges of the triangle $(\beta,
\delta, \phi)$ are not all of the same colour ( by choice of $i$ )
and if $ \beta, \gamma \in \Gamma \backslash F$, then two the
labels in $ \Gamma^+$ on the edges of the triangle $( \beta, \gamma,
\delta )$ are $( \rho, i)$.

For an $L^n_{\infty \omega}$-formula $\varphi $,  define
$\varphi^W$ to be the set $\{ \bar{a} \in W : M \models_W \varphi
(\bar{a}) \}$.

Let $\A$ be the relativized set algebra with domain
$\{\varphi^W : \varphi \in L_n \}$ and unit $W$, endowed with the algebraic
operations ${\sf d}_{ij}, {\sf c}_i, $ etc., in the standard way, and formulas are taken in
$L$.
Let $\cal S$ be the polyadic set algebra with domain  $\wp ({}^{n} M )$ and
unit $ {}^{n} M $. Then the map $h : \A
\longrightarrow S$ given by $h:\varphi ^W \longmapsto \{ \bar{a}\in
{}^{n} M: M \models \varphi (\bar{a})\}$ can be checked to be well -
defined and one-one. It clearly respects the polyadic operations, also because relativized semantics and classical semantics coincide on $L_n$
formulas in the given signature, this is a representation of $\A.$
This follows from the fact that, like in the proof of item (2) of theorem \ref{main}, any permutation $\chi$ of $\omega \cup \{\rho\}$, $\Theta^\chi$
is an $n$-back-and-forth system on $M$. Here $\Theta^\chi$
is the set of partial one-to-one maps from $M$ to $M$ of size at
most $n$ that are $\chi$-isomorphisms on their domains \cite{weak}.

We make the above paragraph more precise.
Let $\chi$ be a permutation of the set $\omega \cup \{ \rho\}$. Let
$ \Gamma, \triangle \in \GG$ have the same size, and let $ \theta :
\Gamma \rightarrow \triangle$ be a bijection. We say that $\theta$
is a $\chi$-\textit{isomorphism} from $\Gamma$ to $\triangle$, cf. \cite[Definitions, 3.8-- 3.9]{Hodkinson}, if for
each distinct $ x, y \in \Gamma$,
if $\Gamma (x, y) = (a, j)$ with $a\in \N$, if $l\in \N$ and $0\leq r<a$ are the unique natural numbers (obtained by deviding $a$ by $N$) such that
$a=Nl+r$, then 
\begin{equation*}
\triangle( \theta(x),\theta(y)) =
\begin{cases}
( N\chi(i)+r, j), & \hbox{if $\chi(i) \neq \rho$} \\
(\rho, j),  & \hbox{otherwise.} \end{cases}
\end{equation*}
If $\Gamma ( x, y) = ( \rho, j)$, then
\begin{equation*}
\triangle( \theta(x),\theta(y)) \in
\begin{cases}
\{( N\chi(\rho)+s, j): 0\leq s < N \}, & \hbox{if $\chi(\rho) \neq \rho$} \\
\{(\rho, j)\},  & \hbox{otherwise.} \end{cases}
\end{equation*}
We now have for any permutation $\chi$ of $\omega \cup \{\rho\}$, $\Theta^\chi$
is the set of partial one-to-one maps from $M$ to $M$ of size at
most $n$ that are $\chi$-isomorphisms on their domains. We write
$\Theta$ for $\Theta^{Id_{\omega \cup \{\rho\}}}$.

We claim that for any any permutation $\chi$ of $\omega \cup \{\rho\}$, $\Theta^\chi$
is an $n$-back-and-forth system on $M$. Clearly, $\Theta^\chi$ is closed under restrictions. We check the
``forth" property. To see why, let $\theta \in \Theta^\chi$ have size $t < n$. We use an argument similar to that used \cite[Theorem 3.10]{Hodkinson}.

Enumerate $\dom(\theta)$, $\rng(\theta),$ respectively as $ \{ a_0,
\ldots, a_{t-1} \}$, $ \{ b_0,\ldots, b_{t-1} \}$, with $\theta(a_i)
= b_i$ for $i < t$. Let $a_t \in M$ be arbitrary, let $b_t \notin M$
be a new element, and define a complete labelled graph $\triangle
\supseteq M \upharpoonright \{ b_0,\ldots, b_{t-1} \}$ with nodes
$\{ b_0,\ldots, b_{t} \}$ as follows.

Choose distinct "nodes"$ e_s < N$ for each $s < t$, such that no
$(e_s, j)$ labels any edge in $M \upharpoonright \{ b_0,\dots,
b_{t-1} \}$. This is possible because $N \geq n(n-1)/2$, which
bounds the number of edges in $\triangle$. We can now define the
colour of edges $(b_s, b_t)$ of $\triangle$ for $s = 0,\ldots, t-1$.
If $M ( a_s, a_t) = ( Ni+r, j)$, for some $i\in \N$ and $0\leq r<N$, then
\begin{equation*}
\triangle( b_s, b_t) =
\begin{cases}
( N\chi(i)+r, j), & \hbox{if $\chi(i) \neq \rho$} \\
\{(\rho, j)\},  & \hbox{otherwise.} \end{cases}
\end{equation*}
If $M ( a_s, a_t) = ( \rho, j)$, then assuming that $e_s=Ni+r,$ $i\in \N$ and $0\leq r<N$,
\begin{equation*}
\triangle( b_s, b_t) =
\begin{cases}
( N\chi(\rho)+r, j), & \hbox{if $\chi(\rho) \neq \rho$} \\
\{(\rho, j)\},  & \hbox{otherwise.} \end{cases}
\end{equation*}
This completes the definition of $\triangle$. It is easy to check
that $\triangle \in \GG$. Hence, there is a graph embedding $ \phi : \triangle \rightarrow M$
extending the map $ Id_{\{ b_0,\ldots, b_{t-1} \}}$. Note that
$\phi(b_t) \notin \rng(\theta)$. So the map $\theta^+ = \theta \cup
\{(a_t, \phi(b_t))\}$ is injective, and it is easily seen to be a
$\chi$-isomorphism in $\Theta^\chi$ and defined on $a_t$.
The converse,``back" property is similarly proved (or by symmetry,
using the fact that the inverse of maps in $\Theta$ are
$\chi^{-1}$-isomorphisms).

The logics $L_n$ and $L^n_{\infty \omega}$ are taken in the
above signature. Let $W$ is simply the set of tuples $\bar{a}$ in ${}^nM$ such that the
edges between the elements of $\bar{a}$ don't have a label involving
$\rho$ (these are $(\rho, i):  i<n$). Their labels are all of the form $(Ni+r, j)$.  We can now show that the
classical and $W$-relativized semantics agree in the sense that
$M \models_W \varphi(\bar{a})$ iff $M \models \varphi(\bar{a})$, for
all $\bar{a} \in W.$   The proof is by induction on $\varphi$.

We prove the hard direction \cite[Proposition 3.13]{Hodkinson}. Suppose that $M \models_W
\exists x_i \varphi(\bar{a})$. Then there is $ \bar{b} \in {}^n M$
with $\bar{b} =_i \bar{a}$ and $M \models \varphi(\bar{b})$. Take
$L_{\varphi, \bar{b}}$ to be any finite subsignature of $L$
containing all the symbols from $L$ that occur in $\varphi$ or as a
label in $M \upharpoonright \rng(\bar{b})$. (Here we use the fact
that $\varphi$ is first-order. The result may fail for infinitary
formulas with infinite signature.) Choose a permutation $\chi$ of
$\omega \cup \{\rho\}$ fixing any $i'$ such that some $(Ni'+r, j)$
occurs in $L_{\varphi, \bar{b}}$ for some $r<N$, and moving $\rho$.

Let $\theta = Id_{\{a_m : m \neq i\}}$. Take any distinct $l, m \in
n \setminus \{i\}$. If $M(a_l, a_m) = (Ni'+r, j)$, then $M( b_l,
b_m) = (Ni'+r, j)$ because $ \bar{a} = _i \bar{b}$, so $(Ni'+r, j)
\in L_{\varphi, \bar{b}}$ by definition of $L_{\varphi, \bar{b}}$.
So, $\chi(i') = i'$ by definition of $\chi$. Also, $M(a_l, a_m) \neq
( \rho, j)$(any $j$) because $\bar{a} \in W$. It now follows that
$\theta$ is a $\chi$-isomorphism on its domain, so that $ \theta \in
\Theta^\chi$. Extend $\theta $ to $\theta' \in \Theta^\chi$ defined on $b_i$,
using the ``forth" property of $ \Theta^\chi$. Let $
\bar{c} = \theta'(\bar{b})$. Now by choice of of $\chi$, no labels
on edges of the subgraph of $M$ with domain $\rng(\bar{c})$ involve
$\rho$. Hence, $\bar{c} \in W$.

Moreover, each map in $ \Theta^\chi$ is evidently a partial
isomorphism of the reduct of $M$ to the signature $L_{\varphi,
\bar{b}}$. Now $\varphi$ is an $L_{\varphi, \bar{b}}$-formula.
We have $M \models \varphi(\bar{a})$ iff $M \models \varphi(\bar{c})$.
So $M \models \varphi(\bar{c})$. Inductively, $M \models_W
\varphi(\bar{c})$. Since $ \bar{c} =_i \bar{a}$, we have $M
\models_W \exists x_i \varphi(\bar{a})$ by definition of the
relativized semantics. This completes the induction.\\

Define $\C$ to be the complex algebra
over $\At\A$, the atom structure of $\A$.
Then $\C$ is the completion of $\A$.
Let $\cal D$ be the relativized set algebra with domain $\{\phi^W: \phi\text { an $L_{\infty\omega}^n$ formula }\}$,  unit $W$
and operations defined like those of $\A$.
Then we have ${\C}\cong \D$, via the map $X\mapsto \bigcup X$ \cite[Lemma 5.3]{Hodkinson}.

Consider the following relation algebra atom structure 
$\alpha(\G)=(\{{\sf Id}\}\cup (\G\times n), R_{\sf Id}, \breve{R}, R_;)$, where:

The only identity atom is $\sf Id$. All atoms are self converse,
so $\breve{R}=\{(a, a): a \text { an atom }\}.$
The colour of an atom $(a,i)\in \G\times n$ is $i$. The identity $\sf Id$ has no colour. A triple $(a,b,c)$
of atoms in $\alpha(\G)$ is consistent if
$R;(a,b,c)$ holds $(R;$ is the accessibility relation corresponding to composition). Then the consistent triples 
are $(a,b,c)$ where:
\begin{itemize}

\item one of $a,b,c$ is $\sf Id$ and the other two are equal, or

\item none of $a,b,c$ is $\sf Id$ and they do not all have the same colour, or

\item $a=(a', i), b=(b', i)$ and $c=(c', i)$ for some $i<n$ and
$a',b',c'\in \G$, and there exists at least one graph edge
of $G$ in $\{a', b', c'\}$.

\end{itemize}
 
$\C$ is not representable because $\Cm(\alpha(\G))$ is not representable  and 
${\sf Mat}_n(\alpha(\G))\cong \At\A$. To see why, for each  $m  \in {\Mat}_n(\alpha(\G)), \,\ \textrm{let} \,\ \alpha_m
= \bigwedge_{i,j<n}  \alpha_{ij}. $ Here $ \alpha_{ij}$ is $x_i =
x_j$ if $ m_{ij} = \Id$ and $R(x_i, x_j)$ otherwise, where $R =
m_{ij} \in L$. Then the map $(m \mapsto
\alpha^W_m)_{m \in {\Mat}_n(\alpha(\G))}$ is a well - defined isomorphism of
$n$-dimensional cylindric algebra atom structures.
Non-representability follows from the fact that  $\G$ is a bad graph, that is, 
$\chi(\G)=N<\infty$ \cite[Definition 14.10, Theorem 14.11]{HHbook}.
The relation algebra atom structure specified above is exactly like the one in Definition 14.10 in {\it op.cit}, except that we have $n$ colours 
rather than just three, and thats precisely why
the set $n$--dimensional basic matrices forms an $n$--dimensional cylindric basis, so we could lift the result from $\RA$s to cylindric algebras of any finite dimension $>2$. 
However, the proof progressed the other way round, starting with a $\CA_n$ constructed model-theoretically, and 
then defining the relation algebra atom structure.
But in any case, we obtained our result for both $\CA_n$s ($2<n<\omega)$
and $\RA$s in one go.

 \end{proof}

Let $2<n<\omega$. The $n$--dimensional 
$\A_m$s  used in the last item of theorem \ref{SL} witnessing non -finite axiomatizability of $\bf El K$ for any $\bold K$, such that $\Nr_n\K_{\omega}\subseteq \bold K\subseteq 
\sf RK_n$ are finite and not representable, hence they are completely additive, atomic and 
outside ${\bf El}{\sf SRK_n}\supseteq {\sf LK}_n$. Furthermore any non-trivial ultraproduct of such algebras 
is also atomic and completely addtive (the last two properties together are first order definable hence preserved under taking ultraproducts), and is 
in  ${\bf El}\Nr_n\K_{\omega}^{\sf ad}\cap \At\subseteq {\sf LK}_n$. This shows that both ${\bf El}{\sf SRK_n}$ 
and ${\sf LK}_n$ are not finitely axiomatizable. 
Here we re-prove the last result for $\sf LK_n$ 
but we restrict our attension to  $\CA_n$.

Let $\sf LCRA$ be the elementary class of relation algebras satsfying the Lyndon conditions.
We now use the constructions in the previous two theorems \ref{conditional} and \ref{conditional2} to prove the next theorem.
We give two proofs. In each we use bad Monk--like algebras converging to a good one, and in the process, we recover Monk's classical results on non-finite axiomatizability of both 
$\sf RCA_n$  $(2<n<\omega)$,  and representable relation algbras $(\sf RRA)$.

\begin{corollary}\label{lyndon} Let $2<n<\omega$. Then $\sf LCRA$ and $\sf LCA_n$ are 
not finitely axiomatizable.
\end{corollary}
\begin{proof} 
{\bf First proof:} Our first proof addresses only $\sf RA$s and $\CA_3$. It can be lifted to $\CA_n$ 
for finite $n>3$, by using the methods in \cite{ANT}. Here we use a simple 
version of  the construction in \cite{ANT}. 
Let $l\geq 2$. Let $I$ be a set
such $\omega>|I|\geq 3l$ if $l$ is finite, and $|I|=|l|$ otherwise. 
$I$ will be the set of non-identity atoms of the blown up and blurred relation algebra when $I$ is finite. Let 
$J=\{(X: X\subseteq I, |X|=l\}.$ This is a generalization  of 
the set of $3$-complex blurs as defined in \cite[Definition 3.1(ii)]{ANT} allowing infinite $l$.
Let 
$H=\{a_i^{P,W}: i\in \omega, P\in I, W\in J\}.$ When $I$ is finite, this is denoted by $\At$ in \cite[p. 73]{ANT} which is the notation we used in the above 
sketch. In all cases, including infinite $l$, 
it will give rise to an infinite 
atom structure of a relation algebra $\F$  whose underlying set is 
$H\cup {Id}$ as follows: 
Define the tenary relation $E$ on $\omega$ as follows:
$$E(i,j,k)\iff (\exists p,q,r)(\{p,q,r\}=\{i,j,k\} \text { and } r-q=q-p).$$
This is a concrete instance of an {\it index blur} as defined in \cite[Definition 3.1(iii)]{ANT}.
Together $(J, E)$ is an $3$--blur but {\it not} a strong $3$--blur.
Now we specify the consistent triples:
$(a, b, c)$ is consistent $\iff$ one of $a, b, c$ is $\sf Id$ and the other two are equal, or 
$a=a_i^{P,S,p}, b=a_j^{Q,Z,q}, c=a_k^{R,W,r}$ where either  $S\cap Z\cap W=\emptyset \text { or both } 
e(i,j,k) \text { and } |\{P,Q,R\}|\neq 1$ (we are avoiding mononchromatic triangles).
That is if for $W\in J$,  $E^W=\{a_i^{P,W}: i\in \omega, P\in W\},$
then $$a_i^{P,S};a_j^{Q,Z}=\bigcup\{E^W: S\cap Z\cap W=\emptyset\}
\bigcup \{a_k^{R,W}: E(i,j,k), |\{P,Q,R\}|\neq 1\}.$$

When $I$ is finite, the two partitions $P_1$ and $P_2$ respectively 
are defined as follows:
For $P\in I$, let $H^P=\{a_i^{P,W}: i\in \omega, W\in J, P\in W\}$.
Then 
$$P_1=\{H^P: P\in I\}, \ \, P_2=\{E^W: W\in J\}.$$
Pending on $l$ and $I$, call these atom structures ${\cal F}(l, I).$
When $l<\omega$, then ${\cal C}m{\cal F}(l, 3l)$ is not representable, because in this case one can embed the finite 
relation algebra $\mathfrak{E}_I(2,3)$ with non-identity atoms $I=3l$ 
into $\Cm{\cal F}(l, 3l)$.
Indeed, the composition in $\Cm{\cal F}(l, 3l)$ is defined such that its restriction on the first partition satisfies:
$$ H^P;H^Q=\bigcup \{H^Z: |\{Z,P;Q\}|\neq 1\}=\bigcup \{H^Z: Z\leq P;Q\text { in } \M\}.$$
Now non-representability of $\C=\Cm{\cal F}(l, 3l)$ follows from the fact that 
$\mathfrak{E}(2,3)$ can be represented, if at all, only on finite 
sets, which is not the case with $\C$ 
since it is infinite. 
The universe of ${\cal R}_l=\Tm{\cal F}(l, 3l)$ is the following set:
$$R_l=\{X\subseteq F: X\cap E^W\in {\sf Cof}(E^W), \text{ for all } W\in J\}.$$
The algebra ${\cal R}_l$, which we denote simply by $\cal R$ as we did in the above sketch, is representable. For any
$a\in H$ and $W\in J,$  set
$U^a=\{X\in R: a\in X\}\text { and } U^{W}=\{X\in R: |X\cap E^W|\geq \omega\}.$
Then the principal ultrafilters of $\cal R$ are exactly $U^a$, $a\in H$ and $U^W$
are non-principal ultrafilters for $W\in J$ when $E^W$ is infinite.
Let  $J'=\{W\in J: |E^W|\geq \omega\},$
and let ${\sf Uf}=\{U^a: a\in F\}\cup \{U^W: W\in J'\}.$
${\sf Uf}$ is the set of ultrafilters of $\cal R$ which is used as colours
to represent $\cal R$, witness \cite[pp. 75-77]{ANT}. 

Let ${\cal F}$ be a non- trivial ultraproduct of ${\cal F}(l, 3l)$, $l\in \omega$. 
Then ${\cal F}$ is a completely representable atom structure obtained by blowing up and blurring
$\mathfrak{E}_{\omega}(2, 3)$. Here $\mathfrak{E}_{\omega}(2,3)$ embeds into $\Cm(\cal F)$, but now the former algebra is representable
plainly on infinite sets as the base of the representation.
This does not tell us that $\cal F$ is completely representable, but at least it does not contradict it. 
And indeed, it can be shown without much difficulty 
$\Cm(\cal F)$ is actually completely representable. Here the set $J$ of blurs is infinite and $\sf Uf$ as defined above would be 
a proper subset of the set of all ultrafilters of the term algebra.

Thus ${\cal R}_l=\Tm{\cal F}(l, 3l)$ are $\sf RRA$'s without a complete representation with a completely 
representable ultraproduct, and $\Cm {\cal F}(l, 3l)=\Cm\At({\cal R}_l)$, $l\in \omega$, are non--representable relation algebras, with a completely representable ultraproduct; since
the last two ultraproducts have the same atom structure $\F$.
Since such algebras have 
$3$--dimensional cylindric basis, we get the analogous result for $\sf LCA_3$.

{\bf Second proof:} Let $\A_l$ be the atomic $\sf RCA_n$ 
constructed from $\G_l$, $l\in \omega$ where $\G_l$ has nodes $\N$ and edge relation $E_l$ defined by
$(i,j)\in E_l\iff 0<|i-j|<N_l$, or a disjoint countable union of $N$ cliques, such that for $i<j\in \omega$, $n(n-1)/n\leq N_i<N_j.$
Then  $\Cm\A_l$ with $\A_l$ based on $\G_l$, as constructed in theorem \ref{conditional2} is not representable.
So $(\Cm(\A_l): l\in \omega)$ is a sequence of non--representable algebras,
whose ultraproduct $\B$, being based on the ultraproduct of graphs having arbitrarily large chromatic number, 
will have an infinite clique, and so $\B$ will be completely representable \cite[Theorem 3.6.11]{HHbook2}.
Likewise, the sequence $(\Tm(\A_l): l\in \omega)$ is a sequence of representable, 
but not completely representable algebras, whose ultraproduct is completey representable.
The same holds for the sequence of 
relation algebras $(\R_l:l\in \omega)$ constructed as in the proof of theorem \ref{conditional2},  for 
which $\Tm\At\A_l\cong \Mat_n\At\R_l$.

\end{proof}

\begin{remark} The model--theoretic ideas used in item  (2) of theorem \ref{main} and construction in the proof of theorem \ref{conditional2}
are quite similar, in the overall structure;
they follow closely the model-theoretic framework in \cite{Hodkinson}.

In both cases, we have finitely many shades of red outside a  Monk-like and rainbow
signature,  that were used as labels to construct an $n$-- homogeneous model $M$ in the
expanded signature. Though the shades of reds are {\it outside} the signature, they were used as labels
during an $\omega$--rounded game played on labelled finite graphs--which can be seen as finite models in  the extended signature having size $\leq n$--
in which \pe\ had a \ws,  enabling her to
construct the required $M$ as a countable limit of the finite graphs played during the game. The construction, in both cases, entailed that any subgraph (substructure)
of $M$ of size $\leq n$, is independent of its location in $M$;
it is uniquely determined by its isomorphism type.

A relativized set algebra $\A$ based on $M$ was constructed
by discarding all assignments whose edges are labelled
by these shades of reds,  getting a set of $n$--ary sequences $W\subseteq {}^nM$. This $W$ is definable in $^nM$ by an $L_{\infty, \omega}$ formula
and the semantics with respect to $W$ coincides with classical Tarskian semantics (on $^nM$) for formulas of the
signature taken in $L_n$ (but not for formulas taken in $L_{\infty, \omega}^n$).

This was proved in both cases using certain $n$ back--and--forth systems, thus $\A$ is representable classically,
in fact it (is isomorphic to a set algebra that) has base $M$.
The {\it heart and soul} of the proof, is to replace the reds labels by suitable
non--red binary relation symbols within an $n$ back--and--forth system, so that one  can
adjust that the system maps a tuple $\bar{b} \in {}^n M \backslash W$ to a tuple
$\bar{c} \in W$ and this will preserve any formula
containing the non--red symbols that are
`moved' by the system.  In fact, all
injective maps of size $\leq n$ defined on $M$ modulo an appropriate
permutation of the reds will
form an $n$ back--and--forth system. 

This set algebra $\A$
will further be atomic, countable, and simple (with top element $^nM$). The subgraphs of size $\leq n$ of $M$ whose edges are not labelled by any shade of red are 
the atoms of $\A$, expressed syntactically by $\sf MCA$ formulas.

The  \de\  of $\A$, namely, the complex algebra of its atom structures, in symbols $\C=\Cm\At\A,$
has top element $W$, but it is not in $\bold S\Nr_n\CA_{n+3}$ in case of the rainbow construction, least representable,
and it  is  not representable in the
Monk-algebra case, so that it will be outside $\bold S\Nr_n\CA_{n+k}$ for some $k\geq 1$. But as opposed to the rainbow
construction,  we cannot tell  from the Monk--like
construction, based on a Monk-like first order theory stipulated above , what is the least such $k$. Nevertheless, in this last case,
$\At\A$  turns out isomorphic to the atom structure consisting of $n$--basic matrices on a weakly, but not strongly
representable, relation atom  structure.

In case of both constructions `the shades of  red' -- which can be intrinsically identified with
non--principal ultrafilters in $\A$,  were used as colours, together with the principal ultrafilters
to represent completely $\A^+$, inducing a representation of $\A$.  Non--representability of $\Cm\At\A$ in the Monk case, used Ramsey's theory. The non neat--embeddability of
$\Cm\At\A$ in the rainbow case, used {\it the finite number of greens} that gave us information on
when $\Cm\At\A$   `stops to be representable.'  The reds in both cases has to do with representing $\A$.

The model theory used for both constructions is almost identical.
Nevertheless,  from the algebraic point of view,
there is a crucial difference between the Monk--like algebras used above and the rainbow algebra used in item (2) of theorem 
\ref{main}.
The non--representability of $\Cm\At\A$ can be tested by a game between the two players \pa\ and \pe.
In the rainbow algebra \ws's of the two players are independent, this is  reflected by the fact
that we have  two `independent parameters' $\sf G$ (the greens)  and $\sf R$ (the reds).
In  Monk--like algebras \ws's are interlinked, one operates through the other; hence only one parameter is the source of colours,
namely,  the graph $\G$. Representability of the complex algebra in
this case depends only on the chromatic number of $\G$, via an application of Ramseys' theorem.

In both cases two players operate using `cardinality of a labelled graph'.
\pa\ trying to make this graph too large for \pe\ to cope, colouring some of
its edges suitably. For the rainbow case, it is a red clique formed during the play.
It might be clear in both cases (rainbow and Monk--like algebras),  to see that \pe\ cannot win the infinite game, but what might not
be clear is {\it when does this happens; we know it eventually happen in finitely many round, but till which round \pe\ has not lost yet}.

In rainbow constructions, one can contol this by varying the green parameter of the construction in \cite{Hodkinson} as we did truncating the number of greens to be 
$n+1$. The structures $\sf G$ and $\sf R$, having any relative strength
gives flexibility and more control over the rainbow  game lifted
from an \ef\ forth--game on these structures.
The number of pebbles used by \pa\ in the graph game
used, determines  exactly when $\Cm\At\A$ `stops to be representable'. This idea can be implemented semantically
like Hodkinson's proof {\it excluding $n+3$--flat representability}
or syntactically, via {\it a blow up and blur} construction. We chose the second approach, proving that for $2<n<\omega$, the class 
$\bold S\Nr_n\CA_{n+k}$ is not atom canonical for any $k\geq 3.$

Now what if in the `Monk construction' based on $\G$,  we have $\chi(\G)=\infty?$ for a graph $\G$.  
Let us approach the problem abstractly. 
Let $\G$ be a graph. One can  define a family of first order structures (labelled graphs)  in the signature $\G\times n$, denote it by $I(\G)$
as follows:
For all $a,b\in M$, there is a unique $p\in \G\times n$, such that
$(a,b)\in p$. If  $M\models (a,i)(x,y)\land (b,j)(y,z)\land (c,l)(x,z)$, then $| \{ i, j, l \}> 1 $, or
$ a, b, c \in \G$ and $\{ a, b, c\} $ has at least one edge
of $\G$.
For any graph $\Gamma$, let $\rho(\Gamma)$ be the atom structure defined from the class of models satisfying the above,
these are maps from $n\to M$, $M\in I(\G)$, endowed with an obvious equivalence relation,
with cylindrifiers and diagonal elements defined as Hirsch and Hodkinson define atom structures from classes of models,
and let $\M(\Gamma)$ be the complex algebra of this atom structure.

We define a relation algebra atom structure $\alpha(\G)$ as above.
Then the relation algebra to have an $n$--dimensional cylindric basis
and, in fact, the atom structure of $\M(\G)$ is isomorphic (as a cylindric algebra
atom structure) to the atom structure $\Mat_n$ of all $n$-dimensional basic
matrices over the relation algebra atom structure $\alpha(\G)$. 
It is plausible that one can prove that 
$\alpha(\G)$ is strongly representable $\iff \M(\G)$ is representable 
$\iff \G$ has infinite chromatic number, so that one gets the result, that the class of strongly represenatble algebras for both $\RA$s and 
$\CA$s of finite dimension at least three, is not elementary in one go.

The underlying idea here is that shade of red $\rho$
will appear in the {\it ultrafilter extension} of $\G$, if it 
has infinite chromatic number, as a reflexive node 
\cite[Definition 3.6.5]{HHbook2} and its $n$--copies $(\rho, i)$,  $i<n$, 
can be used to completely represent
$\M(\G)^{+}$.

\end{remark}

\section{Omitting types} 

Algebraic logic is most effective and attractive when it has non-trivial repercussions on (first order) logic.
Indeed, this section is the climax of the two parts of the paper, where the progression of the previously proved results culminate.  
We apply the algebraic
results obtained so far, to various omitting types theorems for variants of first order logic.

We formulate our results only for $\CA$s to economize on space. 
The other cases can be dealt with exactly in the same way. 
Recall that $L_n$ denotes first order logic restricted to the first $n$ variables.

Henceforth, because we deal with $\CA$s only, we deviate from our earlier notation for several subclasses of $\CA$s, 
that was introduced earlier 
differently, to deal with all cases uniformly, always referring to $\K$. 
For example we denoted the class of dimension complemented algebras in $\K_{\alpha}$, 
$\alpha$ an infinite ordinal, by
${\sf DKc}_{\alpha}$. Here we denote this class for $\CA_{\alpha}$s  by the common more familiar notation adopted in \cite{HMT2}, namely,
by $\sf Dc_{\alpha}$. By the same token, for an ordinal $\alpha$, restricting our attention to $\CA$s, $\sf Cs_{\alpha}$, $\sf Ws_{\alpha}$ $\sf Gs_{\alpha}$, $\sf Gws_{\alpha}$, and $\sf Lf_{\alpha}$
denote the classes of  set algebras,  weak set algebras, generalized set algebras, generalized weak set algebras, and locally finite 
algebras, of dimension $\alpha$, respectively. For $\alpha<\omega$, $\sf Cs_{\alpha}=\sf Ws_{\alpha}$, $\sf Gs_{\alpha}=\sf Gws_{\alpha}$
and $\sf Lf_{\alpha}=CA_{\alpha}$. We should also keep in mind that for every ordinal $\alpha$, we have 
$\RCA_{\alpha}=\bold I\sf Gs_{\alpha}=\bold I\sf Gws_{\alpha}$.

We start with an algebraic lemma implied by the results obtained in the first two items of 
theorem \ref{main}, but it implies neither.

\begin{lemma}\label{lemma} Let $2<n<\omega$. Then there exists an atomic countable $\sf RCA_n$ 
that does not have a complete infinitary $n+3$--flat representation.
In fact, there exists such an algebra, 
that is further simple, and does not 
have a complete $n+3$--square representation.
\end{lemma}
\begin{proof}
We first consider infinitary $n+3$--flatness. 
Consider the term algebra $\E$ of the $\CA$ reduct of $\PEA_{\Z, \N}$ used in the first item of 
theorem \ref{main},  and the $\CA$ reduct of the algebra $\A$ used in second item of the same theorem, which we continue to denote by $\A$, 
each is of dimension $n$.

Such algebra are are countable and atomic; $\A$ is simple.  They have 
no complete infinitary $n+3$--flat representations.
Indeed, if $\E$ has an infinitary $n+3$--flat complete representation, 
then it would be, by lemma \ref{flat} in $\bold S_c\Nr_n\CA_{n+3}$, which means that $\Rd_{ca}\PEA_{\Z, \N}$ is also in $\bold S_c\Nr_n\CA_{n+3}$ 
(for they share the same atom structure) which is impossible. By the same token, if $\A$ has a complete infinitary $n+3$--flat  representation $M$, 
this would induce an
infinitary $n+3$--flat representation of $\Rd_{ca}\Cm\At\A$.
This implies by the same item in the same lemma 
that $\Rd_{ca}\Cm\At\A\in \bold S\Nr_n\CA_{n+3}$, 
which is also impossible.  

We now show that  $\A$ does not have a 
complete $n+3$--square representation. An $n+3$--square representation of $\A$ will imply, by lemma \ref{Thm:n}, 
that \pe\ has a \ws\ in $G^{n+3}_{\omega}(\At\A)$, which is impossible because by item (2) of theorem \ref{main} 
 \pa\ has a \ws\ in this game.
\end{proof}

Let $T$ be an $L_n$ theory in a signature $L$. An $L$--formula $\alpha$ is called an {\it atom}, if $\alpha$ is consistent with $T$
and  for every $L$--formula $\psi$, one of $\psi\to \alpha$  or
$\psi\to \neg \alpha$ is valid. A {\it co--atom} is the negation of an atom.
$T$ is  {\it atomic}, if for every $L$--formula $\phi$ consistent with $T$, 
there is  an  atom $\alpha$, such that $T\models \alpha\to \phi$,

A set of $L$--formulas $\Gamma$ consistent with $T$ is {\it principal}, if there exists a formula $\alpha$ that is 
consistent  with $T$,  such that $T\models \alpha\to \beta$ for all $\beta\in \Gamma$. 
Otherwise, it is {\it non--principal}. 

$\Fm_{T}$ denotes the Tarski--Lindenbaum algebra obtained by factoring out the set of formulas $Fm$ by the congruence relation
$\phi\sim \psi\iff T\models \phi\leftrightarrow \psi.$ 
Then it is easy to see that 
$$\text  {$T$ is atomic }\iff \Fm_{T} \text { has an atomic Boolean reduct.}$$

Let $2<n<m\leq \omega$. From now on we deal with $\Fm_T\in \RCA_n$ when $T$ is complete in a signature $L$, 
so that $\Fm_T$ is simple. We say that $M$ is an $m$--relativized model of $T$, if there exist an embedding $f:\Fm_T\to \wp(1^M)$, equivalently $M$ is a relativized 
$m$-- representation of $\Fm_T$.

\begin{definition} Let $2<n<\omega$, let $\A\in \RCA_n$ be simple, and $X\subseteq \A$. Then an $m$--relativized model 
$M$ of $\A$ omits $X$, if there exists an embedding $f:\A\to \wp(1^M)$ such that $\bigcap f(x)=\emptyset$.
If $T$ is a complete $L_n$ theory and $\Gamma$ is a set of formulas,  then $\Gamma$ is omitted in an $m$--relativized model $M$
of $T$ if $M$ omits $\{\phi_T:\phi\in \Gamma\}(\subseteq \Fm_T)$.
\end{definition}

\begin{theorem}\label{OTT}
Let $2<n<\omega$. Then there is a countable, atomic, complete and consistent $L_n$ theory $T$, such that the
non--principal type consisting of co--atoms cannot be omitted
in an $n+3$--square model, 
 \end{theorem}
\begin{proof}   
We focus on $\A$ as in the second item of theorem \ref{main}.   
We can identify $\A$ with $\Fm_{T}$ for some countable and consistent
theory $T$ using $n$ variables, and because $\A$ is atomic,
$T$ is an atomic theory.

Let $\Gamma=\{\neg\phi: \phi_T\in \At\Fm_{m,T}\}$.
Then it is easy to see, because $T$ is  atomic, that $\Gamma$ is 
a non--principal type. But  $\Gamma$ cannot be omitted
in an {\it $n+3$--square} model 
for such a  model necessarily gives  a complete $n+3$--square representation 
of $\A$, which contradicts the previous lemma.
Hence $\Gamma$ cannot
be omitted in any $m$--square model, for $n\leq m\leq \omega$.

The same conclusion holds for uncountable $m$, for in this case we have 
$M$ is a complete $m$--square representation of $\A\iff  M$ is an $\omega$--square complete representation of $\A\iff$ 
$M$ is a complete representation of $\A$,
because $\A$ is countable, and we know that $\A$ has no complete representation.

Since $\A$ is simple, then $T$ is also 
complete, thus  $T$ and $\Gamma$ 
are  as required.
\end{proof}

We remind the reader of (the contrapositive of) the usual omitting types theorem:  If $T$ is a countable consistent 
first order theory, and $\Gamma$ is an $m$, ($m<\omega)$ type, that is, a set of formulas having at most $m$ free variables, that is realized 
in every model of $T$, then $\Gamma$ is {\it isolated by a formula}, in the sense that, there exists a formula $\alpha$ consistent with $T$ such
that $T\models \alpha\to \beta$ for all $\beta\in \Gamma$. We call $\alpha$ a $k$ witness, $1\leq k<\omega$, 
if it uses $k$ variables.

\begin{theorem}\label{con}\begin{enumarab}
\item  Let $n>3$ be finite.  Assume that  there exists a finite relation algebra $\R$ as in the statement of theorem \ref{conditional} but with the weaker condition of
existence of $n$--blur. In particular, $\R\notin \bold S\Ra\CA_{n+1}$. Then  for any $2<m<n$, there is an $L_m$ atomic theory $T$ in a signature with only one binary relation,
and a non--principal type $\Gamma$ using only three variables and one free variable,  such that $\Gamma$ cannot be omitted in an $n+1$--flat model.
\item If, in addition, $\R$ has a {\it strong} $n$-blur,  then there is an atomic $L_m$ theory $T$, and a type $\Gamma$ 
as above, such that $\Gamma$ is realizble in every $n+1$--flat model, but any formula  isolating $\Gamma$ 
has to contain more than $n$ variables.
\end{enumarab}
\end{theorem}
\begin{proof} We  prove the theorem without the refinements imposed on the type $\Gamma$
and the language having only one binary relation. Such conditions can be easily obtained from what we prove using
exactly the techniques in \cite[Theorem 3.3.1]{Sayed} and \cite{ANT}.

For the first part $\cal R$, obtained by blowing up and blurring $\R$ as in the proof of 
theorem \ref{conditional}, will be representable
and $\Cm\At{\cal R}\notin \bold S\Ra\CA_{n+1}$, because $\R$ embeds into $\Cm(\At\R)$, and by assumption 
$\R\notin \bold S\Ra\CA_{n+1}$. 
By $n$--blurness, and $m<n$, $\cal R$ will have an $m$--dimensional cylindric basis, 
and $\Tm{\sf Mat}_m(\At{\cal R})\in \RCA_m$ will have no complete $n$--flat representation.

This is so, because
$\Cm(\At(\cal R))$ embeds into it $\Cm(\Tm({\sf Mat}_m(\At{\cal R}))$ and 
$\Cm(\At{\cal R})\notin \bold S\Ra\CA_{n+1}$ so that $\Cm({\sf Mat}_n{\At\cal R}))\notin \bold S\Nr_m\CA_{n+1}$. 
As in the proof of lemma \ref{lemma}, 
a complete $n+1$--flat representation for $\Tm{\sf Mat}_n(\At{\cal R})$ induces an $n+1$--flat
representation of  $\Cm(\Tm({\sf Mat}_n{\At\cal R}))$.
The rest follows from the first item of  lemma \ref{flat} and the reasoning 
in theorem \ref{OTT} replacing $\A$ by the also countable  term algebra $\Tm({\sf Mat}_n{\At\cal R})$
and square by flat. 

For the second part, by strong $n$--blurness one gets, exactly like in the sketch of proof of theorem \ref{conditional},  a countable $\C_m\in \sf RCA_m\cap \Nr_m\CA_n$, 
such that $\Cm\At\C_m\notin \bold S\Nr_m\CA_{n+1}$.
In more detail, using the notation in the sketch given above,  the \de\ completion of 
$\cal R$ obtained by blowing up and blurring $\R$  is outside $\bold S\Ra\CA_{n+1}$,
$\sf Mat_{n}(\At\cal R)$ is  an $n$--dimensional cylindric basis, since $\R$ has an $n$--blur 
$(J, E)$,  and this $n$--blur is strong, 
hence there exists 
$\C_{n}={\sf Bb}_{n}(\R, J, E)\in \CA_{n}$ 
such that $\Tm\sf Mat_{n}(\At{\cal R})\subseteq \C_{n}\neq \Cm{\sf Mat}_{n}(\At\cal R)$, and finally 
because $(J, R)$ is a strong $m$--blur, as well, then $\C_m={\sf Bb}_m(\R, J, E)\cong \Nr_m\C_{n}$ is in 
$\sf RCA_m\cap \Nr_m\CA_{n}$.  But $\Cm\At\C_m\notin \bold S\Nr_m\CA_{n+1}$, because 
$\Cm\At\cal R$ embeds into the $\Ra$ reduct of $\Cm\At\C_m$,  
$\R$ embeds into $\Cm\At\cal R$, and by assumption 
$\R\notin \bold S\Ra\CA_{n+1}$.

The $L_m$ theory $T$ corresponding to $\C_m$,  and the non--principal type $\Gamma$ corresponding to its co--atoms, 
as defined in the proof of  theorem \ref{OTT} are as required, witness  \cite[Theorem 3.3.1]{Sayed}.
We give more details:

For brevity, let $\A$ be $\mathfrak{C}_m$ as described above.
Let $\Gamma'$ be the set of  atoms of $\A$.
We can and will assume that $\A$ is simple (this is proved in \cite{ANT}).
Then $\A=\Fm_T$ for some countable consistent
theory $T$ using $m$ variables and
and so  $T$ is atomic.
Let $\Gamma=\{\phi: \phi_T\in \Gamma'\}$. Then we claim
that $\Gamma$ is realized in each $n+1$
flat model of $T$.  Indeed, consider
such  an $n+1$  model $M$ of $T$. Let $V\subseteq {}^nM$ be such that $\bigcup_{s\in V}\rng(s)=M$.
For a formula $\phi$ let $\phi^{M}_c$ be the set of assignments in $V$ that satisfy $\phi$ in the clique guarded sematics, that is
$\phi^{M}_c=\{s\in V: M, s\models _c \phi\}$.
If  $\Gamma$ is not realized in $M$,
then this gives an $n+1$ complete flat representation of $\A=\Fm_T$, defined via $\phi_T\mapsto \phi^M_c$,
since $\bigcup_{\phi\in \Gamma} \phi^{M}_c={}1^M=V$.

For the second part concerning sensitivity of witnesses to
the number of variables used. Suppose for contradiction that $\phi$ is a formula in the signature $L$, consistent with $T$,  using $n$ variables, such that 
$T\models \phi\to \alpha$, for all $\alpha\in \Gamma$.
Then $\A$ is simple, and so we can assume
without loss of generality, that it is (isomorphic to) a set algebra with base $M$, say. 
Let $M$ be the corresponding standard model (in a relational signature)
to this set algebra in the sense of \cite[section 4.3]{HMT2}. Let $\phi^{M}$ denote the set of all assignments classically
satisfying $\phi$ in $M$.
We have  $M\models T$ and $\phi^{M}\in \A$,
because $\A\in \Nr_m\CA_{n}$.

But $T\models \exists \bar{x}\phi$, hence $\phi^{M}\neq 0,$
from which it follows that  $\phi^{M}$ must intersect an atom $\alpha\in \A$,  since the latter is atomic.
Let $\psi$ be the formula, such that $\psi^{M}=\alpha$. Then it cannot
be the case  that $T\models \phi\to \neg \psi$,
hence $\phi$ is not an $n$--witness,
and we are done.

\end{proof}

{\bf Question:} Does $\R$ in the first item exist? 
This is necessary for $\R$ in the second item to exist. Candidates for such relation algebras can 
be found in \cite{HHM} or \cite[Definition 15.2]{HHbook}, but it seems that they need some modifications to fit the bill.

The above theorem does not work for infinite $n$, because in this case $\C_m\in \Nr_m\CA_{\omega}$ which means, by theorem \ref{complete}, since
$\C_m$ is countable and atomic, 
that it is completely representable. This induces a representation of $\Cm\At\C_m$, which is impossible, because we know 
that $\At\C_m$ is not strongly representable.  In other words, for $2<m\leq k\leq \omega$ and $m\in \omega$, 
there is a countable, atomic, simple $\A\in \sf RCA_m\cap \Nr_m\CA_{m+k}$ that is not completely representable 
$\iff$ $k<\omega$.\\

One strategy to circumvent negative results (like non--finite axiomatizability) in algebraic logic is to pass to
`nice expansions' of the class in question. Biro \cite{Biro} showed that for $2<n<\omega$, expanding $\sf RCA_n$ 
with finitely many {\it first order definable operations}, to be defined next,  does not conquer non--finite axiomatizability.
Thus $L_n$ enriched with the corresponding first order definable connectives is still severely incomplete 
(as long as these connectives are finite) relative to any finite Hilbert--style axiomatization. 

We show next that such expansions  do not conquer failure of omitting types, as well. 
First order definable expansions of finite variable 
fragments of first order logic was initiated by Jonsson \cite{j} in the context of relation algebras, 
and were further studied by B\'iro, Givant, N\'emeti, Tarski, S\'agi and others 
\cite{Biro, n, TG},  and naturally extrapolated to $L_n$ \cite{n}. 
Our approach in what follows is mostly algebraic.
 
\begin{definition}\label{first}
Let $2<n<\omega$.
Let $\Cs_{n,t}$ denote the following class of signature $t$:
\begin{enumroman}
\item $t$ is an expansion of the signature of $\CA_{n}.$
\item  $\bold S\Rd_{ca}\Cs_{n,t}=\Cs_{n}.$ In particular, every algebra in $\Cs_{n,t}$ is a Boolean
field of sets with unit ${}^{n}U$ say,
that is closed under cylindrifications and contains diagonal elements.
\item For any $m$-ary operation $f$ of $t$, there exists a first order formula $\phi$ with free variables among the first $n$
and having exactly $m,$ $n$--ary relation symbols
$R_0, \ldots, R_{m-1},$ such that,
for any set algebra ${\A}\in \Cs_{n,t}$
with base $M$, $X_0, \ldots, X_{m-1}\in {\A}$, and $s\in {}^nM$
we have:
$s\in f(X_0,\ldots, X_{m-1})\iff\ \langle M, X_0,\ldots, X_{m-1}\rangle\models \phi[s].$
Here $R_i$ is interpreted as $X_i,$ and $\models$ is the usual satisfiability relation.

\item With $f$ and $\phi$ as above,
$f$ is said to be {\it a first order definable operation with $\phi$ defining $f$},
or simply a first order definable
operation, and $\Cs_{n,t}$ is said to be a {\it first order definable
expansion} of $\Cs_n.$ (The defining $\phi$ is not unique.)

\item $\RCA_{n,t}$ denotes the class ${\bf SP}\Cs_{n,t}$, i.e. the class of all subdirect products
of algebras
in $\Cs_{n,t}.$ We also refer to
$\RCA_{n,t}$ as a {\it first order 
definable expansion} of $\RCA_{n}.$
\end{enumroman}
\end{definition}
Fix $2<n<\omega$. Like $\sf RCA_{n}$, one can show that $\sf RCA_{n,t}$ is a discriminator variety (the discriminator terrm is ${\sf c}_{(n)}x$), 
that is not finitely axiomatizable, 
if the number of variables used in the first order definable operations 
is finite \cite{Biro}.  Let $L_n^t$ be the corresponding (algebraizable) logic using $n$ variables. 
This logic is obtained from $L_n$ by adding connectives definable by the first order formulas used to define the operations in $t$.

A famous theorem of Vaught's for first order logic says that an atomic countable theory has an atomic countable model. 
Now restricting Vaught's theorem to 
(the algebraic counterpart of) $L_n^t$, 
it takes the following form, which we denote by $\sf VT$:

\begin{definition} Let everything be as in the last item of the previous definition. 
Then ${\sf RCA}_{n,t}$ has $\sf VT$, if whenever $\A\in \RCA_{n,t}$ is simple, countable, and atomic,  
there exists $\B\in \Cs_{n,t}$ with base $M$ say, 
and an isomorphism $f:\A\to \B$ such that  
$\bigcup_{x\in \At\A} f(x)={}^nM$.
\end{definition}
Notice that in theorem \ref{OTT}, we have actually showed that $\sf VT$ fails if we 
consider atomic $n+3$--square relativized models, when $n>2$ is finite. In what follows 
$\sf OTT$ abbreviates `omitting types theorem'. The result in the following theorem is mentioned without proof 
on \cite[p.87]{ANT}:

\begin{theorem}\label{firstorder} Let $2<n<\omega$. 
Let $\sf RCA_{n,t}$ be a first order definable 
expansion of $\sf RCA_n$. If the operations in $t$ are first order definable 
by formulas using only finitely many $k>n$ variables, then $\sf RCA_{n,t}$ does not have $\sf VT$. 
In particular, $\sf OTT$ fails for $L_n^t$. 
\end{theorem}
\begin{proof} Let $2<n<\omega$. Let $\A\in \RCA_n\cap \Nr_n\CA_{k+1}$ be simple and 
countable, with no complete representation. Such an algebra exists \cite{ANT}. 
Indeed, like in \cite[Lemma 5.1]{ANT},  take $l\geq 2(k+1)-1$, $m\geq (2(k+1)-1)l$, $m\in \omega$, and take $\R$ in theorem \ref{conditional} to be the finite
integral relation algebra ${\mathfrak E}_m(2, 3)$, where $m$ is the number of non-identity atoms in ${\mathfrak E}_m(2,3)$, 
and compostion in the latter algebra is defined by allowing all triangles 
except monochromatic ones \cite[p.83]{ANT}. 
This finite relation algebras has a strong $k+1$--blur, hence 
using the notation in theorem \ref{conditional}, we have 
$\A={\sf Bb}_n({\mathfrak E}_m(2, 3), J, E))$ with atom structure ${\sf Mat}_n(\At{\cal R})$, where $\cal R$ is the representable
relation algebra obtained by blowing up and blurring ${\mathfrak E}_m(2,3)$,  
is isomorphic to $\Nr_n{\sf  Bb}_{k+1}({\mathfrak E}_m(2, 3), J, E)$,  
where ${\sf  Bb}_{k+1}({\mathfrak E}_m(2, 3), J, E)$ has atom structure 
${\sf Mat}_{k+1}(\At{\cal R})$.  Furthermore, $\A$ is not strongly representable, that is $\Cm{\sf Mat}_n(\At{\mathfrak E}_m(2, 3))$
is not representable,  so it cannot be completely representable.

Now to simplify notation, we let $\A=\Nr_n\B$ with $\B\in \CA_{k+1}$. Without loss of generality, 
we can assume that we have only one extra operation $f$ definable by a first order restricted formula 
$\phi$, say, using $n<k<\omega$ variables with at most $n$ free variables. Now $\phi$ defines a $\CA_{k}$ term $\tau(\phi)$
which, in turn, defines the unary operation $f$ on $\A$, via $f(a)=\tau(\phi)^{\B}(a)$.  This is well defined, in the sense that $f(a)\in \A$, because
$\A\in \Nr_n\CA_{k+1}$ and the first order formula $\phi$ defining $f$ has at most $n$ free variables. 

Call the expanded algebra $\A^*$.  Then $\A^*\in \sf RCA_{n,t}$ and is still simple. For if $J$ is a proper 
ideal of $\A^*$, then $J$ is a Boolean ideal and for $a\in A$, and $i<n$, if $a\in J$, then ${\sf c_i}a\in J$, so that 
$J$ is proper ideal of $\A$, and this cannot be, because $\A$ is simple.
Obviously $\At\A=\At\A^*$. Now assume for contradiction that $\A^*$ has $\sf VT$. Then   there exists $\B\in \Cs_{n,t}$ with base $M$, say, 
and an isomorphism $f:\A^*\to \B$ such that $\bigcup_{x\in \At\A^*} f(x)={}^nM$.  
But then $f':\A\to \Rd_{ca}\B$, defined via $a\mapsto f(a)$, 
is obviously an isomorphism (of the $\CA$ reducts), 
giving a complete representation of  $\A$. This is a 
contradiction and we are done.
\end{proof}

For the same reason stated after theorem \ref{con} if $t$ has (infinitely many) 
operations that use infinitely many variables, 
then it is possible that the omitting types hold for this 
(infinite) first order definable extension of $L_n$. For example if $k\geq n$, define the connective ${\sf c}_k$ by $\exists v_kR_0$. 
Adding ${\sf c}_k$ for all $n\leq k<\omega$, captures all first order definable formulas, in the sense that if $t$ is the expanded signature 
and $\A\in \sf RCA_{n,t}$ then $\Rd_{ca}\A\in \Nr_n\CA_{\omega}$.

To prove our next theorem we need a couple of lemmas:
\begin{lemma} Let $\A, \B$ be algebras having the same signature. Then the following two conditions are equivalent:
\begin{enumroman}
\item $h:\A\to \B$ is a homomorphism.
\item $h$ is a subalgebra of $\A\times \B$ and $h$ is a function with $\dom h=A.$
\end{enumroman}
\end{lemma}
\begin{proof} \cite[Theorem 0.3.37]{HMT2}.
\end{proof}

We formulate the next lemma in a form more general than needed counting in infinite dimensions, because the proof is the same for all dimensions.
The argument used is similar to \cite[Theorem 5.3.15]{HMT2}; the latter theorem addresses relation algebras.

\begin{lemma}\label{lift}
Let $\alpha$ be an arbitrary ordinal and let $\bold K=\{\A\in \CA_{\alpha+\omega}: \A=\Sg^{\A}\Nr_{\alpha}\A\}$. Let $\A,\B\in \bold K$
and suppose that $f:\Nr_{\alpha}\A\to \Nr_{\alpha}\B$ is an isomorphism.
Then there exists an embedding $g:\A\to \B$ such that $f\subseteq g.$
\end{lemma}
\begin{proof} Let $g=\Sg^{\A\times \B}f$. It suffices to show, by the previous lemma,
that $g$ is an injection with domain $A$. It obviously has co domain $B$. Now
$$\dom g=\dom\Sg^{\A\times \B}f=\Sg^{\A}\dom f=\Sg^{\A}\Nr_{\alpha}\A=\A.$$
Though $\Dc_{\alpha+\omega}$ may not be closed under finite direct products but $\bold K$ is for
the following reasoning; assume that $\C$, $\D\in \bold K$, then we have 
$\Sg^{\C\times \D}\Nr_{\alpha}(\C\times \D)= \Sg^{\C\times \D}(\Nr_{\alpha}\C\times \Nr_{\alpha}\D)=
\Sg^{\C}\Nr_{\alpha}\C\times  \Sg^{\D}\Nr_{\alpha}\D=\C\times \D$.

By symmetry it is enough to show that $g$ is a function.  We first prove the following (*)
 $$ \text { If } (a,b)\in g\text { and }  {\sf c}_k(a,b)=(a,b)\text { for all } k\in \alpha+\omega\sim \alpha, \text { then } f(a)=b.$$
Towards this aim, we proceed as follows:
$$(a,b)\in \Nr_{\alpha}\Sg^{\A\times \B}f=\Sg^{\Nr_{\alpha}(\A\times \B)}f=\Sg^{\Nr_{\alpha}\A\times \Nr_{\alpha}\B}f=f.$$
Here we are using that $\A\times \B\in \bold K\subseteq \sf Dc_{\alpha+\omega}$, 
so that  $\Nr_{\alpha}\Sg^{\A\times \B}f=\Sg^{\Nr_{\alpha}(\A\times \B)}f.$
In case $\alpha$ is finite $\A\times \B$ is locally finite.

Now suppose that $(x,y), (x,z)\in g$.
Let $k\in \alpha+\omega\sim \alpha.$ Let $\Delta$ denote symmetric difference. Then
$$(0, {\sf c}_k(y\Delta z))=({\sf c}_k0, {\sf c}_k(y\Delta z))={\sf c}_k(0,y\Delta z)={\sf c}_k((x,y)\Delta(x,z))\in g.$$
Also,
$${\sf c}_k(0, {\sf c}_k(y\Delta z))=(0,{\sf c}_k(y\Delta z)).$$
Thus by (*) we have  $$f(0)={\sf c}_k(y\Delta z), \text { for any } k\in \alpha+\omega\sim \alpha.$$
Hence ${\sf c}_k(y\Delta z)=0$ and so $y=z$.
\end{proof}
Let $\sf covK$ be the cardinal used in \cite[Theorem 3.3.4]{Sayed}.
The cardinal $\mathfrak{p}$  satisfies $\omega<\mathfrak{p}\leq 2^{\omega}$
and has the following property:
If $\lambda<\mathfrak{p}$, and $(A_i: i<\lambda)$ is a family of meager subsets of a Polish space $X$ (of  which Stone spaces of countable Boolean algebras are examples)  
then $\bigcup_{i\in \lambda}A_i$ is meager. For the definition and required properties of $\mathfrak{p}$, witness \cite[pp. 3, pp. 44-45, corollary 22c]{Fre}. 
Both cardinals $\sf cov K$ and $\mathfrak{p}$  have an extensive literature.
It is consistent that $\omega<\mathfrak{p}<\sf cov K\leq 2^{\omega}$ \cite{Fre},
so that the two cardinals are generally different, but it is also consistent that they are equal; equality holds for example in the Cohen
real model of Solovay and
Cohen.  Martin's axiom implies that  both cardinals are the continuum.

If $\alpha$ is any ordinal $\A\in \sf RCA_{\alpha}$ and $\bold X=(X_i: i<\lambda)$ is  family of subsets of $\A$, we say that {\it $\bold X$ is omitted in $\C\in \sf Gws_{\alpha}$}, 
if there exists an isomorphism 
$f:\A\to \C$ such that $\bigcap f(X_i)=\emptyset$ for all $i<\lambda$. When we want to stress the role of $f$, 
we say that $\bold X$ is omitted in $\C$ via $f$. If $X\subseteq \A$ and $\prod X=0$, 
then we may refer to $X$ as a non-principal type of $\A$.

\begin{theorem}\label{i} Let $\A\in \bold S_c\Nr_n\CA_{\omega}$ be countable.  Let $\lambda< 2^{\omega}$ and let 
$\bold X=(X_i: i<\lambda)$ be a family of non-principal types  of $\A$.
Then the following hold:
\begin{enumarab}
\item If $\A\in \Nr_n\CA_{\omega}$,  then $\bold X$ can be omitted in a $\sf Gs_n$,
\item Every $<\mathfrak{p}$ subfamily of $\bold X$ can be omitted in a $\sf Gs_n$; in particular, every countable 
subfamily of $\bold X$ can be omitted in a $\sf Gs_n$,
\item  If $\A$ is simple, then every $<\sf cov K$ subfamily 
of $\bold X$ can be omitted in a $\sf Cs_n$, 
\item It is consistent, but not provable (in $\sf ZFC$), that $\bold X$ can be omitted in a $\sf Gs_n$,
\item If $\A\in \Nr_n\CA_{\omega}$ and $|\bold X|<\mathfrak{p}$, then $\bold X$ can be omitted $\iff$ every countable subfamily of $\bold X$ can be omitted.   
If $\A$ is simple, we can replace $\mathfrak{p}$ by $\sf covK$. 

\item If $\A$ is atomic, {\it not necessarily countable}, but have countably many atoms, 
then any family of non--principal types can be omitted in an atomic $\sf Gs_n$; in particular, 
$\bold X$ can be omitted in an atomic $\sf Gs_n$; if $\A$ is simple, we can replace $\sf Gs_n$ by 
$\sf Cs_n$.
\end{enumarab}
\end{theorem}
\begin{proof} (1) follows from theorem \ref{Shelah}. For (2) and (3), we can assume that $\A\subseteq_c \Nr_n\B$, $\B\in \Lf_{\omega}$. 
We work in $\B$. Using the notation on \cite[p. 216 of proof of Theorem 3.3.4]{Sayed} replacing $\Fm_T$ by $\B$, we have $\bold H=\bigcup_{i\in \lambda}\bigcup_{\tau\in V}\bold H_{i,\tau}$
where $\lambda <\mathfrak{p}$, and $V$ is the weak space ${}^{\omega}\omega^{(Id)}$,  can be written as a countable union of nowhere dense sets, and so can 
the countable union $\bold G=\bigcup_{j\in \omega}\bigcup_{x\in \B}\bold G_{j,x}$.  
So for any $a\neq 0$,  there is an ultrafilter $F\in N_a\cap (S\sim \bold H\cup \bold G$)
by the Baire's category theorem. This induces a homomorphism $f_a:\A\to \C_a$, $\C_a\in \sf Cs_n$ that omits the given types, such that
$f_a(a)\neq 0$. (First one defines $f$ with domain $\B$ as on p.216, then restricts $f$ to $\A$ obtaining $f_a$ the obvious way.) 
The map $g:\A\to \bold P_{a\in \A\sim \{0\}}\C_a$ defined via $x\mapsto (g_a(x): a\in \A\sim\{0\}) (x\in \A)$ is as required. 
In case $\A$ is simple, then by properties of $\sf covK$, $S\sim (\bold H\cup \bold G)$ is non--empty,  so
if $F\in S\sim (\bold H\cup \bold G)$, then $F$ induces a non--zero homomorphism $f$ with domain $\A$ into a $\Cs_n$ 
omitting the given types. By simplicity of $\A$, $f$ is injective.

To prove independence, it suffices to show that $\sf cov K$ many types may not be omitted because it is consistent that $\sf cov K<2^{\omega}$. 
Fix $2<n<\omega$. Let $T$ be a countable theory such that for this given $n$, in $S_n(T)$, the Stone space of $n$--types,  the isolated points are not dense.
It is not hard to find such theories. One such (simple) theory is the following: 

Let $(R_i:i\in \omega)$ be a countable family of unary relations and for each disjoint and finite subsets
 $J,I\subseteq \omega$, let $\phi_{I,J}$ be the formula expressing `there exists $v$ such that $R_i(v)$ holds for all $i\in I$ and $\neg R_j(v)$ holds for all
$j\in J$. Let $T$ be the following countable theory $\{\phi_{I, J}: I, J\text { as above }\}.$  
Using a simple compactness argument one can show that 
$T$ is consistent. Furthermore, for each $n\in \omega$, $S_n(T)$ does not have isolated types at all, 
hence of course the isolated types are not dense in $S_n(T)$ for all $n$. 
Algebraically, this means that if $\A=\Fm_T$, then for all $n\in \omega$, $\Nr_n\A$ is atomless.
Another example, is the  theory of random graphs.

This condition excludes the existence of a prime model for $T$ because $T$ has a prime model $\iff$ the isolated points in 
$S_n(T)$ are dense for all $n$. A prime model which in this context is an atomic model,  omits any family of non--principal types (see the proof of the last item).
We do not want this to happen.

Using exactly the same argument in \cite[Theorem 2.2(2)]{CF}, one can construct 
a family $P$ of non--principal $0$--types (having no free variable) of  $T$,
such that $|P|={\sf covK}$ and $P$ cannot be omitted.
Let $\A=\Fm_T$ and for $p\in P$, let $X_p=\{\phi/T:\phi\in p\}$. Then $X_p\subseteq \Nr_n\A$, and $\prod X_p=0$,
because $\Nr_n\A$ is a complete subalgebra of $\A$.

Then we claim that for any $0\neq a$, there is no set algebra $\C$ with countable base
and $g:\A\to \C$ such that $g(a)\neq 0$ and $\bigcap_{x\in X_p}f(x)=\emptyset$.
To see why, let $\B=\Nr_n\A$. Let $a\neq 0$. Assume for contradiction, that there exists
$f:\B\to \D'$, such that $f(a)\neq 0$ and $\bigcap_{x\in X_i} f(x)=\emptyset$. We can assume that
$B$ generates $\A$ and that $\D'=\Nr_n\D$, where $\D\in \Lf_{\omega}$.
Let $g=\Sg^{\A\times \D}f$.
By lemma \ref{lift} $g$ is a homomorphism with $\dom(g)=\A$ and $g$ omits $P$. 
This contradicts that  $P$, by its construction,  cannot be omitted.
Assuming Martin's axiom, 
we get $\sf cov K=\mathfrak{p}=2^{\omega}$; together with the above arguments this proves (4).

We prove (5). Let $\A=\Nr_n\D$, $\D\in {}\sf Lf_{\omega}$ is countable. Let $\lambda<\mathfrak{p}.$ Let $\bold X=(X_i: i<\lambda)$ be as in the hypothesis.
Let $T$ be the corresponding first order theory, so
that $\D\cong \Fm_T$. Let $\bold X'=(\Gamma_i: i<\lambda)$ be the family of non--principal types in $T$ corresponding to $\bold X$. 
If $\bold X'$ is not omitted, then there is a (countable) realizing tree
for $T$, hence there is a realizing tree for a countable subfamily of $\bold X'$ in the sense of \cite[Definition 3.1]{CF}, 
hence a countable subfamily of $\bold X'$ 
cannot be omitted. Let $\bold X_{\omega}\subseteq \bold X$ be the corresponding countable 
subset of $\bold X$. Assume that $\bold X_{\omega}$ can be omitted in a $\sf Gs_n$, via $f$ say. 
Then  again by lemma \ref{lift},  $f$ can be lifted (like above) to $\Fm_T$ omitting $\bold X'$, 
which is a contradiction.  We leave the part when $\A$ is simple to the reader. 

We finally prove (6): If $\A\in \bold S_n\Nr_n\CA_{\omega}$, is atomic and has countably many atoms, 
then any complete representation of $\A$, equivalently, an atomic representation of $\A$, equivalently, a representation of $\A$ 
omitting the set of co--atoms,  which exists by theorem \ref{complete}, is as required. 
If $\A$ is simple and completely representable, 
then it is completely represented  by a $\sf Cs_n$. 
\end{proof}
Using the reasoning in the first items in both theorems \ref{complete} and \ref{i}, the following omitting types theorem can now be easily proved.
This theorem is stronger than the $\sf OTT$ proved in \cite{Sayed} because $\sf Dc_{\alpha}\subset \bold S_c\Nr_{\alpha}\CA_{\alpha+\omega}$
and the strictness of the inclusion can be witnessed on countable algebras.

For example the cylindric reduct of the (countable) $\alpha$-dimensional algebra $\B\in \sf RQEA_{\alpha}$ (taking $\F=\mathbb{Q}$, say) 
used in the first item of theorem \ref{SL} is in $\bold S_c\Nr_{\alpha}\CA_{\alpha+\omega}$, but it is not in 
${\sf Dc}_{\alpha}$, because if $s\in {}^{\alpha}\mathbb{Q}^{(\bold 0)}$, 
and  $x=\{s\}\in \B$, then $\Delta x=\alpha$.

However, the proof is the same, for one works with a dimension complemented countable dilation $\D\in \Dc_{\alpha+\omega}$ 
of $\A$, as in the hypothesis, such that $\A\subseteq_c \Nr_{\alpha}\D$ and $\A$ generates $\D$. Then using \cite[Theorem 3.2.4]{Sayed}, one constructs a 
representation of $\D$ omitting the given  $< \mathfrak{p}$ (still) non--principal types (by the condition $\subseteq_c)$; its restriction to $\A$ 
the obvious way, gives the required representation omitting the $< \mathfrak{p}$ non-principal 
types of $\A$.

\begin{theorem}\label{w} Let $\alpha$ be a countable infinite ordinal. Let $\A\in \bold S_c\Nr_{\alpha}\CA_{\alpha+\omega}$ be countable.
Let $\lambda < \mathfrak{p}$ be a cardinal, and let $\bold X=(X_ i<\lambda)$ be a family of non-principal 
types of $\A$. Then there exists $\C\in \sf Gws_{\alpha}$ that omits $\bold X$ via an embedding $f:\A\to \C$. 
If $\A$ is simple, then $< {\sf covK}$ many non-principal 
types can be omitted.
\end{theorem}
The last two items in theorem \ref{main}, tell us that for any ordinal $\alpha\geq \omega$, 
there are uncountable atomic $\RCA_{\alpha}$s that are not completely representable (by generalized weak set algebras); 
the last item constructs such an algebra
$\A\in \Nr_{\alpha}\CA_{\alpha+\omega}\subseteq \bold S_c\Nr_{\alpha}\CA_{\alpha+\omega}$.
Using the usual argument, we get that 
the non--principal type consisting of co--atoms of $\A$ cannot be omitted, hence the previous theorem \ref{w} 
does not generalize to the uncountable case. But there is a missing (natural) part:\\

{\bf Question:} Can we replace $\bold S_c$ 
by $\bold S$ in theorem \ref{w}? Perhaps a more tangible instance of the question would be: 
Is there a countable atomic $\A\in \sf RCA_{\omega}$ that is {\it not} completely representable by generalized weak set algebras of dimension $\omega$? \\
Although Hirsch and Hodkinson \cite{HH} showed that the class of completely representable $\RCA_{\omega}$s is {\it not } elementary, and even more 
they explicitly constructed 
a countable atomic $\B\in \RCA_{\omega}$ with no complete representation, but {\it complete representability was excluded with respect to $\sf Gs_{\omega}$s}. 
In principal such a $\B$ can be completely represented in a {\it relativized sense by a $\sf Gws_{\omega}$.} 
The proof in \cite{HH} that $\B$ lacks a `square' complete representation does not kill this possibility.\\

Our next theorem holds for any $L_n(2<n<\omega)$ theory $T,$
that has quantifier elimination \cite[Theorem 3.3.10]{Sayed}, because if $T$ is such,  then $\Fm_T$  will be
in $\Nr_n\CA_{\omega}$. 

The maximality condition stipulated in the hypothesis (in the next theorem) cannot be omitted, else if $\lambda=\omega$,
we get an independent statement, witness \cite[Theorem 3.2.8]{Sayed} and for $\lambda=2^{\omega}$ we get a 
false one by the last item 
of theorem \ref{main}. To see why,  note that if $\B\in \Nr_n\CA_{\omega}$ is atomic and has no complete representation ($\B$ will necesarily have uncountably many atoms by theorem \ref{complete}), 
then viewed as an uncountable $L_n$ theory, 
the non--principal type consisting of co--atoms, formed as in the proof of theorem \ref{OTT} {\it which cannot be a Boolean  ultrafilter} 
as shown next,   cannot be  omitted.  For any infinite cardinal $\kappa$, such a $\B$ of cardinality $2^{\kappa}$  
was given in the last item 
of theorem \ref{main}.

\begin{theorem}\label{Shelah} Assume that $2<n<\omega$, that $\A\in \Nr_n\CA_{\omega}$ has cardinality $\lambda$, $\lambda$ a regular cardinal, and that $0\neq a\in \A$.
Let  $\kappa< {}2^{\lambda}$,
and let $\bold G=(F_i: i<\kappa)$ be a system of non--principal ultrafilters of $\A$.
Then there exists a set algebra $\C$ with base $U$, such that $|U|\leq \lambda$, and a homomorphism 
$f:\A\to \C$, such that $f(a)\neq 0$ and $\bigcap f(F_i)=\emptyset$ 
for each $i<\lambda$.
\end{theorem}
\begin{proof}
The proof is a three folded restricted of a result of Shelah proved for uncountable first order theories. The first restriction is to the countable case, 
the second is to $L_n$ theories, and the third is to $L_n$ theories that have quantifier elimination:

Shelah \cite{Shelah} proved the following:
Suppose that $T$ is a first order theory,
$|T|\leq \lambda$ and $\phi$ is a formula consistent with $T$,  then there exist models $\M_i: i<{}^{\lambda}2$, each of cardinality $\lambda$,
such that $\phi$ is satisfiable in each,  and if $i(1)\neq i(2)$, $\bar{a}_{i(l)}\in M_{i(l)}$, $l=1,2,$, $\tp(\bar{a}_{l(1)})=\tp(\bar{a}_{l(2)})$,
then there are $p_i\subseteq \tp(\bar{a}_{l(i)}),$ $|p_i|<\lambda$ and $p_i\vdash \tp(\bar{a}_ {l(i)})$ ($\tp(\bar{a})$ denotes the complete type realized by
the tuple $\bar{a}$) \cite[Theorem 5.16, Chapter IV]{Shelah}.

We consider only the case $\lambda=\omega$.
Now assume that  $\A=\Nr_n\B$, where $\B\in \sf Lf_{\omega}$ is countable, and let $(F_i: i<\kappa)$
with $\kappa<{}2^{\omega}$ be the given family of non--principal ultrafilters.
Then  we can assume that $\B=\Fm_T$ for some countable consistent theory $T$ and that $a=\phi_T$ ($\phi$ a formula consistent with $T$).
Let $\Gamma_i=\{\phi/T: \phi\in F_i\}$.
Let ${\bold F}=(\Gamma_j: j<\kappa)$ be the corresponding set of types in $T$,
they are non--principal and complete $n$--ary types in $T$.

Let $(\M_i: i<2^{\omega})$ be the constructed set of countable
models for $T$ that overlap only on principal maximal  types in which $\phi$ is satisfiable. We know that such models exist
by Shelah's aforementioned result.

Assume for contradiction that for all $i<2^{\omega}$, there exists
$\Gamma\in \bold F$, such that $\Gamma$ is realized in $\M_i$.
Let $\psi:{}2^{\omega}\to \wp(\bold F)$,
be defined by
$\psi(i)=\{F\in \bold F: F \text { is realized in  }\M_i\}$.  Then for all $i<2^{\omega}$,
 $\psi(i)\neq \emptyset$.
Furthermore, for $i\neq j$, $\psi(i)\cap \psi(j)=\emptyset,$ for if $F\in \psi(i)\cap \psi(j)$ then it will be realized in
$\M_i$ and $\M_j$, and so it will be principal.

This implies that $|\bold F|=2^{\omega}$ which is impossible. Hence we obtain a model $\M\models T$ omitting $\bold F$
in which $\phi$ is satisfiable. The map $f$ defined from $\A=\Fm_T$ to ${\sf Cs}_n^{\M}$ (the set algebra based on $\M$ \cite[Definition 4.3.4]{HMT2})
via  $\phi_T\mapsto \phi^{\M},$ where the latter is the set of $n$--ary assignments in
$\M$ satisfying $\phi$, omits $\bold G$.
\end{proof}

{\bf Question:} Can we replace $\Nr_n\CA_{\omega}$ by the larger $\bold S_c\Nr_n\CA_{\omega}$ in theorem \ref{Shelah}? \\
The problem here is that if $\A\subseteq_c \Nr_{\alpha}\B$, $\B\in \sf Lf_{\omega}$ and if $F$ is a non--principal ultrafilter of $\A$, then there is no guarantee that the 
filter generated by $F$  in $\Nr_n\B$ remains maximal and non--principal. It can be neither.

\end{document}